\documentclass[10pt, twoside,reqno]{amsart}

\usepackage{amsmath,amsthm,amssymb,amsfonts,epsfig,color,graphicx,enumerate,enumitem,accents}
\usepackage{rotating,multirow,color,enumerate,url,mathrsfs,mathrsfs}
\usepackage{arydshln}
\usepackage{bbm}
\usepackage{cases}
\usepackage{esint}
\usepackage{gensymb}
\usepackage{epstopdf}
\usepackage{comment}
\usepackage{leftidx}
\usepackage{graphicx}
\usepackage{mathtools}
\usepackage{enumerate}
\usepackage{subfig}
\usepackage{amsaddr}
\newtheorem{theorem}{Theorem}[section]
\newtheorem{corollary}[theorem]{Corollary}
\newtheorem{lemma}[theorem]{Lemma}
\newtheorem{proposition}[theorem]{Proposition}
\newtheorem{conjecture}[theorem]{Conjecture}

\theoremstyle{definition}

\newtheorem{remark}[theorem]{Remark}
\newtheorem{example}[theorem]{Example}
\newtheorem*{remark*}{Remark}

\numberwithin{equation}{section}

\newcommand{\nlOper}{{\mathcal{A}}}
\newcommand{\Ctwo}{{\mathscr{C}}}
\newcommand{\Jtwo}{{\mathcal{J}}}
\newcommand{\Atwo}{{\mathscr{A}}}
\newcommand{\Obsq}{{\Ob \times \Ob \backslash \Delta}}

\newcommand{\nlComp}{{\Comp_{\mathrm{nl}}}}

\newcommand{\med}{\medskip\noindent}
\newcommand{\e}{\varepsilon}
\def\d{\delta}
\def\a{\alpha}

\def\f{\varphi}
\def\la{\lambda}
\newcommand{\weak}{\stackrel{*}{\rightharpoonup}}

\newcommand{\R}{\mathbb{R}}
\newcommand{\N}{\mathbb{N}}
\newcommand{\NN}{\mathcal{N}}

\newcommand{\Rd}{\R^d}

\newcommand{\ident}{{\mathrm{id}}}
\newcommand{\Ident}{{\mathrm{Id}}}
\newcommand{\Comp}{{\mathcal{C}}}

\newcommand{\ind}{\mathbf{\chi}}
\newcommand{\ov}{\overline}

\newcommand{\A}{\mathcal{A}}
\newcommand{\B}{\mathcal{B}}
\newcommand{\K}{\mathcal{K}}
\newcommand{\LL}{\mathcal{L}}

\newcommand{\weakstar}{{\overset{\ast}{\rightharpoonup}}}
\newcommand{\mbf}[1]{{\mathbf{#1}}}
\newcommand \res{\mathop{\hbox{\vrule height 7pt width .5pt depth 0pt
			\vrule height .5pt width 6pt depth 0pt}}\nolimits}

\newcommand{\opnorm}[1]{{\left\vert\kern-0.25ex\left\vert\kern-0.25ex\left\vert #1 
		\right\vert\kern-0.25ex\right\vert\kern-0.25ex\right\vert}}

\def\O{\Omega}
\newcommand{\Ob}{{\overline{\Omega}}}
\newcommand{\bO}{{\partial\Omega}}

\newcommand{\argu}{{\,\cdot\,}}

\newcommand{\Mes}{{\mathcal{M}}}

\newcommand{\Sddp}{{\mathcal{S}^{d\times d}_+}}

\newcommand{\DIV}{{\mathrm{Div}}}
\newcommand{\dive}{{\mathrm{div}}}

\newcommand{\sig}{{\sigma}}

\newcommand{\tr}{{\mathrm{Tr}}}

\newcommand{\PP}{\mathcal{P}}

\newcommand{\pairing}[1]{{\left \langle #1 \right \rangle}}
\newcommand{\norm}[1]{\Arrowvert #1 \Arrowvert}
\newcommand{\abs}[1]{{\left \lvert #1 \right \rvert}}
\newcommand{\Lip}{{\mathrm{Lip}}}

\newcommand{\eps}{\varepsilon}

\newcommand{\D}{\mathcal{D}}
\newcommand{\C}{{\mathrm{C}}}

\newcommand{\U}{{\mathcal{U}}}

\newcommand{\Ha}{\mathcal{H}}
\newcommand{\IM}{\mathrm{Im}}

\newcommand{\Sdd}{{\mathcal{S}^{d \times d}}}

\newcommand{\mres}{\mathbin{\vrule height 1.6ex depth 0pt width
		0.13ex\vrule height 0.13ex depth 0pt width 1.3ex}}

\newcommand{\one}{{{\bf 1}
		\kern-0,28em \rm l}}

\DeclareMathOperator{\spt}{spt}
\DeclareMathOperator{\rk}{rank}
\newcommand{\bK}{\overline{\mathcal{K}}}

\DeclareMathOperator{\co}{co}
\DeclareMathOperator{\cobar}{\overline{co}}

\def\dis{\displaystyle}

\begin{document}

	\title{Optimal design versus maximal Monge-Kantorovich metrics}

	

	\author{Karol Bo{\l}botowski \and   Guy Bouchitt\'{e}}


	\date{\today}
	
	\begin{abstract} A remarkable connection between optimal design and Monge transport was initiated in the years 1997 
	in the context of the minimal elastic compliance problem and where the euclidean metric cost was naturally involved. In this paper we present different variants in optimal design of mechanical structures, in particular focusing on the optimal pre-stressed elastic membrane problem.  We show that the underlying metric cost is associated with an unknown maximal monotone map which maximizes the Monge-Kantorovich distance between two measures.
 In parallel with the classical duality theory leading to existence and (in a smooth case) to PDE optimality conditions,
 we present a  general geometrical approach arising from a two-point scheme in which  geodesics with respect to the optimal metric
 play a central role.  
 These two  aspects are enlightened by several explicit examples and also by numerical solutions in which optimal structures very  often turn out to be truss-like 
i.e supported by piecewise affine geodesics. In case of a discrete load, we are able to relate
 the existence of such truss-like solutions  to an extension property of maximal monotone maps which is of independent interest and that we propose here as a conjecture.

\end{abstract}

\maketitle



	\dedicatory{}
	
	
	\maketitle

	\vskip1cm
	
	
	\section{introduction}
	\label{sec:introduction}
	The analysis of the behaviour of elastic structures has always been a central
	problem in Mathematics and in Engineering. In the last decades, the optimal design of such structures took benefit  of the dramatic improvement of the powerful tools of
	calculus of variations and geometric measure theory which have been developed meanwhile. Among them homogenization and $\Gamma$-convergence techniques allowed  decisive breakdowns
	as for instance the emergence of  topological optimization methods \cite{allaire-book} which are now very popular  in civil and mechanical engineering.

For the first time, in the year 1997, a remarkable connection between optimal design and Monge transport problem was discovered. 
	It concerned specifically the classical optimal compliance problem in the  scalar case \cite{bouchitte1997}  that may be related to designing a heat conductor and then it was  developped further in the framework of elasticity \cite{bouchitte2001}. This new approach turned out to be very fruitful as it was  possible to consider concentrated loads $f$ and
	low dimensional structures as competitors, in particular trusses of bars as they appear in Michell problem \cite{bouchitte2008} in case of a vanishing volume fraction limit of available elastic material.
From the Monge-Kantorovich optimal transport point of view, the main limitation of this approach is that the underlying cost is always related
to the Euclidean metric on the ambient space $\R^d$ ($d=2$ or $d=3$ in practice). 
 
 In this work we will bring to the fore a new family of optimal design problems 
 for which the Monge-Kantorovich approach mentioned above can be used but needs to be adapted with a major modification: the transport cost is
 now an unknown that will be determined  by solving
a  maximization problem in a suitable class of admissible  metric costs on $\R^d$ (including the Euclidean one). 
This family includes the two-dimensional problem of designing a membrane subject to an out-of-plane load $f$.  Two approaches can be distinguished:
 
 \med
(A)  \quad	 The \textit{optimal elastic membrane problem} which relies on a non-linear model inspired by the von K\'{a}rm\'{a}n's plate theory (see e.g. \cite{ciarlet1980}, Section 6 in \cite{fox1993}, Section II.4 in \cite{lewinski2000}); the design variable is a mass distribution $\mu$ for which the in-plane stress field $\sigma$ is an effect of elastic response to the loads; 
	
\med	
	(B)\quad  The simplified \textit{optimal pre-stressed membrane problem} which relies on the classical linear model (see e.g. Section IV.10.3 in \cite{courant1989}) where  the design variable is the (non-negative) in-plane stress field $\sigma$ which is subject to the equilibrium constraint.

Note that in both cases the transverse stiffness of the membrane does not depend  directly on the constitutive law of the underlying elastic material but of the in-plane stress field $\sigma$  which point-wisely is a positive semi-definite symmetric tensor.
It turns out that the two problems above are in some sense  equivalent to each other and
surprisingly also to another 3D design problem that fits into a class of long standing engineering problems of \textit{form finding}:

\med
	(C)\quad The \textit{optimal vault problem} where over a horizontal 2D reference domain one is to find a surface $z = z(x_1,x_2)$ and elastic material's distribution on this surface. The shell/vault thus constructed ought to carry the vertical load $f$ by means of compression only; vertical position of the load is a design variable, i.e. $f$ tracks the shape $z$ of the vault.
	
		\medskip
The  equivalence between (C) and (B) is established in a recent work of the first author \cite{bolbotowski2021} where (C) appears as a generalization of the Rozvany-Prager optimal arch-grid problem \cite{rozvany1979}, \cite{czubacki2020} (the latter one being recovered when the vault must compose of two mutually orthogonal families of arches). 

\medskip
A detailed presentation of the different optimal design problems sketched above will be skipped for the sake of conciseness but
we ask to keep in mind that they all lead precisely to the same mathematical framework, up to changing the mechanical interpretation of the paramaters coming into play.
In the present paper we will focus on model (B) namely {\em the optimal pre-stressed membrane problem}.
	The motivation for choosing  this variant is that it can be obtained through a simple modification of the classical optimal compliance model in the spirit of \cite{bouchitte2001, bouchitte1997} and then 
	we may describe step by step how we pass from the Monge OT approach involving the Euclidean cost to a model requiring the identification of an optimal metric cost.

%
%

	
	\medskip
	In a first step let us briefly describe the scalar mass optimization
	and its variant we call {\em free material design problem}  (see \cite{bolbotowski2020} for the vectorial variant).
	Given a bounded domain $\Omega\subset \R^d$ (design region),  $\Sigma_0$ being a compact subset of $\Ob$ and a bounded signed Radon measure
	 $f \in \Mes(\Ob;\R)$, we consider an unknown 
	mass distribution $\mu \in \Mes_+(\Ob)$ of material and define the compliance to be
	$$\mathcal{E}_{\Omega,f,\Sigma_0}(\mu)= \mathcal{E}  (\mu) = 
	\sup \left\{\int u\, df - \frac{1}{2} \int \abs{\nabla u}^2 d\mu \ : \ u\in \D(\R^d) \ ,\ u=0 \ \text{on $\Sigma_0$} \right\}.$$
	The mass optimization problem (MOP) with respect to a given amount of mass $m$ reads as follows
	\begin{equation} \tag(MOP)\label{MOP}  \beta(m):= \inf \left\{ \Comp(\mu) :  \mu \in \Mes_+(\Ob),\, \int d\mu \leq m \right\}
	\end{equation}
	If $\Sigma_0$ is void (pure Neumann problem) we need to assume that $f$ is balanced, i.e. $\int f^+= \int f^-$.
	Then if $\Omega$ is convex it is shown in \cite{bouchitte2001} that
	$ \dis\beta(m)=  \frac { (W_1(f^+,f^-))^2}{2 m}\ ,$
	where 
	\begin{equation} \label{MK}
	W_1(f^+,f^-)= \inf \left\{\int_{\Ob\times\Ob} |x-y| \,\gamma(dxdy)\, : \, \gamma \in \Gamma(f^+,f^-) \right\}  
	\end{equation}
	denotes the Monge-Kantorovich distance and $\gamma \in \Gamma(f^+,f^-)$ means that $\gamma$ as a measure 
	on $\Mes_+\bigl(\Ob^2\bigr)$ (transport plan) admits  $f^+,f^-$ as marginals.
	Notice that, if $\Omega$ is non-convex, we simply need to substitute  $|x-y|$ with the geodesic   distance in $\Omega$  between $x$ and $y$.
	In case $\Sigma_0$ is a non-empty compact subset of $\partial\Omega$ and $f$ is a non-negative measure (that is $f^-=0$), the latter formula for the infimum
	of $\mathrm{(MOP)}$ can be recast from the Monge distance of $f$ to $\Sigma_0$ namely (see \cite{bouchitte1997})
	\begin{equation} \label{magic}
	\beta(m) :=\inf(\mathrm{MOP})  = \frac { \bigl(W_1(f,\Sigma_0)\bigr)^2}{2 m} ,
	\end{equation}
	where
	\begin{equation}\label{def:W1-Sigma}
		W_1(f, \Sigma_0):= \min \Big\{ W_1(f, \nu)\,: \, \nu \in \Mes_+(\Sigma_0)\Big\}= \int \mathrm{dist} (x,\Sigma_0) \, f(dx). 
	\end{equation}
	Then a geometric characterization of optimal $\mu$ can be deduced from the geodesics transport rays connecting points in the support of $f$ to $\Sigma_0$.
	Let us mention that the case where $\Sigma_0=\partial \Omega$ is classical in sandpile models (see for instance \cite{cardaliaguet}).   
	
	Next, an anisotropic   generalization of the $\mathrm{(MOP)}$ problem can be considered in which, instead of looking at optimal mass distributions, we search for optimal  conductivity tensor field $\sigma\in  \Mes(\Ob;\Sddp)$ where
	$\Sddp$ is the set of positive semi-definite symmetric tensors:
	\begin{equation}\tag (FMD)
	\label{FMD}
	\inf \biggl\{ \Comp(\sigma) \,:\, \sigma \in \Mes(\Ob;\Sddp),  \  \frac1{d} \int \tr\, \sigma \leq m \biggr\}
	\end{equation} 
	where, for a prescribed Dirichlet region $\Sigma_0$, the compliance reads
	\begin{equation}
	\label{eq:comp_FMD}
	\Comp(\sigma) = \sup \left\{ \int u \, df - \frac{1}{2} \int \pairing{\sigma, \nabla u \otimes \nabla u} \ : \  u\in \D(\R^d) ,\ u=0 \ \text{on $\Sigma_0$}  \right\}.
	\end{equation}
	With the notations above, we see that for $\sigma= \Ident \, \mu$ it holds  $ \Comp(\sigma) =\mathcal{E}  (\mu)$ and 
	the choice of the constraint to be the integral of the trace of $\sigma$, although debatable, intends to be the natural counterpart of the mass constraint in (MOP). Note that in the vector case of linear elasticity, the design variable $\sigma$ is rather a fourth order tensor (inducing Hooke's law) and  the related version of (FMD) is more involved (see the recent work \cite{bolbotowski2020}).
	Nevertheless, as will be seen later, the (FMD) variant is very close fo the initial (MOP) problem and optimal solutions $\sigma$ can be recovered in the same way by selecting geodesics with respect to the Euclidean metric.
	
	\medskip
	We may now readily pass to the {\it optimal pre-stressed membrane model}. As in the (FMD) problem, the unknown design variable is an element of $\Mes(\Ob;\Sddp)$ whereas now the tensor measure $\sigma$ represents the in-plane stress in the membrane occupying a plane horizontal domain $\Omega$. The positivity condition imposed on $\sigma$ rules out compressive stress, which is reasonable when the membrane is very thin and thus perfectly immune to buckling. We assume that the membrane is subject to vertical pressure $f\in \Mes(\Ob;\R)$ and to an in-plane load exerted on the boundary only. The latter forces depend on the designer and play the role of a \textit{pre-load} that generates the \textit{pre-stress} $\sigma$ in the whole domain which in turn provides stiffness against $f$. Since in the interior of the design region only the out-of-plane component of load is non-zero the in-plane equilibrium requires that $\DIV\,\sig = 0$ in the distributional sense in $\Omega$. Virtually, it is the {\em divergence free condition} on $\sigma$ that converts the (FMD) problem for heat conductor to the optimal design model for a pre-stressed membrane:
	\begin{equation} \tag(OM)
	\label{OM}
	\inf \biggl\{ \Comp(\sigma) \,:\, \sigma \in \Mes(\Ob;\Sddp), \, \DIV \sigma= 0 \ \text{in $\Omega$}, \ \frac1{d}\int \tr\, \sigma \leq m \biggr\},
	\end{equation} 
	where the load $f$ enters through $\Comp(\sigma)$ (defined in \eqref{eq:comp_FMD}) together with $\Sigma_0$ being the part of the boundary where the membrane is pinned in the vertical direction; the function $u$ represents deflection (the out-of-plane displacement) of the membrane. 

%
%
	
\medskip
	A first observation is that the  problems (FMD)  and (OM) do not share the same infimum in general. This is due to the divergence constraint which rules out many possible competitors. In particular:
	\begin{itemize}
		\item  An isotropic tensor field of the kind $\sigma= a(x)\, \Ident\
		 \mathcal{L}^2\mres\Omega$ is not admissible unless $a$ is constant.
		\item Let  $\sigma = p(x)\, \tau(x)\otimes \tau(x)\, \Ha^1 \mres C$ with $C$ being a simple curve, $p(x)$ a positive weight  and $\tau(x)$  a unit vector vector. Then $\sigma$  is admissible iff $C$ is a straight line connecting two points of $\partial\Omega$ while  $\tau$ is a constant vector parallel to $C$ and  $p$ is constant.
	\end{itemize}
	In fact as will be seen later, pre-stress tensor fields supported  by networks of bars or, as they should be called within the membrane model, \textit{strings} are favoured in the (OM) problem. To give a flavour of the geometry of solutions, we illustrate in Figure \ref{fig:intro} below the optimal configurations for (FMD) and (OM)  in the case where $f$ is a single Dirac pressure exerted at a point of a square membrane which is pinned along all its boundary, i.e. $\Sigma_0=\partial \Omega$. For (OM) problem the support of optimal measure $\sigma$ is described by a finite union of strings of different thickness that are tied at the loaded point. Such lower dimensional solutions shall be referred to as \textit{trusses} or \textit{truss structures}. In Figure \ref{fig:intro} the arrows indicate the direction of the gradient flow of the deflection function $u$.

	
%
%
%
	\begin{figure}[h]
		\centering
		\subfloat[]{\includegraphics*[trim={0cm 0cm -0cm -0cm},clip,width=0.2\textwidth]{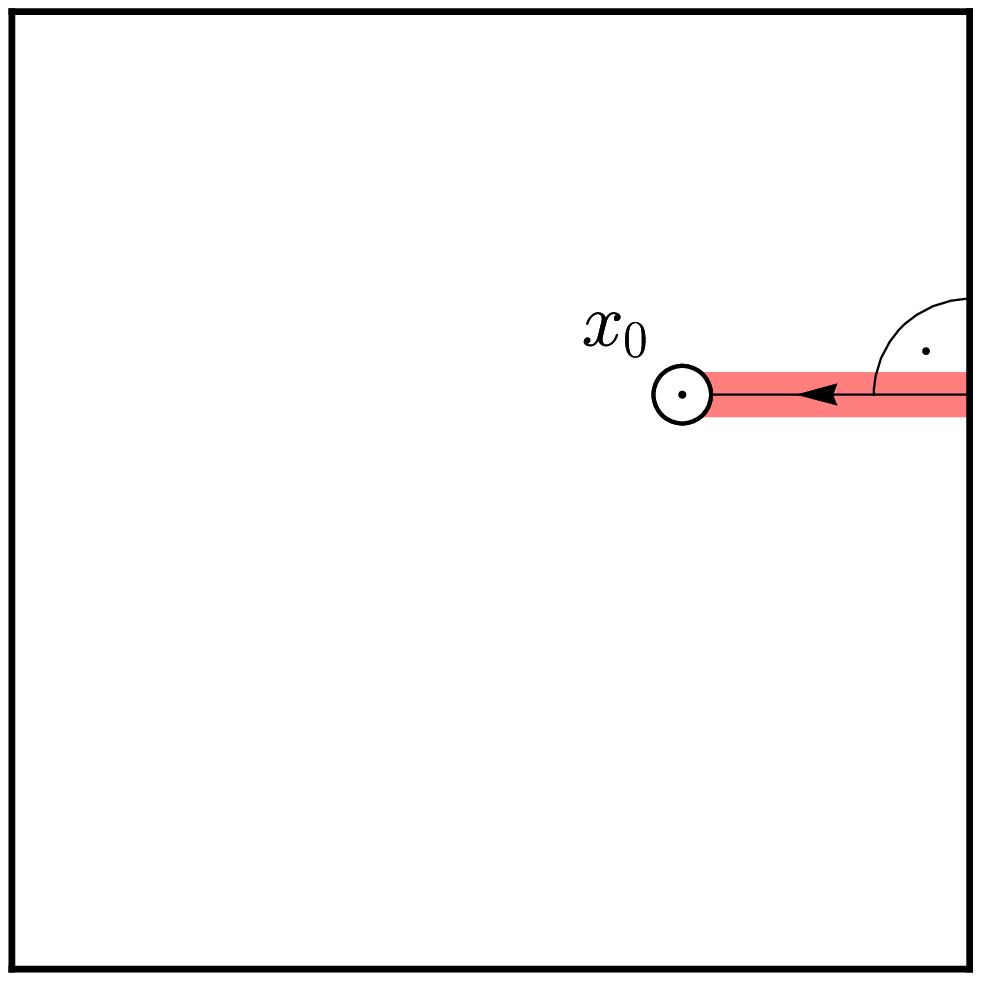}}\hspace{3cm}
		\subfloat[]{\includegraphics*[trim={0cm 0cm -0cm -0cm},clip,width=0.2\textwidth]{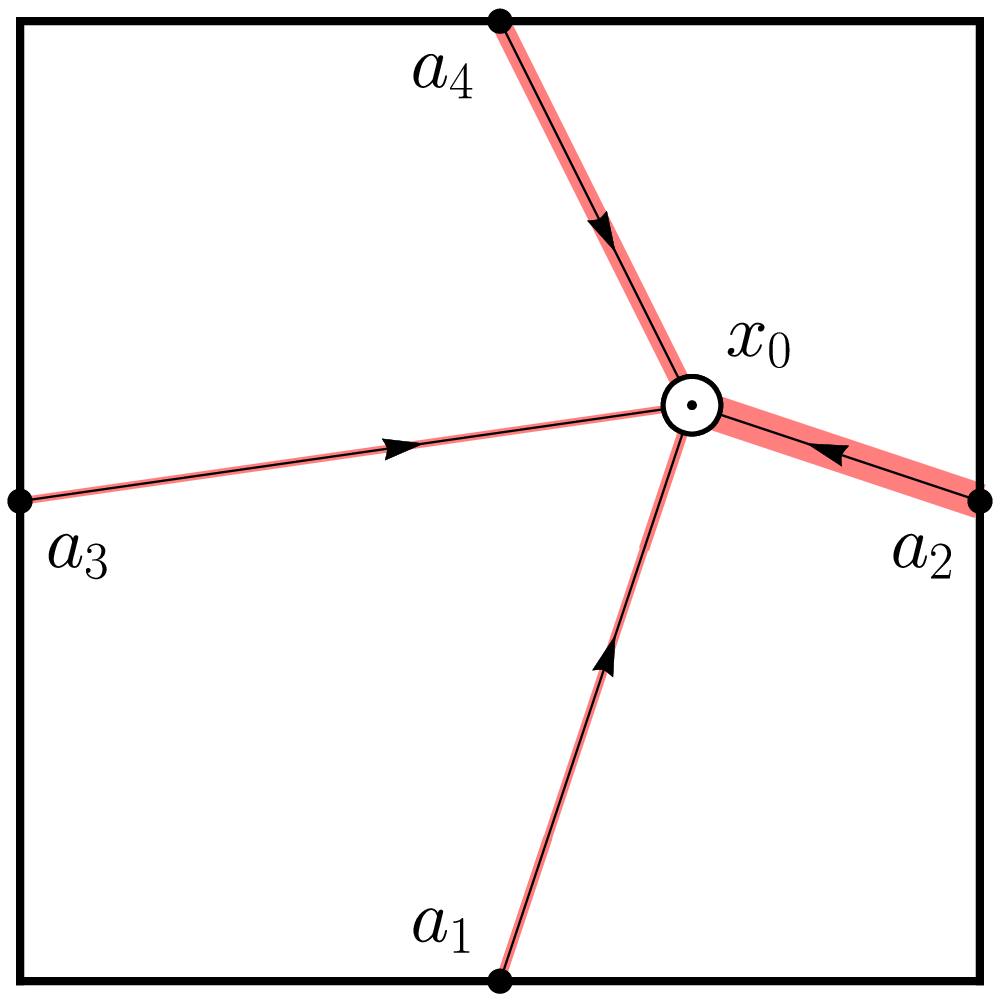}}
		\caption{An optimal solution $\sigma$ in the case of a point-source $f= \delta_{x_0}$ and of a square domain $\Omega$: (a) for the (FMD) problem; (b) for the (OM) problem. $\Omega$. The points $a_i$ are centres of the square's sides.}
		\label{fig:intro}       
	\end{figure}

	
	An unexpected discovery we wish to promote in this paper is that the (OM)  problem has a very deep relation with another  interesting issue in geometry and  optimal transport theory: the search of optimal metrics in a suitable class which maximize the associated Monge distance between two measures. 
	More precisely, if we consider the membrane problem for $\Sigma_0=\partial\Omega$ and $f\in\Mes_+(\O)$,
	then it holds that 
	$ \min (\mathrm{OM}) = \frac{Z_0^2}{2 m_0}$
	where $Z_0= Z_0(f, \Omega)$ is given by
	\begin{equation}\label{optmetric}
	Z_0= \sup \Big\{ W_{c_v} (f,\partial\Omega) \ : \ v\in C^{\infty}(\R^2;\R^2), \  e(v)\ge 0,\ v=\ident 
	\ \text{in $\R^2\setminus\Ob$} \Big\},
	\end{equation}
	where $e(v)$ is the symmetrized gradient of $v$, $c_v$ is the geodesic distance associated with the metric tensor $e(v)$ and  $W_{c_v}$
	stands for the Monge distance related to the cost $c_v$.    
	Our results include the existence of an optimal {\em maximal monotone map} $v$ for a relaxed version of the right hand member of \eqref{optmetric}. 
	Note that the problem  of maximizing a geodesic distance among particular classes of {\em scalar} metrics has been considered in a different context  by several authors \cite{buttazzo2004}, \cite{Conti2011}.
	
	\bigskip
	
	The paper is organized as follows: 
	
	\medskip
	
	In Section \ref{sec:MKfree}, we revisit the link between the (FMD) problem and Monge-Kantorovich theory in the spirit of  \cite{bouchitte2001}. In particular we establish an equivalence between (MOP) and (FMD) and  we give a full description of the optimal measure $\sigma$ in terms of the transport rays connecting the support of $f$  to the boundary of $\partial\Omega$.
	In addition we show that the  strict inequality $\inf \mathrm{(FMD)} < \inf \mathrm{(OM)}$ holds unless the load $f$ is supported 
	on a geometrically identifiable compact subset of $\Omega$. 
	
	In Section \ref{sec:optimembrane}, we show the existence of an optimal  $\sigma\in \Mes(\Ob;\Sddp)$ for (OM) and 
	 we develop a primal-dual framework based on the introduction of an additional unknown  {\em horizontal} vector field  $w:\Omega \to  \R^2$ vanishing on $\partial\Omega$ whose symmetrized distributional gradient   will play the role of a Lagrange multiplier for the divergence free constraint. Accordingly we are led to a dual problem 
	\begin{equation} \label{constraint}
	Z_0:=\sup \left\{ \pairing{f,u}\ :\  e(w)+ \frac1{2}\, \nabla u \otimes \nabla u \le \Ident\right\}
	\end{equation}
	where   pairs $(u,w)$ are in duality  with measures $(\lambda, \sigma) \in \Mes(\Ob;\R^d) \times \Mes(\Ob;\Sdd)$  ($\lambda$ coresponds to the transverse internal force in the membrane caused by its deflection). Then, upon rewriting the constraint in \eqref{constraint}  in terms of an equivalent two points conditions, namely:
	\begin{equation} \label{two-points}
	\frac1{2}\,|u(y)-u(x)|^2 +  \pairing{w(y)- w(x),y-x} \le |x-y|^2 \qquad \forall (x,y) \in (\Ob)^2,
	\end{equation}
	we put forward an alternative duality scheme which fits perfectly to characterize truss-like optimal pairs $(\la,\sigma)$ in the sense that they are decomposable in the form $(\la_\pi,\sigma_\Pi)$ given in  \eqref{def:llpi}, \eqref{def:sigmaPi}. Here, as far as they exist, 
	 $\pi$ and $\Pi$ are measures on $\Ob\times \Ob$ which play the role of 
	Lagrange multipliers of the two-point constraint \eqref{two-points}.
	
	\medskip
	
	In Section \ref{opticond}, we give necessary and sufficient conditions of optimality for two pairs $(\la,\sigma)$ and $(u,w)$
	assuming that $(u,w)$ is Lipschitz regular.  These conditions are particularized in the case of a truss configuration $(\la,\sigma)= (\la_\pi,\sigma_\Pi)$.
	Then, examples of explicit optimal configurations are established in the radial case or for the load $f$ being a single Dirac mass
	(confirming in particular the optimality of the structure depicted in Figure \ref{fig:intro}(b)).

	\medskip
	In  Section \ref{sec:geodesic}, we exploit the two-point condition \eqref{two-points} to
	 establish a connection between the optimal membrane problem $\mathrm{(OM)}$ and 
	the search of a monotone map $v= \ident - w$ maximizing a Monge-Kantorovich distance as stated in \eqref{optmetric}.
	To that aim we begin with a preliminary subsection which could be considered of independent interest where  we define  the intrinsic pseudo-distance $c_v$ associated with a maximal monotone map $v: \R^d\to \R^d$ which agrees with the identity outside $\Ob$.
	Then we prove the existence of a maximal monotone  $v$ associated with the worst Monge-Kantorovich metric \eqref{optmetric}  and derive a saddle point characterization of 
	an optimal pair $(v,\gamma)$ where $\gamma$ is selected among the optimal transports plans solving
	$W_{c_v} (f, \partial\Omega)=\inf \bigl\{\int_{\Ob\times\Ob} c_v(x,y) \,\gamma(dxdy)\, :\, \gamma \in \Gamma(f,\nu),\  \nu \in \Mes_+(\partial\Omega)\bigr\}. $ 
	Next we give a general criterium of optimality for a truss solution and establish the existence of such a solution in case of a finitely supported load $f$ assuming an extension property for monotone maps that we conjecture to be true.

	\medskip
	In Section \ref{sec:numerics}, we present several numerical simulations for the (OM) problem taking for design subset $\Omega$ a squared domain in $\R^2$. It turns out that in most cases optimal stress measures $\sigma$ exhibit a truss structure. 	The numerical method is based on a duality scheme  which involves the two-point condition \eqref{two-points} restricted to a discrete subset of $\Ob \times \Ob$. It is worked out through a conic programming algorithm introduced recently in \cite{bolbotowski2021}. 
	
	Eventually we provide in the appendix several classical tools of convex analysis, a short survey about tangential calculus with respect to a measure
	and some useful approximation properties of convex functions of measures.	
	 	
	\medskip
	To conclude this introduction, let us point out that the existence issue for the coupling measures $(\pi,\Pi)$ allowing the truss representation of 
	a solution $(\la,\sigma)$ (see \eqref{def:llpi},\,\eqref{def:sigmaPi}) is not ensured in general for the infinite dimensional setting. However  we expect it to be true in the case of a finitely supported measure $f$ once the extension property  for monotone metrics conjectured in Section \ref{sec:geodesic} can be proved.

	\subsection*{Acknowledgements}  The first author is grateful to the Laboratoire IMATH, Universit\'{e} de Toulon for hosting his two-month visit there in the fall of 2019 where this research was initiated. He	would also like to thank the National Science Centre (Poland) for the financial support and acknowledge the Research Grant no 2019/33/B/ST8/00325 entitled "Merging the optimum design problems of structural topology and of the optimal choice of material characteristics. The theoretical foundations and numerical methods".
	
	%
	
	%
	%
\subsection*{Notations} Throughout the paper we will use the following notations: 

\begin{enumerate}[leftmargin=1.2\parindent]

\item [-] $\Omega$ denotes a  bounded  domain of $\Rd$ that in general we assume  to be convex; although our mechanical context requires $d = 2$, the mathematical arguments will often be valid for any natural \nolinebreak $d$; 

\item [-] $\Sigma_0$ will be a compact subset of $\bO$ on which a Dirichlet condition is prescribed;

\item [-] the Euclidean norm of $z\in\R^d$ is denoted by $|z|$; $S^{d-1}$ denores the unit sphere $\{|z|=1\}$;

\item[-] by $\Delta$ we denote the diagonal of $\R^d$, namely $\Delta = \bigl\{(x,x) : x \in \Rd \bigr\}$;

\item [-] by $\Sdd$ we shall see the space of $d\times d$ symmetric matrices, while $\Sddp$ will be its subset whose elements are positive semi-definite.
Given $A,B \in \Sdd$, we will write $A\le B$ if $B-A\in \Sddp$; $\tr A$ denotes the trace of $A$, $\rk A$ the rank of $A$;  $\Ident$ denotes the identity matrix while $\ident$ denotes the identity map on $\Rd$;

\item [-]  by $\ind_B$ we will denote the indicator function of the set $B$ taking value $0$ in $B$ and 
	$+\infty$ outside. Instead we denote by $\one_B$  the characteristic function of $B$ taking value $1$ in $B$ and $0$ outside;  

\item [-] if $A\subset \R^d$ is an open subset, $\D(A)$ denotes be  space of $C^\infty$ functions compactly supported in $A$; 
$\D(\Ob\setminus \Sigma_0)$ denotes the set of restrictions to $\Ob$ of elements in $\D(\R^d\setminus\Sigma_0)$;

\item [-]  $C^0(\Ob)$ denotes the Banach space of continuous functions on $\Ob$ while $C_{\Sigma_0}(\Ob)$ (resp. $C_0(\O)$) denotes the subset of $C^0(\Ob)$ consisting of functions  vanishing in $\Sigma_0$ (resp. in $\bO$);

\item [-] $\Lip(\O)$ (or $\Lip(\Ob)$) stands for the space of Lipschitz continuous functions on $\O$  (resp $\Ob$) while $\Lip_0(\O)$ (resp. $\Lip_{\Sigma_0}(\O)$) denotes 
the subspace of elements vanishing on  $\bO$ (resp. $\Sigma_0$);

\item [-] for $k>0$, $\Lip_k(\O)$ is the subset of $\Lip(\O)$ of functions $u$ such that $|u(x) - u(y)|\le k\; |x-y|$ for all $(x,y)$
(if $\O$ is convex, it coincides with $\{u\in W^{1,\infty}(\O): |\nabla u|\le k\ \text{a.e.}\}$);

\item [-] $\Mes_+(\Rd)$ denotes the space of Borel measures on $\R^d$ with values in $[0,+\infty]$. Unless  explicitely specified, we will additionaly assume that elements of $\Mes_+(\Rd)$ are finite on compact subsets; the topological support of $\mu \in \Mes_+(\Rd)$ is denoted  $\spt(\mu)$ 
while $\mu\res A$; represents its  trace on a Borel subset $A\subset \Rd$;   $\Mes_+(A)$ will the subset of elements  $\mu\in\Mes_+(\Rd)$ such that $\mu = \mu\res A$ (or such that $\spt(\mu) \subset A$ if $A$ is closed);

 \item [-] $\Mes(\Ob)$ (resp. $\Mes(\Ob;\R^d)$ or $\Mes(\Ob;\Sddp)$)
is the space of signed finite Radon measures on  $\Rd$ which are compactly supported in $\Ob$ (resp. Borel regular measures from $\Ob$ to $\R^d$ or $\Sddp$); given $\nu\in \Mes(\Ob;\R^d)$
and $\mu \in \Mes_+(\Ob)$, then $\nu\ll \mu$ means that $\nu =\zeta \mu$ for a suitable $\zeta\in L^1_\mu(\Ob;\R^d)$  whereas $\nu\perp \mu$ means that $\mu$ and $\nu$ are mutually singular;

 \item [-] for every Borel set $A$, $\PP(A):=\{ \mu \in \Mes_+(A)\ :\ \mu(A)=1\}$ denotes the set of probalities on $A$;
 
\item [-]  given an open subset $A$, $\D'(A)$ denotes the set of distributions on $A$ (the dual of $\D(A)$); $\D'(\Ob)$ stands for the subset of $\D'(\R^d)$
consisting of distributions  supported in $\Ob$;  to a distribution in $\D'(\R^d)$ we may associate its trace on any open subset $A$ defining
 a unique element of  $\D'(A)$.

\item [-] the distributional divergence of a matrix field  $\sigma \in \Mes(\Ob;\Sddp)$  is an element in $\D'(\Rd;\Sddp)$ that we will be denoted $\DIV(\sigma)$ while $\dive \la \in \D'(\Rd;\Rd)$ will stand for the standard distributional divergence acting on a vector measure  $\la\in \Mes(\Ob;\Rd)$; 
for $A$ being an open subset of $\Rd$, the equality $-\dive \la =f$ on $A$ means that the two distributions have the same trace on $A$;

\item [-]  the topological support of a  function  $f$ (resp.of a measure $\mu$) will be denoted $\spt(f)$ (resp. $\spt(\mu)$);

\item [-] the bracket $\pairing{\argu,\argu}$ shall be used to denote a canonical scalar product in the finite dimensional space of vectors or matrices, whilst in the case of infinite dimensional spaces we shall use the same bracket while sometimes specifying the functional spaces involved  in the lower index;
	
\item [-]	$C^{0,\frac1{2}}(\Ob)$ denotes the space of $\frac1{2}$-H\"older continuous functions on $\O$, while $C^1(\Ob)$
denotes its subclass consisting of continuously differentiable functions;

\item [-] given $\Sigma_0\subset \bO$, for every $x\in \Ob$ we denote by $d(\cdot,\Sigma_0)$ the euclidean distance to $\Sigma_0$ 
and by $p_{\Sigma_0}(x)$ the subset of $\Sigma_0$ defined by
		\begin{equation}\label{Gamma0}
		p_{\Sigma_0}(x) = \Big\{ z\in \Sigma_0 \ : \ d(x,\Sigma_0) = |x-z| \Big\};
		\end{equation}
		the graph of $p_{\Sigma_0}$ as a map from $\Ob$ to subsets of $\Sigma_0$ will be denoted $G_{\Sigma_0}$;	
\item [-] for every $(x,y)\in \Ob\times \Ob$ such that $x\not=y$, we denote by $\lambda^{x,y}$ and $\sigma^{x,y}$ the elements of $\Mes(\Ob;\R^d)$ and $\Mes(\Ob;\Sddp)$, respectively, defined by:
\begin{equation}\label{def:sigma^xy}
	\lambda^{x,y} =  \tau^{x,y}\, \Ha^1 \mres [x,y], \qquad  \sigma^{x,y} = \tau^{x,y}\otimes \tau^{x,y}\, \Ha^1 \mres [x,y], \qquad \tau^{x,y}= \frac{y-x}{|y-x|}
\end{equation}
(by convention, we set  $\lambda^{x,y}=0$ and  $\sigma^{x,y}=0$ if $x=y$). Note that $\lambda^{y,x}=-\lambda^{x,y}$ while $\sigma^{y,x}=\sigma^{x,y}$; 
		
\item [-] for any $\pi\in \Mes(\Ob\times \Ob)$, we denote by $\int \la^{x,y} \,\pi(dxdy)$ the measure $\la_\pi\in \Mes(\Ob;\R^d)$ such that
\begin{equation}\label{def:llpi}
\pairing{\la_\pi, \psi} \ :=\ \int \pairing{\la^{x,y},\psi}  \, \pi(dxdy)  \qquad  \forall\,\psi \in C^0(\Ob,\R^d);
\end{equation} 
	
\item [-] for any $\Pi\in \Mes_+(\Ob\times \Ob)$, we denote by $\int \sigma^{x,y} \,\Pi(dxdy)$ the measure $\sigma_\Pi\in \Mes(\Ob;\Sddp)$ such that 
\begin{equation} \label{def:sigmaPi}
\pairing{\sigma_\Pi, \Psi} \ :=\ \int \pairing{\sigma^{x,y},\Psi}  \, \Pi(dxdy) \qquad \forall\,\Psi \in C^0(\Ob,\Sddp).
\end{equation}

\end{enumerate}

	\bigskip
	
	\section{Monge-Kantorovich approach for the free material design problem}\label{sec:MKfree}
	
	Throughout the whole section we will assume  that the  load $f$ is a non-negative measure that we normalize to satisfy $\int f=1$. 
	As the Dirichlet condition $u=0$ is prescribed on $\Sigma_0$, it is not restrictive to assume that $f(\Sigma_0)=0$.
	To simplify the presentation, we also assume that the design  $\Omega$ is a convex domain. Note that this convexity assumption can be removed if we assume that   $\Sigma_0=\partial\Omega$.
	Keeping the notations from the introduction, we consider the optimal design of heat conductor (FMD) for a given mass $m$:
	\begin{equation}\label{alpham}
	\alpha(m) :=  \inf \biggl\{ \Comp(\sigma)  \ :\ \sigma \in \Mes(\Ob;\Sddp), \  \frac1{d} \int \tr\, \sigma \le m\biggr\},
	\end{equation}
	with the compliance $\Comp(\sigma)$ being defined in \eqref{eq:comp_FMD}.
	Here the mass constraint is intended as the overall integral of the arithmetic mean of the  eigenvalues of the conductivity tensor.
	This normalization, although it may look arbitrary, is convenient in order to compare with the (MOP) problem where the infimum  
	is restricted to the subclass of isotropic conductivity tensor fields $\sigma = \Ident \,\mu$. Therefore $\beta(m) $ defined in \eqref{magic}
	satisfies the inequality $\alpha(m) \le \beta(m) .$ A more precise relation will be derived in Proposition \ref{mop=fmd}. 	
	Let us first show how we can handle the mass parameter $m$ by introducing the reduced problem associated with (FMD), namely
	\begin{equation}
	\label{reducedFMD}
	Z\ :=\ \inf \biggl\{ \Comp(\sigma)  +  \int \tr\, \sigma  \ :\ \sigma \in \Mes(\Ob;\Sddp)\biggr\}.
	\end{equation} 
	By exploiting the $2$-homogeneity with respect to $u$ in the definition \eqref{eq:comp_FMD} of $\Comp(\sigma)$, we easily infer the scaling properties:
	\begin{equation}\label{scaling}
	\Comp(t \, \sigma)  = \frac1{t}\, \Comp(\sigma)\quad \forall t>0 \quad , \qquad \alpha(m)= \frac{\alpha(1)}{m}\quad \forall m>0.
	\end{equation}
	\begin{lemma}\label{2hom} The infimum problem \eqref{reducedFMD} admits at least one solution. All such solutions satisfy 
		the equi-repartition principle 	$\Comp(\sigma)=\int \tr \sigma= \frac{Z}{2}$.
		Moreover, for given $m>0$, $\widetilde{\sigma}$ is optimal for $\a(m)$ if and only if $\sigma =\frac{Z}{2 md}\,\widetilde{\sigma} $ is optimal for \eqref{reducedFMD}.
		Accordingly, the value function in \eqref{alpham} is given by\ $ \alpha(m) = \frac{Z^2}{4 md}$ and any minimizer $\sigma$ for \eqref{reducedFMD}
		is optimal for $\a(m_0)$ for $m_0:= \frac{Z}{2d}$ and vice-versa.
	\end{lemma}
	
	\begin{proof} \ The existence of an optimal $\sigma$ for \eqref{reducedFMD} is a consequence of the direct method of Calculus of Variations
		that we apply on the space $\Mes(\Ob;\Sddp)$ equipped with the weak* topology.  Indeed the functional $\sigma\in \Mes(\Ob;\Sddp) \mapsto \Comp(\sigma)\in [0,+\infty]$ is convex lower semicontinuous as a supremum over $u\in \D(\Rd)$  of the affine weakly* continuous functions : $L_u(\sigma)= \int u \, df - \frac{1}{2} \int \pairing{\sigma, \nabla u \otimes \nabla u}$. 
		On the other hand, the functional $\sigma\in \Mes(\Ob;\Sddp) \mapsto  \int \tr \,\sigma$ is convex l.s.c. with weakly* compact sublevel sets 
		(notice that the trace coincides with the restriction  to $\Sddp$ of a norm on symmetric matrices). 
		Let $\sigma$ be any solution for \eqref{reducedFMD}. Then the function
		$t\in \R_+ \mapsto  \Comp(t\sigma) + \int \tr(t\sigma) = \frac1{t} \Comp(\sigma) + t \int \tr\, \sigma\ $
		is minimal at $t=1$, thus $ \Comp(\sigma)= \int \tr\, \sigma= \frac{Z}{2}$ as claimed and $\frac{Z}{2}=\alpha(\frac{Z}{2d})$. 
		In addition, by using \eqref{scaling}, we get $\alpha(m) =\, \frac{Z}{2md}\, \alpha\left(\frac{Z}{2d}\right)\, =\ \frac{Z^2}{4md}.$
		Eventually we notice that $\widetilde{\sigma}:= \frac{2md}{Z} \, \sigma$ is admissible for $\alpha(m)$ while,
		from \eqref{scaling}, we infer that $\Comp(\widetilde{\sigma})= \frac{Z}{2md}\, \Comp(\sigma)= \frac{Z^2}{4md}.$ 
		Thus  $\widetilde{\sigma}$ is optimal for $\alpha(m)$. 
		The converse implication can be derived in a similar way. The last statement is obvious since, for 
		any solution $\sigma$  of \eqref{reducedFMD}, it holds that  $\Comp(\sigma) = \frac{Z}{2}= \a(m_0)$  while $\int \tr\, \sigma= \frac{Z}{2}=\a(m_0)$.

	\end{proof}
	
	%
	\begin{proposition} \label{mop=fmd} For $Z$ defined by \eqref{reducedFMD} let
		\begin{equation} \label{def:If}
		I(f,\Sigma_0) := \sup \Big\{ \pairing{f,u} \ :\ u \in \mathrm{Lip}_1(\Omega) \ , u=0 \ \text{on $\Sigma_0$}\Big\}.
		\end{equation}
		The following statements hold true:
		\begin{itemize}[leftmargin=1.5\parindent]
			\item [(i)] Let $W_1(\cdot,\cdot)$ denote the Monge distance defined  in \eqref{MK}, then
			\begin{equation} \label{dualMK} I(f,\Sigma_0) = W_1(f,\Sigma_0):= \min \Big\{ W_1(f,g)  : g \in \PP(\Sigma_0)\Big\} = \int d(x,\Sigma_0) \, f(dx);
			\end{equation}
		\item [(ii)]  The following equality holds: \ $ Z = \sqrt{2} \ I(f,\Sigma_0)$;
		
		\item [(iii)] Let $\sigma$ be a solution to \eqref{reducedFMD}. Then  $\mu= \tr\, \sigma$ solves the mass optimization problem $\mathrm{(MOP)}$ 
		with $\int \mu= \frac{Z}{2}$. As a consequence $\mu(\Sigma_0)=0$ and the  value functions for $\mathrm{(MOP)}$ and $\mathrm{(FMD)}$ are linked by the 
		relation $\beta(m) = d\, \alpha(m).$ 
		 Furthermore $\sigma$ is the  rank-one tensor measure given by 
		\begin{equation}\label{rank1} 
		\sigma\ =\ (\nabla_\mu \bar u \otimes \nabla_\mu \bar u) \ \mu \qquad  \text{where} \qquad   \bar u:= d(x,\Sigma_0).
		\end{equation}
		($\nabla_\mu \bar u$ denotes the $\mu$-tangential gradient  of the Lipschitz function $\bar u$ as defined in Proposition \ref{Tmu}).
		\end{itemize}
	\end{proposition}
	
	\begin{proof}\ For the assertion (i), we refer to \cite{bouchitte2001}. The convexity assumption on $\Omega$ ensures that the geodesic distance in $\Ob$ coincides with the Euclidean one. Let us establish  (ii). In order to use a compactness argument, we go back to the constrained problem \eqref{alpham} noticing that, by Lemma \ref{2hom}, we recover the desired equality by showing that $\alpha(d^{-1}) = \frac{1}{2}\, \bigl(I(f,\Sigma_0)\bigr)^2.$
		The latter equality is  a consequence of the following chain of equalities: 
		\begin{align*} \alpha(d^{-1}) = \inf_{\substack{\sigma \in \Mes(\Ob;\Sddp) \\ \int \tr \sigma \le 1}} \ \sup_{u\in \D(\Ob\setminus \Sigma_0)}  \left\{ \pairing{f,u}    -   \frac{1}{2}\! \int \pairing{\sigma, \nabla u \otimes \nabla u} \right\}
		&=\sup_{u\in \D(\Ob\setminus \Sigma_0)}\ \inf_{\substack{\sigma \in \Mes(\Ob;\Sddp) \\ \int \tr \sigma \le 1}}  \left\{ \pairing{f,u} -   \frac{1}{2} \int\! \pairing{\sigma, \nabla u \otimes \nabla u} \right\}\\
		&= \sup_{u\in \D(\Ob\setminus \Sigma_0)} \left\{ \pairing{f,u}- \frac1{2}\|\nabla u\|_\infty^2  \right\} \ =\ \frac{1}{2}\, \bigl(I(f,\Sigma_0)\bigr)^2
		\end{align*}
		where in the first  line, we switch infimum and supremum by applying  Ky Fan's Theorem (see Theorem \ref{ky-Fan} in Appendix) to  the convex-concave Lagrangian 
			$\mathcal{L}(\sigma,u) = \pairing{f,u} -  \frac{1}{2}\int \pairing{\sigma,\nabla u \otimes \nabla u} $  taken on $X\times \D(\Ob\setminus \Sigma_0)$ where $X:=\bigl\{\sigma \in \Mes(\Ob;\Sddp) : \int \tr \sigma \le 1 \bigr\}$ is convex and compact for the weak* topology and 
		where to pass from the first to   the second line, we optimize with respect to $\sigma\in X$ by taking tensors of the form $\sigma= \tau\otimes \tau \, \delta_x$ where $\delta_x$ is the Dirac mass at $x$ and $\tau$ is a unit vector in $\R^d$. The last equality can be readily obtained by
writing competitors $u$ in the form $u= t v$ where $v\in  \mathrm{Lip}_1(\Omega)$ and $t\in \R$ and by maximizing in $t$ first and then 
with respect to $v$.

		Let us now prove the assertion (iii). Let $m_0= \frac{Z}{2d}$ as given in Lemma \ref{2hom} and let $\sigma$ solve \eqref{reducedFMD}.
		 Then $\mu:=\tr\, \sigma$
		satisfies $\int \mu = \frac{Z}{2}$  and, in view of the assertions (i) and (ii) and taking \eqref{magic} into account, we infer that:
		$$   \Comp(\Ident\,\mu) \ \ge \ \frac { \bigl(I(f,\Sigma_0)\bigr)^2}{2 \int \mu} \ =\  \frac{ Z}{2}  .$$
		Since  $0\le \sigma \le \Ident\,\mu$, it follows from  definition \eqref{eq:comp_FMD} that  $\Comp(\sigma) \ge \Comp(\Ident\,\mu)$, hence
		$  Z = \int \tr \sigma + \Comp(\sigma) \ge \frac{Z}{2}  +  \Comp(\Ident\,\mu)    \ge Z .$
		As a consequence, we are led to the equalities:\  
		\begin{equation}\label{comp=}
		\Comp(\sigma)  = \Comp ( \Ident\,\mu) =\frac{Z}{2}. 
		\end{equation}
		It follows that $\mu$ is optimal for  $(\mathrm{MOP})$ subject to the mass constraint $\int \mu= \frac{Z}{2} $. 
		In particular  $\mu(\Sigma_0)= 0$ as a consequence of \cite[Prop 3.7]{bouchitte2001} and 
	 we have  $\Comp ( \Ident\,\mu)= \beta(d \, m_0)= \frac1{d} \beta(m_0) $. 
On the other hand, it holds that  $\Comp(\sigma)= \alpha(m_0)$ since  $\sigma$ is optimal for $(\mathrm{FMD})$ with the upper bound on the mass being $m_0$. 
Therefore we  obatin the
		 equality $\beta(m_0) = d\, \alpha(m_0)$ that we  extend to all $m>0$ by the scaling property.

 To conclude the proof of Proposition \ref{mop=fmd} it remains to show that any optimal $\sigma$ 
		 is uniquely determined in terms of its trace $\mu:= \tr\, \sigma $ by the relation \eqref{rank1}.
		To that aim, we exploit \eqref{comp=} and the fact that the function  $\bar u(x):= d(x, \Sigma_0)$ is optimal in \eqref{def:If}.
	It turns out that optimality of $\mu$ for $(\mathrm{MOP})$ implies that  $|\nabla_\mu \bar u|=1\ $ $\mu$-a.e (see  \cite{bouchitte2001}). Next,
		we rewrite the supremum problem involved in the definition of  the compliance $ \Comp(\Ident\,\mu)$ (resp. $ \Comp(\sigma)$) by extending to Lipschitz competitors as follows:  
		\begin{align} \label{relax-mu} \Comp(\Ident\,\mu) &= \sup \left\{\pairing{f,v} - \frac{1}{2} \int \abs{\nabla_\mu v}^2 d\mu \ : \ v\in \mathrm{Lip}(\Omega), \ v=0 \ \text{on $\Sigma_0$} \right\} \\
	\label{relax-sigma} \Comp(\sigma) &= \sup \left\{\pairing{f,v} - \frac{1}{2} \int \pairing{ S, \nabla_\mu v \otimes \nabla_\mu v}\, d\mu \ : \ v\in \mathrm{Lip}(\Omega),\ v=0 \ \text{on $\Sigma_0$} \right\}	
\end{align}
where $S \in L^\infty_{\mu}(\Ob;\Sddp)$ satisfies $\sigma = S \mu$ and $\tr S=1$. To justify the equalities \eqref{relax-mu}, \eqref{relax-sigma}, it is enough to approximate any element $v\in \mathrm{Lip}(\Omega)$ vanishing on $\Sigma_0$ by an equi-Lipschitz sequence $v_n$ 
in $\D(\R^d\setminus \Sigma_0)$ and apply the assertion (ii) of Proposition \ref{Tmu}.
		Taking into account \eqref{comp=} and that $|\nabla_\mu \bar u|=1$, one checks easily that $\bar v:= \sqrt{2}\, \bar u$ is optimal in \eqref{relax-mu}. Indeed by \eqref{dualMK}, we have  
		\begin{equation*}\label{vopti}
		\Comp(\Ident\,\mu)\ \le\ \pairing{f,\bar v} - \frac{1}{2} \int \abs{\nabla_\mu \bar v}^2 d\mu\ =\ \sqrt{2}\,  I(f,\Sigma_0) - \int \mu\ =\
	 \frac{Z}{2} =
			\Comp(\Ident\,\mu).
		\end{equation*}
	 On the other hand by \eqref{relax-sigma}, we have
		$ \Comp(\sigma)  \ge 
	 \pairing{f,\bar v} - \frac1{2} \int \pairing{S,\nabla_\mu \bar v \otimes \nabla_\mu \bar v}\, d\mu.$
		Hence, considering \eqref{comp=} and assertion (ii),  we deduce that
		$$ \int \pairing{S,\nabla_\mu \bar v \otimes \nabla_\mu \bar v}\, d\mu \ \ge\ \int |\nabla_\mu \bar v|^2 \, d\mu.$$
		Obviously the same inequality holds after substituting $\bar v$ with $\bar u$. Then, since $\tr\, S=1$ \ $\mu$-a.e., we may localize to obtain:
		$$ \pairing{S,\nabla_\mu \bar u \otimes \nabla_\mu \bar u} = |\nabla_\mu \bar u|^2\quad \mu\text{-a.e.}$$
		and then  conclude that $S=\nabla_\mu \bar u \otimes \nabla _\mu \bar u\ $ as claimed in \eqref{rank1}. 
		\end{proof}
	
	The rest of the section is devoted to the representation of optimal $\sigma$ through transport rays connecting the support of $f$ to $\Sigma_0$.
		\begin{theorem}\label{sliced-sigma} Let $\bar \gamma \in \Mes_+(G_{\Sigma_0})$ be a pairing measure with first marginal being equal to $f$.
		Then, the tensor measure $ \bar \sigma:= \frac1{\sqrt{2}} \int \sigma^{x,y} \, \bar\gamma(dxdy)$ \ is optimal for \eqref{reducedFMD}.
		Conversely, any optimal measure $\sigma$ for \eqref{reducedFMD} can be represented in this form for a suitable $\gamma \in \Mes_+(G_{\Sigma_0})$.
	\end{theorem}
	
	\begin{proof}\ Let $\mu= \tr\, \bar\sigma$. Then, since $\int \tr\, \sigma^{x,y} = |x-y|$, we have
		$$\int \mu = \frac1{\sqrt{2}} \int_{\Ob\times\Ob} |x-y| \, \gamma(dxdy) =\frac1{\sqrt{2}}  \int d(x,\Sigma_0) \, f(dx)= \frac1{\sqrt{2}}\, I(f,\Sigma_0)= \frac{Z}{2}.$$
		This implies that $\sigma$ meets the mass constraint for $m_0= \frac{Z}{2d}$ given in Lemma \ref{2hom}. 
		On the other hand, by \cite[Thm 4.6]{bouchitte2001}, the measure $\mu$ is optimal for $\mathrm{(MOP)}$ for that prescribed mass $m_0$. It follows that
		$$ \Comp(\bar \sigma) \le \Comp(\Ident\,\mu) =  \frac{I(f,\Sigma_0)^2} {2 m_0} = \frac{Z}{2},$$
		hence the optimality  of $\bar\sigma$ in \eqref{reducedFMD} since $\Comp(\bar \sigma) + \int \tr \bar \sigma \le Z$.
		
		Conversely let $\sigma$ be optimal for \eqref{reducedFMD} and let $\mu=\tr\, \sigma$. By assertion (iii) of Proposition \ref{mop=fmd} we know that $\mu$ is optimal for
		$\mathrm{(MOP)}$ and that $\sigma$ is rank-one according to  \eqref{rank1}. As a result, it is enough to show the existence of a transport plan $\gamma \in\Mes_+(G_{\Sigma_0})$
		such that $\mu$ is represented by the slicing formula:
		$$ \pairing{\mu, \varphi} \ =\ \frac1{\sqrt{2}} \int_{\Ob\times \Ob} \left(\int_{[x,y]}\varphi\, d\Ha^1\right)\! \gamma(dxdy) \qquad  \forall\, \varphi\in C^0(\R^d).$$
		This is a consequence of Lemma \ref{NeuSmi} given hereafter that we apply to $\mu$ whose mass is  $m_0=\frac{Z}{2} = \frac{I(f,\Sigma_0)}{\sqrt{2}}$.
	\end{proof}
	
	\begin{lemma}\label{NeuSmi} \ Let $f\in \PP(\Ob)$ be a probability measure on $\Ob$, $I(f,\Sigma_0)$ defined by \eqref{def:If} and $\bar u=d(x,\Sigma_0)$.  Then:
		\begin{itemize}[leftmargin=1.5\parindent]
			\item [(i)]\ Let $\mu\in \Mes_+(\Ob)$ such that $\Comp(\Ident\,\mu)<+\infty$. Then there exists $g\in \PP(\Sigma_0)$ such that   
			\begin{equation}\label{dir=neu}
			\Comp(\Ident\,\mu) = \sup_{v\in C^1(\Ob)} \left\{\pairing{f- g,v} - \frac1{2} \int |\nabla v|^2 \, d\mu\right\}=
			\min_{q\in L^2_\mu(\Ob;\R^d)} \left\{ \frac1{2} \int |q|^2 \, d\mu\, :\,  -\dive (q\,\mu) \!=f- g \ \text{ in $\D'(\R^d)$}\right\}. 
			\end{equation}
			
						\item [(ii)]\ Assume further that $\mu$ is optimal for $\mathrm{(MOP)}$ with mass $m=\frac{Z}{2}$. Then the minimum on the right hand side of \eqref{dir=neu} is reached for a vector field $\bar q$ such that we have
			$|\bar q|= \frac{I(f,\Sigma_0)}{m}$ and $\bar q=\frac{I(f,\Sigma_0)}{m} \nabla_\mu \bar u\ $ holding $\mu$-a.e.
			Moreover, there exists a suitable $\bar\gamma \in \Mes_+(G_{\Sigma_0}) $ with $\bar\gamma\in \Gamma(f,g)$ such that the vector measure $\bar\lambda= q\,\mu $ can be decomposed as follows:
			$$ \pairing{\bar\lambda, \psi} \ =
			\ - \int_{\Ob\times \Ob} \left(\int_{[x,y]} \pairing{\psi, \frac{y-x}{|y-x|}}\, d\Ha^1\right) \bar\gamma(dxdy) \qquad  \forall \psi\in C^0(\R^d;\R^d).$$

		\end{itemize}
	\end{lemma}
	Note that the assertion (i) above (whose validity requires that $f\ge 0$) provides an equivalence principle between a Dirichlet condition on $\Sigma_0$ and a Neumann condition associated with a suitable source term $g$ supported on $\Sigma_0$.

	
		\begin{proof}[Proof of the assertion (i)] \ We show that as a measure $g$ satisfying the assertion (i) we may take any minimizer for the problem:
			\begin{equation}\label{optig}
			\min\Big\{ G(\nu) \ :\ \nu\in\PP(\Sigma_0) \Big\}, \qquad  G(\nu) :=  \sup_{v\in C^1(\Ob)} \left\{\pairing{f- \nu,v} - \frac1{2} \int |\nabla v|^2 \, d\mu\right\}.
			\end{equation}
 Next by applying again  the commutation argument for convex concave Lagrangians (see Theorem \ref{ky-Fan}), we get
			\begin{align*}
			\min\Big\{ G(\nu) \ :\ \nu\in\PP(\Sigma_0) \Big\}\ &=\  \inf_{\nu\in\PP(\Sigma_0)}  \sup_{v\in C^1(\Ob)} \left\{\pairing{f- \nu,v} - \frac1{2} \int |\nabla v|^2 \, d\mu\right\}\\
			&= \sup_{v\in C^1(\R^d)} \inf_{\nu\in\PP(\Sigma_0)}   \left\{\pairing{f- \nu,v} - \frac1{2} \int |\nabla v|^2 \, d\mu\right\}\\
			&= \sup_{v\in C^1(\R^d)} \left\{\Big\langle f,v- \sup_{\Sigma_0} v \Big\rangle - \frac1{2} \int |\nabla v|^2 \, d\mu \right\}\\
			&= \sup_{u\in \mathrm{Lip}_+(\R^d)} \left\{\pairing{f,u} - \frac1{2} \int |\nabla_\mu u|^2 \, d\mu \ : u =0 \ \text{on $\Sigma_0$}\right\} 
		=\, \Comp(\Ident\,\mu),
 			\end{align*}
			where:
			\begin{enumerate}
				\item[-] in the third line we used the fact that for every $v\in C^1(\R^d)$ the non-negative Lipschitz function $u= (v -\sup_{\Sigma_0} v)^+$ vanishes on $\Sigma_0$ while its energy $\pairing{f,u}$ is larger than $\pairing{f,v}$ since $f\ge 0$ ;
				\item[-] in the last line, we used the fact that, for a positive load $f$, the supremum in the definition of the compliance functional is unchanged if we restrict to non-negative Lipschitz functions $u$.
			\end{enumerate}
			By taking $g$ to be a minimizer in \eqref{optig}, we are led to the first equality in \eqref{dir=neu}. The second equality is a byproduct of classical duality in the Hilbert space $L^2_\mu$ after noticing that the divergence condition holding in $\D'(\R^d)$ is equivalent to the equality  $\pairing{f-g,v}= \int \pairing{q,\nabla v}\, d\mu$ holding for every $v\in C^1(\R^d)$.
		
		\medskip
		\noindent
		{\em Proof of the assertion (ii)}\ If $\mu$ is optimal for $\mathrm{(MOP)}$ with mass $m_0$, we know from \cite[Thm 2.3 and  Thm 3.9]{bouchitte2001} that
		any $\bar  q$ solving \eqref{dir=neu} has a constant norm $|\bar q|= \frac{I(f,\Sigma_0)}{m_0}$ while the vector measure $\bar \lambda = \bar q\, \mu$ is optimal for 
		the PDE formulation of the Monge distance between $f$ and $g$:
		$$ W_1(f,g) =\min_{\lambda \in \Mes(\Ob;\R^d)} \left\{ \int |\lambda| \ :\ -\dive \lambda =f- g \ \text{ in $\D'(\R^d)$}\right\}.$$
Next, by applying a deep argument in geometric measure theory due to S.K. Smirnov (see \cite{Smirnov}  and  Proposition 2.3 in \cite{Dweik}), we may decompose the optimal vector measure $\bar\lambda$ into the form $\bar\lambda =  \int_{\Ob\times \Ob}  \lambda^{x,y} \, \bar\gamma(dxdy)$ (for definition of $\lambda^{x,y}$ see \eqref{def:sigma^xy})  with  $\bar\gamma\in \Gamma(f,g)$ being a suitable pairing measure such that 
 $$ \int |\bar\lambda| = \int_{\Ob\times \Ob} |\lambda^{x,y}|  \, \bar\gamma(dxdy) =   \int_{\Ob\times \Ob} |x-y| \, \bar\gamma(dxdy).$$
Then, recalling that by \eqref{dualMK} one has $ W_1(f,g)= \int d(x,\Sigma_0) \, f(dx)$, we infer that
$\int_{\Ob\times \Ob}  |x-y| \, \bar\gamma(dxdy) = \int d(x,\Sigma_0) \, f(dx).$
It follows that $ \bar\gamma$ is supported in $G_{\Sigma_0}$ as claimed. 

	\end{proof} 	
	Thanks to Theorem \ref{sliced-sigma} we are able to characterize the special geometric configurations for which the inequality $\inf \mathrm{(FMD)}\le \inf \mathrm{(OM)}$
	mentioned in the introduction becomes an equality.	
	Recalling  the definition \eqref{Gamma0} of $p_{\Sigma_0}$, let us introduce:
	\begin{equation}\label{def:Mridge}
	M(\Omega, \Sigma_0):= \Big\{x \in \Ob\setminus\Sigma_0    \ :\   x \in\co\bigl(p_{\Sigma_0}(x)\bigr) \Big\}.
	\end{equation}
		Notice that $M(\Omega, \Sigma_0)$ is empty if $\Sigma_0=\co(\Sigma_0)$. Otherwise it is a non-empty compact subset of $\co (\Sigma_0)$ as shown in Lemma \ref{ridge} below.
	\begin{corollary}\label{fmd<om} Let $f\in \PP(\Ob)$.  Then the reduced $\mathrm{(FMD)}$ problem  \eqref{reducedFMD} admits a divergence free solution  if and only if $\spt(f)\cap \Omega \subset M(\Omega, \Sigma_0)$. 
		If this condition is violated, then we have the strict inequality \ $\inf \mathrm{(FMD)}< \inf \mathrm{(OM)}$.
	\end{corollary}

	\begin{proof} Assume that $\spt(f) \subset M(\Omega, \Sigma_0)$. By Choquet's theorem, for $f$-a.e. $x$, there exists a probability $p^x$
	supported in $p_{\Sigma_0}(x)$ whose barycenter $[p^x]$ satisfies $[p^x] = x$. As the map $(x,p) \in \Ob\times \PP(\Sigma_0) \mapsto [p]-x$ is Borel regular, we can select $p^x$ so that $x \to p^x$ is $f$-measurable (see for instance \cite{castaing}). Then  a plan $\bar\gamma\in \Mes_+(G_{\Sigma_0})$ with first marginal $f$  is obtained by setting for every $\f\in C^0(\Ob\times\Ob)$:
		\begin{equation} \label{desint}
		\pairing{\bar\gamma, \varphi} := \int_{\Ob} \bigl\langle p^x,\varphi(x,\cdot) \bigl\rangle \, f(dx).\end{equation}
		Then, by invoking Theorem \ref{sliced-sigma}, the tensor measure $ \bar \sigma:=\frac1{\sqrt{2}} \int \sigma^{x,y} \, \bar\gamma(dxdy)$ is optimal for \eqref{reducedFMD}. Furthermore, in view of definition \eqref{def:sigma^xy}, we may evaluate the divergence of $\sigma$ against a test function $\psi \in \D(\Omega;\R^d)$:
		\begin{align*}
		\sqrt{2}\,  \pairing{\DIV \,\bar\sigma, \psi} &= - \int_{\Ob\times\Ob} \pairing{\sigma^{x,y}, \nabla \psi} \, \bar\gamma(dxdy) = -\int_{\Ob\times\Ob} \pairing{\psi(y)-\psi(x),\frac{y-x}{|y-x|} } \, \bar\gamma(dxdy)\\
		&= \int_\Omega  \left(\int \pairing{ \psi(x),\frac{y-x}{|y-x|}}  \, p^x(dy) \right) f(dx) \\
		&= \int_{\Omega}  \frac{\pairing{\psi(x), [p^x]-x}} {d(x,\Sigma_0)} \, f(dx) =0 \ ,
		\end{align*}
		where in the second line we used \eqref{desint} and the fact that $\psi$ vanishes on $\partial \Omega$ hence $p^x$-a.e.,
		while in the last line  we exploit the fact that $|y-x| =d(x,\Sigma_0)$ for all $y\in  p_{\Sigma_0}(x)$, thus $ p^x$-a.e.

		In order to prove the converse implication, we  assume that \eqref{reducedFMD} admits $\sigma$ as a divergence free solution.
		Invoking again Theorem \ref{sliced-sigma}, we may write $\sigma= \frac1{\sqrt{2}} \int \sigma^{x,y} \, \gamma(dxdy)$ for a suitable $\gamma\in \Mes_+(G_{\Sigma_0})$ with the first marginal being equal to $f$. This pairing measure $\gamma$ admits a disintegration of the form \eqref{desint} for a measurable family of probabilities $\{p^x\}$ such that
	$\spt(p^x)\subset p_{\Sigma_0}(x)$ for $f$-a.e $x\in \Ob$.
		Then, after the same computations as before,  we are led to the equality
		$$ \sqrt{2}\,  \pairing{\DIV \,{\sigma}, \psi} =  \int_{\Ob\setminus\Sigma_0}  \frac{\pairing{\psi(x), [p^x]-x}} {d(x,\Sigma_0)} \, f(dx) =0, $$
		holding for every test function $\psi \in \D(\Omega;\R^d)$. It follows that $[p^x] = x$ for $f$-a.e. $x\in\Omega$, thus $\spt(f)\cap \Omega \subset M(\Omega, \Sigma_0)$ since $M(\Omega,\Sigma_0)$ is closed by Lemma \ref{ridge}.
		We conclude the proof of Corollary \ref{fmd<om} by noticing that the existence of a divergence free solution to  \eqref{reducedFMD} induces the equality $\min \mathrm{(FMD)}= \min \mathrm{(OM)}$
		and vice-versa.
	\end{proof}
	
	\begin{remark}\label{selectgamma} The requirement for $f$ given in Corollary \ref{fmd<om} implies that $\spt(f)\cap\Omega$ is contained in $\cobar(\Sigma_0)$ as well as in the closure of the set where  $d(\cdot, \Sigma_0)$ is not differentiable.
		In general these sets are strictly larger than $ M(\Omega, \Sigma_0)$. 
		For instance, if $\Omega=\{|x_1|\!<\!a , |x_2|\!<\!b \}$ with $a\le b$ and $\Sigma_0=\partial\Omega$, we infer that $\inf \mathrm{(FMD)}< \inf \mathrm{(OM)}$ unless
		$ \spt(f) \subset M(\Omega)=\{x_1=0 ,\ |x_2|\le b\!-\!a\}.$
		The fact that $p_{\Sigma_0}$ can be multi-valued (namely at points of non-differentiability of $d(\cdot,\Sigma_0)$)
		implies the  existence of multiple  solutions to the $\mathrm{(FMD)}$ problem. 
		In the example above, if $f$ is a Dirac mass located at a point $x_0=(0,t_0)$ with $|t_0|< b-a$, then all solutions to $\mathrm{(FMD)}$ arise from the family of pairings 
		$\{\gamma_\theta :\ 0\le \theta\le 1\}$ where
		$ \gamma_\theta = (1\!-\!\theta)\ \delta_{(0,t_0)}\otimes \delta_{(-a,t_0)} \, +\, \theta\ \delta_{(0,t_0)}\otimes \delta_{(a,t_0)}.$
		Among them the only one which meets  the divergence free constraint is obtained for $\theta =\frac1{2}$ and the optimal stress is $\bar \sigma= e_1\otimes e_1 \, \Ha^1 \mres [(-a,t_0),(a,t_0)]$. 
	\end{remark}

	\begin{lemma}\label{ridge}\  Assume that $\Sigma_0$ is a strict subset of $\co(\Sigma_0)$. Then  $M(\Omega, \Sigma_0)$ is a non-empty compact subset of $\co(\Sigma_0)\setminus\Sigma_0$.
		If $\Sigma_0=\partial\Omega$, it coincides with the {\rm high ridge} of $\Omega$ defined by: 
		\begin{equation}\label{def:ridge}
		M(\Omega) := \Big\{ x\in \Omega \ :\ d(x,\partial\Omega) \ge  d(z,\partial\Omega)\ \ \forall z\in \Omega\Big\}.
		\end{equation}
	\end{lemma}
	\begin{proof} \ Let $u=d(\cdot,\Sigma_0)$ and define, for every $x\notin \Sigma_0$,  the convex compact set of $S^{d-1}$:
		$$ C_0(x):= \co \left(\left\{ \frac{x-y}{|x-y|}\, :\, y\in p_{\Sigma_0}(x)  \right\}   \right)$$
		It turns out that $C_0(x)$ coincides with the Clarke's gradient of $u$ on $\R^d\setminus\Sigma_0$. Moreover, by \cite[Prop 4.4.1, Thm 3.2.6]{cannarsa}, $u$ is locally
		semi concave in $\R^d\setminus \Sigma_0$ and its lower directional derivative 
		$$u^0_-(x,\theta) :=\ \liminf_{h\to 0^+,\, y\to x} \frac{u(y+h \theta) - u(y)}{h}\ $$
		is lower semicontinuous and satisfies $|u^0_-(x,\theta)|\le |\theta|$, while for every $(x,\theta) \in (\R^d\setminus \Sigma_0)\times \R^d$:
		$$ u^0_-(x, \theta):= \lim_{h\to 0^+} \frac{u(y+h \theta) - u(y)}{h} = \min \Big\{ \pairing{p,\theta} \, :\, p\in C_0(x) \Big\}.$$
		In view of definition \eqref{def:Mridge} and noticing that $u^0(x,\theta)>0$ if $x\in \Sigma_0$ and $\theta$ has a direction pointing  inward to the convex set $\Omega$, we
		deduce  the following equivalences:
		\begin{equation}\label{clarke=0}
		x \in M(\Omega, \Sigma_0)\quad \Longleftrightarrow\quad  x\notin \Sigma_0 \ \text{and}\  C_0(x)\supset\{0\} \quad \Longleftrightarrow  \quad u^0_-(x, \theta) \le 0  \quad \forall\, \theta\in\R^d .
		\end{equation}
		It follows from the lower semicontinuity of $u^0_-$ that $M(\Omega, \Sigma_0)$ is a closed subset of $ \co(\Sigma_0)\setminus \Sigma_0$.
		
		If $\Sigma_0=\partial\Omega$, then the function $u$ is concave on $\Ob$ (see for instance \cite{zolesio}) and $C_0(x)$ coincides with the superdifferential $\partial^+u(x)$ 
		in the sense of convex analysis. Therefore the condition $0\in C_0(x)$ in \eqref{clarke=0} is  equivalent to  saying that the maximum of $u$ on $\Ob$ is reached at $x$.
		Thus in this case, we obtain the equality  $M(\Omega, \Sigma_0)= M(\Omega)$ with $M(\Omega)$ defined in \eqref{def:ridge}, which in turn is a non-empty convex compact subset.
		
		Eventually we have to show that $M(\Omega, \Sigma_0)$ is non-empty in the general case. Since we assumed that $K_0:=\co(\Sigma_0)\nsupseteq\Sigma_0$, there exists
		$\bar x\in K_0$ such that $ u(\bar x) =\max_{K_0}  u >0$. We argue that $\bar x\in M(\Omega,\Sigma_0)$.
		If $\bar x$ belongs to the interior of $K_0$ then $u^0_-(\bar x,\theta)\le 0$ for every $\theta\in \R^d$ and our claim follows from \eqref{clarke=0}.
		In fact we may always reduce ourselves to this case by considering a finite subset $\Sigma_{\bar x}\subset \Sigma_0$ whose convex hull $K_{\bar x}$ contains $\bar x$ while having minimal cardinality
		(at most $d+1$).  Noticing that $u(\bar x) =\max_{K_{\bar x}} u$, we apply the same arguments to the restriction of $u$ to the affine subspace  $\bar x + V$ spanned by $K_{\bar x}$.
		As $\bar x\notin \Sigma_{\bar{x}}$, by construction $\bar x$ belongs to the relative interior of $K_{\bar x}$ while $p_{\Sigma_0}(\bar x)= p_{\Sigma_{\bar x}}(\bar x)$.
		Therefore $u^0(\bar x, \theta) \le 0$ for $\theta\in V$ and, by applying the counterpart of  \eqref{clarke=0} in $V$ , we deduce that
		$$0\in \co \left(\left\{ \frac{\bar x-y}{|\bar x-y|}\, :\, y\in p_{\Sigma_{\bar x}}(\bar x)\right\}\right),$$
		thus arriving at the same conclusion as the right handside above is a subset of $C_0(x)$.
		
		\end{proof}


\begin{remark}\label{signed-f-false} \ The validity of Corollary \ref{fmd<om} requires that $f$ is a non-negative measure.
	In case of a signed load $f$, Example \ref{divers} (see in particular the configuration depicted in Fig. \ref{fig:miscellaneous}(c)) furnishes a counter-example in which $\Sigma_0=\bO$ and, despite 
	$f$ not being supported in $M(\O)$,  a solution $\sigma$ to $\mathrm{(FMD)}$ exists such that  $\DIV\, \sigma=0$ in $\O$.
	\end{remark}

	\medskip

	%
	%
	%
	
	\section{The optimal pre-stressed membrane problem }\label{sec:optimembrane}
	
	In this section $\Omega$ denotes a bounded convex domain of $\R^d$  and $\Sigma_0$ a closed subset of $\bO$. 
	we investigate the optimal design problem described in the introduction:
	\begin{equation} \tag(OM)
	\label{OM}
	\alpha_0(m) \ :=\ 	\inf \biggl\{ \Comp(\sigma) \,:\, \sigma \in \Mes(\Ob;\Sddp), \ \DIV \sigma= 0 \ \text{in $\Omega$}, \ \frac1{d}\int \tr\, \sigma \leq m \biggr\}
	\end{equation} 
	where $\Comp(\sigma)=\Comp_{\O,f,\Sigma_0}(\sigma)$ is defined in \eqref{eq:comp_FMD}. In practice $d=2$ and   $f\in \Mes(\Ob)$  represents the vertical pressure exerted on the membrane, $u$ the deflection function
	and $\Sigma_0$ the part of the boundary where the membrane is pinned in the vertical direction. We may assume that $f(\Sigma_0)=0$.
	%
	Recall that in Section \ref{sec:MKfree} (see Corollary \ref{fmd<om}) we  showed that the infimum in $\mathrm{(OM)} $ is stricly larger than the infimum of $\mathrm{(FMD)} $
	(that is $\alpha(m)\le \alpha_0(m)$), except if $f$ is supported in the set $M(\Omega,\Sigma_0)$ defined in \eqref{def:Mridge}. 
	Although the two problems exhibit in general  very different solutions, they have some common features, in particular the $2$-homogeneity argument used in Lemma \ref{2hom}.
	Then it is easy to derive that 
	$$ \alpha_0(m)\ =\ \frac{Z_0^2}{4\, m\, d}\qquad \text{(null indices  recall the divergence free constraint)}$$ 
	where $Z_0$ denotes the infimum of the following 
	reduced membrane problem 
	\begin{equation}\label{reducedom}
	Z_0 \ :=\ \inf\ \biggl\{ \Comp(\sigma)  +   \int \tr\, \sigma  \ :\ \sigma \in \Mes(\Ob;\Sddp), \ \   \DIV \sigma= 0 \ \text{in $\Omega$}\biggr\}.
	\end{equation}
	As will be seen later optimal tensor measures $\sigma$ for \eqref{reducedom} exist and every optimal $\ov \sigma$ satisfies the  equi-repartition of energy principle
	\begin{equation}\label{equipart}
	\int \tr\, \ov\sigma\ =\ \Comp(\ov\sigma)\ =\ \frac{Z_0}{2}. 
	\end{equation}
	Nonetheless, the duality argument leading to  Proposition \ref{mop=fmd} has to be modified in a significant way. In order to account for the divergence free constraint in $\Omega$ it is necessary to introduce, as a Lagrange multiplier, the  symmetrized gradient $$e(w):= \frac1{2} \,\Big(\nabla w + (\nabla w)^{\mathrm{T}} \Big) , $$
	of a smooth function $w\in \mathrm{Lip}(\Ob;\R^d)$ vanishing on $\partial\Omega$.   
	Accordingly, the counterpart of the supremum problem \eqref{def:If} will read as follows
	\begin{equation}\label{dualOM}  I_0(f,\Sigma_0) \ :=\ \sup\left\{ \int u \, df \ :\ (u,w)\in \mathcal{K}\right\},
	\end{equation}
	where $\mathcal{K}$ denotes the convex subset consisting of all Lipschitz pairs $ (u,w)\in \mathrm{Lip}(\Ob)^{1+d}$ such that:
\begin{subnumcases}{\label{def:calK}}
     u=0 \quad \text{on $\Sigma_0$},\qquad  w=0  \quad \text{on $\bO$},& \label{dirichlet}\\
   \frac{1}{2}\,\nabla u \otimes \nabla u + e(w)  \leq \Ident \qquad \text{a.e in } \Omega. & \label{pointwise}
\end{subnumcases}

	The natural duality involved will be now between continuous pairs $(u,w)$ (that is the deflection and in-plane deformation of the membrane)  and measures $(\lambda, \sigma) \in \Mes(\Ob;\R^d) \times \Mes(\Ob;\Sdd)$  (transverse force and in-plane stress).
	As a primal problem we will consider the $\mathrm{(OM)}$ problem \eqref{reducedom}  rewritten in the form:
	\begin{equation}\tag{$\mathcal{P}$}
			\inf  \Big\{ J(\lambda,\sigma) \ : \ (\lambda,\sigma) \in \A  \Big\} 
	\end{equation}
	where $J$ is a suitable functional on measures (see \eqref{defJ}) and the admissible set $\A$ is defined by
			\begin{equation}
			\label{defA}
			\A:=\Big\{(\lambda,\sigma) \in \Mes(\Ob; \R^d\times \Sddp) \ :-\dive\, \lambda =f \quad \text{in $\R^d\setminus \Sigma_0$}, \quad \DIV\, \sigma =0  \quad \text{in $\Omega$} \Big\}.
		\end{equation}
	
In parallel we  put forward an alternative  duality scheme based on a two-point equivalent of the condition \eqref{pointwise}
(see forthcoming Lemma \ref{lem:two_point_condition}), namely:
\begin{equation}
			\label{eq:two_point_condition}
			\frac{1}{2}\, \abs{u(y)-u(x)}^2 + \pairing{w(y)- w(x),y-x} \leq \abs{x-y}^2 \qquad \forall (x,y)\in \Ob\times\Ob .
			\end{equation}
This complementary approach is useful in order to characterize solutions of $(\mathcal{P})$ 
which are decomposable in the spirit of Theorem \ref{sliced-sigma}, i.e. which are of the form $\la_\pi$ and $\sigma_\Pi$ (see definitions \eqref{def:llpi} 
and \eqref{def:sigmaPi}) for
suitable  scalar measures $(\pi,\Pi)$ on $\Ob\times\Ob$.
 These measures act as Lagrange multipliers of the two-point constraint and they will encode an optimal truss-like solution $(\lambda_\pi, \sigma_\Pi)$
 to $(\mathcal{P})$ if they are minimal in the following problem:
\begin{equation}\tag{$\mathscr{P}$}
	\inf\Big\{ \Jtwo(\pi,\Pi)\, :\, (\pi,\Pi) \in \Atwo \Big\}
\end{equation}
 where $\Jtwo$ is a suitable  a suitable convex local functional on measures (see \eqref{defJtwo}) and $\Atwo$ denotes the class of
pairs $(\pi,\Pi) \in \Mes(\Ob\times\Ob;\R^2)$ such that $(\lambda_\pi, \sigma_\Pi)\in\A$. One can check easily that 
 \begin{equation}\label{defAtwo}
	(\pi,\Pi) \in \Atwo \quad \iff \quad   \begin{cases} 
	(i) &\int \big(u(y)-u(x) \bigr)\, \pi(dx dy) = \pairing{f,u}  \qquad \ \forall\, u\in C_{\Sigma_0}(\Ob),\\
	(ii) & \int \langle w(y)-w(x),\tau^{x,y} \rangle\, \Pi(dx dy) = 0   \qquad \forall\, w\in C_0(\Omega;\Rd),\\
	(iii) &  \Pi\ge 0 \quad \text{on $\Obsq$}
	\end{cases} 
\end{equation}
Note that, by a density argument, conditions (i) and (ii) above hold once they are checked for smooth pairs $(\f,\phi)
\in \D(\Ob\setminus \Sigma_0)\times \D(\O;\R^d)$.

\medskip
The main point of this section consists in showing  the following equalities
\begin{equation} \label{Z0=I0}   I_0(f,\Sigma_0) = Z_0 = \min (\mathcal{P}) =  \inf (\mathscr{P})
	\end{equation}
that render the zero duality gap and moreover allow to choose truss structures $(\lambda_{\pi},\sigma_{\Pi})$ as minimizing sequences for the $\mathrm{(OM)} $ problem.
This result is established in Section \ref{duality} after a preliminary Section \ref{prelim} devoted to some properties of the 
convex set $\K$ and to the  duality properties of
the functionals $J$ and $\Jtwo$ which appear in problems $(\mathcal{P})$ and $(\mathscr{P})$, respectively. 	
 Before proceeding several remarks are in order:
 
 \begin{remark}\label{rem-dual}\ If we restrict the supremum in the right hand member of \eqref{dualOM} to pairs $(u,w)\in \K$ such that $w\equiv 0$, we recover $I(f,\Sigma_0)$ defined in \eqref{def:If} up to factor $\sqrt{2}$. This discrepancy is due to
the $\frac1{2}$ factor in front of $\nabla u\otimes\nabla u$ that we set in order to better fit to the mechanical viewpoint and in particular 
with applications of our theory to several issues in civil engineering mentioned in the introduction.
 \end{remark}
 
\begin{remark}\label{rem-relaxdual}\
 In contrast with the $\mathrm{(FMD)}$ problem where the dual attainment in \eqref{def:If} is straightforward, here a major difficulty
	is that we cannot ensure the existence of Lipschitz maximizers for \eqref{dualOM}. 
	Indeed, as will seen later, the convex subset  $\mathcal{K}$ is merely bounded in $C^{0,\frac1{2}}(\Omega)\times W^{1,1}(\Omega;\R^d)$
	and relaxed solutions solutions $(u,w)$ may appear as a limit of maximizing sequences and this limit could be even discontinuous.
	A complete characterization of the closure of $\mathcal{K}$ in $C_0(\Omega)\times L^1(\Omega;\R^d)$ that employs a suitable extension of the two-point condition \eqref{eq:two_point_condition}
	will appear in Section \ref{sec:geodesic} within the framework  of maximal monotone maps. \end{remark}
 
 \begin{remark}\label{rem-notruss}\ A major drawback of the two-point duality strategy  is the lack of solutions to  problem $(\mathscr{P})$ in the general case, which will be confirmed by a simple counter-example  (see Remark \ref{notruss}).  This failure of existence 
 is mainly due to  the fact that minimizing sequence $(\pi_n,\Pi_n) = (\alpha_n \, \Pi_n,\Pi_n)$, despite the natural estimate 
 $ \sup_n \Jtwo(\pi_n,\Pi_n) = \sup_n   \int  (1+ \frac{ \a_n^2}{2}) \, |x-y|\, \Pi_n(dxdy) <+\infty$, can exhibit 
 very large mass concentrations of $\Pi_n$ on the diagonal $\Delta$.
 The same kind of difficulty arises in the mathematical approach towards Michell's truss problem developed in \cite{bouchitte2008}.
 Nevertheless, we expect that for finitely supported loads $f$ a solution $(\pi,\Pi)$ always exists thanks to an extension argument for monotone maps that we present  in Section \ref{sec:geodesic} as a conjecture. Should it be true, the optimal discrete measure $\Pi$ induces an optimal design $\sigma_{\Pi}$ composed of bars (or strings) as it is confirmed analytically through the examples  
 in Section \ref{opticond} and numerically through the simulations in Section \ref{sec:numerics}.
 
 \end{remark}
	
\begin{remark}\label{pi>0}\
 The admissible set $\Atwo$ for $(\mathscr{P})$ \textit{a priori}  involves signed measures $\pi$
 while  $\sigma_\Pi\ge 0$ requires that $\Pi\ge 0$.
 In fact the infimum of  $(\mathscr{P})$ is unchanged if one restrict to  measures $\pi$
 of the kind $\pi=\alpha \Pi$ with $\a\ge 0$ or, as will appear more natural in the context of Monge-Kantorovich problem (see Section \ref{MK_sec}), with $\a\le 0$. Indeed, if $(\a \, \Pi, \Pi)$ is an element of the admissible set $\Atwo$, then
 so is $(|\a|\, \tilde{\Pi}, \tilde{\Pi})$ where $\tilde{\Pi}$ is defined by $\langle \tilde{\Pi}, \f \rangle := \int_{\a\ge 0} \f(x,y) \, d\Pi + \int_{\a< 0} \f(y,x) \, d\Pi$ for every $\f\in C^0(\Ob^2)$. Clearly this new  admissible pair shares the same energy i.e. $\Jtwo(|\a|\, \tilde{\Pi}, \tilde{\Pi})= \Jtwo(\a\, \Pi, \Pi).$  This will be also  the case for the admissible pair  $(-|\a|\, \widehat{\Pi}, \widehat{\Pi})$
 where $\widehat{\Pi}$ is the image of $\tilde{\Pi}$ under the map $(x,y)\mapsto (y,x)$.  

\end{remark}	
	

	\subsection{Preliminary results } \label{prelim} 
	
	\medskip
	\begin{lemma}
		\label{lem:two_point_condition}
		Let there be given $(u,w) \in \mathrm{Lip}(\Ob)^{1+d}$; then the pointwise constraint \eqref{pointwise} is equivalent to the two-point condition
		\eqref{eq:two_point_condition}. In particular they imply that the function $v:= \ident -w$ is monotone on $\Ob$. In addition
		the equality is reached in \eqref{eq:two_point_condition} for $(x,y)$ if and only if $u$
		and $\pairing{v(\cdot),y-x}$  are affine functions on $[x,y]$.
			\end{lemma}

	\begin{proof}
		First let us assume that condition \eqref{pointwise} is satisfied. Let   $(x,y) \in \Ob^2$  be  any pair of distinct points and set $\tau = \frac{y-x}{\abs{y-x}}$. Since the trace of  $(u,w)$ on the segment  $[x,y]$ is Lipschitz continuous, we have
		\begin{align*}
		\pairing{w(y)- w(x),y- x} &=\abs{y- x}  \int_{[x,y]} \pairing{\nabla w,\tau} \, d\Ha^1 =\abs{y- x} \int_{[x,y]} \pairing{e(w),\tau \otimes \tau} \, d\Ha^1 \\
		\abs{u(y)- u(x)}^2 &= \abs{\int_{[x,y]}  \pairing{\nabla u,\tau } \,d\Ha^1}^2 \leq \abs{y- x} \int_{[x,y]} \pairing{\nabla u \otimes \nabla u, \tau \otimes \tau } \,d\Ha^1 ,
		\end{align*}
		where in the second line we used  Schwarz's inequality.
		From the above and taking \eqref{pointwise} into account, we deduce that
		\begin{equation*}
		\frac{1}{2}\, \abs{u(y)- u(x)}^2 + \pairing{w(y)- w(x),y- x} \leq  \abs{y- x} \int_{[x,y]} \pairing{\frac{1}{2}\,\nabla u \otimes \nabla u + e(w), \tau \otimes \tau } \,d\Ha^1 \  \leq \abs{y-x}^2.
		\end{equation*}
		Conversely assume that \eqref{eq:two_point_condition} holds and let us chose a point of differentiability $x\in \Omega$ for $(u,w)$.
		Then, by taking any   $\tau \in S^{d-1}$ and $y= x + h\, \tau$ for small $h>0$, we infer  directly from the two-point condition \eqref{eq:two_point_condition} that
		\begin{equation*}
		\pairing{\frac{1}{2}\,\nabla u(x) \otimes \nabla u(x)+ e(w),\tau \otimes \tau}
		=\lim_{h\to 0} \left\{ \frac{1}{2} \, \left(\frac{u(x + h \, \tau) -u(x)}{h} \right)^2 \! + \pairing{\frac{w(x+h\, \tau)-w(x)}{h},\tau} \right\} \leq 1.
		\end{equation*}
		Due to arbitrariness of $\tau$ we are led to the desired inequality \eqref{pointwise} since the differentiability of the Lipschitz map $(u,w)$ holds  a.e. $x\in \Omega$. The asserted equivalence is established. Moreover, by \eqref{eq:two_point_condition}, the function $v=\ident-w$ satisfies
		$\frac{1}{2}\, \abs{u(y)- u(x)}^2 \le \pairing{v(y)- v(x),y- x}$, hence the monotonicity property. 
	 In order to check the last statement let us consider a pair $(x,y)$ where  \eqref{eq:two_point_condition} holds with an equality. Then, the Schwarz's inequality mentioned above becomes an equality and therefore $\pairing{\nabla u,\tau }$ is a constant $\alpha$ on $[x,y]$. 
	Then the scalar Lipschitz function $\f(t)=\pairing{v(x+t \tau), \tau}$ satisfies $\f' \ge \frac{\alpha^2}{2}$ a.e. while 
	$\f(1)-\f(0)= \frac1{2} \frac{(u(y)-u(x))^2}{|y-x|^2} =  \frac{\alpha^2}{2}$. Thus $\f$ has a constant slope $\frac{\alpha^2}{2}$.  
	\end{proof}
	
	\begin{remark}  \ If $\Omega$ is a general domain the equivalence stated in Lemma \ref{lem:two_point_condition} is still valid if  we restrict the condition  \eqref{eq:two_point_condition}
		to those pairs $(x,y)$ satisfying $[x,y]\subset \Ob$. 
		
	\end{remark}

	\subsubsection*{ Construction and properties of the functional $J$}  As directly related to the constraint \eqref{pointwise}, we consider the following subset
	\begin{equation} \label{def:C}
	\C:=\left\{ (z,M)\in \R^d \times \Sdd\, :\, \frac1{2}\, z\otimes z +M\le \Ident   \right\}
	\end{equation}
	which can be seen as the level set $\{g\le 1\}$ of the function  $g: \R^d \times \Sdd \to \R_+$ defined by
	\begin{equation}\label{def:g}
	g(z,M) \ :=\ \rho^+ \!\left( \frac1{2}\, z\otimes z + M \right)\!,
	\end{equation}
	$\rho^+$ being the semi-norm on $\Sdd$ given by
	\begin{equation}\label{def:rho+}
	\rho_+ (A) \ :=\ \sup \Big\{ \pairing{S,A}\,:\, S\ge 0,\ \ \tr \, S\le 1 \Big\}.
	\end{equation}
		
	\begin{lemma}\label{rho+} 
		For $A\in \Sdd$ let  $\sigma_A$ be the set of eigenvalues of $A$. Then:
		\begin{itemize}[leftmargin=1.5\parindent]
			\item[(i)]  $\displaystyle \rho_+ (A)=  \min \{s\ge 0 : A \le  s\, \Ident\} = \max \big\{\lambda_+: \lambda \in \sigma_A \big\}$.
			\item[(ii)] Assume moreover that $\rho_+ (A)=1$ and by ${\rm i}_A$ denote the multiplicity of eigenvalue $1$.
			Then any $S$ optimal in \eqref{def:rho+} satisfies:
			$ \rk(S) \ \le\ {\rm i}_A \ \le \ \rk(A).$
			In particular, we have  $\rk(S) \le d-1$ if $A\not=\Ident$.
		\end{itemize}
		
	\end{lemma}
	\begin{remark}\label{rankS}
		The rank-one property of optimal $\sigma=S\mu$ in the $\mathrm{(FMD)}$ problem  obtained in Proposition \ref{mop=fmd} can be seen as a consequence of Lemma \ref{rho+} applied to the rank-one tensor function $A=\nabla_\mu u\otimes \nabla_\mu u$. For $d=2$ note also that  $\rk(S)=1$ is true unless $A= \rho_+(A) \, \Ident$.
	\end{remark}
	\begin{proof} The assertion (i) is straightforward by evaluating the supremum in \eqref{def:rho+} with $S= a\otimes a$ where $|a|=1$.
		Let now $A$ and $S\ge 0$ be such that $1= \rho_+(A) = \pairing{S,A}$ and $\tr S =1$.
	We can chose an orthonormal base $\{a_i ,\, 1\!\le\! i\!\le\! d\}$ so that $A =\sum_i \lambda_i \, a_i\otimes  a_i $ and $S =\sum_i \mu_i\, a_i\otimes  a_i $ 
		where the real eigenvalues $\lambda_i, \mu_i$ satisfy:
		$$ 1= \lambda_1 = \max_i  \lambda_i, \qquad  \mu_i\in [0,1], \quad \sum_i \mu_i =1  .$$
		Then the equality $1=\pairing{S,A}$ is equivalent to $\sum_i \mu_i (1-\lambda_i) =0$. Since $\lambda_i\le 1$,  this implies that $\mu_i=0$
		for every index $i$ such that $\lambda_i<1$. Thus $\rk(S)= \sharp\bigl(\{ i : \mu_i>0\}\bigr) \le  \sharp\bigl(\{ i : \lambda_i=1\}\bigr)= i_A$.
	\end{proof}
	
	\begin{lemma}\label{CA}\ The function $g$ defined in \eqref{def:g} is convex continuous. Therefore $\C$ is
	a closed  (unbounded) convex subset of $\R^d \times \Sdd$. Its support function is given by
		\begin{equation}\label{C*}
		\ind_\C^*(\theta, S) =\begin{cases} \tr\, S + \frac1{2} \pairing{S, q\otimes q}  & \text{if $S\in \Sddp,\ q\in \Rd$  \ and \  $\theta = S q$,   }\\
		+\infty & \text{if $S\notin \Sddp$ \ or if \  $\theta \notin  \IM(S)$ }
		\end{cases}
		\end{equation}
		and we have the following lower bound: 
		\begin{equation}\label{coercif}
		\ind_\C^*(\theta, S)\ \ge\ \frac1{2}\, \tr\, S +  |\theta|.
		\end{equation}
	\end{lemma}
	We notice that the expression given in \eqref{C*} for $\theta \in \IM(S)$ is independent of the choice of $q$ such that $\theta= S\, q$. With a small abuse of notation, we will sometimes write $\ind_\C^*(\theta, S)=\tr S + \frac1{2} \pairing{S^{-1} \theta,\theta}.$ 
	\begin{proof} \ It is is easy to check that $g$ is locally bounded and that  $g = \sup_{|a|=1} g_a$  where, for every $a\in\R^d$, $g_a$ is the convex continuous function  given by
		$$ g_a(z,M) := \frac1{2} |\pairing{z,a}|^2 + \pairing{ M a, a}.$$
		It follows that $g$ is convex continuous hence
		$\C$ is a  closed convex subset of $\R^d \times \Sdd$. 
		
		Let us now prove \eqref{C*}. First, by considering  pairs $(0, - s\, \Ident)$  which belong to $\C$ for $s>0$ being arbitrarily large,
		we infer that $\ind_\C^*(\theta, S) = +\infty$ unless $S\ge 0$. Next, assuming that $S\ge 0$, we  compute
		\begin{align*}
		\ind_\C^*(\theta, S) &= \sup_{z\in \R^d} \sup_{M\in\Sdd} 
		\left\{\pairing{\theta,z} + \pairing{M,S}\ :\ \frac1{2}\, z\otimes z +M\le \Ident  \right\}\\
		&=  \sup_{z\in \R^d}  \left\{\pairing{\theta,z} +  \sup_{M\in\Sdd} \Big\{\pairing{M,S}\ :\ M\le \Ident- \frac1{2}\, z\otimes z\Big\}   \right\} \\
		&=   \tr\, S  + \sup_{z\in \R^d}\left\{  \pairing{\theta,z} - \frac1{2} \, \pairing {Sz,z}   \right\}.
		\end{align*} 
		Clearly the supremum with respect to $z$ is infinite if $\theta$ is not orthogonal to the kernel of $S$. If it is not the case, then
		$\theta\in \IM(S)$ and the concave function $z\mapsto \pairing{\theta,z} - \frac1{2}  \pairing {Sz,z}$ achieves its maximum 
		at $z=q$ for $q$ being any solution of $S q=\theta$, thus rendering the maximum equal to $\frac1{2} \pairing{Sq,q}$.
		
		Eventually, we deduce \eqref{coercif} by noticing that the right hand side of the inequality coincides with the support function
		of $\bigl\{ (z,A): |z|\le 1 \ ,\ 2 A\le  \Ident \bigr\}$ which clearly is a subset of $\C$.
	\end{proof}
	\medskip
	
	Following \cite{Goffman}, to the one-homogeneous  integrand 
	$\ind_\C^*:\R^d \times \Sdd \to [0,+\infty]$ and any pair $(\lambda,\sigma)\in \Mes(\Ob;\R^d \times \Sdd)$ we can  associate a scalar measure defined for all Borel subsets $B\subset\Ob$ by:
	$$ \int_{B} \ind_\C^*(\lambda,\sigma)  :=  \int_{B} \ind_\C^*\Big(\frac{d\lambda}{dm},\frac{d\sigma}{dm}\Big) \, dm$$
	for $m$ being any measure in $\Mes_+(\Ob)$ such that $(\lambda,\sigma) \ll m$ (the choice of $m$ is immaterial due to the homogeneity of $\ind_\C^*$).
	By fixing $B=\Ob$ we obtain a functional depending on $(\lambda,\sigma)$: 
	\begin{equation}\label{defJ}
	J(\lambda,\sigma) := \int_{\Ob} \ind_\C^*(\lambda,\sigma)  \qquad\quad \forall\, (\lambda,\sigma)\in \Mes(\Ob;\R^d \times \Sdd). 
	\end{equation}
	Convex one-homogeneous functionals on measures of this type have been studied in \cite{bouvalIHP}. In particular, $J$ can be characterized in terms of the duality between $\Mes(\Ob;\R^d \times \Sdd)$ and $C^0(\Ob;\R^d\times\Sdd)$:
	\begin{lemma}\label{1hom}
		The following statements hold true:
		\begin{itemize}[leftmargin=1.5\parindent]
			\item [(i)] The functional $J$ is convex, weakly* lower semicontinuous and $1$-homogeneous on $\Mes(\Ob;\R^d \times \Sdd)$.
			Moreover, in order that  $J(\lambda,\sigma)<+\infty$, it is necessary that $\sigma \in \Mes(\Ob;\Sddp)$ and $\lambda \ll\mu:=\tr\, \sigma$. In this case 
			$$  J(\lambda,\sigma) = \begin{cases}   \int \tr \,\sigma + \frac1{2} \int \pairing{Sq,q}\, d\mu & \text{where  $\sigma= S\mu$ and $ \lambda=Sq\, \mu$}, \\
			+\infty & \text{if} \ \mu \Big( \big\{\frac{d\lambda}{ d\mu} \notin \IM(S)\big\} \Big)>0.
			\end{cases}$$
			\item [(ii)] $J$ is the support function of the set $\U_\C$ of continuous selections of $\C$ defined by:
			$$\U_\C:=\Big\{(\psi,\Psi)\in C^0(\Ob;\Rd \times \Sdd)\, :\, (\psi(x),\Psi(x))\in \C \ \ \forall \, x\in \Ob\, \Big\}.$$
			In particular, for every $\mu\in\Mes_+(\Ob)$ and  $(q,S)\in L^1_\mu(\Ob; \R^d\times \Sddp)$, we have 
			$$ J(q\mu,S\mu)  \ =\ \sup \left\{ \int \Big(\pairing{q,\psi} + \pairing{S,\Psi} \Big)\, d\mu \, :\, (\psi,\Psi)\in \mathcal{U}_\C\right\};$$
			\item [(iii)] Let $\Comp(\sigma)$ be the compliance functional defined in \eqref{eq:comp_FMD}. Then we have:
			\begin{equation}\label{dualcomp}
			\Comp(\sigma) + \int \tr \, \sigma \ =\ \min   \Big\{ J(\lambda,\sigma) \, : \,  \lambda \in \Mes(\Ob;\R^d), \ -\dive \lambda = f \text{ in $\D'(\R^d\setminus \Sigma_0)$} \Big\}.
			\end{equation}
			
		\end{itemize}
		
	\end{lemma}
	\begin{proof}\ For (i),  we notice  that  $\ind_\C^*(z,0)<+\infty$ implies that $z=0$. Therefore if $\sigma =S\mu$ and  $\lambda = {\theta\, \mu + \theta_s \, m_s}$ is the Lebesgue decomposition of $\lambda$ with respect to $\mu$ (with $m_s \perp \mu$), then $J(\lambda,\sigma)= \int {\ind_\C^*(\theta,S)}\, d\mu +
	  \int {\ind_\C^*(\theta_s,0)}\, dm_s <+\infty $ implies that $\theta_s=0$, thus $(\lambda,\sigma) = (\theta,S)\,\mu$. The integral representation of $J(\lambda,\sigma)$ follows from
		the definition \eqref{defJ} and Lemma \ref{CA}.
		
		For the assertion (ii), we refer to \cite[Thm 5]{bouvalIHP} and \cite[Thm 11]{valadier}. Let us now establish assertion (iii). 
		Let $\sigma= S\mu$ with $S\in L^\infty_{\mu}(\Ob;\Sddp)$ and $\tr\, S=1$. In view of the expression for $J$ obtained in (i) we have to show that:
		$$  \inf \left\{ \frac1{2}\int  \pairing{Sq,q} \, d\mu\  :\  -\dive(Sq \mu)= f\ \text{ in $\D'(\R^d\setminus \Sigma_0)$}\right\}=
		\sup_{\substack{u\in C^1(\R^d)\\ u=0\, \text{on $\Sigma_0$}}} \left\{  \pairing{f,u} - \frac1{2}\int  \pairing{S\,\nabla u,\nabla u}\, d\mu \right\}.$$
		This equality follows from standard duality arguments in $L^2_\mu$ where in addition we exploit the density result of Lemma \ref{density} to show that the divergence condition  $-\dive(Sq \mu)= f$  in $\D'(\R^d\setminus \Sigma_0)$
		is equivalent to the equality  $\pairing{f,v}= \int \pairing{Sq, v} d\mu$ holding for every $v\in C^1(\R^d)$ that vanishes on $\Sigma_0$. 
	\end{proof}
	\subsubsection*{ Construction of the functional $\Jtwo$} \
	We introduce the closed (unbounded) convex subset of $\R^2$:
\begin{equation}\label{def:Ctwo}
 \Ctwo := \left\{ (s_1,s_2) \in \R^2 \, :  \, \frac{1}{2} \, (s_1)^2 + s_2\,  \leq 1 \right\}
\end{equation}
which is directly related with the  two-point condition \eqref{eq:two_point_condition}. Indeed, if for distinct pairs $(x,y) \in \Ob \times \Ob$ we set
\begin{equation}
	\label{defLambda}
	 \zeta_1(x,y) = u(y)-u(x), \qquad \zeta_2(x,y):=\pairing{w(y)-w(x),\frac{y-x}{\abs{y-x}}},
\end{equation}
then we can rewrite \eqref{eq:two_point_condition} as  $(\zeta_1,\zeta_2) \in \abs{x-y}\, \Ctwo$ for all $(x,y) \in \Ob\times\Ob$ (where $\zeta_2(x,y) = 0$ for $x=y$ by convention).


In what follows, a pair of scalar measures $(\pi, \Pi)$ on $\Ob \times \Ob$ will play the role of Lagrange multipliers for the two-point constraint.
The convex functional $\Jtwo(\pi,\Pi)$  involved in the truss problem $(\mathscr{P})$ 
will be associated with the support function $\ind_{\Ctwo}^*$ of the convex set $\Ctwo$. 
Recalling definition \eqref{def:C} and formula \eqref{C*}, it turns out that this convex one homogeneous integrand $\ind_{\Ctwo}^*$ is closely related to $\ind_\C^*$:
\begin{lemma}\label{Ctwo}  The support function of $\Ctwo$ is a one-homogeneous convex l.s.c. non-negative function given by
		\begin{equation}\label{Ctwo*}
		\ind_{\Ctwo}^*(t_1, t_2) =  \begin{cases}
		\left(1+\frac{1}{2}\,a^2 \right) t_2   & \text{if \ $t_2 \geq 0$ \ and \  $t_1 = a\, t_2$, }\\
		+\infty & \text{if \ $t_2 <0$ \ or if \ $t_2=0, t_1 \neq 0$ . }
		\end{cases}
		\end{equation}
		It satisfies the relation 
		\begin{equation}
		\label{Ctwo_C}
			\ind_{\Ctwo}^*(t_1, t_2) = \ind_\C^*(t_1\,\tau,\,t_2\, \tau \otimes \tau) \qquad \forall\, \tau \in S^{d-1}.
		\end{equation}	

%
\end{lemma}
\begin{proof}\ The formula \eqref{Ctwo*} follows directly from \eqref{C*} once we have shown  \eqref{Ctwo_C}.
 For every $(z,M) \in \C$ and any $\tau \in S^{d-1}$ we have $ \frac{1}{2} \bigl(\pairing{z,\tau} \bigr)^2\! + \pairing{M,\tau\otimes \tau} = \pairing{\frac{1}{2}\, z\otimes z +M,\tau\otimes \tau} \leq 1$, thus $(s_1,s_2)=\bigl(\pairing{z,\tau}, \pairing{M,\tau \otimes \tau} \bigr)$ belongs to $\Ctwo$ and therefore 
  $$\ind_{\Ctwo}^*(t_1, t_2)\ \ge \ \sup \Big\{ t_1 \, \pairing{z,\tau} + t_2 \, \pairing{M,\tau \otimes \tau}\ :\ (z,M)\in \C \Big\}  =  \ind_\C^*(t_1\,\tau,\, t_2\, \tau \otimes \tau). $$
 To obtain the opposite inequality, we associate to any element $(s_1,s_2) \in \Ctwo$  the pair $(z,M)\in \C$ defined by $z= s_1\, \tau$ and  $s_2\, \tau \otimes \tau$. We are led to
 $$ \ind_\C^*(t_1\,\tau,t_2\, \tau \otimes \tau) \ \ge \ \sup \Big\{ s_1\, t_1 + s_2\, t_2 : (s_1,s_2) \in \Ctwo\Big\}  =  \ind_{\Ctwo}^*(t_1, t_2).$$ 
 \end{proof}

\medskip
Following the Goffman-Serrin construction, we may associate  to any
 $(\pi,\Pi) \in\Mes(\Ob\times\Ob;\R^2)$ the non-negative Borel measure  $\ind_{\Ctwo}^*(\pi,\Pi)$ ranging in $[0,+\infty]$.
  The functional $\Jtwo$ is defined as follows:
  \begin{equation}\label{defJtwo}
	\Jtwo(\pi,\Pi) := \int_{\Ob \times \Ob} \abs{x-y}\  \ind_{\Ctwo}^*\bigl(\pi, \Pi\bigr)(dxdy) \qquad \forall\, (\pi,\Pi) \in \Mes(\Ob\times\Ob;\R^2).
\end{equation}
\begin{lemma}\label{1hom_two}
 The functional $\Jtwo$ is convex, weakly* lower semicontinuous and $1$-homogeneous on $\Mes(\Ob\times\Ob;\R^2)$.
In order that  $\Jtwo(\pi,\Pi)<+\infty$, it is necessary that $\Pi \ge 0$ and $\pi \ll \Pi$ on $\Obsq$. In this case, it holds that
  \begin{equation}
\label{formula_Jtwo}
\Jtwo(\pi,\Pi) = \int_{\Ob\times\Ob} |x-y|\,  \left( 1 + \frac{1}{2} \left(\frac{d\pi}{d\Pi}\right)^2\right) \Pi(dx dy).
\end{equation}
Furthermore, for every $(\pi,\Pi) \in \Mes(\Ob;\R^2)$  we have the inequality
		\begin{equation}\label{J<Jtwo}
			J(\lambda_\pi, \sigma_\Pi) \ \le\ \Jtwo(\pi, \Pi).
		\end{equation}
%
%
		
		
%
%
%
\end{lemma}

\begin{remark} Although one has $J(\a \lambda^{x,y}, \sigma^{x,y})= \Jtwo( \a \delta_{(x,y)},  \delta_{(x,y)})$  for every pair $(x,y)$ and $\a \in\R$,
 the inequality \eqref{J<Jtwo} is strict in general. This happens for instance when $\Pi$ charges two pairs $(x,y)$ and $(x',y')$ such that  $\Ha^1([x,y] \cap [x',y'])>0$. We expect however that it is possible to avoid this discrepancy by substituting $(\pi,\Pi)$ with a suitable equivalent 
pair $(\tilde{\pi}, \tilde{\Pi})$ such that  $(\lambda_{\tilde{\pi}}, \sigma_{\tilde{\Pi}})= (\lambda_\pi, \sigma_\Pi)$ and no overlapping between segments $[x,y]$  occurs as $(x,y)$ runs over $\spt(\tilde{\Pi})$.
\end{remark}
\begin{proof} 
	Since the  integrand  $\ind_{\Ctwo}^*$ is convex, 1-homogeneous and lower semicontinuous from $\R^2$ to $[0,+\infty]$, 
	we know that, for every open subset $B \subset \Ob\times\Ob$, the map 
	$$ (\pi,\Pi)\in \Mes(\Ob\times\Ob;\R^2) \mapsto \ind_{\Ctwo}^*(\pi,\Pi)(B)\ $$
	is weakly* lower semicontinous. In particular applying this to each level set $B_t=\{|x-y|>t\}$, it follows  fom Fatou's lemma that 
	$$ \liminf_n \Jtwo(\pi_n,\Pi_n) =\liminf_n \int_0^{+\infty} \ind_{\Ctwo}^*(\pi_n,\Pi_n) (B_t)\, dt \ge\int_0^{+\infty}
	 \ind_{\Ctwo}^*(\pi,\Pi)(B_t)\, dt = \Jtwo(\pi,\Pi) $$
	 whenever  $(\pi_n,\Pi_n)\weak (\pi,\Pi)$, hence the desired lower semicontinuity property of $\Jtwo$.
Further, if $\pi= \alpha \, \Pi + \pi_s$ with $\pi_s \perp \Pi$ denotes the Lebesgue-Nikodym decomposition of $\pi$ with respect to $\Pi$, we have
$\ind_{\Ctwo}^*(\pi,\Pi)= \ind_{\Ctwo}^*(\alpha \, \Pi,\Pi)+ \ind_{\Ctwo}^*(\pi_s,0)$, hence by Lemma \ref{Ctwo} for every $t>0$:
\begin{align*}
 \ind_{\Ctwo}^*(\pi,\Pi)(B_t)= \begin{cases}  \int_{} (1+ \frac{\a^2}{2})\, d\Pi \ \ & \text{if \ $\Pi\mres B_t \ge 0$ \ and \   $|\pi_s|(B_t)=0$,} \\ +\infty & \text{otherwise}.
 \end{cases}\end{align*}
 Since $\cup_{t>0} B_t= \Obsq$, 
 by integrating with respect to $t$ between $0$ and $+\infty$,  we recover the conditions required for the finiteness of $\Jtwo(\pi,\Pi)$ and the 
 equality \eqref{formula_Jtwo}.
 
  Let us now prove \eqref{J<Jtwo}. It is not restrictive to assume that $\Jtwo(\pi,\Pi)<+\infty$ so that $\Pi\ge 0$ and $\pi= \a \Pi$ on 
  $\Obsq$ for a suitable Borel function $\a$.
%
Then, for any continuous selection $(\psi,\Psi)$ of the convex set $\C$ (see \eqref{def:C}), we have
	\begin{align*}
		\pairing{\lambda_\pi,\psi} + \pairing{\sigma_\Pi,\Psi} &= \int_{\Obsq} \Big(\a(x,y) \pairing{\lambda^{x,y},\psi} +\pairing{\sigma^{x,y},\Psi}\Big)\, \Pi(dx dy)\\
		&=\int_{\Obsq} \left( \int_{[x,y]} \Big(\a(x,y)\, \pairing{\tau^{x,y},\psi} + \pairing{\tau^{x,y} \otimes \tau^{x,y},\Psi} \Big) d\Ha^1 \right) \Pi(dx dy)\\
		&\leq \int_{\Obsq} \left( \int_{[x,y]} \ind_{\Ctwo}^*  \bigl(\a(x,y),1\bigr) d\Ha^1 \right) \Pi(dx dy)\\
		& = \int_{\Obsq} |x-y| \ \ind_{\Ctwo}^* \! \left(\pi,\Pi\right)\!(dxdy)  \le	\Jtwo(\pi,\Pi)
	\end{align*}
	where, to pass from  second to the third line,  we used that $\bigl(\psi(z),\Psi(z)\bigr)$ belongs to the convex set $\C$  for all $z \in [x,y]$ so that, in virtue of \eqref{Ctwo_C}, we have  $ \bigl(\pairing{\tau^{x,y},\psi(z)},\pairing{\tau^{x,y} \otimes \tau^{x,y},\Psi(z)} \bigr) \in \Ctwo$.
	We conclude with the desired inequality upon recalling the assertion (ii) of Lemma \ref{1hom}
	and by taking the supremum of the left hand side with repect to all pairs $(\psi,\Psi) \in \U_\C$. 	
	\end{proof}
\subsubsection*{A density result and generalized integration by parts}

We first need the following  approximation result:
	\begin{lemma}
		\label{density}
		Let us assume $(u,w) \in (\mathrm{Lip}(\Ob))^{1+d}$ such that $u=0$ on $\Sigma_0$ and  $w=0$  on $\bO$; then there exists a sequence $(u_n,w_n)$ and  and a constant $M>0$ such that: 
		\begin{itemize}[leftmargin=1.5\parindent]
			\item [(i)] \ $(u_n,w_n)\in \D(\R^d\setminus \Sigma_0)\times \D(\Omega)$ and  $(u_n,w_n) \to (u,w)$  uniformly in $\Ob$;
			\item[(ii)]\ $\mathrm{Lip}(u_n)  +  \mathrm{Lip}(w_n) \le M $;
			\item [(iii)] \ $\dis\limsup_{n\to\infty}\   \|g\bigl(\nabla u_n, e(w_n)\bigr)\|_{L^\infty(\Omega)} \le \| g\bigl(\nabla u, e(w)\bigr)\|_{L^\infty(\Omega)}$,
			where $g$ is defined in \eqref{def:g}. 
		\end{itemize}
	\end{lemma}
	
	%
	%
	
	\begin{proof}  Let $(u,w)$ be as given in the statement. It is not restrictive to assume that $0$ belongs to the interior of the convex set $\Omega$
		and by  $p_\O$ we denote the element of $\mathrm{Lip}_1(\R^d;\Ob)$ defined by
		$$  p_\O(x) = x \quad \text{if \ $x\in\Ob$}, \qquad   p_\O(x) = \frac{x} {j_\O(x)}\quad   \text{if \ $\R^d\setminus \Ob$}$$
		where $j_\O(x) = \min \bigl\{t \ge 0 : x \in t\, \Ob \bigr\}$ is the gauge of $\O$. 
		We proceed in three steps: 
		
		\med
		{\em Step1}\quad  We construct a Lipschitz extension $(\tilde u, \tilde w)$ of $(u,w)$ to whole $\R^d$ by setting
		$$  \tilde w = 0 \quad \text{in \ $\R^d\setminus \Ob$}, \qquad \tilde u = u \circ p_\O .$$
		This extension is bounded and satisfies 
		$$ \| g\bigl(\nabla \tilde u, e(\tilde w)\bigr)\|_{L^\infty(\R^d)}   \ =\ \| g\bigl(\nabla u, e(w)\bigr)\|_{L^\infty(\Omega)}.$$
		Indeed, if $k$ denotes the right hand side, then  $g\bigl(\nabla (\frac{u}{\sqrt{k}}) , e(\frac{w}{k}) \bigr)\le 1$ a.e. and by  
		Lemma \ref{lem:two_point_condition} we have the two-point condition
		\begin{align*}\frac{1}{2}\, \abs{\tilde u(y)- \tilde u(x)}^2 + \pairing{\tilde w(y)-\tilde  w(x),y- x} &\ =\ \frac{1}{2}\, \abs{u\bigl(p_\O(y)\bigr) -  u\bigl(p_\O(x)\bigr)}^2 + \pairing{ w\bigl(p_\O(y)\bigr)-  w\bigl(p_\O(x)\bigr),p_\O(y)- p_\O(x)}\\
		& \ \leq \ k\, \abs{p_\O(y) - p_\O(x)}^2 \  \leq\  k\,  \abs{y-x}^2
		\end{align*}  
		holding for all $(x,y)\in (\R^d)^2$. 
		
		\med
		{\em Step2}\quad In this step we construct a sequence   $(u_n,w_n)\in \mathrm{Lip}(\R^d;\R^{1+d})$ which fulfils conditions 
		(ii),\,(iii)  of the lemma, but instead of (i), we merely require that $\spt(u_n) \cap\Sigma_0=\varnothing$ and $\spt(w_n) \Subset \O$.
		To that aim, we consider a small parameter $\d>0$ and we contract the pair $(\tilde u, \tilde w)$  defined in Step 1 as follows:
		$$   u_\d (x) = \frac1{1+ \d} \ \tilde u\bigl((1+\d)\, x\bigr), \qquad   w_\d (x) = \frac1{1+ \d}\ \tilde w\bigl((1+\d)\, x\bigr).$$
		Clearly we have  $g\bigl( \nabla u_\d, e(w_\d)\bigr)(x)= g\bigl(\nabla \tilde u, e(\tilde w)\bigr)\big((1+\d)\,x\big)\le \| g\big(\nabla u, e(w)\big)\|_{L^\infty(\Omega)}$
		while $\spt(w_\d) \subset \frac{\Ob}{1+\d} \Subset \O$. By construction we also have 
		$u_\d=0$ on $\Sigma_0$ but unfortunately $\spt(u_\d)$ may touch a part of $\Sigma_0$.
		To remedy this we slightly modify $u_\d$ to arrive at $v_\d$ vanishing on $\{ |u_\d|\le \d\}$; we consider:
		$$   v_\d := u_\d -\d  \quad \text{on $\{ u_\d\le -\d\}$}, \qquad   v_\d := 0  \quad \text{on $\{ |u_\d|\le \d\}$}, \qquad  v_\d := u_\d +\d  \quad 
		\text{on $\{ u_\d \ge\d\}$}.$$ 
		Then, clearly, $v_\d$ is a Lipschitz function such that $\spt(v_\d) \cap\Sigma_0=\varnothing$. 
		Besides one easily checks that $ \nabla v_\d \otimes \nabla v_\d \le  \nabla u_\d \otimes \nabla u_\d$ a.e. hence $ g\big(\nabla v_\d, e(w_\d)\big)\le g\big( \nabla u_\d, e(w_\d)\big)$.
		Eventually, since $|u_\d-v_\d| \le \delta$ and $(\tilde u, \tilde w)$ is bounded and Lipschitz, we have that $(v_\d, w_\d)\to (u,w)$ uniformly in $\Ob$ while keeping the same Lipschitz constant. 
		Hence the Step 2 is completed by taking $(u_n,w_n) =  (v_{\d_n}, w_{\d_n}) $ as approximating sequence with $\d_n \to 0$.
		
		\med
		{\em Step3}\quad  To each pair $(u_n,w_n)$ constructed in Step 2  we apply a smooth convolution kernel $\theta_\e(x)= \e^{-d} \theta(\frac{x}{d})$ where $\theta$ is a radial symmetric element of $\D^+(\R^d)$ such that $\int \theta =1$. Since $\spt(w_n) \Subset \O$ and $\spt(u_n) \Subset \R^d\setminus \Sigma_0$,
		the pair $(u_n^\e, w_n^\e) :=(u_n * \theta_\e, w_n * \theta_\e)$ belongs to $\D(\R^d\setminus \Sigma_0)\times \D(\Omega)$  for $\e$ small enough
		and still satisfies condition (ii).
		On the other hand, since the integrand $g$ is convex lower semicontinuous, it follows from Lemma \ref{mollig} applied to 
		$\xi= \big(\nabla u_n, e(w_n)\big) $ that  
		$$   g\big( \nabla u_n^\e, e(w_n^\e)\big) \ \le\  g\big( \nabla u_n, e(w_n)\big) \ \le\  \| g\big(\nabla u, e(w)\big)\|_{L^\infty(\Omega)}.$$
		Then, by  using a classical diagonalization argument we may choose a sequence $\e_n\to 0$ so that $(u_n,w_n):=(u_n^{\e_n},w_n^{\e_n})$
		satisfies the required conditions (i),\,(ii),\,(iii). The proof of Lemma \ref{density} is complete.
			\end{proof}
	
	As a consequence of Lemma \ref{density}  elements   $(u,w)\in \mathcal{K}$  (see \eqref{def:calK}) which are characterized by condition 
		$\| g(\nabla u, e(w))\|_{L^\infty(\Omega)}\le 1$  can be approximated uniformly by smooth elements of $\K$ that satisfy (i),\,(ii),\,(iii).	Moreover,  by using tangential differential calculus (see Appendix \ref{appendix_mu_calculus}) with respect to a measure, we can deduce an integration by parts formula (see \eqref{intpartmu} below):
	\begin{corollary}\label{dualmu}
		Let  $(u,w) \in \K$ and $\mu\in \Mes_+(\Ob)$. Then:
		\begin{itemize}[leftmargin=1.5\parindent]
			\item[(i)] \ $\nabla_\mu u= 0$  \ $\mu$-a.e. in $\Sigma_0$ and $e_\mu(w)= 0$ \ $\mu$-a.e. in $\bO$;
						\item[(ii)] \ for all  $(\theta,S)\in L^1_\mu(\Ob; \R^d\!\times \!\Sddp)$
			with\,  $-\dive (\theta \mu) =f$ in $\D'(\Sigma_0^c)$ and $\DIV (S \mu)=0$ in $\D'(\Omega)$ it holds that:
			\begin{align}\label{intpartmu}
			&\pairing{f,u} = \int \pairing{\theta, \nabla_\mu u} \, d\mu + \int \pairing{S, e_\mu(w)} \, d\mu, \\
			\label {ineq0}
			& \pairing{\theta, \nabla_\mu u} + \pairing{S, e_\mu(w)} \le  \chi_\C^* (\theta,S) \quad \text{$\mu$-a.e.}
			\end{align} 
		\end{itemize}
	\end{corollary}
	%
		
	\begin{proof} By Lemma \ref{density} there exits an equi-Lipschitz approximating sequence $(u_n,w_n)\in \K$ converging uniformly to $(u,w)$
		and such that $u_n, w_n$ are smooth and compactly supported in $\Sigma_0^c$ and $\O$ respectively. It follows that  
		\begin{equation}\label{Linftymu}
		\big(\nabla_\mu u_n, e_\mu(w_n) \big) \, \weakstar \,  \big(\nabla_\mu u, e_\mu(w)\big)  \quad \text{in $L^\infty_\mu(\Ob; \R^d\times \Sdd)$}.
		\end{equation}
		The assertion (i) is then a consequence of the fact that, for each $n$,  $\nabla_\mu u_n = P_\mu \nabla u_n$ vanishes on $\Sigma_0$ while 
		$e_\mu(w_n)= P_\mu \bigl(e(w_n)\bigr) P_\mu $ vanishes on $\bO$.
		
		Let now  $(\theta,S)\in L^1_\mu(\Ob; \R^d\!\times \!\Sddp)$ as given in the lemma. The condition  $-\dive (\theta \mu) =f$ in $\D'(\Sigma_0^c)$
		implies (see Proposition \ref{byparts}) that $\theta(x)\in T_\mu(x)$ for $\mu$-almost all $x$ in the open subset $\Sigma_0^c$. In the same way the condition
		$\DIV (S \mu)=0$ in $\D'(\Omega)$ implies that $P_\mu S P_\mu = S$ holds $\mu$-a.e. in $\O$.
		Therefore, recalling that $u_n$  vanishes in a neighbourhood of $\Sigma_0$, we have
		\begin{align*} 
		\pairing{f,u_n} =  \int_{\Ob\setminus \Sigma_0} \pairing{\theta, \nabla u_n} \, d\mu + \int_\O \pairing{S, e(w_n)} \, d\mu 
		=  \int_{\Ob\setminus \Sigma_0} \pairing{\theta, \nabla_\mu u_n} \, d\mu + \int_\O \pairing{S, e_\mu(w_n)} \, d\mu.
		\end{align*} 
		Passing to the limit $n\to\infty$ with the help of \eqref{Linftymu} and of the assertion (i)  we obtain  \eqref{intpartmu}.
		In order to derive \eqref{ineq0}, we start with the inequality:
		$$\int_{B\setminus \Sigma_0}  \pairing{\theta, \nabla u_n} \, d\mu + \int_{B\cap\O} \pairing{S, e(w_n)}\, d\mu  \le  \int_B \chi_\C^* (\theta,S)\, d\mu$$
		which holds for any Borel set $B\subset \Ob$ since  $(\nabla u_n, e(w_n)) \in \C$ everywhere in $\Ob$. 
		Passing to the limit with the same arguments as before we get
		$$\int_{B}  \pairing{\theta, \nabla_\mu u} \, d\mu + \int_{B} \pairing{S, e_\mu(w)}\, d\mu  \le  \int_B \chi_\C^* (\theta,S)\, d\mu$$
		thus \eqref{ineq0} follows by localizing. 
	\end{proof}

	\medskip

	\subsection{Primal and dual formulations }\label {duality}
	
	 We recall our primal problem 
	\begin{equation}\tag{{$\mathcal{P}$}}
		\inf_{(\lambda,\sigma) \in \Mes(\Ob; \R^d\times \Sddp) } \Big\{ J(\lambda,\sigma) \, : \, -\dive\, \lambda =f \quad \text{in $\R^d\setminus \Sigma_0$}, \ \ \DIV \,\sigma =0  \quad \text{in $\Omega$} \Big\}
	\end{equation}
	where $J$ is the functional on measures defined in \eqref{defJ}.
	By  minimizing  with respect to $\lambda$ first, with the help of \eqref{dualcomp} we immediately recover that $\inf (\mathcal{P})=Z_0$
	where $Z_0$ is the infimum of the reduced $\mathrm{(OM)}$ problem \eqref{reducedom}. 
	We also recall the truss variant of ($\mathcal{P}$) given by
	\begin{equation} \tag{{$\mathscr{P}$}}
		\inf\Big\{ \Jtwo(\pi,\Pi) : (\pi,\Pi) \in \Atwo \Big\}
	\end{equation}
	where $\Jtwo$ and $\Atwo$ are defined  by \eqref{defJtwo} and \eqref{defAtwo} respectively.
	In virtue of Lemma \ref{1hom_two} and of the definition of the admissible set $\Atwo$, we clearly have the inequality $\inf (\mathcal{P}) \le \inf (\mathscr{P})$.
	
	\medskip
	The dual problem is a relaxed version of the supremum problem \eqref{dualOM}
	that we shall write  as 
	\begin{equation}\tag{{$\mathcal{P}^*$}}
		\sup\left\{ \int u \, df \ : (u,w)\in \bK \right\}
	\end{equation} 
	where  $\bK$ denotes the closure of  the convex constraint $\mathcal{K}$, defined in \eqref{def:calK},
	as a subset of the Banach space $C^0(\Ob)\times L^1(\Omega)$. 
	The choice of this topology is induced by the following estimate result:
	
	\begin{lemma}\label{esti} \ Let $R$ denote the diameter of $\Omega$. Then, for every $(u,w)\in \mathcal{K}$ it holds that:
		\begin{itemize}[leftmargin=1.5\parindent]
			\item [(i)] \ $\|w\|_{L^\infty(\Omega)} \leq R$;
			\item[(ii)]\  $|u(x)-u(y)| \le \sqrt{2R}\,  |x-y|^{\frac1{2}}$ \ for every $(x,y)\in \Ob\times\Ob$; 
			\item[(iii)]\  $ \int_\Omega |Dw| + \int_\Omega |\nabla u|^2 \, dx  \le d \, |\Omega| + c_d \, R^d$ (with $c_d$ depending on $d$ only).
		\end{itemize}
		As a consequence, if $\Sigma_0$ is non-empty, $\mathcal{K}$ is a bounded convex subset
		of $(C^{0,\frac1{2}}(\Ob)\cap W^{1,2}(\Omega)) \times (W^{1,1}\cap L^\infty)(\Omega;\R^d)$, hence it is relatively compact
		in $C^0(\Ob)\times L^1(\Omega;\R^d)$.
	\end{lemma}
	The complete characterization of the compact $\bK$ (as a subset of $C^0(\Ob) \times BV(\Ob;\R^d)$) will be given in  Section \ref{sec:geodesic}.
	\begin{proof} Let us take an element $(u,w) \in \mathcal{K}$. By Lemma \ref{lem:two_point_condition} the two point condition \eqref{eq:two_point_condition}
		holds for all $(x,y)\in \Ob\times\Ob$. In particular, we have $\pairing{w(y)-w(x),y-x} \le |x-y|^2$ while $w(y)=0$ for all  $y\in \bO$.
		Fix an  arbitrary $x\in \Omega$. Then, for every every $\tau\in S^{d-1}$, there exists
		a unique real $t\in (0,R)$  such that $x'= x + t \tau\in \bO$. Thus we get  $-\pairing{w(x),\tau} \le t\le R$ and  $ |w|(x) \le R $ as $\tau$ is arbitrary. 
		This proves (i). Furthermore, if $x,y\in \Omega$ are distinct and $\tau^{x,y}= \frac{y-x}{|y-x|}$, there exists $s,t >0$ such that $x'=x-s\tau^{x,y},\ y'=y + t\tau^{x,y}$
		belong to $\bO$ and $s+t+ |x-y| = |y'-x'|\le R$. We deduce from  the above that $-\pairing{w(y)-w(x),\tau^{x,y}}\le s+t \le R$, thus 
		$  - \pairing{w(y)-w(x),y-x} \le  R |x-y|.$
		Eventually the  assertion (ii) follows by exploiting again \eqref{eq:two_point_condition}:
		$$ \frac1{2}\, | u(y) - u(x)|^2\ \le\ |x-y|^2 -  \pairing{w(y)- w(x),y- x}\ \le\ 2R \, |x-y| .$$
		The assertion (iii) is a consequence of applying \cite[Prop 5.1]{alberti1999} to 
		the Lipschitz monotone map  $v:\R^d \to \R^d$ where  $v(x) :=x-w(x)$ in $\Ob$ and $v(x)=x$ in $\R^d\setminus\Ob$.
		Then, owing to Remark 5.2 in \cite{alberti1999}, it holds that
		$$ \int_\Omega |D v| \le C_d R^{d-1}\, \mathrm{osc}(v,\Omega)\qquad \Big(\mathrm{osc}(v,\Omega):=\sup \bigl\{|v(x) - v(y)|,\  (x,y) \in \Omega^2 \bigr\} \Big)  $$
		for a suitable universal constant $C_d$. Since by (i) the oscillation of $v$ in $\Omega$ is not larger than $2\, R$,  we get 
		$ \int_\Omega |D v| \le 2 C_d\, R^d$. On the other hand, \eqref{pointwise} implies that 
		$ \frac1{2}\, \nabla u\otimes \nabla u \le e(v)$, hence $\frac1{2}\, |\nabla u|^2 \le \DIV\, v $ 
		by taking the traces. All in all, we deduce 
		the estimate (iii) with $c_d= 6 C_d$.   
	\end{proof}
	
	As  a consequence of the compactness of $\bK$, the linear continuous form $(u,w)\in C^0(\Ob)\cap\times L^1(\Omega;\R^d) \mapsto \pairing {f,u}$ 
	achieves its maximum and  $(\mathcal{P}^*)$ admits (possibly non-Lipschitz) solutions. 
	In addition, by applying the  H\"older estimate (ii) and the Dirichlet condition $u=0$ on $\Sigma_0$, we get:
	\begin{equation}\label{supfini}
	\max\, (\mathcal{P}^*) = I_0(f,\Sigma_0)\ \le\ \sqrt{2} \ {\rm diam}(\Omega) \int |f|  <+\infty .
	\end{equation}
	
	An important tool for proving the  equalities announced in \eqref{Z0=I0} is  the following perturbation of the dual problem: 
	given $K$ being a closed subset of $\Ob\times\Ob$ we define  $h_K : (p,q) \in C^0(\Ob\times\Ob;\R^2) \to [-\infty,0]$ to  be  the functional defined by
		\begin{equation}\label{def:h} h_K(p,q):= \inf_{u,w}\ 
		 \left\{ - \pairing{f,u} \, \left\vert\, 
		 \def\arraystretch{1.2}
		 \begin{array}{ll}(u,w)\in \bigl({\rm Lip}(\Ob)\bigr)^{d+1}, \quad u=0 \ \text{in $\Sigma_0$}, \quad w=0 \ \text{in $\bO$,}\\
		 \bigl(\zeta_1(u)+p,\, \zeta_2(w) + q\bigr) \in \, |x-y|\, \Ctwo\quad 
\text{for all $(x,y)\in K$} \end{array}\right.
		 \right\}
		\end{equation}
		where 
		\begin{equation} \label{def:zeta} \zeta_1(u)(x,y) :=u(y)-u(x), \qquad \zeta_2(w)(x,y):=\pairing{w(y)-w(x),\tau^{x,y}}.\end{equation}
		
		\begin{proposition}\label{key} Let $K$ be a compact subset of $\Ob^2$ and define the (possibly empty) convex subset
		$\Atwo_K:=\bigl\{(\pi,\Pi)\in \Atwo \, : \,\spt\big((\pi,\Pi)\big) \subset K \bigr\}$.
		  Then:
		\begin{itemize}[leftmargin=1.5\parindent]
		\item [(i)] The conjugate of $h_K$ in the duality between $C^0(\Ob\times\Ob;\R^2)$ and $\Mes(\Ob\times\Ob;\R^2)$ is given by
$$  h_K^*(\pi,\Pi) =\begin{cases}  \Jtwo(\pi,\Pi) & \text{if $ \ (\pi,\Pi)\in \Atwo_K,$} \\
+\infty& \text{otherwise};
\end{cases} $$
\item [(ii)] $h_K$ is convex; it is finite at $(0,0)$ if and only if $\Atwo_K\cap \mathrm{dom}(\Jtwo)$ is non-empty. 
If this is the case and if moreover $K$ does not intersect the diagonal $\Delta$, 
then $h_K$ is continuous at  $(0,0)$  and there exists a minimizer for the problem
\begin{equation}\label{PK}
\inf \Big\{ \Jtwo(\pi, \Pi) \ :\ (\pi, \Pi) \in \A_K\Big\} \tag{$\mathscr{P}_K$}
\end{equation}

\item [(iii)] The function $h:=h_{\Ob\times \Ob}$ is finite and lower semicontinuous at $(0,0)$.

\end{itemize}

\end{proposition}

\begin{proof} \  Let us compute the conjugate of $h_K$ at $(\pi,\Pi)$. We have 
\begin{align*}
h_K^*(\pi,\Pi) =& \sup_{p, q, u, w}\ \left\{ \pairing{\pi,p} + \pairing{\Pi,q} +\pairing{f,u}\  \left\vert\,  \def\arraystretch{1.2}
\begin{array}{ll}(u,w)\in \bigl({\rm Lip}(\Ob)\bigr)^{d+1}, \quad u=0 \ \text{in $\Sigma_0$}, \quad w=0 \ \text{in $\bO$,}\\
\bigl(\zeta_1(u)+p,\, \zeta_2(w) + q\bigr) \in \, |x-y|\, \Ctwo\quad 
\text{for all $(x,y)\in K$} \end{array}\right. \right\}\\
=&\ \sup_{\tilde{p},\tilde{q}}\ \Big\{ \pairing{\pi,\tilde{p}} + \pairing{\Pi,\tilde{q}} \  :\ \left(\tilde{p},\tilde{q}\right) \in \, |x-y|\, \Ctwo\quad 
\text{for all $(x,y)\in K$}  \Big\}\\
&+ \sup_{u,w}\ \Big\{\pairing{f,u} - \pairing{\pi,\zeta_1(u)} - \pairing{\Pi,\zeta_2(w)} \ :\ (u,w)\in ({\rm Lip}\big(\Ob)\big)^{d+1},\ u=0 \ \text{in $\Sigma_0$},\ w=0 \ \text{in $\bO$}\Big\},
\end{align*}
where, to pass from the first to the second line,  we have set $(\tilde{p},\tilde{q}) =(\zeta_1(u)+p, \zeta_2(w) + q)$ which  
 in fact run over all $C_0(\Ob\times\Ob;\R^2)$ selections of the multifunction $|x-y|\, \Ctwo$ whatever is $(u,w)$. This allows to split the supremum into the sum $a(\pi,\Pi) + b(\pi,\Pi)$ where $a$ is the supremum with respect to $(\tilde{p}, \tilde{q})$ in the second line and  $b$ is the supremum with respect to $(u,w)$
 in the third line.
 By linearity, we see that $b(\pi,\Pi)$ coincides with the indicator  of the set of measures $(\pi,\Pi)$ satisfying  the conditions (i) and (ii) in \eqref{defAtwo}
 for Lipschitz test functions thus, by using a density argument, we get $b=\ind_\Atwo$.  
 Next we rewrite $a(\pi,\Pi)$ in the form
 $$ a(\pi,\Pi)= \, \sup \left\{ \int \zeta_1\, d\pi + \int \zeta_2 \, d\Pi \ :
\ \zeta_1,\zeta_2\in C^0(\Ob\times\Ob) \ ,\ (\zeta_1,\zeta_2)\in \Gamma(x,y) \ \text{on\ }  \Ob\times \Ob  \right\}$$
where $\Gamma(x,y) =|x-y|\, \Ctwo$ if $(x,y) \in K$ and  $\Gamma(x,y)=\R^2$ if $(x,y)\in \Ob^2\setminus K$. 
 It is easy to check that $\Gamma$ is lower semicontinuous as a multifunction ranging in the family of closed convex subsets of $\R^2$. The support function of $\Gamma(x,y)$ determines a one-homogeneous lower semicontinuous integrand  $\f:\Ob^2\times \R^2 \to [0,+\infty]$ given by
 $ \f\big((x,y),z\big) = |x-y| \, \ind_{\Ctwo}^*(z)$ if $(x,y)\in K$ and by $\f\big((x,y),z\big) = {\ind_{\{0\}}(z)}$ otherwise.
Then, by applying \cite[Theorem 5]{valadier} and by using once more the Goffmann-Serrin convention, we deduce that
 $$ a(\pi,\Pi) = \int_{\Ob^2} \f\big((x,y), (\pi,\Pi)\big) =  \begin{cases} \Jtwo(\pi,\Pi) & \text{if $\spt(\pi,\Pi) \subset K$,} \\ +\infty & \text{otherwise.}\end{cases}$$
Therefore, recalling that $h_K^*(\pi,\Pi) = a(\pi,\Pi) + b( \pi,\Pi)$, we recover the expression given by the assertion (i).

Lets us now prove the assertion (ii). The convexity of $h_K$ is straightforward due to convexity of the set of elements $\{(u,w), (p,q)\}$ which satisfy the constraint in \eqref{def:h}. Notice that $h_K\le 0$ (by taking $(u,w)=(0,0)$). On the other hand, in view of assertion (i),
 $\Atwo_K\cap \textrm{dom} (\Jtwo)$ is non-empty if and only if 
$\inf h_K^* <+\infty$. In this case we have  $h_K(0,0)\ge h_K^{**}(0,0)= - \inf h_K^*>-\infty$ and  the function $h_K$ is convex, proper
and finite at $(0,0)$.

Assume now that the compact subset $K$ is such that $K\cap \Delta$ is empty. Then there exists $\d>0$ such that $|x-y|\ge \d$ for all $(x,y)\in K$.
By evaluating the infimum in \eqref{def:h} with $(u,w)=(0,0)$ and  in view of \eqref{def:Ctwo}, we obtain that
$$  h_K(p,q) \le 0   \quad\text{whenever} \quad  \max_K \left\{\frac{p^2}{2} + q \right\} \le \d,$$
from which follows that the convex function $h_K$ is continuous at $(p,q)=(0,0)$.  An important consequence is that  $h_K^*$ attains its minimum
on $\Mes(\Ob\times\Ob;\R^2)$ and  that $-h_K(0,0)= - h_K^{**}(0,0)=  \min h_K^*$ (see Lemma \ref{pertu}), hence the assertion (ii).

 The proof of the assertion (iii) is delicate and  requires  technical tools about maximal monotone maps developed in Section \ref{sec:geodesic}. 
 Let $(p_n,q_n)$ be a sequence such that $(p_n,q_n)\to 0$ in $C^0(\Ob^2)$. It is not restrictive to assume that $\abs{p_n(x,y)}\leq 1$ and $\abs{q_n(x,y)}\leq 1$ for any $n$ and $(x,y) \in \Ob \times \Ob$ and that $\sup_n h(p_n,q_n)<+\infty$.
 By  definition \eqref{def:h}, we may choose $(u_n,w_n)\in C_{\Sigma_0}(\Ob) \times (C_0(\O))^d$ so that
\begin{equation}\label{masterineq1}   -\pairing{f,u_n} \le  h(p_n,q_n) +\frac1{n}, \qquad
 \frac1{2}|\zeta_1(u_n)+p_n|^2 + \, |x-y|\, \bigl(\zeta_2(w_n)+ q_n\bigr)  \le |x-y|^2 \quad \forall (x,y) \in (\Ob)^2\ .\end{equation}  
 In particular, recalling definition \eqref{def:zeta}, we deduce that
 $$ \pairing{w_n(y)-w_n(x),y-x} \leq \abs{x-y}^2 -q_n(x,y) \abs{x-y} \leq \abs{x-y}(\abs{x-y}+1). $$
 From this we deduce that $w_n$ is uniformly bounded by slightly modifying the proof of claim (i) in Lemma \ref{esti}. Indeed, by choosing for any  $x \in \Omega$, any vector $\tau \in S^{d-1}$ and a real $t$ such that $|t|\le \mathrm{diam}(\O)$ and $y= x+ t\tau\in \bO$, we obtain that $-\pairing{w(x),\tau} \leq t+1 \leq R+1$ where $R=\mathrm{diam}(\O)$. As $\tau$ is arbitrary, we arrive at the uniform bound $\sup \abs{w_n} \leq 1 + R$ yielding
 in particular that $\zeta_2(w_n)+ q_n \ge - 2 (R+1) -1$. 
 Going back to the inequalities in \eqref{masterineq1}, we obtain the following estimate on $(u_n)$ 
 	\begin{equation}
	\label{eq:equi-cont_u_n}
	\abs{u_n(y) - u_n(x) + p_n(x,y)} \leq C \sqrt{x-y} \qquad 	\forall (x,y) \in \Ob \times \Ob,
	\end{equation}
	where we can take $C = \bigl(6R + 6 \bigr)^{1/2}$. Since $p_n$ tends to zero uniformly in $\Ob$ 
	we infer equi-continuity of $\{u_n\}$ and, since 
	 $u_n = 0$ on  $\Sigma_0$, in virtue of Arzel\`{a}-Ascoli theorem,
   we may extract a subsequence $u_{n_k}$ that converges uniformly to a function $u\in C_{\Sigma_0}(\Ob)$  

 Next, upon extending $w_{n_k}$ by zero to whole $\Rd$, we may choose a subsequence (without further relabelling) $w_{n_k}$ that K-converges to a multifunction $\tilde{\mbf{w}}:\Rd \to 2^{\Rd}$ of a closed graph. This multifunction $\tilde{\mbf{w}}$ is of full domain since for each $x \in \Rd$ the set $\{w_{n_k}(x)\}_k$ is bounded and hence the sequence $\bigl(x,w_{n_k}(x)\bigr)$ admits a cluster point. Moreover, obviously we have $\tilde{\mbf{w}} = \{0\}$ on $ \Ob^{\,c}$.
	For arbitrary pair $(x,y) \in \Ob \times \Ob$ we choose any $\hat{x} \in \tilde{\mbf{w}}(x)$ and $\hat{y} \in \tilde{\mbf{w}}(y)$. Then there exists a sequence $(x_k,y_k) \in \Ob \times \Ob$ such that $(x_k,y_k) \rightarrow (x,y)$ and $\bigl(w_{n_k}(x_{n_k}),w_{n_k}(y_{n_k})\bigr) \rightarrow (\hat{x},\hat{y})$. Directly from \eqref{masterineq1} we have
	\begin{align*}
	\frac{1}{2}\bigl(u_{n_k}(y_{k}) - u_{n_k}(x_{k}) + p_{n_k}(x_{k},y_{k}) \bigr)^2 + \pairing{w_{n_k}(y_{k})-w_{n_k}(x_{k}),y_{k}-x_{k}}
	+ q_{n_k}(x_{k},y_{k}) \, \abs{x_{k}-y_{k}} \leq \abs{x_{k}-y_{k}}^2.
	\end{align*}
	Since $u_n$, $p_n$ and $q_n$ converge uniformly on $\Ob$ we find that in the limit
	\begin{equation}
	\label{eq:two_point_cond_in_the_limit}
	\frac{1}{2}\, \abs{u(y)-u(x)}^2 + \pairing{\hat{y}- \hat{x},y-x} \leq \abs{x-y}^2 \qquad \forall\, \hat{x} \in \tilde{\mbf{w}}(x), \ \hat{y} \in \tilde{\mbf{w}}(y), \quad \forall (x,y) \in \Ob \times \Ob
	\end{equation}
	rendering $\tilde{\mbf{v}}:= \ident -\tilde{\mbf{w}}$ a monotone multifunction of closed graph and of full domain.
	 In case $\tilde{\mbf{v}}$ is not maximal we define $\mbf{v}(x) = \mathrm{co}\bigl(\tilde{\mbf{v}}(x) \bigr)$ arriving at a maximal monotone map $\mbf{v} \supset \tilde{\mbf{v}}$ according to \cite[Corollary 1.4]{alberti1999} ($\mbf{v}$ has a closed graph due to the finite dimension of the domain). 
	Still we have  $\mbf{v}(x)= \{x\}$ for $ x\in \R^d\setminus \Ob$ while \eqref{eq:two_point_cond_in_the_limit} 
	implies that 
	$$ \frac{1}{2}\, \abs{u(y)-u(x)}^2 \le \pairing{\hat{y}- \hat{x},y-x}  \qquad \forall\, \hat{x} \in \mbf{v}(x), \ \hat{y} \in \mbf{v}(y), \quad \forall (x,y) \in \Ob \times \Ob. $$
	Eventually, by taking the infimum with respect to admissible pairs $(\hat{x},\hat{y})$ and 
	with the notations introduced in Section \ref{sec:geodesic}, we arrive at the condition
	$ u(x)- u(y) \le \ell_{\mbf{v}}(x,y)$ holding for all $(x,y)\in \Ob\times\Ob$, 
	which implies the metric inequality  (see assertion (ii) of Proposition \ref{explicitKbar})
	$$ u(x)- u(y) \le c_{\mbf{v}}(x,y) \quad \forall (x,y)\ \in \Ob\times\Ob .$$ 
	Then, keeping in mind that $u\in C_{\Sigma_0}(\Ob)$ while $\mbf{v}$ is an element of $\mbf{M}_\O$,
	$(u,\mbf{v})$ is a competitor for the geometric  version \eqref{revised-dual} of the dual problem. Therefore $\pairing{f,u} \le I_0(f,\Sigma_0)$
	and 
	the desired lower semicontinuity inequality follows since, by the first inequality in \eqref{masterineq1}, we have	
	$$ \liminf_n h(p_n,q_n)\, \ge\,  \liminf_n \big\{-\pairing{f,u_n} \big\} = - \, \pairing{f,u} \, \ge\,  -I_0(f,\Sigma_0)\, = \, h(0,0) .$$

\end{proof}

We are now in position to state the main result of Section  \ref{sec:optimembrane}:

	\begin{theorem}\label{Z=I} \ The following (no-gap) equalities hold:
		$$\max (\mathcal{P}^*) \ = \  \inf (\mathcal{P}) \ =\ \inf (\mathscr{P}) $$ 
			Moreover the problem $(\mathcal{P})$ admits solutions and any minimizer  $(\lambda,\sigma)$ vanishes on $\Sigma_0$.
	\end{theorem}
	\begin{proof} We  begin by proving inequality  $ \inf (\mathcal{P})  \ge \sup (\mathcal{P}^*)$.
		It is  a consequence of the following claim:
		\begin{equation}\label{inf>sup}
		(u,w)\in \K \ \text{and}\   (\lambda,\sigma)\in \A \quad  \Rightarrow \quad \pairing{f,u}  \ \le \ J(\lambda,\sigma).
		\end{equation}
		To show \eqref{inf>sup}, we may assume that $ J(\lambda,\sigma) <+\infty$ so that $\lambda \ll \mu := \tr\,\sigma$ and we can write 
		$(\lambda,\sigma)= (\theta,S) \, \mu$ for a suitable $(\theta,S)$ in $L^1_\mu(\Ob, \R^d \times \Sddp)$ so that
		$J(\lambda,\sigma) =  \int  \chi_\C^* (\theta,S)\, d\mu$.
		Then  we may apply the integration by parts formula \eqref{intpartmu} in Corollary \ref{dualmu} and conclude by 
		integrating the  inequality \eqref{ineq0}. 
		
		On the other hand, by the definition of $\Atwo$ (see \eqref{defAtwo}) and in virtue of inequality $J(\lambda_\pi, \sigma_\Pi) \ \le\ \Jtwo(\pi, \Pi)$ (see Lemma \ref{1hom_two}),  we infer that  $ \inf (\mathscr{P}) \ge \inf (\mathcal{P})$.
		Thus,  in order to establish the pursued equalities, it is enough showing that $\inf (\mathscr{P})\, = \, \sup (\mathcal{P}^*).$ 
		Thanks to the last assertion of  Lemma \ref{key}, the perturbation function $h$ obtained  in \eqref{def:h} for $K=\Ob\times\Ob$
		is convex, finite and l.s.c. at the origin and, by construction, such that $h(0,0)= - I_0(f,\Sigma_0) $.
Therefore,  by Lemma \ref{pertu}, it holds that  $I_0(f,\Sigma_0) =- h(0,0)= -h^{**}(0,0) = \inf h^*$. Then, by applying  assertion (i) of Lemma  \ref{key} and \eqref{supfini}, we deduce the claimed equality
$$ \max (\mathcal{P}^*) = \inf \big\{ \Jtwo(\pi,\Pi) : (\pi,\Pi)\in \Atwo \big\} =  \inf (\mathscr{P}) .$$

		Eventually it remains to show the existence of solutions to $(\mathcal{P})$.
		Let $(\lambda_n,\sigma_n)\in \Mes(\Ob, \R^d\times \Sdd)$  be a minimizing sequence, namely
	$$ J(\lambda_n,\sigma_n) \to Z_0, \qquad (\lambda_n,\sigma_n) \in \A.$$
	Since  $Z_0= I_0(f,\Sigma_0) <+\infty$ (thanks to \eqref{supfini}) and  by the coercivity condition \eqref{coercif}, we infer that
	$(\lambda_n,\sigma_n)$ is uniformly bounded hence  weakly* precompact in $\Mes(\Ob, \R^d\times \Sdd)$. 
	It follows then   from  lower semicontinuity of $J$ and from closedness of the set $\A$ that any cluster point $(\lambda,\sigma)$ is a minimizer for $(\mathcal{P})$. Eventually we observe that, for such  a minimizer,  the pair $(\lambda,\sigma) \mres (\Ob \setminus \Sigma_0)$
		still belongs to $\A$. Therefore  $Z_0= \int_\Ob \ind_\C^*(\lambda,\sigma) \ge \int_{\Ob\setminus \Sigma_0} \ind_\C^*(\lambda,\sigma)$ 
		from which follows  that $\int_{\Sigma_0} \ind_\C^*(\lambda,\sigma)=0$. Therefore, by \eqref{coercif}, we have    $(\lambda,\sigma)\res\Sigma_0=0$ as claimed.
	\end{proof}

	\section{Optimality conditions and examples} \label{opticond}
	\subsection{Optimality conditions for $(\mathcal{P})$}\label{opticalP}
	In virtue of   Theorem \ref{Z=I}, an admissible pair $(\lambda,\sigma)$ for  $(\mathcal{P})$ and a pair $(u,w)\in \bK$
	are solutions to $(\mathcal{P})$ and $(\mathcal{P}^*)$, respectively, if and only if the following extremality condition holds:
	\begin{equation} \label{extremality1}
	\pairing{f,u}\ =\ J(\lambda,\sigma) .
	\end{equation}
	The next step is to find how this equality can be localized in order to obtain a pointwise relation (associated law) 
	between the optimal $(\lambda,\sigma)$ and $(u,w)$. As we are particularly interested with situations where optimal design are concentrated on lower dimensional sets,  the validity of an integration by parts formula 
	applying to possibly singular measures $(\lambda,\sigma)$ turns out to be crucial. 
	To that aim we use some tools of tangential differential calculus which we can apply successfully only in the case of a Lipschitz solution
	to $(\mathcal{P}^*)$ (that is for $(u,w)\in \mathcal{K}$).

	\medskip
	
	Recall that a solution $(\lambda,\sigma)$ to  $(\mathcal{P})$ must satisfy $J(\lambda,\sigma)<+\infty$. Therefore, by applying assertion (i) of
	Lemma \ref{1hom} and setting $\mu =\tr\, \sigma$, such a solution can be represented in terms of a triple $(\mu, q, S)$ as follows
	\begin{equation}\label{triple}
	\sigma = S \mu, \qquad  \tr S = 1 \quad  \mu\text{-a.e.}, \qquad  \lambda = S q\, \mu,
	\end{equation}
	where $(q,S)$ is a suitable element of $L^2_\mu(\Ob;\R^d)\times L^\infty_\mu(\Ob; \Sddp)$.
	We recall that $\Omega$ is a convex domain, $\Sigma_0$ is a non-empty compact subset of $\bO$
	and $f$ is an element of $\Mes(\Ob)$. It is not restrictive to assume that $f(\Sigma_0)=0$.

	\begin{theorem}\label{Optismooth}  Let $(\lambda,\sigma)$ be an element of $\Mes(\Ob; \R^d\times \Sddp)$ given in the form
		$(\lambda,\sigma)= (S q\, \mu,S \mu)$ according to \eqref{triple} and let 
		$(u,w)\in \mathrm{Lip}(\Ob;\R^{1+d})$. 
		Then the pairs $(\lambda,\sigma)$ and $(u,w)$ are optimal for $(\mathcal{P})$ and $(\mathcal{P}^*)$, respectively,
		if and only if all the following conditions are satisfied:
		\begin{align}\label{cns-smooth}
		\begin{cases}
		(i)& u=0 \ \text{ on $\Sigma_0$},\ \  w=0  \ \text{ on $\bO$},\\
		(ii) &   -\dive( Sq \, \mu) = f \ \  \text{in $\D'(\Ob\!\setminus\! \Sigma_0)$}, \quad  \DIV (S\, \mu) =0\  \ \text{in $\D'(\Omega)$,  }\\
		(iii) & \frac{1}{2}\,\nabla u \otimes \nabla u + e(w)  \leq\, \Ident  \quad  \text{$\mathcal{L}^d$-a.e. in $\Omega$,}\\
		(iv)&   Sq = S\, \nabla_\mu u\quad   \text{ $\mu$-a.e.},\\
		(v)&  \pairing{ \frac{1}{2}\,\nabla_\mu u \otimes \nabla_\mu u + e_\mu(w), S} = \tr S \quad  \text{ $\mu$- a.e.},\\
		(vi)&  \mu(\Sigma_0) =0. \\
		\end{cases}
		\end{align}

	\end{theorem}
	\begin{remark}\label{tangential} \ The choice of $q$ satisfying $(iv)$ does not affect the solution $\lambda$ and we may
		drop $(iv)$ taking directly $q= \, \nabla_\mu u$ in $(ii)$.
		On the other hand,  it is not mandatory to represent $\sigma$ in the form $\sigma=S\mu$ with the normalization $\mu=\tr \sigma$ (so that $\tr\, S=1$ in $(v)$). We may alternatively choose any $\tilde{\mu}$ such that $\mu\ll \tilde{\mu}$.
				\end{remark}
		
	\begin{proof}   Let $(\lambda,\sigma)$ and $(u,w)$ be admissible pairs  (that satisfy (i),\,(ii),\,(iii)). Then $(u,w)\in\K$
		and, by  Corollary \ref{dualmu},
		we have the pointwise inequality 
		\begin{equation}\label{ineq1}
		\pairing{Sq, \nabla_\mu u} + \pairing{S, e_\mu(w)}\ \le\  \chi_\C^*(Sq,S) \qquad \text{$\mu$-a.e.}
		\end{equation}
		Then, by  applying the integration by parts  formula \eqref{intpartmu} and
		integrating \eqref{ineq1} with respect to measure $\mu$, we recover the inequality
		$$ \pairing{f,u}=\  \int_{\Ob} \pairing{Sq, \nabla_\mu u}\, d\mu  +\int_\Ob \pairing{S, e_\mu(w))}\, d\mu \le J(\lambda,\sigma) .$$
		Moreover, as $\tr S=1$ and recalling \eqref{C*}, the inequality above becomes an equality if and only if
		\begin{equation}\label{iff=}
		\pairing{Sq, \nabla_\mu u  } + \pairing{S, e_\mu(w)} =  1 + \frac1{2} \pairing{S q,q} \qquad \text{$\mu$-a.e.}
		\end{equation}
		Therefore, under the admissibility conditions (i),\,(ii),\,(iii), checking the optimality of $(\lambda,\sigma)$ and $(u,w)$ (that is the extremality condition
		\eqref{extremality1}) amounts to verifying  whether  \eqref{iff=} holds true.
		
		First we observe that if (iv),\,(v)  are satisfied then by symmetry of $S$ we have  $\pairing{Sq, \nabla_\mu u}= \pairing{Sq, q}=\pairing{S,\nabla_\mu u\otimes \nabla_\mu u}$ 
		holding  $\mu$-a.e. and it is clear that \eqref{iff=} follows. 
				Conversely, assume that \eqref{iff=} holds true. Then, by integrating this  equality on $\Sigma_0$ and taking into account  assertion (i) of
		Corollary \ref{dualmu}, we get 
		$$ 0 =  \int_{\Sigma_0} (\pairing{Sq, \nabla_\mu u}+\pairing{S, e_\mu(w)})\, d\mu = \mu(\Sigma_0) + \frac1{2} \int_{\Sigma_0}\pairing{S q,q}\, d\mu \ge \mu(\Sigma_0),$$
		from which, we infer the condition $(vi)$. On the other hand, after easy computations, we may rewrite \eqref{iff=} as follows:
		$$    \Big\langle S, e_\mu(w)+ \frac1{2}\,  \nabla_\mu u\otimes \nabla_\mu u \Big\rangle = 1 + \bigl\langle S (q-\nabla_\mu u), q-\nabla_\mu u\bigr\rangle \qquad \text{$\mu$-a.e.}$$
		Since $S\in \Sddp$ the right hand side is not smaller than $1$ while, by \eqref{ineq0}  the first term  is not greater than 1. This implies that  $ \pairing{ S (q-\nabla_\mu u), q-\nabla_\mu u} \bigr\rangle$ vanishes $\mu$-a.e. therefore $(iv)$ must hold true and, in turn, so does $(v)$.
		
	\end{proof}

\subsection{Optimality conditions for $(\mathscr{P})$}\label{optiscrP}	In order to characterize truss-like solutions
 we give here  optimality conditions for $(\mathscr{P})$. Note that they do not require the Lipschitz regularity of the pair 
$(u,w)$ solving the dual problem $(\mathcal{P}^*)$ but merely its continuity. In fact a non-smooth generalization of Proposition \ref{OptiTwo} below will appear in Section \ref{sec:geodesic} where the continuity assumption on $w$ is skipped.
We recall the admissibility conditions given in \eqref{defAtwo} for the competitors of problem $(\mathscr{P})$.
\begin{proposition}
	\label{OptiTwo}
	Let $(\pi,\Pi)$ be an element of $\Mes(\Ob\times\Ob;\R\times \R_+)$ given in the form $(\pi,\Pi)= (\alpha\Pi,\Pi)$ with $ \alpha \in L^1_\Pi(\Ob\times\Ob)$ and let $(u,w)\in C^0(\Ob;\R^{d+1})$. Then the pairs $(\pi,\Pi)$ and $(u,w)$ are optimal for, respectively, $(\mathscr{P})$ and $(\mathcal{P}^*)$ if and only the following conditions are satisfied:
	\begin{align}\label{OCtwo1}
	\begin{cases}
	(i)& u=0 \ \text{ on $\Sigma_0$}, \quad  w=0  \ \text{ on $\bO$},\\
	(ii) &   (\alpha\Pi,\Pi) \in \Atwo,\\
	(iii) & \frac{1}{2}\, \abs{u(y)-u(x)}^2 + \pairing{w(y)- w(x),y-x} \leq \abs{x-y}^2 \qquad \forall\,(x,y)\in \Ob \times \Ob,\\
	(iv)&   \alpha(x,y) =\frac{u(y)-u(x)}{|y-x|} \qquad \text{for } \Pi \text{-a.e. } (x,y) ,\\
	(v)&  \frac{1}{2}\, \abs{u(y)-u(x)}^2 + \pairing{w(y)- w(x),y-x} = \abs{x-y}^2 \qquad \text{for } \Pi \text{-a.e. } (x,y). \\
	\end{cases}
	\end{align}
\end{proposition}
\begin{proof}
	The admissibility of  $(\pi,\Pi)=(\alpha\Pi,\Pi)$ in  $(\mathscr{P})$  and  of $(u,w)$ in $(\mathcal{P}^*)$ are equivalent to conditions (i),\,(ii),\,(iii). Based on (i),\,(ii) and by exploiting \eqref{defAtwo} we may write
	\begin{align*}
	\pairing{f,u} &=	\int_{\Ob\times\Ob} \left(\alpha(x,y) \bigl( u(y) - u(x) \bigr) + \Big\langle w(y)-w(x),\frac{y-x}{|y-x|} \Big\rangle \right) \Pi(dxdy) \\
	&\le \int_{\Ob\times\Ob} |x-y| \left(1 + \frac{1}{2}\bigl(\alpha(x,y)\bigr)^2\right) \Pi(dxdy)  = \Jtwo(\alpha\Pi,\Pi),	\end{align*}
where to pass to the second line we use the pointwise inequality $\alpha\, \zeta_1(u) +\zeta_2(w) \le \abs{x-y}\,\ind_\Ctwo^*(\alpha,1)$ (see \eqref{Ctwo*}) taking into account that, by condition (iii),  the pair
$\bigl(\zeta_1(u),\zeta_2(w)\bigr)$ defined in  \eqref{defLambda}  belongs to $\abs{x-y}\,\Ctwo$ for every $(x,y) \in \Ob\times \Ob$. Therefore, in virtue of the equality $\inf \mathscr{P}= \sup \mathcal{P}^*$
obtained in Theorem \ref{Z=I}, the optimality of $(\pi,\Pi)$ and $(u,w)$ is equivalent to the localized equality:
$$\alpha(x,y) \bigl( u(y) - u(x) \bigr) + \Big\langle w(y)-w(x),\frac{y-x}{|y-x|} \Big\rangle
= |x-y| \left(1 + \frac{1}{2}\bigl(\alpha(x,y)\bigr)^2\right) \qquad \text{for } \Pi \text{-a.e. } (x,y) \,$$
that, after multipying by $|x-y|$, we can rewrite as
	\begin{equation*}
		\frac{1}{2}\, \abs{u(y)-u(x)}^2 + \pairing{w(y)- w(x),y-x} -\abs{x-y}^2 = \frac{1}{2} \Big(  u(y)-u(x) - |y-x| \, \alpha(x,y) \Big)^2.
	\end{equation*}
From (iii), we see that the latter equality holds if and only if (iv) and (v) are both satisfied.
\end{proof}

\begin{remark} \ The equality in  (v) implies that $u$ and $\pairing {\ident-w, y-x}$ is affine on every segment $]x,y[$ whenever $(x,y)\in \spt \Pi$.
In particular, for  any such a pair $(x,y)$, we have $\Ha^1([x,y]\cap\Sigma_0)=0$. Indeed, if $\Ha^1([x,y]\cap\Sigma_0) >0$, then $x,y\in \bO$ and, as $u$
is affine on $[x,y]$, we get  $u(x)=u(y)=\pairing{w(y)- w(x),y-x}=0$, which is in contradiction with (v). We
thus recover condition (vi)  of Theorem \ref{Optismooth}.

\end{remark}

	%
	%


	\subsection{ Optimal configurations in the radial case} \label{circular}
		For $R>0$ we consider a circular domain $\Omega = \bigl\{x\in \R^2 \, : \,  \abs{x} < R \bigr\}$ and  a Dirichlet condition on the whole boundary
		(that is $\Sigma_0=\bO$). We also assume  that the load $f$ is radial (note that $f$ can be a signed measure). It is then completely described by the
		bounded repartition function
		\begin{equation}\label{repartition}
		F(t) \ :=\ \int_{\{|x|\le t\}} \ f(dx),\qquad t\in [0,R].
		\end{equation}
		Then, working in polar coordinates $(r,\theta)$, it is easy to find  a solution $(\bar u, \bar w), (\bar\lambda,\bar \sigma)$ to  \eqref{cns-smooth} in the form
		of radial functions:
		\begin{equation}\label{radial}
		\bar u(x)= u(r), \quad \bar w(x)= w(r)\, e_r, \quad \lambda = \alpha(r)\,z(r) \, e_r \,\mu, \quad \sigma =\alpha(r)\,  e_r\otimes e_r\, \mu, \quad   \mu =\mathcal{L}^2\mres \O 
		\end{equation}
		for suitable  $u, w\in {\mathrm{Lip}\bigl((0,R)\bigr)}$ and $z,\alpha$ in $L^1_{\rm loc}\bigl((0,R)\bigr)$.
		In fact, as noticed in Remark \ref{tangential}, we may assume that  $z= u'$ on the subset $\{\alpha\not=0\}$. Thus \eqref{cns-smooth} has a  solution of the form \eqref{radial} if and only if:
		\begin{align}
		\begin{cases}\label{cnsradial}
		(i')& u(R)=0,\quad  w(0)= w(R)=0, \quad \alpha \ge 0\  \text{a.e. in $(0,R)$,} \\
		(ii') &   - \dive \bigl(\alpha(r)\, u'(r) \, e_r\bigr) = f,\quad \DIV\bigl( \alpha(r) \,  e_r\otimes e_r\bigr)= 0 \qquad \text{in $\D'(\Omega)$}, \\
		(iii') & \frac{1}{2}\,|u'|^2+ w' \le 1, \quad \frac{w}{r}\le 1 \qquad \text{a.e. in $(0,R)$},\\
		(iv') &  \frac{1}{2}\,|u'|^2+ w'=1 \qquad \text{a.e. in $\{\alpha>0\}$},
		\end{cases}
		\end{align}
		where in $(i')$ we use $(i)$ and express that $\bar w= w(|x|)\, e_r$ is continuous at $x=0$ and in addition that $\sigma\ge 0$. The conditions $(iii'),\,(iv')$ are equivalent to $(iii),\,(iv)$ since 
		$$\frac1{2} \nabla \bar u \otimes \nabla \bar u + e(\bar w) = \Big(\frac1{2} \abs{u'}^2 + w' \Big)  e_r\otimes e_r +  \frac{w}{r}\, e_\theta\otimes e_\theta, \qquad \tr S= \alpha(r).
		$$
		It turns  out that \eqref{cnsradial} admits  a unique solution. Indeed, solving the equations $(ii')$ leads to
		$$ \alpha(r) = \frac{D}{r} \quad,\quad u'(r)= z(r)  = - \frac{F(r)}{2 \pi D} \ .$$
		for a suitable positive constant $D$. This can be checked by noticing that $\alpha$ needs to satisfy $\alpha' + \alpha /r =0$ in $(0,R)$
		and that conversely the distributional divergence of $\frac1{r} \, e_r\otimes e_r$ vanishes on the whole ball $\Omega$ (including the origin).
		On the other hand, the expression for $u'(r)$ is deduced by integrating the first equation of  $(ii')$ on the subset $\{ |x| <r\}$ 
		and by using Green's formula and \eqref{repartition}. 
		Accordingly, in order to match with $(i')$ and with the equality constraint in $(iv')$, we deduce that:
		\begin{equation}\label{explicituw}
		u(r) =  \frac1{2\pi D} \ \int_r^R F(t) \, dt, \qquad w(r)= r -   \frac1{8 \pi^2 D^2}\int_0^r F^2(t)\, dt
		\end{equation}
		where the constant $D$ is determined  by the condition  $w(R)=0$, i.e.
		\begin{equation}\label{D}
		D  = \frac1{2 \pi} \left(\frac1{2R} \int_0^R F^2(t)\, dt\right)^{\frac1{2}}.
		\end{equation}
		Note that $u$ and $w$ are both Lipschitz (since $|F(t)| \le  \int |f|$) and that the first inequality in $(iii')$ is an equality
		while the second one (i.e. $w\le r$ ) follows directly from \eqref{explicituw}. Therefore all the required conditions in \eqref{cnsradial} are satisfied. 
		
		\begin{remark}\label{wnot0}
		We observe that $w$ can never be identically zero except in the case where $F(t)$ is constant, that is when $f$ is a Dirac mass at the origin. In this special case, we have $\inf \mathrm{(FMD)}= \inf \mathrm{(OM)}$ which, upon recalling Corollary \ref{fmd<om}, relates to the fact that
		  the high ridge of $\O$  given by \eqref{def:ridge} satisfies $M(\O)=\{0\}$.
		\end{remark}
		\begin{remark}\label{notruss} If we confine ourselves to loads $f$ furnishing a strictly increasing function $F$ (a basic example of such load is a uniform density), then the solution to $(\mathcal{P}^*)$  given by \eqref{explicituw} provides  a strictly concave $u$. In particular, for any pair $(x,y) \in \Obsq$ the function $\bar u$ is not affine on $[x,y]$. Thus, owing to Lemma \ref{lem:two_point_condition},  the equality condition (v) in \eqref{OCtwo1} required for solving $\mathscr{P}$ cannot be satisfied whatever we take $(\pi,\Pi)$ as an admissible pair in $\Atwo$.  
   That way we get a {\em counter-example to the existence of a truss solution} for the optimal membrane problem.
\end{remark}
				\begin{remark}\label{mixedsol}\ In general the solution to the primal problem $({\mathcal P})$ is not unique.
				For instance let us consider a radial load $f$ which  does not charge the open subset $\{|x| < r_0\}$ for some $r_0\in (0,R)$.  
		Then, besides the optimal stress solution  
	$\bar \sigma= \frac{D}{R} \, e_r \otimes e_r $ with  $D$ given in \eqref{D}, we may also take another solution mixing a distributed stress tensor and a concentrated one, as for instance 
	\begin{equation*}\label{sigma-variant}
		 \sigma = D \, \left( \frac{e_r \otimes e_r }{r} \ \mathcal{L}^2\mres\{ r >r_0 \} + 
		e_\theta \otimes e_\theta \ \Ha^1\mres \{ |x|=r_0\}\right) ,
		\end{equation*}
		which clearly satisfies $\DIV \,\sigma=0$ in $\D'(\O)$. 
 One checks easily the optimality conditions \eqref{cns-smooth} for $u,w$ defined by \eqref{explicituw} while taking 
 $\mu = \mu_a+\mu_s$ where  $\mu_a= \mathcal{L}^2\mres\{ r >r_0 \}$ and $\mu_s= \Ha^1\mres \{ |x|=r_0\}$. The $\mu$ tangential projector $P_\mu$ is then 
  the identity $\mu_a$-a.e. while $P_\mu= e_\theta\otimes e_\theta$ holds  $\mu_s$-a.e. 
 Noticing that $F$ defined in \eqref{repartition} vanishes on $[0,r_0)$, we have  $u(r)=u(r_0)$ and $w(r)=r$ for all $r\in [0,r_0]$. Then, writing $\sigma=S \mu$, we get 
 $$S =\begin{cases} \frac{ D }{r} \, e_r \otimes e_r & \mu_a \text{-a.e.,}\\ D\, e_\theta \otimes e_\theta& \mu_s\text{-a.e.,} \end{cases} \qquad     \big(\nabla_\mu \bar u,e_{\mu}(\bar w)\bigr)
		=\begin{cases}  \big( u'(r) \, e_r,\, w'(r)\, e_r\otimes e_r + \frac{w(r)}{r}\, e_\theta\otimes e_\theta \big)   & \text{$\mu_a$-a.e.,} \\ 
		\big(0, \, \frac{1}{r_0}\,e_\theta\otimes e_\theta\big)  & \text{$\mu_s$-a.e.}\end{cases}$$
By taking $q=\nabla_\mu \bar u$ we see that all the conditions \eqref{cns-smooth} are satisfied.
%
	\end{remark}
\medskip

\subsection {Optimal configurations  for a  one-force load} \label{disk}\

\subsubsection*{ Case of a disk}\ 	Let $\Omega \subset \R^2$ be a disk of radius $R_0$ centered at the origin. We consider a load $f = \delta_{x_0}$ with $x_0 \in \Omega$ and any Dirichlet zone $\Sigma_0 \subset \bO$ such that $x_0 \in \mathrm{co}(\Sigma_0)$. It is convenient to introduce the geometric parameter  $d_0:= \sqrt{R_0^2- |x_0|^2}$.
	Let $p\in \PP(\Sigma_0)$ be any probability measure satisfying the barycenter condition 
\begin{equation}\label{barycenter}
x_0 = \int_{\Sigma_0} x\, p(dx).
\end{equation} 
Then we consider the pair $(\pi,\Pi) = (\a \Pi,\Pi) \in \big(\Mes_+(\Ob\times \Ob)\big)^2$ defined  by
\begin{equation}\label{optipiPi}
\pi=  p(dx) \otimes \delta_{x_0}(dy), \qquad \Pi= \frac{|x_0-x|}{\sqrt{2}\, d_0}  \, p(dx) \otimes \delta_{x_0}(dy), \qquad\alpha(x,x_0)=\frac{\sqrt{2}\, d_0}{|x-x_0|}\quad  \text{for  $x\in \Sigma_0$}.  \end{equation}
We claim that such a pair solves  the problem $(\mathscr{P})$, thus providing  a truss solution $(\lambda_\pi,\sigma_\Pi)$ to $(\mathcal{P})$ in virtue of Theorem \ref{Z=I}. 	
Clearly $(\pi,\Pi)$ is an admissible competitor. Indeed $\Pi\ge 0$ and  condition (i) of  in \eqref{OCtwo1}  is fulfilled since $p(\Sigma_0)=1$ 
while condition (ii) follows from \eqref{barycenter}. In view of Proposition \ref{OptiTwo}, it remains to find an admissible pair  $(\bar u,\bar w) \in C_0(\Omega;\R^{d+1})$
satisfying the conditions $(iii)$,\,$(iv)$,\,$(v)$ in \eqref{OCtwo1}. 
To that aim, we proceed in polar coordinates $(r,\theta)$ with respect to $x_0$ so that the domain $\Omega$ is parametrized as
	$$ \Omega :=\Big\{ x_0 + r \, e_r(\theta)\ :\ \theta\in [0,2\pi),\ \ 0 \leq r < \rho(\theta) \Big\}\ ,$$
	where $e_r(\theta):= (\cos \theta, \sin \theta)$ and  $\rho(\theta):[0,2\pi) \to [d_0,2R_0-d_0]  $  a Lipschitz function. 
	Then we propose the following $\bar u,\bar w$ constructed from the function $h$ whose graph coincides with the conical surface in $\R^3$ with $(x_0,1)$ as its vertex and containing  $\bO \times \{0\}$, namely:
	\begin{equation}
		\label{cone_uw}
		(\bar u,\bar w) = \left(\sqrt{2}\, d_0, 2\, x_0\right)  h \qquad \text{where} \quad   h(x_0 + r \, e_r) := 1- \frac{r}{\rho(\theta)}.
	\end{equation}
Since   $(\bar u,\bar w)$ vanishes on the whole boundary $\bO$ and  $(\bar u , \bar w)(x_0) = (\sqrt{2} \, d_0, 2\, x_0)$,
the equalities in $(iv)$ and $(v)$ in \eqref{OCtwo1} become, respectively
$$  \alpha(x,x_0) = \frac{\bar u(x_0)}{|x-x_0|} = \frac{\sqrt{2}\, d_0}{|x-x_0|}, \qquad d_0^2 +  \pairing{2 x_0,x_0-x} = |x_0-x|^2. $$
The first equality then follows  from the explicit form of  $\alpha(x,x_0)$ in \eqref{optipiPi} while the second equality equivalent to $d_0^2 + |x_0|^2= |x|^2$  holds at every $x\in \bO$ where $\abs{x} = R_0$, hence  $p$-a.e. 
Eventually it remains  to check the two-point condition $(iii)$ which is the tricky  part. 
As  $(\bar u,\bar w)$ is Lipschitz, we may check the equivalent condition \eqref{pointwise} (see Lemma \ref{lem:two_point_condition}), namely
that at every point  of differentiability $x\in \Omega$ the eigenvalues of the symmetric tensor $A(x):= \frac{1}{2}\, \nabla \bar u(x) \otimes \nabla \bar u(x) +e(\bar w)(x)$ do not exceed 1. 
We compute using the frame $(e_r,e_\theta)$ 
(where $e_\theta(\theta) = (-\sin \theta, \cos \theta)$). In view of  \eqref{cone_uw} we have for $x=x_0 + r\, e_r$:
$$  \nabla \bar u = \sqrt{2}\, d_0\,  \nabla h\ =\ \sqrt{2}\, d_0\, (a\, e_r + a'  e_\theta), \qquad  e(\bar w) = x_0 \otimes \nabla h+ \nabla h \otimes x_0 \qquad \text{where}\quad  a(\theta) :=- \frac1{\rho(\theta)}.$$
 We obtain the decomposition \ $A(x) = A_{1,1}\, e_r\otimes e_r + A_{1,2}\, (e_r\otimes e_\theta + e_\theta \otimes e_r)
+ A_{2,2}\,   e_\theta \otimes e_\theta$ \ in which
\begin{align}\label{tensorA}  A_{1,1}= d_0^2 \, a^2 + 2 \pairing{x_0,e_r}\, a, \ \ \
A_{1,2}= d_0^2\,  a a' +  \pairing{x_0,e_r}\, a' +  \pairing{x_0,e_\theta}\, a,\ \ \
A_{2,2}= d_0^2\, (a')^2 + 2 \pairing{x_0,e_\theta}\, a'.  
\end{align}
Fortunately  the later expressions simplify. Indeed, since  $|x_0 + \rho\, e_r|^2=R_0^2$ , we have $\rho^2 + 2 \rho  \pairing{x_0,e_r}= R_0^2-|x_0|^2= d_0^2$ \
and therefore:
$$  d_0^2\, a^2= 1- 2 \pairing{x_0,e_r}\, a, \qquad   d_0^2\,  a a'= - \pairing{x_0,e_r}\, a' -  \pairing{x_0,e_\theta}\, a,$$
where the second equality is obtained by differentiating the first one with respect to $\theta$. Thus $A_{1,1}=1$ and $A_{1,2}=0$. Next we prove that $A_{2,2}\le 0$.
To that aim we exploit the fact that $\bar u=0$ on $\bO$ so that $\nabla \bar{u}$ is parallel to $x_0 + \rho\,  e_r$. This implies  more relations:
 $$ (\pairing{x_0,e_r} + \rho  ) \, a' =  \pairing{x_0,e_\theta}\, a, \qquad d_0^2\, (a')^2 = \frac{\pairing{x_0,e_\theta}^2}{(\pairing{x_0,e_r} + \rho  )^2} \big(1- 2 \pairing{x_0,e_r}\, a\big).  $$
 Substituting in the expression of $A_{2,2}$ in \eqref{tensorA}, we get:
 $$ A_{2,2} = \frac{\pairing{x_0,e_\theta}^2}{(\pairing{x_0,e_r} + \rho  )^2} \Big( 1- 2 \pairing{x_0,e_r}\, a +  2 (\pairing{x_0,e_r} + \rho  )\, a\Big)
 =   -\frac{\pairing{x_0,e_\theta}^2}{(\pairing{x_0,e_r}+\rho)^2}.   $$
  Summarizing we have proved that $A(x)$ is a diagonal tensor with eigenvalues not larger that $1$. This implies the admissibility condition $(iii)$ and the optimality of  
$(\bar u,\bar w)$ for the dual problem while $(\pi,\Pi)$ given by \eqref{optipiPi} is optimal for $(\mathscr{P})$.
Moreover the minimal energy is given by \ $\min (\mathscr{P}) = \min (\mathcal{P}) = \max (\mathcal{P}^*) = \sqrt{2}\, d_0$.

\subsubsection*{Case of a general convex domain}
The construction used above for the disk paves a way for finding a solution for other shapes of design domains.
More precisely let  $\Omega \subset \R^2$ be a general bounded convex domain, $x_0\in\O$ and $\Sigma_0 \subset \bO$ a closed Dirichlet region. 
 We reinforce the condition $x_0 \in \mathrm{co}(\Sigma_0)$ as follows:
  \begin{equation}\label{bary+}
  \exists y_0 \in\O \ :\ x_0 \in \mathrm{co}\big(\Sigma_0(y_0)\big)\quad \text{where}\quad\Sigma_0(y_0):=\Sigma_0\cap p_{\bO}(y_0)
\end{equation}
where $p_{\bO}(y_0)$ is the minimal set of  $d(y_0,\cdot)$ on $\bO$, {see \eqref{Gamma0}}.
Then we define  
 $$d_0(y_0) :=\sqrt{\bigl(d(y_0,\Sigma_0)\bigr)^2 - \abs{x_0-y_0}^2}. $$
 \begin{proposition}
	\label{one-force_construction}
	Assume that \eqref{bary+} is satisfied for a suitable $y_0\in\O$ and let  $\rho$ be any probability supported  on $\Sigma_0(y_0)$ such that $x_0 = \int x\, \rho(dx).$	  Then the pair $(\pi,\Pi)$ given by
	\begin{equation}\label{optipiPi-y0}
\pi=  \rho \otimes \delta_{x_0} , \qquad \Pi= \frac{|x_0-x|}{\sqrt{2}\ d_0(y_0)}  \, \rho(dx) \otimes \delta_{x_0}(dy) 
	\end{equation}	
	solves the problem $(\mathscr{P})$ for $f = \delta_{x_0}$ and we have\ $\min(\mathcal{P})=\min (\mathscr{P})= \sqrt{2} \, d_0(y_0)$.
	\end{proposition}
%

\begin{proof} \ In the same way as for the disk, the equalities $\rho(\Sigma_0)=1$ and $x_0 = \int x\, \rho(dx)$ imply that $(\pi,\Pi)$ given in \eqref{optipiPi-y0} is an admissible pair for $(\mathscr{P})$.  Next we observe that \eqref{bary+}
 implies that $x_0 \in B(y_0,R_0)$ where $R_0= d(y_0,\bO)$. Without loss of generality we may assume that $y_0$ is the origin so that $d_0(y_0)= \sqrt{R_0^2-|x_0|^2}$. Therefore $(\pi,\Pi)$  given in \eqref{optipiPi-y0} coincides with the optimal pair we found for the disk $B(0,R_0)$ (see \eqref{optipiPi}). 
Let us consider the zero extension to $\O$ of the pair $(\bar u, \bar w)$ given by \eqref{cone_uw} in the disk $B(0,R_0)$. Clearly it is Lipschitz continuous, vanishes on all $\bO$ and  satisfies 
 the pointwise gradient constraint \eqref{pointwise} on all $\O$.  Therefore the condition $(iii)$  of Proposition \ref{OptiTwo} is fulfilled on $\Ob\times\Ob$. On the other hand, the  conditions $(iv)$ and $(v)$ are obviously satisfied in the same way as for the disk. That way we recover the optimality of $(\pi,\Pi)$ for the problem $(\mathscr{P})$ set on the domain $\Omega$. The corresponding minimal 
 energy is given by $\min (\mathscr{P}) = \sqrt{2} \, \sqrt{R_0^2-|x_0|^2}=  \sqrt{2}\, d_0(y_0)$. 
 
\end{proof}
\begin{remark}\label{bary-}
Assume for simplicity that $\Sigma_0=\bO$. Due to the condition \eqref{bary+} which is  crucial in our previous construction, we can only handle  single loads $f=\delta_{x_0}$
when  $x_0$ belongs to the compact subset  
$ K_0:= \bigcup \left\{ \mathrm{co}\big( p_{\bO}(y_0)\big): y_0\in \Ob\right\}$  
which is determined by those points $y_0$ for which $p_{\bO}(y_0)$ has more than one element 
 (i.e. $y_0$ belongs to the skeleton of $\O$).  In contrast  with the square for which  $K_0=\Ob$,
  it happens that $K_0$ is a strict subset of $\Ob$ for an ellipse of large excentricity.

\end{remark}
As an illustration of Proposition \ref{one-force_construction} we can solve the one-force problem for a rectangular design domain
$\Omega $ of sides $2R$ and $2R+L$ where $L > 0$ as presented in Fig. \ref{fig:one_point_force} where also locations of strategic points $a_1,\ldots,a_6$ are specified. We choose $\Sigma_0 = \bO$ and $f = \delta_{x_0}$ for $x_0\in\Omega$.
\begin{figure}[h]
	\centering
	\subfloat[]{\includegraphics*[trim={0cm 0cm -0cm -0cm},clip,width=0.23\textwidth]{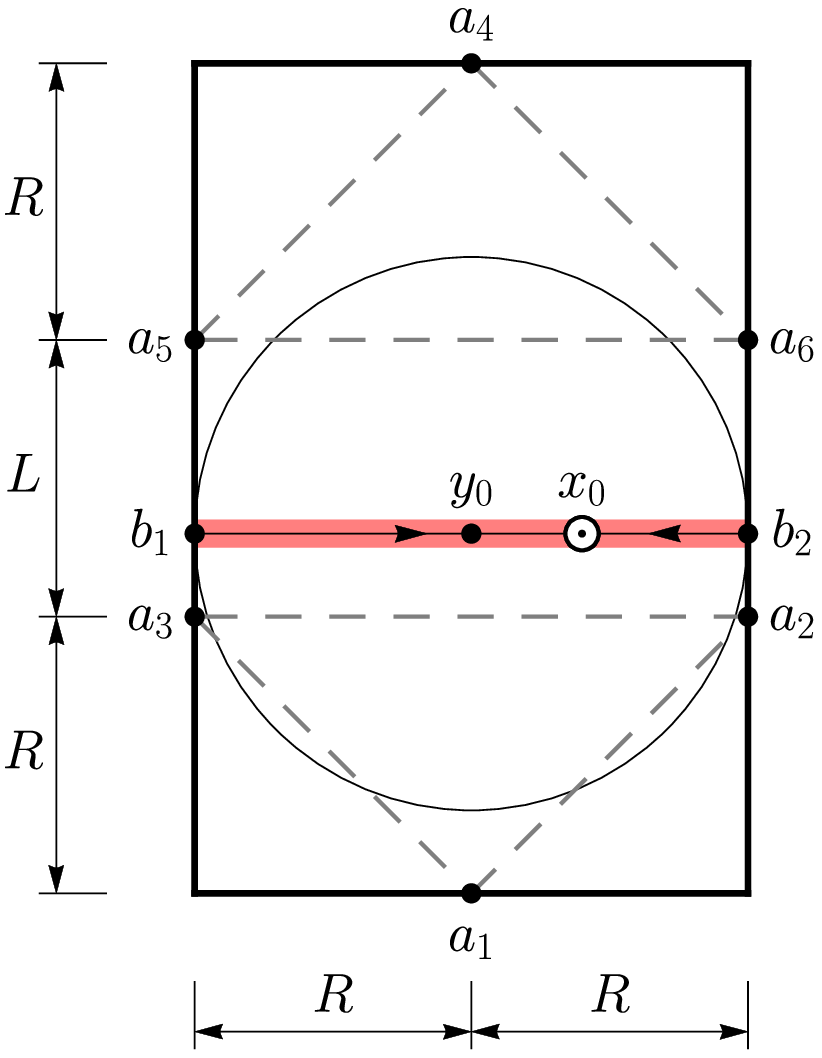}}\hspace{1cm}
	\subfloat[]{\includegraphics*[trim={0cm 0cm -0cm -0cm},clip,width=0.23\textwidth]{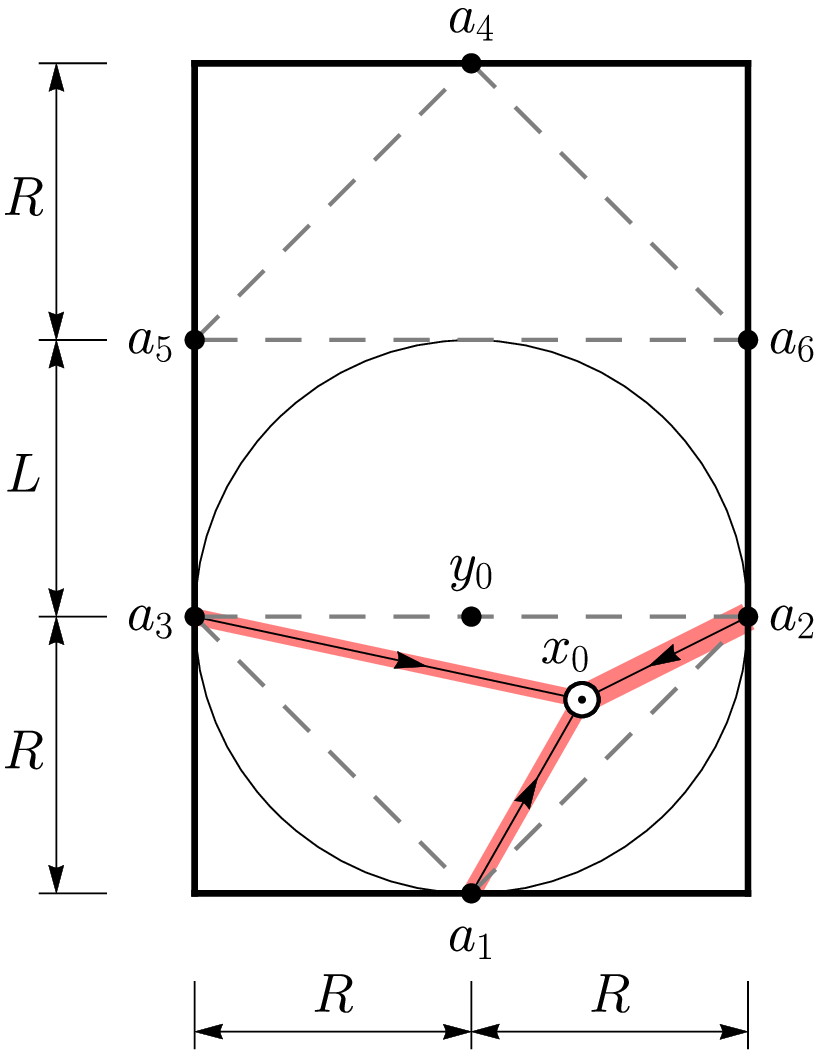}} \hspace{1cm}
	\subfloat[]{\includegraphics*[trim={0cm 0cm -0cm -0cm},clip,width=0.23\textwidth]{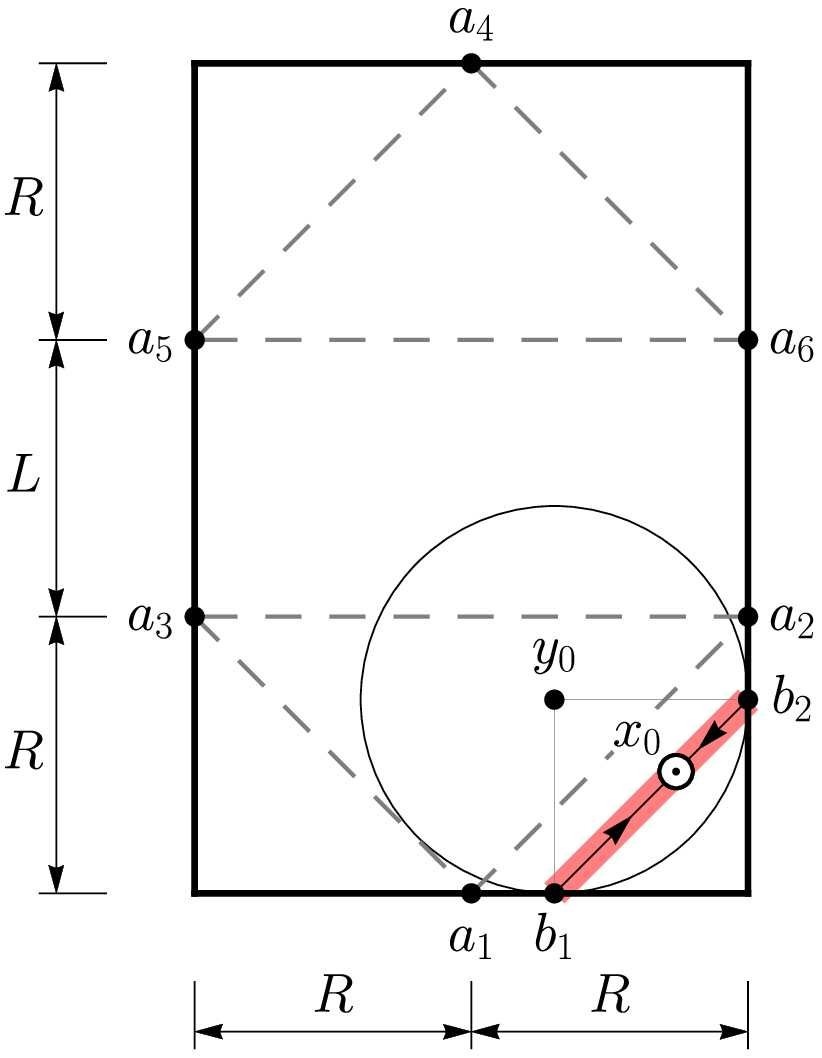}}
	\caption{$y_0$ and  $\bar \sigma$ for three locations of $x_0$}
	\label{fig:one_point_force}       
\end{figure}
The dashed lines in Fig \ref{fig:one_point_force} partition the domain into seven regions.
  For the position $x_0$ lying inside each of those regions, the solution $(\lambda,\sigma)$ given by \eqref{optipiPi-y0}  differs. By symmetry it is not restrictive to assume that $x_0$ belongs to the south-east part of $\Omega$
so that we are reduced to three  cases. For each of them we precise the selected $y_0$ and the probability $\rho$. 
The corresponding optimal $\bar \sigma$ is depicted  in Fig \ref{fig:one_point_force}.

	\smallskip
	\noindent {\bf Case (a)}:  {\em $x_0$ belongs to the interior of $\co\{a_2, a_3, a_5, a_6\}$.}
\  By $b_1$ and $b_2$ we denote the orthogonal projections of $x_0$ onto the longer sides of the domain, see Fig. \ref{fig:one_point_force}(a). The unique $y_0$ satisfying \eqref{bary+} is given by $y_0 = (b_1+b_2)/2$ (thus  $\Sigma_0(y_0) = \{b_1,b_2\}$). The unique element $\rho$ of $\PP(\Sigma_0(y_0))$ such that $x_0 = \int x \,\rho(dx)$ is given by $\rho = \frac{\abs{x_0-b_2}}{\abs{b_1-b_2}} \delta_{b_1}+\frac{\abs{x_0-b_1}}{\abs{b_1-b_2}} \delta_{b_2}$.
	After some simplifications, we arrive at a minimal energy given by  $\min(\mathcal{P}) =\sqrt{2 \,\abs{x_0-b_{1}}\abs{x_0-b_{2}}}$.
The optimal pair $(\lambda,\sigma)$  represented in Fig. \ref{fig:one_point_force}(a) corresponds to a one-dimensional string
$[b_1,b_2]$. 
	
	
	\smallskip
	\noindent{\bf Case (b):} \ {\em  $x_0$ belongs to the closure of $\co\{a_1, a_2, a_3\}$.}
	\ The unique $y_0$ satisfying \eqref{bary+} is given by  $y_0=(a_2+a_3)/2$. Then $\Sigma_0(y_0) = \{a_1,a_2,a_3\}$ and $\rho$ is the unique probability  $\rho = \sum_{i=1}^3 \a_i\, \delta_{a_i}$ with $\a_i\geq 0$ such that $x_0 =  \int x \,\rho(dx) =\sum_{i=1}^3 \a_i\, a_i $. The optimal membrane $\sigma$ consists of three strings tied at $x_0$ as depicted in Fig. \ref{fig:one_point_force}(b).


	\smallskip
	\noindent{\bf Case (c):}  \  {\em  $x_0$ lies in $\co\{a_1, a_2, a_7\}$ .}
	\  The  unique possible $y_0$ is determined by its projections $b_1, b_2$ on the boundary where $b_1 \in [a_1,a_7]$ and $b_2 \in [a_2,a_7]$ are chosen so that the segment $[b_1,b_2]$ contains $x_0$ and  is parallel to $[a_1,a_2]$ (see Fig. \ref{fig:one_point_force}(c)). 
 Then  $\Sigma_0(y_0)=\{b_1,b_2\}$ and $\rho = \frac{\abs{x_0-b_2}}{\abs{b_1-b_2}} \delta_{b_1}+\frac{\abs{x_0-b_1}}{\abs{b_1-b_2}} \delta_{b_2}$.
 Similarly to the Case (a), the optimal membrane is a one dimensional string $[b_1,b_2]$.	
 
 \medskip
 
In addition we can handle the case of a square where $L=0$, $a_3=a_5$ and  $a_2=a_6$. 
For $x_0$ lying in the triangle $a_1 a_2 a_7$, the construction of $y_0$ and $\rho$ are the same as in the Case (c) above rendering the one dimensional string $[b_1,b_2]$ the unique solution. 
In contrast, if  $x_0$ lies in the interior of the rotated square $\co\{a_1, a_2, a_3, a_4\}$, the unique possible $y_0$ is the center
	of the square while  $\Sigma_0(y_0)=\{a_1,a_2,a_3,a_4\}.$ Then there are infinitely many probability measures $\rho= \sum_{i=1}^4 \a_i\, \delta_{a_i}$ that give $x_0 = \sum_{i=1}^4 \a_i\, a_i$. Among the corresponding solutions, we recover in particular the four-string structure that was advertised in the introduction, see Fig. \ref{fig:intro}(b).

	\section{Kantorovich-Rubinstein duality for optimal metrics.}\label{sec:geodesic}
	
	In this section our aim is to attack the dual problem $(\mathcal{P}^*)$ from a different viewpoint which relies on the observation that,
 for every Lipschitz admissible pair $(u,w)\in\K$, the function $v:=\mathrm{id}- w$  satisfies $e(v)\ge 0$ and therefore is a {\em monotone map} from $\Omega$ to $\R^d$. Then the two-point condition \eqref{eq:two_point_condition} rewritten as
	 $$  u(x) - u(y) \le \ell_v(x,y),    \quad \text{where}\quad \ell_v(x,y):= \sqrt{2\pairing{v(y)-v(x), y-x}},$$
can be extended  to a suitable class of functions which allows to characterize $\bK$. 
Moreover, by considering a regularization of the function $\ell_v$  (see Subsection \ref{subsec:monotone}), we attach to
 every $v$ a sub-additive transport cost $c_v$ so that $(\mathcal{P}^*)$ can be rewritten in the form:
	$$  \sup_v \left\{ \sup_u \Big\{  \pairing{f,u}\ :\ u(x) - u(y) \le c_v(x,y)  \ \ \forall\,(x,y) \in \Ob\times\Ob \Big\}\right\},$$
where the second  supremum falls into  the classical Kantorovich-Rubinstein duality framework
 (see for instance  \cite{villani} for more details).

\subsection{ Pseudo-metrics associated with monotone maps}\label{subsec:monotone} 

Given  a multifunction $\mbf{v}:\Rd \rightarrow \R^d$ (being a map from $\Rd$ to subsets of $\Rd$)  we denote by $G_\mbf{v}$ its graph that is
$G_\mbf{v}= \big\{ (x,v)\in \Rd\times\Rd : v\in \mbf{v}(x) \big\}$. Recall that $\mbf{v}$ is  called \textit{monotone} whenever
	\begin{equation*}
		{\pairing{x'_1 - x'_2,x_1-x_2} \geq 0 \qquad \forall\, (x_1, x'_1), (x_2,x'_2)  \in G_\mbf{v}.}
	\end{equation*}
	Moreover we say that $\mbf{v}$ is \textit{maximal} monotone if for any monotone $\tilde{\mbf{v}}$ such that 
	$  G_\mbf{v}\subset  G_{\tilde{\mbf{v}}}$ there holds $\tilde{\mbf{v}} = \mbf{v}$. This implies in particular that $ G_\mbf{v}$ is a closed subset of $\R^d\times\R^d$ (or equivalently that $\mbf{v}$ is upper semicontinuous as a multifunction).
	
	Given a bounded convex domain $\O\subset\R^d$,  
	we will consider the class of maps
	$$ \mbf{M}_\O \ :=\Big\{ \mbf{v}: \R^d \to \R^d \ :\ \mbf{v} \ \text{maximal monotone},\ 
	\mbf{v}= \mathrm{id} \quad \text{in $\R^d\setminus \Ob$}\ \Big\},$$
	and in parallel the following subset of $BV_{\rm loc}(\R^d)\cap L^\infty(\R^d)$:
	$$ \mathcal{B}_\O \ :=\ \Big\{ v\in BV_{\rm loc}(\R^d;\R^d) \ :\ e(v)\ge 0, \ \ v= \mathrm{id} \ \ \text{a.e. in $\R^d\setminus \Ob$}\Big\}.
	$$
	It is important to notice that $ \mbf{M}_\O$ and $\mathcal{B}_\O$ are convex sets. 
	For every $\mbf{v}\in \mbf{M}_\O$ and $a,b \in\R^d$, we define:
	 \begin{equation}\label{lv}
\ell_{\mbf{v}}(a,b):= \min \left\{  \sqrt{2\pairing{b' - a',b-a}}\ :\  a'\in \mbf{v}(a),\ \ b' \in \mbf{v}(b) \right\} 
\end{equation}
and the scalar monotone function:
\begin{equation}\label{v_ab}
\mbf{v}_{a,b}(t):= \pairing{\mbf{v}\big((1-t) \,a + t\, b\big),b-a}, \quad t\in\R.
\end{equation}
Note that $\ell_{\mbf{v}}(a,b)= \sqrt{2} |b-a|$ if $a,b \in \Rd\setminus\Ob$. It turns out that the function  $\ell_{\mbf{v}}$ is  lower semicontinous but the continuity may fail out of the diagonal.
 \begin{lemma} \label{mono1} Let $R$ denote the diameter of $\Omega$ and let $\mbf{v}$ be an element of $\mbf{M}_\O$.  Then:
\begin{itemize}[leftmargin=1.5\parindent]
\item [(i)] For every $x\in \R^d$, $\mbf{v}(x)$ is a non-empty compact convex subset of $\ov {B_R(x)}$ and there holds  $\mbf{v}(x)\supset\{x\}$ 
for every $x\in \bO$.  Moreover, the subset $\{x\in\R^d: \mbf{v}(x)\ \text {is not a singleton}\}$  is contained in a $d- 1$ rectifiable subset of $\Ob$.
\item [(ii)]  The function $\ell_{\mbf{v}}$ is lower semicontinuous on $\Rd\times\Rd$ and the minimum in \eqref{lv} is attained.
\item [(iii)] Let $a,b\in \Rd$ and  $\mbf{v}_{a,b}$ as defined in \eqref{v_ab}.
Then $\mbf{v}_{a,b}$ is a maximal monotone map from $\R$ to $\R$ and  
$$\ell_{\mbf{v}}(a,b)\ =\ \sqrt{2}\, \left( \mbf{v}_{a,b}(1_-) - \mbf{v}_{a,b}(0_+)\right)^{\frac1{2}}.$$ 
Moreover we have the following upper bounds
\begin{equation}\label{lv<}
\ell_{\mbf{v}}(a,b) \le \sqrt{2} \, |b-a|^{1/2} ( R + |b-a|)^{1/2},  \quad    \ell_{\mbf{v}}(a,b)\le \sqrt{2R}\, |b-a|^{1/2} \ \ \text{if  $(a,b) \in \Ob^2$}.
\end{equation}
\end{itemize}
\end{lemma}
\begin{proof}\ The  first property in assertion (i) follows from \cite[Prop.1.2]{alberti1999} and from assertion (i) of Lemma \ref{esti}. 
Since $\mbf{v}(x)=\{x\}$ holds for every $x\in \R^d\setminus\Ob$,  the closed graph of $\mbf{v}$ contains $\{(x,x): x\in \bO\}$
hence  $\mbf{v}(x)\supset\{x\}$ for all $x\in\bO$.  
The second property in assertion (i) is a direct consequence of \cite[Thm 2.2 ]{alberti1999}. 
The lower semicontinuity property stated in (ii)  is straightforward once we know that
for any sequence $(a_n,b_n)\to (a,b)$ the elements $a'_n\in \mbf{v}(a_n) ,b'_n\in \mbf{v}(a_n)$ which realize the minimum  in \eqref{lv} remain in a compact subset of $\Rd$, thus  (possibly after extraction of a subsequence) converge respectively to $a'\in \mbf{v}(a)$ and $b'\in \mbf{v}(b)$.
It is in fact the case since, by the assertion (i), one has $ |a'_n-a_n|+ |b'_n-b_n|\le 2R$ and also due to closedness of the graph $\mbf{v}$. 
 The equality in assertion (iii) is trivial since $\ell_{\mbf{v}}(a,b)= \ell_{\mbf{v}_{a,b}}(0,1)$
while one has  $ \mbf{v}_{a,b}(t)= [\mbf{v}_{a,b}(t_-),\mbf{v}_{a,b}(t_+)]$ for every $t\in\R$.
On the other hand, if $a,b\in \O \times\O$ the intersection of the  line $\big\{(1- t) a + t b: t\in\R \big\}$ with $\bO$ gives two values $t_a, t_b$ such that
$t_a < 0 <1< t_b$ and $(t_b-t_a) |b-a| \le R $. Therefore, since $\mbf{v}$ agrees with the identity in $\Rd\setminus\Ob$, 
we get  
\begin{equation}\label{estiv0}
\mbf{v}_{a,b}(1_-) -\mbf{v}_{a,b}(0_+)\ \le\ \mbf{v}_{a,b}(t_{b}-0) -\mbf{v}_{a,b}(t_{a}+0)\ \le\  R\, |b-a|
\end{equation}
hence we obtain the inequality $ \ell_{\mbf{v}}(a,b)\le  \sqrt{2R} |b-a|^{1/2}$ for all $(a,b)\in\O\times\O$ which can be extended to $\Ob\times\Ob$ by using the lower semicontinuity of $\ell_{\mbf{v}}$ obtained in assertion (ii). Thus the second  inequality in \eqref{lv<} is proved. The first one follows if $(a,b)\in \Ob\times\Ob$ whereas it obviously holds true also  for $(a,b)\in (\R^d\setminus \O )^2$. 
Eventually it is enough to check the case where $a\in \O$ while $b\notin\Ob$. This is done by considering the unique $b'\in\bO\cap [a,b]$
which is associated with a value $t_{b'}<1$. Then 
$$\mbf{v}_{a,b}(1_-) -\mbf{v}_{a,b}(0_+) \le \big(\mbf{v}_{a,b}(t_{b'}-0) - \mbf{v}_{a,b}(t_a+0)\big) +  \big(\mbf{v}_{a,b}(1_-) -  \mbf{v}_{a,b}(t_{b'}+0) \big)
\le  R |b-a| + |b-a|^2.$$
The proof of Lemma \ref{mono1} is complete
\end{proof}

\med
The main properties of  $\mbf{M}_\O$ and its relation with $\mathcal{B}_\O$ are summarized in the next two lemmas.

\begin{lemma}\label{mono2}  For every  $\mbf{v}\in \mbf{M}_\O$ there exists a unique $v\in \mathcal{B}_\O$ such that $\mbf{v}(x)=\{v(x)\}$ for a.e. $x\in \O$. Moreover, for $R$ being the diameter of $\Omega$, one has:
\begin{equation}\label{estiv}
\int_\Ob |Dv|\le 2 C_d R^d \qquad \text{and} \qquad  \pairing{v(y) -v(x),y-x} \le R \, |x-y| \quad \text{a.e. in $\O\times\O$}. 
\end{equation}
Conversely, for every $v\in\mathcal{B}_\O$ there exists a unique $\mbf{v}\in \mbf{M}_\O$ such that  $\mbf{v}(x)=\{v(x)\}$ for a.e. $x\in \O$. Moreover $\mbf{v}(x)$ is a one dimensional segment $[v^-(x),v^+(x)]$ for $\Ha^{d-1}$ a.e. $x$ in the jump set of $v$.
\end{lemma} 
\begin{proof} \ The second inequality in \eqref{estiv} can be deduced from \eqref{estiv0}. The other statements are consequences of \cite[Thm 5.3 and Corollary 1.5]{alberti1999}.

\end{proof}

\med
By the assertion (i) of Lemma \ref{mono1} the Hausdorff distance $D_H$ in $\R^d\times \R^d$ between the graphs of two elements  ${\mbf{v}_1}, {\mbf{v}_2}\in \mbf{M}_\O$ is finite 
and induces a metric $ \mbf{h}({\mbf{v}_1}, {\mbf{v}_2}):=  D_H( G_{{\mbf{v}_1}},  G_{{\mbf{v}_2}}).$
Similarly, we may embed  $\mathcal{B}_\O$ with the $L^1$-distance $d(v_1,v_2):= \int_\O |v_1-v_2| dx$ (recall that $v_1=v_2=\mathrm{id} $ a.e in $\R^d\setminus \O$). 
In view of Lemma \ref{mono2}, we may consider the one to one map between metric spaces  
$$\mbf{i}: \mbf{v} \in (\mbf{M}_\O, \mbf{h}) \mapsto  v\in  (\mathcal{B}_\O,d) .$$

\begin{lemma}\label{homeo}  The set $\mbf{M}_\O$  is a compact metric space and  $\mathbf{i}$ is a homeomorphism between $\mbf{M}_\O$ and $\mathcal{B}_\O$.  
Furthermore, let  $(\mbf{v}_n)$ be a sequence $\mbf{M}_\O$ such that $\mbf{v}_n \to \mbf{v}$. Then $\ell_{\mbf{v}_n} \to \ell_{\mbf{v}}$
in the sense of $\Gamma$-convergence that is
\begin{subnumcases}{\label{Gamma-lv}}
     \liminf_n \ell_{\mbf{v}_n}(a_n,b_n) \ge \ell_{\mbf{v}}(a,b) \ \text{ whenever }\ (a_n,b_n)\to (a,b)\ \text{in}\ \R^d\times\R^d,& \label{gammaliminf}\\
\forall (a,b)\quad    \exists  \ (a_n,b_n) \to (a,b)   \ :\    \limsup_n \ell_{\mbf{v}_n}(a_n,b_n) \le \ell_{\mbf{v}}(a,b).& \label{gammalimsup}
\end{subnumcases}

\end{lemma}
\begin{proof}  \ We observe that, thanks to the assertion (i) of Lemma \ref{mono1},  we have   the estimate $\mbf{h}(\mbf{v}, \mbf{{\rm id}})\le R$  for every
$\mbf{v} \in \mbf{M}_\O$ (this prevents the graph of $\mbf{v}$ to converge to an empty set).  The compactness of $(\mbf{M}_\O, \mbf{h})$ is then a consequence of \cite[Proposition 1.7]{alberti1999}. Let us prove that the one-to-one  map $\mbf{i}$ is continuous. Assume that $\mbf{v}_n \to \mbf{v}$ in $\mbf{M}_\O$
and consider a Lebesgue negligible Borel subset $N$ such that $\mbf{v}_n(x)=\{v_n(x)\}$ and $\mbf{v}(x)=\{v(x)\}$ for all $x\notin N$ and for every $n\in\N$.
Then, for such an $x$, the Hausdorff convergence of the graphs  $G_{\mbf{v}_n}$ implies that $v(x)$ is the unique possible cluster of the sequence $\big(v_n(x)\big)$.
Since this sequence is bounded we infer that $v_n\to v$ a.e. on $\O$ and, by the dominated convergence theorem,  the convergence 
$d(v_n,v)= \int_\O |v_n-v| \, dx\,  \to 0$ follows.

It remains to show the $\Gamma$-convergence property \eqref{Gamma-lv}. For proving \eqref{gammaliminf} we write $\ell_{\mbf{v}_n}(a_n,b_n)= \sqrt{2\pairing{b'_n - a'_n,b_n-a_n}}$
for suitable $a'_n\in \mbf{v}(a_n) \ ,\ b'_n \in \mbf{v}(b_n)$ (which is plausible owing to assertion (ii) of Lemma \ref{mono1}). As $a'_n,\, b_n'$ remain bounded (by assertion (i) of Lemma \ref{mono1})
we may assume  that, up to extraction of a subsequence, it holds that $a'_n\to a', \  b'_n\to b'$. Then the  graph convergence of $\mbf{v}_n$ implies that 
$a'\in \mbf{v}(a), \ b' \in \mbf{v}(b)$, hence $ \liminf_n \ell_{\mbf{v}_n}(a_n,b_n)= \sqrt{2\pairing{b' - a',b-a}}\ge \ell_{\mbf{v}}(a,b)$.
On the other hand, for given $a,b$ we can choose $a'\in \mbf{v}(a) \ ,\ b' \in \mbf{v}(b)$ so that $\ell_{\mbf{v}}(a,b)= \sqrt{2\pairing{b' - a',b-a}}$.
Then there exists $(a_n,a'_n)\to (a,a')$ and $(b_n,b'_n)\to (b,b')$ such that $a'_n\in \mbf{v}(a_n),\ b'_n \in \mbf{v}(b_n)$. Then the property \eqref{gammalimsup} follows since
$$\limsup_n \ell_{\mbf{v}_n}(a_n,b_n) \le  \limsup_n  \sqrt{2\pairing{b'_n - a'_n,b_n-a_n}} = \sqrt{2\pairing{b' - a',b-a}} = \ell_{\mbf{v}}(a,b) .$$

\end{proof}

Next we associate to every $\mbf{v}\in \mbf{M}_\O$ a cost function $c_{\mbf{v}}: \R^d\times \R^d \to \R_+$ defined as follows:
\begin{equation}\label{def:cost}
c_{\mbf{v}}(a,b) :=  \inf\left\{ \sum_{i=1}^{N-1}  \ell_{\mbf{v}}(x_i,x_{i+1})\ :\ x_1=a, \ \ x_N=b,\ \ N\ge 2\right\} 
\end{equation}
where $\ell_{\mbf{v}}$ is defined by \eqref{lv}.
Note that  the infimum in \eqref{def:cost} 
is not reached in general since the number $N$ of intermediate points is not upper bounded.   
Let us also remark that, by using the homeomorphism $\mbf{i}$,  we may associate the same cost to any $v\in \mathcal{B}_\O$  so that the notation $c_v$ could be used as well. 
\begin{theorem}\label{cv-dual}\ Let $\mbf{v}$ be an element of $\mbf{M}_\O$. Then:
\begin{itemize}[leftmargin=1.5\parindent]
\item [(i)]  $c_{\mbf{v}}$ is    continuous sub-additive  and satisfies  
\begin{equation}\label{bounds-c_v}
    \sqrt{2}\ | d(a,\O) -d(b,\O) | \ \le \ c_\mbf{v}(a,b) \ \le\ \ell _\mbf{v}(a,b) \le \sqrt{2} \, |b-a|^{1/2} ( R + |b-a|)^{1/2} \qquad \forall  (a,b) \in \Rd\times\Rd;
  \end{equation} 
\item[(ii)] the following dual representation of the pseudo-metric $c_\mbf{v}$ holds true
 \begin{equation}\label{c_v=sup}
c_\mbf{v}(a,b)  = \max \Big\{ u(b) - u(a)\ :\ u\in C^0(\R^d), \ \ u(x)-u(y)\ \le\ \ell_{\mbf{v}}(x,y) \quad \forall  (x,y) \Big\};
  \end{equation}
  \item[(iii)] for every $(a,b)\in \Rd\times\Rd$ the evaluation function  $\mbf{v}\in (\mbf{M}_\O,\mbf{h})  \mapsto\  c_\mbf{v}(a,b)$  is concave 
  and upper semicontinuous.
  \end{itemize}
  \end{theorem}
  
  \begin{remark}\label{nocontinuous}  \ The evaluation map in  assertion (iii) above is not continuous in general. Indeed, consider for instance 
  the case where $\O=(0,1)$ and $\mbf{v}_n$ is the monotone map associated with the step function $v_n$ on $[0,1]$ such that $v_n(0)=0,\ v_n(1)=1$ and 
  $ v_n'= \frac1{n-1} \sum_{i=1}^{n-1} \delta_{\frac{i}{n}}$. Then, by taking  intermediate points in \eqref{def:cost} very close to each $i/n$ from the left and from the right, it is easy to check that $c_{\mbf{v}_n}(0,1)=0$. Clearly $\mbf{v}_n \to \mbf{v}$ for $\mbf{v}$ being the identity.
 Hence $0=\lim_n c_{\mbf{v}_n}(0,1)< c_\mbf{v}(0,1)=1$.

\end{remark}

  \begin{proof}  The sub-additivity of $c_{\mbf{v}}$ is straightforward from the definition  \eqref{def:cost}; the second inequality in \eqref{bounds-c_v} is obtained  by taking $N=2$ and $(x_1,x_2)=(a,b)$ while the third one (already in \eqref{lv<}) implies that is $c_\mbf{v}$ is continuous on the diagonal hence everywhere by exploiting the sub-additivity property.  
  The first inequality in \eqref{bounds-c_v} will be obtained by applying \eqref{c_v=sup} to the function
  $u(x) = \sqrt{2}\, d(x,\Omega)$ once we can check that $u(x)-u(y)\ \le\ \ell_{\mbf{v}}(x,y)$ for all  $(x,y)$.
  This is trivially the case if $(x,y)\in \Ob^2 \cup (\Rd\setminus\Ob)^2  $ whereas for $x\in\Ob$ and $y\notin \Ob$ it follows from
   inequalities $\ell_\mbf{v} (x,y) \ge  \ell_\mbf{v} (z,y) =  \sqrt{2} (\mbf{v}_{z,y}(1_+) - \mbf{v}_{z,y}(0_-))^{\frac1{2}} =  \sqrt{2} \, |y-z| \ge u(y)$ (note that $u(x)=0$) where $z= [x,y]\cap \bO$.
  
  Let us prove (ii). By the very definition \eqref{def:cost},  for every  $u\in C^0(\Rd)$, we have the equivalence
  $$  u(x)-u(y)\ \le\ \ell_{\mbf{v}}(x,y) \quad \forall  (x,y) \quad \Longleftrightarrow\quad u(x)-u(y)\ \le\ c_{\mbf{v}}(x,y) \quad \forall  (x,y).$$
  This clearly implies that $c_\mbf{v}(a,b)$ is not larger than the right hand side of \eqref{c_v=sup}. The converse inequality is obtained by considering the function $u(x) =c_\mbf{v}(a,x)$ which satisfies $ u(x)-u(y) \le c_{\mbf{v}}(x,y)$ by sub-additivity  and symmetry of $c_\mbf{v}$.
  In addition,  this function $u$ is H\" older continuous by the assertion (i) and \eqref{lv<}. This proves the equality in \eqref{c_v=sup}. Moreover, since the H\"older coefficient is uniformly bounded for all admissible $u$, the supremum  is
  actually a maximum. 
 
  It remains to prove assertion (iii). Let us check first that that the map $\mbf{v} \mapsto c_\mbf{v}(a,b)$ is upper semicontinuous. Let ${\mbf{v}_n} \to \mbf{v}$ in $\mbf{M}_\O$
 and let us show that $\limsup_n c_{\mbf{v}_n} (a,b)  \le c_{\mbf{v}} (a,b)$ for fixed $(a,b)\in (\Rd)^2$. Given a finite set $\{x_i,\ 1\le i\le  N\}$  such that $x_1=a \, ,\, x_N=b$,
 we can choose, for each $i$,  an approximating sequence $x_{i,n}\to x_i$ such that  $\ell_{\mbf{v}_n} (x_{i,n},x_{i+1,n} ) \to \ell_\mbf{v} (x,y)$.
This is indeed a consequence of  the $\Gamma$-convergence of $\ell_{\mbf{v}_n}$ to $\ell_\mbf{v}$ proved in Lemma \ref{homeo}. In accordance with \eqref{def:cost}, for each $n$ 
 we may estimate $c_{\mbf{v}_n} (a,b)$ from above by using the finite sequence of $N+2$ points $\{a,x_{1,n}, \ldots,x_{N,n},b\}$
 and we are led to
 $$\limsup_n c_{\mbf{v}_n} (a,b)  \le \limsup_n \left\{ \ell_{{\mbf{v}_n}} (a, x_{1,n}) + \sum_{i=1}^{N-1}  \ell_{{\mbf{v}_n}}(x_{i,n},x_{i+1,n})  + \ell_{{\mbf{v}_n}} (x_{N,n}, b) \right\}
 =  \sum_{i=1}^{N-1}  \ell_{\mbf{v}}(x_i,x_{i+1}),$$
where to obtain the latter equality we additionally notice that $\lim_n \ell_{{\mbf{v}_n}} (a, x_{1,n}) = \lim_n \ell_{{\mbf{v}_n}} (x_{N,n}, b) = 0$ as a consequence of the estimates \eqref{lv<}. By minimizing the right hand sum with respect to  $x_i$'s we obtain the claimed upper semicontinuity inequality.

Let us prove now the concavity property. It is enough to check the middle point property that is, for given ${\mbf{v}_1}, {\mbf{v}_2} \in \mbf{M}_\O$ and 
setting $\mbf{v} = \frac1{2} ({\mbf{v}_1}+{\mbf{v}_2})$:
\begin{equation}\label{concave-claim}
 c_{\mbf{v}}(a,b) \ge \frac1{2} \big(c_{\mbf{v}_1}(a,b)+ c_{\mbf{v}_2}(a,b)\big).
\end{equation}
In view of assertion (ii), there exist two elements  $u_1, u_2\in C^0(\R^d)$ such that for $i\in \{1,2\}$
  $$ c_\mbf{v_i}(a,b)  = u_i(b) - u_i(a)  \qquad \text{and}\qquad  u_i(x)-u_i(y)\ \le\ \ell_{\mbf{v}_i}(x,y) \quad \forall\,  (x,y) .$$
 Then \eqref{concave-claim} follows once we can show that $u= \frac1{2}(u_1+u_2)$ satisfies the constraint $u(x)-u(y)\ \le\ {\ell_{\mbf{v}}(x,y)}$ 
 for every $(x,y) \in (\R^d)^2$ as well. This is a consequence of the following concavity property of the map $\mbf{v} \mapsto \ell_{\mbf{v}}^2(x,y)$:
 \begin{equation}\label{lv2-concave} 
\ell_{\frac{\mbf{v_1}+ \mbf{v_2}}{2} }^2 (x,y) \ \ge  \  \frac1{2}   \left(  \ell_{\mbf{v_1}}^2 (x,y)+  \ell_{\mbf{v_2}}^2(x,y)\right).
\end{equation}
Indeed, since $|u_i(x)-u_i(y)|\le \ell_{\mbf{v_i}}(x,y)$,  by straightforward computations we will get 
$$ |u(x)-u(y)|^2 \le \frac1{4} \big(\ell_{\mbf{v_1}} (x,y)+  \ell_{\mbf{v_2}}(x,y)\big)^2
 \le \frac1{2} \big(\ell_{\mbf{v_1}} (x,y)^2+ \ell_{\mbf{v_2}} (x,y)^2\big) \le  \ell_{\frac{\mbf{v_1}+ \mbf{v_2}}{2} }^2 (x,y) = \ell_{\mbf{v}} (x,y)^2. $$
Let us proof the claim \eqref{lv2-concave}. Let $\hat x \in  \frac1{2} ({\mbf{v}_1}+{\mbf{v}_2})(x)$ and $\hat y \in  \frac1{2} ({\mbf{v}_1}+{\mbf{v}_2})(y)$.
Then we have $\hat x=  \frac1{2} (\hat x_1 +\hat x_2),\ \hat y=  \frac1{2} (\hat y_1 +\hat y_2) $ for suitable 
$\hat x_i, \hat y_i$ in $\Rd$. Therefore, by the definition of $\ell_{\mbf{v_i}}(x,y)$, we infer that
$$ 2\,  \pairing{\hat{y} - \hat{x},y-x} =  \pairing{\hat{y_1} - \hat{x_1},y-x} +  \pairing{\hat{y_2} - \hat{x_2},y-x} 
\ge \frac1{2} \big(\ell_{\mbf{v_1}}^2 (x,y)+  \ell_{\mbf{v_2}}^2(x,y)\big),$$
hence the  claim  by choosing $(\hat x, \hat y)$ optimal for $\ell_{\frac{\mbf{v_1}+ \mbf{v_2}}{2} } (x,y)$.
The proof of Theorem \ref{cv-dual} is now complete.
\end{proof}

The construction of $c_{\mbf{v}}$ described above in fact provides the largest  sub-additive function below $\ell_{\mbf{v}}$. 
It induces a pseudo-distance in $\R^d$, i.e.  the value $c_{\mbf{v}}(a,b)$ for two distinct points can be zero. In particular this happens if $\mbf{v}$ is tangentially flat  on $[a,b]$. In order to obtain a metric we  need to consider the quotient space  $X_{\mbf{v}}$  of $\R^d$ with respect to the relation $x\sim y$ iff $c_\mbf{v}(x,y)=0$ and to extend the definition of $c_\mbf{v}$ accordingly. Note that, owing to assertion (i) of Theorem \ref{cv-dual}, the class $\dot x$ of an element $x$ is a closed subset of $\Rd$  and that $\dot x=\{x\}$ for $x\in \R^d\setminus\Ob$.
A natural local pseudo-metric associated with $\mbf{v}$ will be given by the Finsler-type function
\begin{equation}\label{pseudoFinsler}
\varphi_\mbf{v}(x,z) =\limsup_{h\to 0+} \frac1{h}\, c_\mbf{v}\big(x,x+ h z\big)
\end{equation}
It is a Borel integrand from $\R^d\times\R^d$ to $[0,+\infty]$ which is convex l.s.c. and positively one homogeneous in $z$. 

We shall consider continuous parametrized curves in the metric space $X_{\mbf{v}}$. Such a curve can 
be represented by a continuous function $\gamma: [0,1] \to \R^d$  where $\Rd$ is equipped with the pseudo-distance $c_\mbf{v}$.
We will abbreviate by saying that $\gamma $ is $c_\mbf{v}$-continuous 
(note that $\gamma$ may jump  from $x$ to $y$ if $c_\mbf{v}(x,y)=0$ !).
To such a curve we associate its length defined by
\begin{equation}\label{def:length}
 L_{\mbf{v}}(\gamma) = \sup\left\{  \sum_{i=0}^{n-1} c_\mbf{v}(\gamma(t_{i+1}), \gamma(t_{i}))\ :\  0\le t_0< t_1 <\dots <t_n \le 1\right\}.
\end{equation}
A curve  $\gamma$ of finite length is called rectifiable. 

The geodesic (or inner) distance $c^g_\mbf{v}(a,b)$ between two points $a,b$  is given by
$$  c^g_\mbf{v}(a,b) \ :=\  \inf \Big\{ L_{\mbf{v}}(\gamma) \, : \, \gamma \ \text{$c_\mbf{v}$-continuous}\ ,  \gamma(0) =a, \ \gamma(1) =b \Big\}.$$
If a minimizer $\gamma$ exists, the image of $\gamma$ will be called a geodesic curve for $c_{\mbf{v}}$ joining $a$ to $b$.
In general one has  $c_\mbf{v}\le c^g_\mbf{v}$ possibly with a strict inequality. Fortunately, in our case, the equality holds true.

\begin{proposition}\label{geo-segment} Let $\mbf{v}$ be an element of $\mbf{M}_\O$ and $a,b \in \R^d$. Then:
\begin{itemize}[leftmargin=1.5\parindent]
\item [(i)] Let $\gamma(t)=(1- t) \, a + t\, b$ for $t\in [0,1]$. Then it holds that 
$c^g_\mbf{v}(a,b) \le L_{\mbf{v}}(\gamma)  \le \ell_\mbf{v}(a,b).$ As a result we have $c^g_\mbf{v}(a,b)= c_\mbf{v}(a,b)$. 
Moreover, the equality $c_\mbf{v}(a,b)= \ell_\mbf{v}(a,b)$ is true if and only if the scalar monotone function 
$\mbf{v}_{a,b}$ has a constant slope.
 \item [(ii)]  Assume that the infimum in \eqref{def:cost} is attained for suitable points $x_1,x_2,\dots , x_N \in \R^d$. Then 
 the polygonal curve  $C= \cup_{i=1}^{N-1} [x_i,x_{i+1}]$ is a geodesic joining $a$ to $b$ while for each $i$ the scalar function  $\mbf{v}_{x_i,x_{i+1}}$ 
is affine on $(0,1)$ .
\item [(iii)]   Assume that $\mbf{v}(x)= \{v(x)\}$ where $v$ is a Lipschitz map 
 and let $\varphi_\mbf{v}$ be given by \eqref{pseudoFinsler}. Then it holds that
\begin{equation}\label{integral-length}
 L_{\mbf{v}}(\gamma) = \int_0^1 \varphi_\mbf{v} \big(\gamma(t),\gamma'(t)\big)\ dt \qquad   \text{for every $\gamma\in {\rm Lip}([0,1];\Rd)$.}
\end{equation}

 \end{itemize}
\end{proposition}
\begin{remark}\label{Finsler}  The extension of the integral representation \eqref{integral-length} to general $\mbf{v}\in \mbf{M}_\O$ is a delicate issue.
It turns out that, for a Lipschitz $\mbf{v}$, the Finsler pseudo-metric $\varphi_\mbf{v}$ satisfies
$$ \varphi_\mbf{v}(x,z) =\limsup_{h\to 0+} \frac1{h}\, \ell_\mbf{v} \big(x,x+ h z\big) = \sqrt{2}\,\big(\pairing{e(v)(x), z\otimes z}\big)^{1/2},$$
 at every point $x$ where $v$ is differentiable.
A natural guess would be that the formula \eqref{integral-length} is still valid for a general $\mbf{v}$ if we take  $\varphi_\mbf{v}(x,z)= \limsup_{h\to 0+} \frac1{h} \ell_\mbf{v}(x,x+ h z)$  allowing infinite values. 
\end{remark}

The main point of the proof of the assertion (i) relies on the following one dimensional lemma:
\begin{lemma}\label{esti-1d}
Let $f: [0,1] \to \R$ be a bounded non-decreasing function and denote by $f'(x)$ the a.e. defined derivative of $f$.
   Then the function 
   $$m_f(s,t)\, := \inf_{\{s= t_0<t_1 < \dots < t_N= t\}}  \sum_{i=0}^{N-1}  \sqrt{\big(t_{i+1}- t_i\big)\big(f(t_{i+1}-0)- f(t_i+0)\big)}$$ 
  is such that for every $0\le s<t\le 1$: 
\begin{equation}\label{masterineq2}
m_f(s,t)   \le\  \int_s^t \sqrt{f'(x)}\, dx 
\ \le\  \sqrt{f(1_-)- f(0_+)}.
\end{equation}
  Moreover, if $(s,t)=(0,1)$, the inequalities above are equalities if and only if $f$ is affine in $(0,1)$. 
\end{lemma}
\begin{proof} 
Without loss of generality we may assume that $f$ is left continuous, in particular that $f(1_-)=f(1)$. Let $df = f'(x) \, dx + f'_s$ denote the decomposition of the Lebesgue-Stieltjes measure $df$ (where $f'_s$ is the singular part of $df$).
Let $\{t_i,\, 0\le i\le N\}$ be a subdivision of $[s,t]$ and $\theta: [0,1]\to \R_+$ a continuous function.
Let $\theta_i= \theta \big(\frac{t_i+ t_{i+1}}{2}\big)$. Then, from inequality
$2 \sqrt{(t_{i+1}- t_i)(f(t_{i+1})- f(t_i+0))} \le  \theta_i\cdot \big(f(t_{i+1})- f(t_i+0)\big) + \frac1{\theta_i}(t_{i+1}- t_i) $
we infer that 
$$m_f(s,t)\ \le\  \sum_{i=0}^N  \frac1{2}\, \theta_i \cdot \big(f(t_{i+1})- f(t_i+0)\big) + \frac1{2\theta_i} (t_{i+1}- t_i) ,$$ 
where  in the right hand side we recognize a Darboux sum related to the integral
 $\int_{]s,t]} \frac{\theta}{2} \, df +  \int_{]s,t]} \frac{1}{2\theta} \, dx.$
 Thus taking the limit as the size of subdivisions tends to zero and by density of continuous functions in $L^1_\mu\big([0,1]\big)$,
 for $\mu= \mathcal{L}^1\mres [0,1] + f'_s$, we derive that
 $$ m_f(s,t)\ \le  \int_s^t   \frac1{2} \left(\theta(x) f'(x) +\frac{1}{\theta(x)}\right)  dx + \int_{[s,t]} \frac{\theta(x)}{2} \, f'_s(dx) $$
for  every positive Borel function $\theta$. Let $\eps>0$ and let $B$ be a Borel subset of full Lebesgue measure in $[0,1]$ such that
 $f'_s(B)=0$. Then, by taking $\theta_\e (x) = (f'(x)\vee \eps)^{-1/2}$ if $x\in B$ while $\theta_\e (x)=0$ otherwise and by applying the dominated convergence theorem,
  we infer that:
 $$ m_f(s,t)  \le  \limsup_{\e\to 0}  \int_s^t    \frac1{2} \left(\theta_\e(x) f'(x) +\frac{1}{\theta_\e(x)}\right)  dx  = \int_s^t \sqrt{f'(x)}\, dx.$$
Furthermore we notice that $f(1)- f(0_+)= \int_0^1 f'(x)\, dx + \int_{(0,1]}  f_s'$. Thus, as a consequence of Schwarz's inequality, we have
$$\int_0^1 \sqrt{f'(x)}\, dx \le \Big(\int_0^1 f'(x)\, dx\Big)^{1/2} \le  \sqrt{f(1)- f(0_+)},$$
with  equalities if and only if $f_s'=0$ on $(0,1]$ and $f'=f(1)- f(0_+)$ a.e.
\end{proof}

\begin{proof}[Proof of Proposition \ref{geo-segment}]\  Let $\gamma(t)=(1- t)\,a + t \,b$ for $t\in [0,1]$ and $f(t)=2 \, \mbf{v}_{a,b}(t)$.
We claim that for every $0\le s<t\le 1$ it holds that $c_\mbf{v}\big(\gamma(s),\gamma(t)\big)\le \int_s^t \sqrt{f'(x)} \, dx$.
Indeed by definition \eqref{def:cost}, we have
$$ c_\mbf{v}(\gamma(s),\gamma(t))\le \inf_{\{s=t_0<t_1 < \dots < t_N=t\}}  \sum_{i=0}^{N-1} \ell_\mbf{v}\big(\gamma(t_i),\gamma(t_{i+1})\big) = m_f(s,t).$$
where in the last equality we used the fact that $\ell_\mbf{v}\big(\gamma(t_i),\gamma(t_{i+1})\big)=  \sqrt{(t_{i+1}- t_i)(f(t_{i+1}-0)- f(t_i+0))}$ 
and the definition of  $m_f(s,t)$ introduced in Lemma \ref{esti-1d}. The claim then follows directly from \eqref{masterineq2}.
Next we can easily deduce an upper bound for the length $L_\mbf{v}(\gamma)$:
$$ L_\mbf{v}(\gamma):= \sup\left\{  \sum_{i=0}^N c_\mbf{v}\big(\gamma(t_{i+1}), \gamma(t_{i})\big)\ :\  t_0=0< t_1 <\dots <t_N =1\right\} \le \int_0^1  \sqrt{f'(x)} \, dx.$$
Since $\int_0^1  \sqrt{f'(x)} \, dx\le \sqrt{f(1_-)-f(0_+)}  = \ell_\mbf{v}(a,b)$, we infer that $c^g_\mbf{v}(a,b) \le  L_\mbf{v}(\gamma) \le \ell_\mbf{v}(a,b)$. As it is true for any
 $(a,b)\in (\Rd)^2$, the geodesic pseudo-distance $c^g_\mbf{v}$ is a sub-additive minorant of $\ell_\mbf{v}$. Therefore it cannot be larger that $c_\mbf{v}$
and the equality  $c^g_\mbf{v}=c_\mbf{v}$ follows.

Eventually, we observe that the equality $c_\mbf{v}(a,b)= \ell_\mbf{v}(a,b)$ is equivalent to equalities $L_\mbf{v}(\gamma)=m_f(0,1)=(f(1_-)-f(0_+))^{1/2}$
which by Lemma \ref{esti-1d} amounts to saying that $f$ is affine. The assertion (i) is proved.
The  assertion (ii) is a straightforward consequence of the former equivalence. Indeed, assume that there exits $x_1,x_2,\dots , x_N \in \R^d$ such that
$  c_\mbf{v}(a,b) = \sum_{i=1}^{N-1} \ell_\mbf{v}(x_i,x_{i+1}).$
Then, by the sub-additivity of $c_\mbf{v}$, we infer that $\sum_{i=1}^{N-1} \big(\ell_\mbf{v}(x_i,x_{i+1})-c_\mbf{v}(x_i,x_{i+1})\big) =0$, from which follows equalities  $\ell_\mbf{v}(x_i,x_{i+1})=c_\mbf{v}(x_i,x_{i+1})= L_\mbf{v}\big([x_i, x_{i+1}]\big) $ for every $i$. Accordingly, the scalar functions
$\mbf{v}_{x_i,x_{i+1}}$  are affine on $(0,1)$ and the length of the polygonal curve $C$ consisting of the union of the segments $[x_i, x_{i+1}]$ satisfies 
$L_\mbf{v}(C) =\sum_{i=1}^{N-1} L_\mbf{v}([x_i, x_{i+1}])=  \sum_{i=1}^{N-1} \ell_\mbf{v}(x_i, x_{i+1})= c_\mbf{v}(a,b)$.
The assertion (iii) is a consequence of \cite[Thm 1.2]{venturini}) where it is proved that the integral 
length representation \eqref{integral-length} holds with $\varphi_\mbf{v}$ defined by \eqref{pseudoFinsler} as well as with its lower version  where  the upper limit as $h\to 0+$ in \eqref{pseudoFinsler} is replaced  by the lower limit.
\end{proof}

\begin{corollary}\label{distance-cost}\ Let $\mbf{v}$ be an element of $\mbf{M}_\O$. Then the quotient space $(X_{\mbf{v}},c_\mbf{v}) $  is a geodesic locally compact metric space. For every $a,b\in \R^d$, one has
$
c_\mbf{v}(a,b)  =
\min \big\{ L_{\mbf{v}}(\gamma)  :  \gamma(0) =a \ ,\ \gamma(1) =b \big\} ,$
where the minimum is reached over parametrized curves $\gamma: [0,1] \to \R^d$ such that\
$ c_\mbf{v}\big(\gamma(s), \gamma(t)\big)\,=\, c_\mbf{v}(a,b) \, |s-t|$ for  all $ s,t \in [0,1].$
 
\end{corollary}

\begin{remark*}  The existence of a geodesic  curve $\gamma: [0,1] \to \R^d$ with a finite Euclidean length is not known
unless $\mbf{v}$ is assumed to be strongly monotone. 
\end{remark*}

\begin{proof}\  $X_{\mbf{v}}$ is a metric space when it is equipped with the distance $d_\mbf{v}(\dot x, \dot y):=  c_\mbf{v}(x,y)$
(definition independent of the choice of the representative in each class).
To show that it is locally compact, it is enough to check that any bounded sequence $(\dot x_n)$ admits at least one cluster point.
It turns out that $(x_n)$ is bounded in $\Rd$ (for the Euclidean norm) as a consequence of the following equality:
\begin{equation}\label{cv-bounded}
\forall a\in\Rd\qquad \lim_{|x| \to \infty}  \frac{c_{\mbf{v}}(a,x)}{|x|} \ =\  
\sqrt{2}.
\end{equation}
Therefore $|x_{n_k}-x|\to 0$ for a suitable subsequence $(x_{n_k})$ and $x\in\Rd$. Since by \eqref{cv-dual} and assertion (i) of Theorem 
\ref{cv-dual} we have
$  c_{\mbf{v}}(x_{n_k},x) \le  \ell_{\mbf{v}}(x_{n_k},x) \le \sqrt{2} \, |x_{n_k}-x|^{1/2} ( R + |x_{n_k}-x|)^{1/2}\ ,$
we conclude that $d_{\mbf{v}}(\dot x_{n_k},\dot x)= c_{\mbf{v}}(x_{n_k},x) \to 0$.
Summarizing, we have shown that $X_{\mbf{v}}$ is a locally compact metric space. Moreover the length of a curve $u:[0,1]\to X_{\mbf{v}}$  can be 
recast from the length $L_{\mbf{v}}(\gamma)$ defined in \eqref{def:length} by taking any $\gamma$ such that $\gamma(t)\in u(t)$ (note that $\gamma$ can be discontinuous if we consider $\Rd$ with the Euclidean norm).
 Then, by  assertion (i) of Proposition \ref{geo-segment}, we have the equality
 $$ d_{\mbf{v}}(\dot a,\dot b) =\inf \Big\{ L_\mbf{v}(u) \ : \ u\in C^0\big([0,1];X_{\mbf{v}}\big), \  \ \gamma(0)=\dot a, \  \gamma(1)=\dot b\, \Big\},$$
  from wich follows that  $(X_{\mbf{v}}, d_{\mbf{v}})$ is a length space.
 The existence of an optimal curve $u$ is then a consequence of Hopf-Rinow theorem  for which we refer to the book \cite{Papadopoulos} (Theorem 2.4.6 in particular). Moreover, $u$ can be constructed  so that it is injective with a constant speed i.e.  $d_\mbf{v}\big(u(s), u(t)\big)\, =\, d_\mbf{v}(a,b)\,  |s-t|$ for all $s,t\in [0,1]$. This is precisely the statement of our corollary.

Eventually, it remains to show the claim \eqref{cv-bounded}. 
The inequality $ \limsup_{|x| \to \infty}  \frac{c_{\mbf{v}}(a,x)}{|x|}\le \sqrt{2}$
follows directly from \eqref{lv<}. To obtain the converse inequality, we consider a positive real $L$ large enough that
 $\Ob \subset\{|x| < L\}$. Since $\mbf{v}={\rm id}$ on the complement of $\Ob$, the tangential component of $\mbf{v}$ is affine on the segment
$[z,x]$ where $z= L \frac{x}{|x|}$ when $x \in \Rd \setminus \Ob$. Therefore, by  Proposition \ref{geo-segment}, we have $c_{\mbf{v}}(x,z)=\sqrt{2}\, |x-z|=\sqrt{2}\, (|x|-L)$.
By exploiting the sub-additivity of $c_{\mbf{v}}$, we end up with $c_{\mbf{v}}(a,x)\ge \sqrt{2} (|x|-L) -M_L(a) $ 
where $M_L(a)=\sup\{ c_{\mbf{v}}(a,z): |z|\le L\} <  +\infty$. It follows that $\liminf_{|x| \to \infty}  \frac{c_{\mbf{v}}(a,x)}{|x|}\ge \sqrt{2}$. 
\end{proof}

\subsection{Dual achievement through maximal monotones maps}
We are now in position to revisit the dual problem $(\mathcal{P}^*)$ introduced in Subsection \ref{duality}
that we are going to recast in the following geometric form:
\begin{equation}\label{revised-dual}
I_0(f,\Sigma_0)= \sup_{ (u,\mbf{v})\in C_{\Sigma_0}(\Ob) \times\mbf{M}_\O} \Big\{  \pairing{f,u} \ :\  u(x_1) - u(x_2)  \le  c_{\mbf{v}}(x_1,x_2) \quad \forall\, (x_1,x_2) \in \Ob \times \Ob\Big\} \tag{$\mathcal{P}_{\rm geo}^*$}
\end{equation}
 \begin{theorem}\label{sol-rev-dual} Assume that $\Sigma_0$ is non empty. Then the supremum in \eqref{revised-dual} is a maximum
 and we have the equality  $ \max (\mathcal{P}^*)=\max (\mathcal{P}_{\rm geo}^*)$. Moreover,
a pair $(u,\mbf{v})\in C_{\Sigma_0}(\Ob) \times\mbf{M}_\O$ solves \eqref{revised-dual} if and only if  
 $(u,w)$ is optimal for  $(\mathcal{P}^*)$, where $w={\rm id}-\mbf{i}(\mbf{v})$. 
 \end{theorem}
 \begin{proof} 
  The existence of an optimal pair $(u,\mbf{v})$ for \eqref{revised-dual} is straightforward. Indeed if $(u_n, \mbf{v}_n)$ is 
 a maximizing sequence, then $\{u_n\}$ is equicontinuous as a consequence of the uniform upper bound estimate in \eqref{bounds-c_v} 
 hence relatively compact in $C_{\Sigma_0}(\Ob)$ (by Ascoli's theorem) while we recall that $\mbf{M}_\O$ is a compact metric space (see Lemma \ref{homeo}).
 By exploiting the upper semicontinuity property of the map $\mbf{v} \in \mbf{M}_\O \mapsto c_\mbf{v}(x_1,x_2)$ holding for all $x_1,x_2$ 
 (see assertion (iii) in Theorem \ref{cv-dual}), we see that any cluster point $(u, \mbf{v})$ satisfies the inequality constraint $u(x_1) - u(x_2)  \le  c_{\mbf{v}}(x_1,x_2)$ and therefore is a solution of \eqref{revised-dual}.
 The equivalence between the formulations $(\mathcal{P}^*)$ and $(\mathcal{P}_{\rm geo}^*)$ of the dual problem is a 
 straightforward consequence of the forthcoming  Proposition \ref{explicitKbar}.

 \end{proof}
Recall that the set of competitors for $(\mathcal{P}^*)$ is the closure in $C^0(\Rd)\times L^1(\Rd;\Rd)$ of the subset $\K$  defined in \eqref{def:calK}.
By Lemma \ref{esti}, we already know that $\bK$  is a bounded convex subset
of $\big(C^{0,\frac1{2}}\cap W^{1,2}(\Omega)\big) \times \big((BV\cap L^\infty)(\Omega;\R^d)\big)$.
The complete characterization of $\bK$ given below is  crucial:  

\begin{proposition}
		\label{explicitKbar}
	Let us be given  $u \in C^0(\Ob)$, such that $u = 0$ on $\Sigma_0$, and $w \in L^1(\O;\Rd)$. Then, the following conditions are equivalent:
		\begin{enumerate}[leftmargin=1.5\parindent]
			\item [(i)] $(u,w) \in \bK$;
			
			\item  [(ii)]  there exists an element $\mbf{v}\in \mbf{M}_\O $ such that $\mbf{v}(x)=\{x- w(x)\}$ \ for a.e.  $x\in\Omega$ and
			\begin{equation}\label{cv2points}
			u(x_1) - u(x_2)\  \le\  c_{\mbf{v}}(x_1,x_2) \quad\quad \forall\, (x_1,x_2) \in \Ob \times \Ob 
			\end{equation}
			(or equivalently \
			$u(x_1) - u(x_2) \le  \ell_{\mbf{v}}(x_1,x_2) \quad \forall\, (x_1,x_2) \in \Ob \times \Ob$);

			\item [(iii)] $u\in W^{1,2}(\O)$ and there exists  $v\in \mathcal{B}_\O $ such that $v={\rm id}-w$ a.e. in $\Omega$ and
			\begin{equation}\label{ineq-measure}
  \frac{1}{2}\,\nabla u \otimes \nabla u \ \mathcal {L}^d \mres\O \ \le\  e(v) \  \quad \text{in $\Mes(\Rd;\Sdd)$ }
\end{equation}
where $e(v)$ is the symmetric part of the distributional derivative $D v \in \Mes(\Rd;\Rd)$.
		\end{enumerate}
		\end{proposition}
\begin{remark}\label{u-affine}  In view of Theorem \ref{cv-dual}, the condition \eqref{cv2points} in (ii)  is in fact equivalent to the weaker condition
 $u(x_1) - u(x_2) \le  \ell_{\mbf{v}}(x_1,x_2)$ for all $(x_1,x_2)\in \Ob\times\Ob$. 
 In addition, if such a function $u$ realizes the equality $u(a) - u(b) = \ell_{\mbf{v}}(a,b)$ at some $(a,b)\in\Ob^2$, then
 $u$ has be affine on $[a,b]$. Indeed, the equality above implies that $\ell_{\mbf{v}}(a,b)=c_{\mbf{v}}(a,b)$ and, by applying assertion (ii) of Proposition \ref{geo-segment}, we know that $\mbf{v}\cdot \tau^{a,b}$ is affine on $]a,b[$. Then the condition \eqref{cv2points} holding as an equality for any $(x_1,x_2)\in [a,b]^2$
 enforces $u$ to be affine as well.
 This observation can be useful when dealing with polygonal geodesic curves corresponding to some optimal $\mbf{v}$.
 \end{remark}
 
\begin{remark}\label{revis-2points}   If $w$ is Lipschitz as it was assumed in Section \ref{sec:optimembrane}
it holds that $\mbf{v}(x)=\{x-w(x)\}$ for all $x\in\Rd$ and  we recover the two-point condition \eqref{eq:two_point_condition} from   
\eqref{cv2points} whilst condition (iii) is  equivalent to \eqref{pointwise}.
\end{remark}

\begin{proof} \ First we show that $(i)\Rightarrow (ii)$: let $(u_n,w_n)\in \K$ be a sequence such that $u_n\to u$ in $C^0(\Ob)$ and $w_n \to w$ in $L^1(\O;\Rd)$. Then $u=0$ on $\Sigma_0$  and  $(u_n,w_n)$ satisfies the two-point conditions \eqref{eq:two_point_condition}. Setting $\mbf{v}_n = \ident - w_n$ we 
easily infer  that $(u_n,\mbf{v}_n)$ satisfies \eqref{cv2points} while $\mbf{v}_n \in \mbf{M}_\O$. Due to the homeomorphism between $\mbf{M}_\O$ and $\mathcal{B}_\O$ (see Lemma \ref{homeo}) we deduce the convergence $\mbf{v}_n \to \mbf{v}$ and $(u, \mbf{v})$ satisfies \eqref{cv2points} thanks to 
the upper semicontinuity of the map $\mbf{v} \mapsto c_\mbf{v}(x_1,x_2)$ (see the assertion (iii) in Theorem \ref{cv-dual}). 

Let us now prove the implication  $(ii)\Rightarrow (iii)$. We know that $v:=\mbf{i}(\mbf{v})$ belongs to $BV\cap L^\infty(\O;\Rd)$.
Let $\omega\Subset \O$ and fix $h\in\Rd$ such that $|h|<\delta:= {\rm dist}(\omega,\bO)$. As the set $\big\{x\in \omega: \mbf{v}(x)=\{v(x)\},\
\mbf{v}(x+h)=\{v(x+h)\} \big\}$ is of full Lebesgue measure in $\omega$, 
 the condition  \eqref{cv2points}  implies that:
\begin{equation} \label{ineqh}
\frac1{2} \big(u(x+h) -u(x)\big)^2 \ \le\ \pairing {v(x+h) - v(x), h}\ \qquad \text{for a.e. $x\in \omega$}.
\end{equation}
  By integrating over $\omega$ and with the help of Schwarz inequality, we deduce that  
$$ \frac1{2} \int_\omega \big(u(x+h) -u(x)\big)^2\, dx\ \le\ |h|^2\, \int_\omega \left|\frac{v(x+ h) - v(x)}{\abs{h}}\right|\, dx \ \le\ C \, |h|^2\ ,$$
where the constant $C$ is finite and depends only of  $\int_\O |Dv|$ (the variation of $v$ as an element of $BV(\O;\Rd)$).  
Since $ \omega$ and $h$ with $|h|<\delta$ can be chosen arbitrarily, the latter uniform upper bound entails that $u\in W^{1,2}(\O)$ (see for instance \cite[Proposition 9.3]{brezis} and subsequent Remark 6).
Now we may also multiply \eqref{ineqh} by a test function $\f\in \D(\O;\R_+)$ and integrate over $\O$ so that, setting  $h=\e \, z$ with $z\in S^{d-1}$ and dividing  by $\e^2$, we obtain:
\begin{equation} \label{ineq2eps}
 \frac1{2} \int_\O \left(\frac{u(x+\e z) -u(x)}{\e}\right)^2 \f(x)\, dx\ \le\  \int_\O \pairing {\frac{v(x+\e z) - v(x)}{\e}, z}\, \f(x) \, dx\ .
 \end{equation}
 As $u\in W^{1,2}(\O)$,  $\bigl(\frac{u(x+\e z) -u(x)}{\e}\bigr)$ remains bounded in $L^2(\O)$ and converges to $\pairing {\nabla u,z}$ therein. Besides we have:
 \begin{align*} \lim_{\e\to 0} \int_\O \pairing {\frac{v(x+\e z) - v(x)}{\e}, z}\, \f(x) \, dx&= \lim_{\e\to 0}  \int_\O \pairing {v(x), z} \left(\frac{\f(x-\e z)-\f(x)}{\e}\right)  dx \\
 &=  -\int_\O \pairing {v(x), z} \, \dive( z\f )\, dx  = \big \langle\pairing {Dv\, z, z}, \f \big \rangle.
\end{align*} 
 Hence, from \eqref{ineq2eps} we infer that
 $\frac1{2} \int_\O  | \pairing{\nabla u(x),z}|^2\, \f(x)\, dx \le  \big \langle\pairing {Dv\, z, z}, \f \big\rangle =  
 \big \langle\pairing {e(v)\, z, z}, \f \big\rangle.$
 The inequality \eqref{ineq-measure} follows by the arbitrariness of $z\in S^{d-1}$ and of $\f\in \D(\O;\R_+)$.

%
Eventually we turn to the implication $(iii) \Rightarrow (i)$. Let $w= {\rm id}-v$. Note that, as $v\in\mathcal{B}_\O $, $w$ vanishes a.e. in $\Rd\setminus \Ob$  and may have a jump on $\bO$ giving a contribution to the singular part $e_s(w)$ of the symmetric tensor measure $e(w)$.
 We need to construct a sequence of Lipschitz functions  $(u_\e,w_\e)\in \K$ converging to $(u,w)$ in $C^0(\Ob)\times L^1(\O;\Rd)$. 
 In fact, we will mostly use  the same arguments as in  the proof of Lemma \ref{density}. 
First we construct an extension $\tilde u$ of $u$ using the same method as in the Step 1 of the proof of Lemma \ref{density}. 
This extension belongs to $C^0\cap W^{1,2}_{\rm loc}(\Rd)$  whilst, due to \eqref{ineq-measure} and recalling the dedinition of $g$ in \eqref{def:g}, the pair $(\tilde u,  w)$ satisfies the inequality
\begin{equation}\label{ineq2uw}
g\big(\nabla \tilde u, \{e( w)\}\big) \le 1 \ \text{a.e. in $\Rd$}, \qquad e_s( w)\le 0\ \text{  in $\Mes(\Rd;\Sdd)$}.
\end{equation}
where $\{e(w)\}$ denotes the Lebesgue density of the absolutely continuous part  of  the measure $e(w)$ (in fact
$\{e(v)\}$ coincides with the symmetric part of the approximate gradient  $w$ which exists a.e., see for instance \cite[Thm 3.84]{ambrosio-fusco}).
Next we associate to the integrand $g$  the  convex l.s.c. and one homogeneous integrand $h:\R^d \times \Sdd\times \R \to [0,+\infty]$ defined by
$$ h(z,M,t) = \begin{cases} t\, \rho_+\left( \frac{z\otimes z}{2 t^2} + \frac{M}{t}\right)  & \text{if $t>0$}\\
\rho_+(M) &\text{if $t=0$ and $z=0$}\\
+\infty &\text{if $t<0$}
\end{cases} 
$$
that is $t g(z/t,M/t)$ for $t>0$  and its limit $g^\infty(z,M):=\lim_{t\to 0} t g(z/t,M/t)$ for $t=0$ (recession function of $g$).  
 Then we may rewrite \eqref{ineq2uw} as an inequality between scalar measures  as follows:
 \begin{equation}\label{ineq3uw}
 h\left(\nabla \tilde u\, \mathcal {L}^d, e(w), \mathcal {L}^d\right)  \le\   \mathcal {L}^d   \quad \text{in $\Mes_+(\Rd)$}.
\end{equation}
Indeed  this equivalence can be easily checked by noticing that the  measure in the left hand side (intended in the sense of Goffman-Serrin \cite{Goffman})  admits the Lebesgue-Nikodym decomposition
$$h\left(\nabla \tilde u\, \mathcal {L}^d, e(w), \mathcal {L}^d\right) = h\big(\nabla \tilde u, \{e(w)\}, 1\big) \mathcal {L}^d  + h\big(0, e_s(w),0\big)
= g\big(\nabla \tilde u, \{e( w)\}\big)\, \mathcal {L}^d + \rho_+\big(e_s(w)\big).$$
Assume for a moment that  $\spt(\tilde u) \Subset\R^d\backslash \Sigma_0$ while $\spt( w)\Subset \O$.
Then we consider $(u_\e, w_\e) :=({\tilde u} * \theta_\e, {w} * \theta_\e)$ for $\theta_\e(x)= \e^{-d} \theta\big(\frac{x}{\e}\big)$ being  a smooth convolution kernel (with  $\theta$  radial symmetric and  $\int \theta =1$). Then $u_\e=0$ on $\Sigma_0$ and $(u_\e,w_\e) \to (u,w)$ in  $C^0(\Ob)\times L^1(\O;\Rd)$  as $\e\to 0$. Moreover, denoting  $\xi:=\left(\nabla \tilde u\, \mathcal {L}^d, e(w), \mathcal {L}^d\right)$, we have
$ \xi * \theta_\e= \left(\nabla \tilde u_\e, e(w_\e), 1\right) \mathcal {L}^d$, hence   by 
applying Lemma \ref{mollih} and in virtue of \eqref{ineq3uw}, we infer  the following inequalies in $\Mes_+(\Rd)$:
		$$   g \big( \nabla u_\e, e(w_\e)\big) \, \mathcal{L}^d  \ \le\  g\big( \nabla \tilde u, \{e(w)\}\big) \, \mathcal{L}^d \, +
		 \, \rho_+ \big (e_s( w )\big) \  \le\       \mathcal{L}^d.   $$
 In particular we have   $g\big( \nabla u_\e, e(w_\e)\big)\le 1$  a.e. and by taking  the restriction of $(u_\e, w_\e)$ to $\Ob$
we obtain  the desired sequence in $\K$ converging to $(u,w)$.
Eventually the restriction on the support of $(u,w)$ can be dropped by repeating word for word the dilation argument used in the second step of the proof of Lemma \ref{density}.

\end{proof}

\subsection{Maximal Monge-Kantorovich metrics and saddle point formulation}\label{MK_sec}

All along this subsection we will assume that the load $f=f_+-f_-$ is a signed measure where $f_+, f_-\in\Mes_+(\Ob\setminus \Sigma_0)$. It will be sometimes useful to consider alternative  decompositions $f=\mu - \nu$ where $\mu, \nu\in\Mes_+(\Ob\setminus \Sigma_0)$ are not necessarily singular with respect to each other 
(in contrast with the case of the usual Jordan decomposition).
 Up to measures concentrated on $\Sigma_0$, such measures $\mu$ and $\nu$ will act as a source and target measures in the Monge-Kantorovich transport problem associated with the $c_{\mbf{v}}$-cost.
Let us fix  an element $\mbf{v}\in \mbf{M}_\O$ and define 
\begin{equation}\label{distMKv}
 W^{\Sigma_0}_{c_{\mbf{v}}}(\mu,\nu) := \inf \Big\{ W_{c_{\mbf{v}}}(\mu + \mu_0, \nu + \nu_0 )\ :\ \mu_0,\nu_0 \in \Mes_+(\Sigma_0)  \Big\} 
\end{equation}
%
where the Monge-Kantorovich $c_{\mbf{v}}$-distance between two elements $\rho_1,\rho_2\in\Mes_+(\Ob)$ is defined  by
\begin{equation*}
	W_{c_{\mbf{v}}}(\rho_1,\rho_2) := \inf \left\{\int_{\Ob\times\Ob} c_{\mbf{v}}(x,y) \,\gamma(dxdy)\ :\ \gamma \in \Gamma(\rho_1,\rho_2) \right\}.
\end{equation*}
Note that $W_{c_{\mbf{v}}}$ and $W^{\Sigma_0}_{c_{\mbf{v}}}$ are symetric and that $W_{c_{\mbf{v}}}(\rho_1,\rho_2)<+\infty$ if and only if $ \int \rho_1= \int \rho_2$. By using the lower semicontinuity of the map $(\rho_1,\rho_2) \mapsto W_{c_{\mbf{v}}}(\rho_1,\rho_2)$, one checks easily that the infimum
in \eqref{distMKv} is actually a minimum. 

\begin{remark}\label{MKdSigma} \  If $f$ is a probability on $\Ob\setminus \Sigma_0$, it is easy to check that 
 $$W^{\Sigma_0}_{c_{\mbf{v}}}(f,0)=\min \Big\{ W_{c_{\mbf{v}}}(f, g)\ :\ g \in \PP(\Sigma_0)  \Big\} = \int c_{\mbf{v}}(x,\Sigma_0)\, f(dx).$$  
Then we recover the Monge-Kantorovich $c_{\mbf{v}}$-distance between $f$ and
 $\Sigma_0$ (i.e. $W_{c_{\mbf{v}}}(f,\Sigma_0)$) in a similar way as in \eqref{def:W1-Sigma}, \eqref{dualMK}. 
In this sense \eqref{distMKv} generalizes the notion of distance to $\Sigma_0$ for a general signed measure {$f = \mu - \nu$}.  
\end{remark}


The next preliminary result is an adaptation of the Kantorovich-Rubinstein duality theorem  (see for instance \cite[Thm 1.14]{villani}) to our framework:

\begin{lemma}\label{Rubin}\ Let $\mu,\nu\in\Mes_+(\Ob\setminus\Sigma_0)$ and $\mbf{v}\in \mbf{M}_\O$. Then
\begin{equation}\label{MKv}
 W^{\Sigma_0}_{c_{\mbf{v}}}(\mu,\nu) =  \max_{ u\in C_{\Sigma_0}(\Ob)} \Big\{  \pairing{\mu-\nu,u} \ :\  u(x_1) - u(x_2)\  \le\  c_{\mbf{v}}(x_1,x_2) \quad \forall\, (x_1,x_2) \in \Ob \times \Ob\Big\}. 
\end{equation}
As a consequence $ W^{\Sigma_0}_{c_{\mbf{v}}}(\mu,\nu)= W^{\Sigma_0}_{c_{\mbf{v}}}(f_+,f_-)$ whenever $\mu-\nu=f$.
Furthermore, the infimum in \eqref{distMKv} is attained on the convex weakly* compact subset 
 \begin{equation}\label{gbound}
	\Mes(\mu,\nu;\Sigma_0) := \left\{ (\mu_0,\nu_0)\in \bigl( \Mes_+(\Sigma_0) \bigr)^2 \  :\ \text{$ \int (\mu + \mu_0)= \int (\nu + \nu_0),\ \ \int \nu_0 \le \int \mu ,\ \  \int \mu_0 \le \int \nu$}\right\}.
\end{equation} 
\end{lemma}

\begin{proof}\  Let $x_0\in \Sigma_0$. By the uniform H\"older estimate \eqref{lv<},\,\eqref{bounds-c_v} and Ascoli's theorem, the set
 $$ \B:=\Big\{u\in C^0(\Ob)\, :\ u(x_0)=0, \ \ u(x_1)-u(x_2)  \le  c_{\mbf{v}}(x_1,x_2) \ \ \forall\, (x_1,x_2) \in \Ob\times\Ob\Big\}$$
is a convex compact subset of  $C^0(\Ob)$.
Since the cost function $c_{\mbf{v}}$ is continuous and sub-additive, we know from Kantorovich-Rubinstein duality theorem 
that for every $(\mu_0,\nu_0)\in ( \Mes_+(\Sigma_0))^2$, the following equality holds:
$$  W_{c_{\mbf{v}}}(\mu+ \mu_0,\nu + \nu_0)\ =\ \sup_{ u\in \B} \Big\{ \pairing{(\mu +\mu_0)-(\nu+ \nu_0) , u} \Big\}. $$
By applying Ky-Fan's result (Theorem \ref{ky-Fan}) which allows switching the symbols $\inf$ and $\sup$, we are led to 
\begin{align*}
W^{\Sigma_0}_{c_{\mbf{v}}}(\mu,\nu) \ &=\ \inf_{(\mu_0,\nu_0)\in  (\Mes_+(\Sigma_0))^2} \, \sup_{ u\in \B}\  \pairing{(\mu+ \mu_0)-(\nu +\nu_0), u} \\
&\ = \sup_{ u\in \B}\, \left( \pairing{\mu-\nu,u} + \inf_{(\mu_0,\nu_0)\in (\Mes_+(\Sigma_0))^2} \pairing{\mu_0-\nu_0, u} \right)
\end{align*}
hence \eqref{MKv} by noticing that the infimum in the bottom line is finite if and only if $u=0$ on $\Sigma_0$.

As a consequence of \eqref{MKv}, we have $W^{\Sigma_0}_{c_{\mbf{v}}}(\mu,\nu)=W^{\Sigma_0}_{c_{\mbf{v}}} (f_+,f_-)$   whenever $\mu-\nu= f_+-f_-$.
It remains to show the last assertion of Lemma \ref{Rubin}. As $\Sigma_0$ is compact, it can be readily checked that
 $\Mes(\mu,\nu;\Sigma_0)$ is a weakly* convex compact subset of $(\Mes_+(\Sigma_0))^2$ thus, recalling \eqref{distMKv}, we have:
$$ W^{\Sigma_0}_{c_{\mbf{v}}}(\mu,\nu) \le   \min \Big\{ W_{c_{\mbf{v}}}(\mu + \mu_0, \nu + \nu_0 )\ :\ (\mu_0,\nu_0)\in \Mes(\mu,\nu;\Sigma_0) \Big\} \,$$
where the existence of a minimizer in the right side is ensured by the lower semicontinuity of the function $(\mu_0,\nu_0)\mapsto W_{c_{\mbf{v}}}(\mu + \mu_0, \nu + \nu_0 )$.
Consequently, our argument boils down to showing that the  inequality above is in fact an equality.  This will be indeed the case if for every $(\mu_0,\nu_0)$  we are able to construct measures $(\tilde{\mu}_0,\tilde{\nu}_0)\in \Mes(\mu,\nu;\Sigma_0)$ such that 
 $W_{c_{\mbf{v}}}(\mu + \tilde{\mu}_0, \nu + \tilde{\nu}_0 )\le W_{c_{\mbf{v}}}(\mu + \mu_0, \nu + \nu_0 )$. 
 To show this claim, we may assume that $W_{c_{\mbf{v}}}(\mu + \mu_0, \nu + \nu_0 )<+\infty$. Hence $\int (\mu + \mu_0)= \int (\nu + \nu_0)$ and 
there exists  an element $\gamma \in \Gamma(\mu + \mu_0, \nu + \nu_0 )$  such that $W_{c_{\mbf{v}}}(\mu + \mu_0, \nu + \nu_0 ) = 
 \int c_{\mbf{v}} \, d\gamma$.  Set $\tilde\gamma= \gamma - \gamma_0$ where $\gamma_0:= \gamma\mres \Sigma_0\times\Sigma_0$.
 Clearly the marginals of $\tilde\gamma$, further denoted by $\tilde{\mu}$ and $\tilde{\nu}$, satisfy 
 $\tilde{\mu}\mres \Ob\setminus\Sigma_0 =\mu$ and $\tilde{\nu}\mres \Ob\setminus\Sigma_0 =\nu$.  
 Let $\tilde{\nu}_0 = \tilde{\nu}\mres \Sigma_0$ and $\tilde \mu_0 = \tilde{\mu}\mres \Sigma_0$.
 Then it holds that $\tilde{\mu}=\mu + \tilde{\mu}_0$ and $\tilde{\nu}=\nu + \tilde{\nu}_0 $. Furthermore we have
 $ \int \tilde{\mu}_0 = \tilde{\mu}(\Sigma_0) = \gamma(\Sigma_0\times\Sigma_0^c) \le \int \nu $ while $\int \tilde{\nu}_0 = \tilde{\nu}(\Sigma_0) = \gamma(\Sigma_0^c\times\Sigma_0)\le \int \mu.$
 It follows that $(\tilde{\mu}_0,\tilde{\nu}_0)$ belongs to $\Mes(\mu,\nu;\Sigma_0)$, hence follows our claim since:
 $$  W_{c_{\mbf{v}}}(\mu + \tilde{\mu}_0, \nu + \tilde{\nu}_0 )\le \int c_{\mbf{v}} \, d\tilde\gamma\le \int c_{\mbf{v}}\, d\gamma \le W_{c_{\mbf{v}}}(\mu + \mu_0, \nu + \nu_0 ) .  $$
\end{proof}
 
An important consequence of Lemma \ref{Rubin} is that solving the dual problem $(\mathcal{P}^*)$ amounts to identifying a maximal monotone map $\mbf{v}$ for which the 
Monge-Kantorovich distance of $f$ to $\Sigma_0$ is maximal:

 \begin{theorem}\label{worsemetric}  Let $\Sigma_0\subset \bO$ be closed non empty and $f\in \Mes(\Ob\backslash \Sigma_0)$. Then
\begin{equation}\label{maximalMK}  I_0(f,\Sigma_0) = \max (\mathcal{P}^*) \ =\ \max\, \Big\{  W^{\Sigma_0}_{c_{\mbf{v}}}(f_+,f_-)\ : \  \mbf{v}\in \mbf{M}_\O \Big\}.
\end{equation}
\end{theorem}
\begin{proof}  The equality $\sup (\mathcal{P}^*)= \sup \left\{  W^{\Sigma_0}_{c_{\mbf{v}}}(f_+,f_-)\, :\, \mbf{v}\in \mbf{M}_\O\right\}$
follows from the formulation \eqref{revised-dual} where, for every  $\mbf{v}\in \mbf{M}_\O$, we compute the supremum with respect to $u$  by applying \eqref{MKv} for $\mu=f_+$ and $\nu=f_-$. Clearly if $(\bar u, \bar{ \mbf{v}})$ solves \eqref{revised-dual}, then $\bar {\mbf{v}}
$ satisfies
$W^{\Sigma_0}_{c_{\bar {\mbf{v}}}}(f_+,f_-)\ge W^{\Sigma_0}_{c_{\mbf{v}}}(f_+,f_-)$ for all $\mbf{v}\in \mbf{M}_\O$.

 \end{proof} 
As a corollary we now state  a saddle point formulation for an optimal pair $(\gamma,\mbf{v}) \in \Gamma(\mu,\nu;\Sigma_0) \times \mbf{M}_\O$
with respect to the Lagrangian $ L(\mbf{v}, \gamma) = \int_{\Ob\times\Ob} c_{\mbf{v}}\, d\gamma$ where
$$\Gamma(\mu,\nu;\Sigma_0) := \Big\{ \gamma \in \Mes_+(\Ob\times\Ob)\ :\ \exists (\mu_0,\nu_0)\in\Mes(\mu,\nu;\Sigma_0), \ \ \gamma\in \Gamma(\mu + \mu_0,\nu + \nu_0 ) \Big\}. $$

 \begin{corollary}\label{saddle1}  For every $\mu,\nu \in \Mes_+(\Ob\backslash \Sigma_0)$ such that $f=\mu-\nu$ there exists a pair $(\ov \gamma, \ov {\mbf{v}}) \in  \Gamma(\mu,\nu;\Sigma_0) \times \mbf{M}_\O$ solving the saddle point problem:
 \begin{equation}\label{saddle-gamma-v}
\int_{\Ob\times\Ob}  \, c_\mbf{v}\, d \ov\gamma \ \le\ \int_{\Ob\times\Ob}  \, c_\mbf{\ov v}\, d \ov\gamma \ \le\ \int_{\Ob\times\Ob}  \, c_\mbf{\ov v}\, d\gamma  \qquad \forall \, \gamma \in  \Gamma(\mu,\nu;\Sigma_0), \ \ \forall\, \mbf{v}\in \mbf{M}_\O .
\end{equation}
Furthermore, for any such an optimal pair  $(\ov \gamma, \ov {\mbf{v}})$, there exists an optimal potential $\ov u\in C_{\Sigma_0}(\Ob)$ such that $(\ov u, \ov {\mbf{v}})$ solves the dual problem \eqref{revised-dual} and satisfies the relations
\begin{equation}\label{extremality}
\ov u(y)- \ov u(x)  \le c_\mbf{\ov v} (x,y) \quad \forall\, (x,y)\in \Ob^2, \qquad \ov u(y)- \ov u(x)  = c_\mbf{\ov v} (x,y) \quad \ \text{for $\ov\gamma$-a.e. $(x,y)\in\Ob^2$}.
\end{equation}
\end{corollary}

 \begin{remark}\label{optimaltriple} \ By applying  Corollary \ref{saddle1} with $\mu=f_+$ and $\nu=f_-$, we see that a pair $(\ov u, \ov {\mbf{v}})$ is a solution of the dual problem \eqref{revised-dual} if and only if there exists a suitable $\ov \gamma$ such that the triple $(\ov \gamma, \ov {\mbf{v}}, \ov u)$ 
  belongs to   $\Gamma(f_+,f_-;\Sigma_0) \times \mbf{M}_\O\times C_{\Sigma_0}(\Ob)$ and satisfies 
the two conditions \eqref{saddle-gamma-v} and \eqref{extremality}. 

 \end{remark}

 \begin{proof} \  As can be readily checked $\Gamma(\mu,\nu;\Sigma_0)$ is a convex compact subset of $\Mes_+(\Ob\times \Ob)$ endowed with the weak* topology. 
 Besides, from  Lemma \ref{homeo}  we know that the convex set $\mbf{M}_\O$ equipped with the distance $\mbf{h}$ is compact.
 The existence of a saddle point for \eqref{saddle-gamma-v} then follows from  the last assertion of Theorem \ref{ky-Fan}. Indeed,   
the Lagrangian $L(\mbf{v}, \gamma) := \int_{\Ob\times\Ob} c_{\mbf{v}}\, d\gamma$ is convex lower semicontinuous with respect to $\gamma\in \Gamma(\mu,\nu;\Sigma_0) $  and, as a consequence of combining assertion (iii) of Theorem \ref{cv-dual} with Fatou's lemma, it is  concave upper semicontinuous with respect to $\mbf{v}\in  \mbf{M}_\O$.

Let now $(\ov \gamma, \ov {\mbf{v}})$ satisfy \eqref{saddle-gamma-v}. Then, the left hand side inequality in \eqref{saddle-gamma-v} 
 implies that $W^{\Sigma_0}_{c_{\mbf{ v}}}(\mu ,\nu)\le \int_{\Ob^2}\, c_\mbf{\ov v} \, d \ov\gamma$ while the right hand side  implies that
 $\int_{\Ob^2}\, c_\mbf{\ov v} \, d \ov\gamma = W^{\Sigma_0}_{c_{\mbf{\ov v}}}(\mu ,\nu)$. Therefore $\mbf{\ov v}$ is a maximizer in  the right hand side of \eqref{maximalMK} and  we obtain the equalities 
\begin{equation}\label{optiv}
I_0(f,\Sigma_0) = \max \eqref{revised-dual} =  W^{\Sigma_0}_{c_{\mbf{\ov v}}}(\mu ,\nu) = \int_{\Ob^2}\, c_\mbf{\ov v} (x,y)\, \ov\gamma(dxdy) .
\end{equation}
Next, by applying Lemma \ref{Rubin} with $\mbf{v}=\mbf{\ov v}$, there exists an optimal potential $\ov u$ for \eqref{MKv}.  
Clearly the equalities in \eqref{optiv} imply that $(\ov u, \ov {\mbf{v}})$ solves \eqref{revised-dual}
 while we have 
 $$\int_{\Ob^2}\, c_\mbf{\ov v} (x,y)\, \ov\gamma(dxdy) = \pairing{f,\ov u} = \pairing{\mu - \nu,\ov u} =\int_{\Ob^2} (\ov u(y)- \ov u(x)) \, \ov\gamma(dxdy).$$
Since $\ov u$ stisfies  $\ov u(y)- \ov u(x) \le  c_\mbf{\ov v} (x,y)$ for all $(x,y)$, the  equality above is possible if and only if  $ \ov u$ satisfies the equality condition  in \eqref{extremality}. The last assertion of Corollary \ref{saddle1} follows.
  \end{proof}

  \subsection{Extended characterization of optimal truss solutions}
  
   As announced in Section \ref{opticond}, we can now derive the non-smooth extension of the optimality conditions given in Proposition \ref{OptiTwo}. 
 For every $\mbf{v}\in \mbf{M}_\O$ and  $(a,b)\in \R^d\times\R^d$, we define 
 \begin{equation}\label{def:ell+}
\ell_{\mbf{v}}^+(a,b):= \max \left\{  \sqrt{2\pairing{b' - a',b-a}}\ :\  a'\in \mbf{v}(a), \ b' \in \mbf{v}(b) \right\}. \end{equation}
Note that $\ell_{\mbf{v}}^+=\ell_{\mbf{v}}$ at those points where $\mbf{v}$  is single valued or regular enough. Moreover, for a boundary point $a\in\bO$, we may assign $a'=a$ in the maximum  \eqref{def:ell+}, namely (see the proof of Lemma \ref{masterineq} below):
 \begin{equation}\label{bO-ell+}
\frac1{2} \big(\ell_{\mbf{v}}^+(a,b)\big)^2 = \begin{cases} \max \big\{ \pairing{b'- a,b-a} : \ b' \in \mbf{v}(b) \big\}&  \text{if $(a,b)\in \bO\times\O$,}\\
|b-a|^2 & \text{if $(a,b)\in \bO\times\bO$.}
\end{cases}\end{equation}
The following technical result holds:
 \begin{lemma}\label{masterineq} Let us be given $\mbf{v}\in \mbf{M}_\O$ and $\Pi\in\Mes_+(\Ob\times\Ob)$ satisfying condition $(ii)$ in \eqref{defAtwo}.
 Then we have
\begin{equation}\label{magicineq}  \frac1{2} \int_{\Ob\times\Ob} \frac {\ell_{\mbf{v}}^2}{|x-y|}\, \Pi(dxdy) \le   \int_{\Ob\times\Ob} |x-y|\, \Pi(dxdy) ,
\end{equation}
 with an equality if only if $\ell_{\mbf{v}}= \ell_{\mbf{v}}^+$ \ $\Pi$-a.e. 
\end{lemma}
\begin{proof} \ If $\mbf{v}$ is smooth and single valued, then there exists  $w\in C_0(\Ob;\R^d)$  such that $\mbf{v}(x) =\{x-w(x)\}$
 and the assumption on $\Pi$ implies that 
\begin{equation}\label{magic=}
\int_{\Ob\times\Ob} 
\Big \langle v(y)-v(x), \frac{y-x}{|y-x|} \Big \rangle \, \Pi(dxdy)   =  \int_{\Ob\times\Ob} |x-y|\, \Pi(dxdy),
\end{equation}
 which, in view of \eqref{lv}, means that \eqref{magicineq} holds with an equality.
 Consider now a general $\mbf{v}\in \mbf{M}_\O$.
We observe that, if we let $\eta$ be the sum of the marginals of $\Pi$, the equality \eqref{magic=} can  be extended by density 
to all element $v$ in the closure of $C_0(\O;\R^d)$ in $L^1_\eta(\Ob;\R^d)$, in particular to  any bounded Borel function $v:\Ob\to \R^d$ such that $v=\ident$ \ $\eta$-a.e. in $\bO$. The idea consists now in choosing for $v$ a suitable $\eta$-measurable selection  of the  multifunction
$x\mapsto \mbf{v}(x)$. More precisely, let  $v_0: \Ob \to \R^d$ be  defined by:
$$ v_0(x) = s_d\big(\mbf{v}(x)\big) \quad \text{if $x\in \O$}, \qquad  v_0(x) =x   \quad \text{if $x\in \bO$}\,$$
where, for every convex compact subset $A\subset \R^d$, $s_d(A)$ stands for the Steiner center given by  
$$ s_d(A) :=\ \frac1{\omega_d} \int_{S^{d-1}} z \ \ind_A^*(z) \, \Ha^{d-1} (dz). $$ 
where $\omega_d$ denotes the volume of the unit ball $B^d$.
 It is well established that the map $A\mapsto s_d(A)$ is Lipschitz with respect to the Hausdorff distance. In addition, (see \cite[page 50]{Schneider} and \cite{Saint-Pierre})   $s_d(A)$ belongs to the relative interior of $A$ (note that  this includes the case of a singleton $A=\{a\}$ where  $s_d(A)=a$). It follows that the function $v_0$ defined above is Borel and 
satisfies $v_0(x)=x$ \ $\eta$-a.e. in $\bO$ and  $v_0(x)\in \mathrm{ri}(\mbf{v})(x)$ \ $\eta$-a.e. in $\O$.
  Therefore   $v_0$ satisfies the equality \eqref{magic=}  while $\Pi$-a.e.
 $  \frac1{2} \ell_{\mbf{v}}^2(x,y) \le \pairing{v_0(y)-v_0(x), x-y}\le  \frac1{2} (\ell_{\mbf{v}}^+)^2 (x,y).$
 By integrating with respect to $\Pi$ we are led to
 $$\frac1{2} \int_{\Ob\times\Ob} \frac {\ell_{\mbf{v}}^2}{|x-y|}\, \Pi(dxdy) \le   \int_{\Ob\times\Ob} |x-y|\, \Pi(dxdy) \le \frac1{2} \int_{\Ob\times\Ob} \frac {(\ell^+_{\mbf{v}})^2}{|x-y|}\, \Pi(dxdy).$$
The inequality \eqref{magicineq} follows and it becomes an equality provided that $\ell_{\mbf{v}}^+=\ell_{\mbf{v}}$ holds $\Pi$-a.e. 
Conversely, let us assume that in \eqref{magicineq} we have equality. Then, since the equality \eqref{magic=} holds for $v_0$, for $\Pi$-a.e. $(x,y)$ we obtain:
\begin{equation}\label{v0opti}
\pairing{v_0(x)-v_0(y), x-y} = \frac1{2} \,\ell_{\mbf{v}}^2(x,y) =  \min \Big\{ \pairing{y'-x', y-x}\, :\, x'\in \mbf{v}(x), \  y'\in \mbf{v}(y)  \Big \}.
\end{equation}
This means that the minimum of the linear form $(x',y') \mapsto \pairing{y'-x', y-x}$ on the convex compact subset $\mbf{v}(x)\times \mbf{v}(y)$
is reached at $\big(v_0(x),v_0(y)\big)$. In particular, for $\Pi$-a.e. $(x,y)\in \O^2$, we will have that $\big(v_0(x),v_0(y)\big) \in \mathrm{ri}(\mbf{v})(x)\times \mathrm{ri}(\mbf{v})(y)$
and this is not possible unless the linear form remains constant on $\mbf{v}(x)\times \mbf{v}(y)$. We thereby deduce that $\ell_{\mbf{v}}^+=\ell_{\mbf{v}}$ is satisfied  $\Pi$-a.e on $\O\times\O$. 
Let us now consider a pair $(x,y) \in \bO\times\O$ that satisfies \eqref{v0opti}.
By prescribing  $x'=x$ in the infimum defining $\ell_{\mbf{v}}$ we get  
$$ \frac1{2}\, \ell_{\mbf{v}}^2(x,y)=\pairing{x-v_0(y), x-y}  \ge  \min \Big\{ \pairing{ y' - x, y-x} \,:\,  y'\in \mbf{v}(y) \Big\}  \geq \frac1{2}\, \ell_{\mbf{v}}^2(x,y)
,$$
which proves that each inequality is an equality and, as a result, the affine function $y'\in \mbf{v}(y) \mapsto \pairing{ y' - x, y-x}$ is minimal at $v_0(y)\in \mathrm{ri}(\mbf{v})(y)$ thus it must be constant. It follows that  
$$\frac1{2} (\ell_{\mbf{v}})^2(x,y) = \max \Big\{ \pairing{ y' - x, y-x} \, : \,  y'\in \mbf{v}(y)   \Big\}= \frac1{2} (\ell^+_{\mbf{v}})^2(x,y), $$
where in the second equality we used \eqref{bO-ell+}.
If $(x,y) \in \bO\times\bO$, \eqref{bO-ell+} provides the equality $\frac1{2} (\ell^+_{\mbf{v}})^2(x,y) = |x-y|^2 $ while, owing to the first equality in \eqref{v0opti} and the fact that $v_0(x) = x$ on $\bO$,
we have $\frac1{2}(\ell_{\mbf{v}})^2(x,y)= |x-y|^2$.  

Eventually we are left to prove  \eqref{bO-ell+} while showing the inequality $\ge$ is the trivial part.  We recall that $\mbf{v}(x)=\{x\}$ if $x\notin \Ob$
while $\mbf{v}(x)\supset\{x\}$ if $x \in \bO$  (see (i) in Lemma \ref{mono1}). 
Let $(a,b)\in\bO\times\Ob$. By convexity of $\O$ we have $a_n = a + \frac1{n} (a-b)\notin \O$, thus $\mbf{v}(a_n)\supset\{a_n\}$  for every $n\in\N^*$. By the monotonicity property we deduce that, for every $a' \in \mbf{v}(a)$,  
$ 0\le n \pairing{a_n-a', a_n-a} = \frac1{n} |a-b|^2 + \pairing{ a-a', a-b}$
hence $\pairing{ a-a', b-a}\le 0$   by sending $n\to\infty$. 
Accordingly we get $ \pairing{ b'-a', b-a}\le \pairing{ b'-a, b-a}$ for every $a' \in \mbf{v}(a)$ and $b'\in \mbf{v}(b)$.  If in addition  $b\in\bO$,  we may chose $b'=b$ on the right hand side of the latter inequality. This proves the desired upper bound for
$\ell^+_{\mbf{v}}(a,b)$ hence  \eqref{bO-ell+}.

 \end{proof}
 
 \begin{theorem}\label{extendedOptiTwo}
	Let $(\pi,\Pi)$ be an element of $\Mes(\Ob\times\Ob;\R\times \R_+)$ given in the form $(\pi,\Pi)= (\alpha\Pi,\Pi)$ with $ \alpha \in L^1_\Pi$ and let $(u,\mbf{v})\in C_{\Sigma_0}(\Ob)\times \mbf{M}_\O$. Then the pairs $(\pi,\Pi)$ and $(u,\mbf{v})$ are optimal for, respectively, $(\mathscr{P})$ and \eqref{revised-dual} if and only the following conditions are satisfied:
	\begin{align}\label{OCtwo}
	\begin{cases}
	(i)&  (\alpha\Pi,\Pi) \in \Atwo ,\\
	(ii) &  u(y) - u(x) \leq \ell_{\mbf{v}}(x,y) \qquad  \forall (x,y) \in \Ob \times \Ob    ,\\
	(iii)& |u(y) - u(x)| = \ell_{\mbf{v}}(x,y) \qquad  \text{for } \Pi \text{-a.e. } (x,y) , \\
	(iv) & \alpha(x,y) =\frac{u(y)-u(x)}{|y-x|} \qquad \text{for } \Pi \text{-a.e. } (x,y) ,\\
	(v)& \ell_{\mbf{v}}(x,y)= \ell_{\mbf{v}}^+(x,y)\qquad \text{for } \Pi \text{-a.e. } (x,y).		\end{cases}
	\end{align}
	\end{theorem}
	 \begin{proof} Conditions $(i)$ and $(ii)$ are equivalent to the admissibility of $(\pi,\Pi)$ and $(u,\mbf{v})$ for  $(\mathscr{P})$ and \eqref{revised-dual} respectively. Therefore they are assumed to hold true along the proof. By using (ii) and Minkowski inequality, we derive the following  inequalities
 \begin{equation}\label{minko}
\a(x,y) \big( u(y)-u(x) \big)\, \le \, \a(x,y)\, \ell_{\mbf{v}}(x,y) \, \le\, \frac1{2}\, |x-y|\, \a^2(x,y)  + \frac1{2}\, \frac {\ell_{\mbf{v}}^2(x,y)}{|x-y|},
\end{equation}
which clearly  become equalities  if and if the equalities of conditions (iii) and (iv) are satisfied.
By integrating with respect to $\Pi$ and thanks to \eqref{magicineq} we infer that
 \begin{align*}\pairing{f,u}&
 = \int_{\Ob\times\Ob} \a(x,y) \big( u(y)-u(x)\big)\, \Pi(dxdy)\\
 &\le   \frac1{2} \int_{\Ob\times\Ob} |x-y|\, \a^2(x,y) \, \Pi(dxdy) \ + \  
 \frac1{2}  \int_{\Ob\times\Ob}  \frac {\ell_{\mbf{v}}^2(x,y)}{|x-y|}\, \Pi(dxdy)\\
 &\le   \int_{\Ob\times\Ob} |x-y| \left(1+ \frac{\a^2}{2}\right)\, \Pi(dxdy) = \Jtwo(\pi,\Pi).
\end{align*}
It follows that  the extremality condition $\pairing{f,u} = \Jtwo(\pi,\Pi)$ holds true if and only the inequalities above are equalities. 
In view of \eqref{minko} and of Lemma \ref{masterineq} this is equivalent to the triple of conditions $(iii),(iv),(v)$.
We conclude the proof by recalling that, in virtue of Theorem \ref{Z=I}, equality $\pairing{f,u} = \Jtwo(\pi,\Pi)$  characterizes  optimal 
 pairs $(\pi,\Pi)$ and $(u,\mbf{v})$ among the admissible ones.
 \end{proof}
 
 \begin{remark}\label{equirepartition}  \ From conditions $(iii),(iv),(v)$ and Lemma \ref{masterineq}, we get the equality
  \begin{equation}\label{equitwo}
\int  |x-y| \, d\Pi =  \int  |x-y| \, \frac{\alpha^2}{2} \, d\Pi  \quad \left(= \frac{Z_0}{2}\right)
\end{equation}
which  is in fact  an alternative form of the energy equi-repartition principle \eqref{equipart}. Note that \eqref{equitwo} can be recovered  
from the minimality at $t=1$ of $t\mapsto \Jtwo(\alpha \Pi, t \Pi)$ (where $(\alpha \Pi, t \Pi)\in \Atwo$ for all $t>0$).	
	\end{remark}  	
\begin{remark} \ Recall that condition $(ii)$  can be written equivalently with  $\ell_{\mbf{v}}(x,y)$ replaced by $c_{\mbf{v}}(x,y)$ while
condition  $(iv)$ implies that $\ell_{\mbf{v}}(x,y)=c_{\mbf{v}}(x,y)\ \ \Pi \text{-a.e.}$
On the other hand the regularity condition $(v)$ cannot be dropped.	Indeed, let us consider the following one dimensional example where $\O=(-1,1)$, $\Sigma_0=\{-1\}$ and $f=\sqrt{2} \, \delta_0$.
Then, with $ \Pi= \delta_{(-1,0)}+ \delta_{(0,1)}$ and $ \pi=\sqrt{2}\, \delta_{(-1,0)}$, we obtain an admissible pair $(\pi,\Pi)\in\Atwo$.
Next we define a maximal monotone map $\mbf{v}$ by setting 
$$  \mbf{v}(x)= \{x\} \quad \text{ if $-1\leq x < 0$}, \qquad \mbf{v}(0)= [0,1], \qquad \mbf{v}(x)= \{1\} \quad \text{ if $0<x\le 1$}.$$
and an element of $u\in C_{\Sigma_0}(\Ob)$ by setting $u(x) = \sqrt{2}\, (1+ x\wedge 0)$.
It can be easily checked that all the conditions required in Theorem \ref{extendedOptiTwo} are satisfied except  $(v)$ since we have
$\ell_{\mbf{v}}(-1,0)=\a(-1,0)= \sqrt{2}$ and  $\ell_{\mbf{v}}(0,1)=\a(0,1) =0$ while $\ell_{\mbf{v}}^+(-1,0) = 2$ and $\ell_{\mbf{v}}^+(0,1) = \sqrt{2}$. Then\
$   \int_{\Ob\times\Ob} |x-y|\, d\Pi  =
 \int_{\Ob\times\Ob} |x-y|\, \a^2\, d\Pi=2 .$
Accordingly we get a
 a duality gap since $ \pairing{f,u} = 2  <  \Jtwo(\pi,\Pi)=3$.
 The interested reader can check easily that $I_0(f,\Sigma_0) = 2 \sqrt{2}$ and optimality is reached by taking instead\
 $  \Pi=\frac{\sqrt{2}}{2} \left(\delta_{(-1,0)}+ \delta_{(0,1)}\right)$ , $ \pi=\sqrt{2}\, \delta_{(-1,0)}$,
while, for $x\in [-1,1]$,  $ u(x) = 2 + 2 (x\wedge 0)$,\  $\mbf{v}(x)=\{ 1 + 2 (x\wedge 0)\}$.

 \end{remark}
  
 As a consequence of  Theorem \ref{extendedOptiTwo},  from a solution to $(\mathscr{P})$
  we may recover a  saddle point solution as given in Corollary \ref{saddle1}:
\begin{corollary}\label{saddlePi}
	 Let  $(\ov\pi,\ov\Pi)$ and $(\ov u,\ov {\mbf{v}})$ be optimal pairs for, respectively $(\mathscr{P})$ and \eqref{revised-dual}.
	 Then the pair
$\bigl(\ov{\mbf{v}},\ov{\gamma} \bigr)$ with $\ov{\gamma}=\abs{\ov{\pi}} $
solves \eqref{saddle-gamma-v} with $\mu$ and $\nu$  being the traces on $\Ob\setminus\Sigma_0$ of  the left and right marginal of $\abs{\ov{\pi}}$ respectively. 
\end{corollary}

\begin{proof}\ We will utilize the optimality conditions $(iii), (iv), (v)$ given in \eqref{OCtwo} (Theorem \ref{extendedOptiTwo}). In particular, we will use the fact that $c_\mbf{\ov v}(x,y)=\ell_\mbf{\ov v}(x,y)= |\ov \a| \, |x-y|$ for $\Pi$-a.e. $(x,y) \in \Ob\times\Ob$ while 
$\ov{\gamma}=\abs{\ov{\pi}} = |\ov\a| \, \ov\Pi$. Thus, taking \eqref{equitwo} into account, we obtain
\begin{equation}\label{saddle=}  \int_{\Ob\times\Ob}  \, c_\mbf{\ov v}\, d \ov\gamma  = \int_{\Ob\times\Ob}  \, \ell_\mbf{\ov v}\ |\ov \a|\,  d \ov\Pi =
\int_{\Ob\times\Ob}  \, \frac{\ov\a^2}{ |x-y|} \,  d \ov\Pi = Z_0 .\end{equation}
On the other hand,  by exploiting \eqref{magicineq} and Minkowski's inequality, for every $\mbf{v}\in \mbf{M}_\O$ we have:
\begin{equation} \label{saddle<}\int_{\Ob\times\Ob}  \, c_\mbf{ v}\, d \ov\gamma  \le \int_{\Ob\times\Ob}  \, \ell_\mbf{v}\ |\ov \a|\,  d \ov\Pi 
 \le  \int_{\Ob\times\Ob} \frac{\left(\ell_\mbf{v}\right)^2}{2|x-y|} \,d \ov\Pi +\int_{\Ob\times\Ob} 
|x-y| \, \frac{\ov \a^2}{2} \, d \ov\Pi \le\  \Jtwo\big(\ov\a\, \ov\Pi,\ov\Pi \big)\ =\ Z_0.
\end{equation}
Eventually, by employing the principle of maximal Monge distance \eqref{maximalMK}, we infer that for every $\gamma\in  \Gamma(\mu,\nu;\Sigma_0)$:
\begin{equation}\label{saddle>}  \int_{\Ob\times\Ob}  \, c_\mbf{\ov v}\, d \gamma\ \ge\   W^{\Sigma_0}_{c_{\mbf{v}}}(\mu,\nu)
\ =\ W^{\Sigma_0}_{c_{\mbf{v}}}(f_+,f_-)\ =\ Z_0\ .
\end{equation}
In view of \eqref{saddle=}, \eqref{saddle<} and \eqref{saddle>}, we deduce that $(\ov{\mbf{v}},\ov{\gamma})$ is a saddle point. 
\end{proof}

 \subsection{Finitely supported loads: a route to existence of truss solutions}
 
 Despite the lack of solutions to problem $(\mathscr{P})$ in general (see Remark \ref{notruss}), we strongly believe 
 that such solutions do exist  in the case of a finitely supported load. Our argument rests upon an extension property for
 $c_{\mbf{v}}$ pseudo-metrics that we shall propose here as a conjecture. 

\subsubsection*{Extension property}  Let us be given a monotone multi-function $\mbf{v}_0$ whose domain $D_0:=\mathrm{dom} (v_0)$ contains $\R^d\setminus\Ob$.
Then  $\mbf{v}_0$ admits at least one maximal monotone extension and any such  extension is defined over whole $\R^d$, see \cite{alberti1999}. 
With $\mbf{v}_0$ we may associate the sub-additive function $c_{\mbf{v}_0}: D_0\times D_0 \to \R_+$ given by 
$$c_{\mbf{v}_0}(a,b) :=  \inf\left\{ \sum_{i=1}^{N-1}  \ell_{\mbf{v}_0}(x_i,x_{i+1})\ :\ x_1=a, \ x_N=b, \ \ N\ge 2, \ \ \{x_i\} \subset D_0 \right\}$$
where $\ell_{\mbf{v}_0}$ is defined as in \eqref{lv}. 
We will say that the monotone map $\mbf{v}_0$ of domain $D_0$  has the extension property if there exits a maximal monotone map $\mbf{v}$ of domain $\R^d$ such that $  \mbf{v}\supset \mbf{v}_0$ and  $c_{\mbf{v}}= c_{\mbf{v}_0}$ in  $D_0\times D_0$. 

 \begin{conjecture}  \label{finiteconj} Let $S$ be a finite subset of $\O$ and $D_0=S \cup (\R^d\setminus \Ob)$. Then any monotone map 
 $\mbf{v}_0:D_0\to\R^d$  such that $\mbf{v}_0=\ident$ in $\R^d\setminus\Ob$ 
 has the extension property.
 
 \end{conjecture}
 
 \begin{proposition}\label{optitruss} Assume that  $\Sigma_0=\bO$ and that $S:=\spt(f)$ is a finite subset of $\O$. Then, if Conjecture \ref{finiteconj}
 holds  true, there exists a   solution  $(\pi,\Pi)$  to $(\mathscr{P})$ that satisfies 
 $$ \spt(\Pi) \subset (S\times \bO) \cup  (S\times S \setminus\Delta).$$
 
 \end{proposition} 
 
 \begin{proof} Since $S$ is finite, the set $K:= \big(S\times \bO \big) \cup  \big((S\times S) \setminus\Delta\big)$ is closed and at a positive distance from the diagonal $\Delta$. We may then apply the assertion (ii) of Proposition \ref{key} obtaining the existence of an optimal solution
 for the problem  \eqref{PK}.
 Clearly we have $\min (\mathscr{P}_K) \ge \inf (\mathscr{P})$. To show the converse inequality, we use the duality relation 
 \begin{equation}\label{nogapK} \min (\mathscr{P}_K)=- h_K^{**}(0,0) =- h_K(0,0) = \sup \Big\{  \pairing{f,u} \ :\ (u,v) \in \mathscr{B}_K \Big\},
 \end{equation}
 where $h_K$ is given by \eqref{def:h} while, after the change of variable $v= \ident-w$, the admissible set associated with $h_K(0,0)$
 becomes
 \begin{equation*}
 \mathscr{B}_K :=\Big\{ 
 (u,v)\in \big({\rm Lip}(\Ob)\big)^{d+1}\ : \  u=0 \ \text{on $\bO$},\ v=\ident \ \text{on $\bO$}, \ \ 
		 u(x) -u(y) \le  \ell_v(x,y)\quad 
\forall (x,y)\in K 
		 \Big\}.
		\end{equation*}
Let us set $D_0:= S \cup \R^d\setminus \Ob$ and let $(u,v)$ be an element of $\mathscr{B}_K$ that we implicitly extend by setting $u(x)=0,\ v(x)=x$
for $x\in \R^d\setminus\Ob$. 
 Then, denoting by $\mbf{v}_0: D_0 \to \Rd$ the monotone map defined by ${\mbf{v}_0} =\{v(x)\}$, we obviously obtain
$ u(x)-u(y) \le \ell_{{\mbf{v}_0}}(x,y) $ for all $(x,y)\in D_0\times D_0$. Indeed the condition $ u(x) -u(y) \le  \ell_v(x,y)$ holding for $(x,y)\in K$
extends to $D_0\times D_0$ by the symmetry of  $\ell_v$ and the fact that $u=0$ on $\R^d\setminus\O$.  Owing to the definition of $c_{{\mbf{v}_0}}$ given above (and the fact that $\ell_{{\mbf{v}_0}}= \ell_{{v}}$ on $D_0\times D_0$),  it is then straightforward that $ u(x)-u(y) \le c_{{\mbf{v}_0}}(x,y) $ for all $(x,y)\in D_0\times D_0$.
Following the conjecture such a ${\mbf{v}_0}$ enjoys the extension property, thus there exist  a maximal monotone extension $\tilde{\mbf{v}}\in \mbf{M}_\O$
such that $c_{\tilde{\mbf{v}}}= c_{{\mbf{v}_0}}$ in  $D_0\times D_0$. Then we define $\tilde{u}: \O \to \R$ by
$$  \tilde{u} (x) := \inf \Big\{  u(y) + c_{\tilde{\mbf{v}}}(x,y)\ :\ y\in D_0 \Big\}.$$
By construction  $\tilde{u}$  satisfies $\tilde{u}(x) - \tilde{u}(y) \le c_{\tilde{\mbf{v}}}(x,y)$ for all $(x,y)\in \Ob\times\Ob$
while  $\tilde{u}=u$ on $D_0$. It follows that  $\tilde{u}\in C_0(\O)$ and that $\pairing{f,\tilde{u}} =\pairing{f,u}$.
The new pair $(\tilde{u}, \tilde{\mbf{v}})$ is therefore
admissible for the dual problem \eqref{revised-dual}. Accordingly, thanks to Theorem \ref{Z=I}, we infer that 
$$  \inf (\mathscr{P}) =I_0(f,\bO) \ge   \pairing{f,u}.$$
The latter inequality being true for all $(u, v)\in \mathscr{B}_K$, we may pass to the supremum with the help of \eqref{nogapK} thus concluding
 with the desired
inequality $ \inf (\mathscr{P}) \ge  \sup \big\{  \pairing{f,u}: (u,v) \in \mathscr{B}_K \big\} =\min (\mathscr{P}_K)$.

 \end{proof}


\section{Approximation by finite truss structures and numerical simulations}
\label{sec:numerics}

In virtue of Theorem \ref{Z=I}, we know that solutions $(\lambda,\sigma)$ to $(\mathcal{P})$ can be reached as weak* limits of sequences of the kind $(\lambda_{\pi_n},\sigma_{\Pi_n})$ for $(\pi_n,\Pi_n)$ being  a minimizing sequence for the problem $(\mathscr{P})$ 
(recall that such measures $(\lambda_{\pi_n},\sigma_{\Pi_n})$ represent truss-like membrane structures which are composed of families of straight strings). This legitimates employing a two-point numerical scheme treatment of  problem $(\mathscr{P})$ where  we narrow the search down to finite trusses spanned by a fixed finite grid populating $\Ob$. Accordingly, we will be using the discrete variant of the dual problem $(\mathcal{P}^*)$ where the two-point condition \eqref{eq:two_point_condition} is set only for pairs of points in the grid.

For simplicity  we will assume from now on that $\O$ is  a  square domain in $\R^2$, namely we take  $\O$ to be the unit square $Q=(-1/2, 1/2)^2$.
 A priori the load will be any signed measure $f \in \Mes(\Ob;\R)$. For any mesh parameter $h>0$, we consider a  finite  grid of regularly spaced nodes:
\begin{equation*}
\mathrm{X}_h = \Ob \cap \left\{  (k_1 h,k_2 h) : (k_1,k_2) \in \mathbbm{Z}^2  \right\}.
\end{equation*}
By a suitable choice of $h$, we may assume that every vertex of $\Omega$ is an element of $\mathrm{X}_h$. 
In the sequel  $f_h \in \Mes(\mathrm{X}_h;\R)$ denotes a discrete measure that approximates $f$
in the sense of  tight convergence as $h \to 0$. A natural choice  is
$f_h = \sum_{x\in \mathrm{X}_h}  f\big(Q_h(x)\big)\, \delta_x$ where $Q_h(x)=x+ hQ$.
Then we set the discretized problem $({\mathscr{P}}_h)$ to match with \eqref{PK} in Proposition \ref{key} while taking $K = K_h := \mathrm{X}_h \times \mathrm{X}_h \setminus \Delta$. If we agree that the set $\mathscr{A}_{K_h}$ of admissible
 measures $(\pi,\Pi)$ is defined for the discrete load $f_h$ in the same way as in \eqref{defAtwo}, it is easy to check that $\Atwo_{K_h}$ is non-empty whenever $\Sigma_0 \cap \mathrm{X}_h \neq \varnothing$. The dual discrete problem $(\mathcal{P}^*_h)$ is recast as $-h_{K_h}(0,0)$ according to  \eqref{def:h}. Since $K_h \cap \Delta = \varnothing$, by applying the assertion (ii) of Proposition \ref{key} we obtain the zero-gap equality $\inf ({\mathscr{P}}_h) = \sup (\mathcal{P}^*_h)$ and the existence of a solution for problem $({\mathscr{P}}_h)$. Following \cite{bolbotowski2021}, the problems $({\mathscr{P}}_h)$ and  $(\mathcal{P}^*_h)$ are  handled as a pair of \textit{conic quadratic programs} \cite{andersen2003} that we implement in $\text{MATLAB}^{\tiny{\textregistered}}$ with the use of the $\text{MOSEK}^{\tiny{\textregistered}}$  toolbox. 

Along the forthcoming examples, we will display the numerical solution $(\pi,\Pi)$ to $(\mathscr{P})$ for  which 
$\pi\le 0$ (see Remark \ref{pi>0}). With regard to the solution $(u,w)$ to $(\mathcal{P}^*)$, we will often limit ourselves to display  the scalar function $u$
(interpolated from its values on $\mathrm{X}_h$).

\begin{example}[\textbf{Four point forces}]
\label{ex:four_forces}

For the Dirichlet zone  $\Sigma_0=\bO$ the load $f = \sum_{i=1}^4 \delta_{x_i}$ consists of four symmetrically spaced point forces: $x_i = (\pm \alpha  ,\pm \alpha )$ with $\alpha = 0.2$ (symbolized by $\odot$). In Fig. \ref{fig:4_point_loads}(a,b) we present measures  $\sigma_\Pi, \lambda_\pi$ arising from a solution $(\pi,\Pi)$ of $({\mathscr{P}}_h)$, whereas Fig. \ref{fig:4_point_loads}(c) shows an optimal $u$.
 Direction of $\lambda_\pi$, being a vector valued measure, is marked by the use of arrows.

\begin{figure}[h]
		\centering
		\subfloat[]{\includegraphics*[trim={0cm 0cm -0cm -0cm},clip,height=0.28\textwidth]{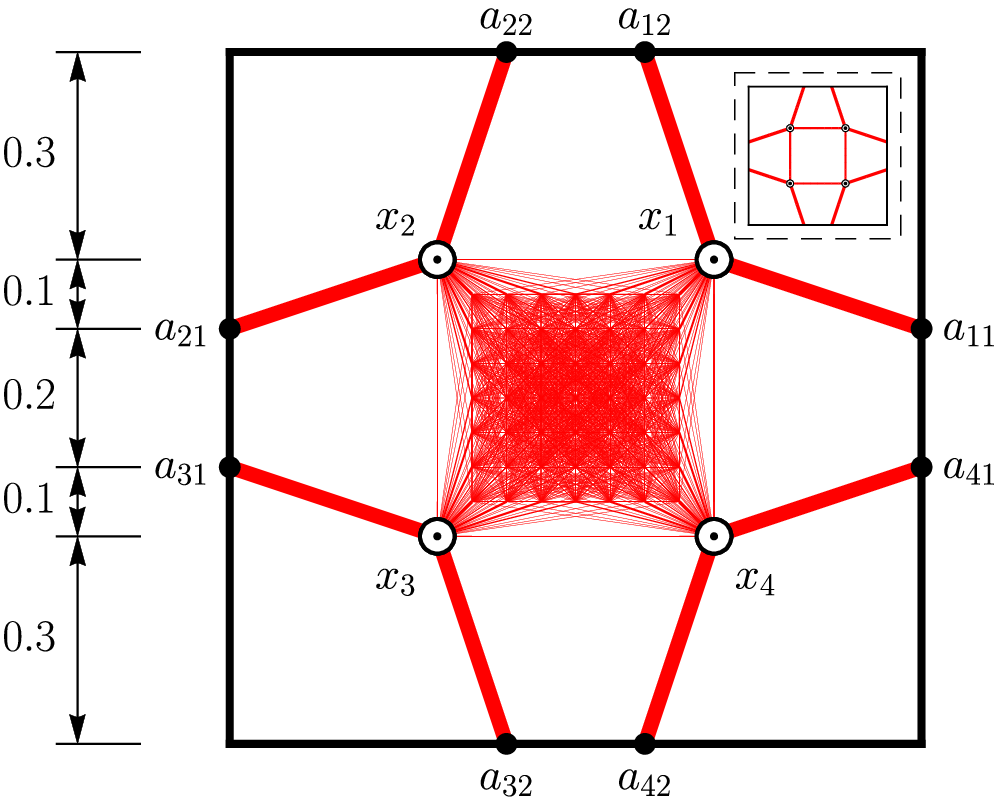}}\hspace{0.4cm}
		\subfloat[]{\includegraphics*[trim={0cm 0cm -0cm -0cm},clip,height=0.28\textwidth]{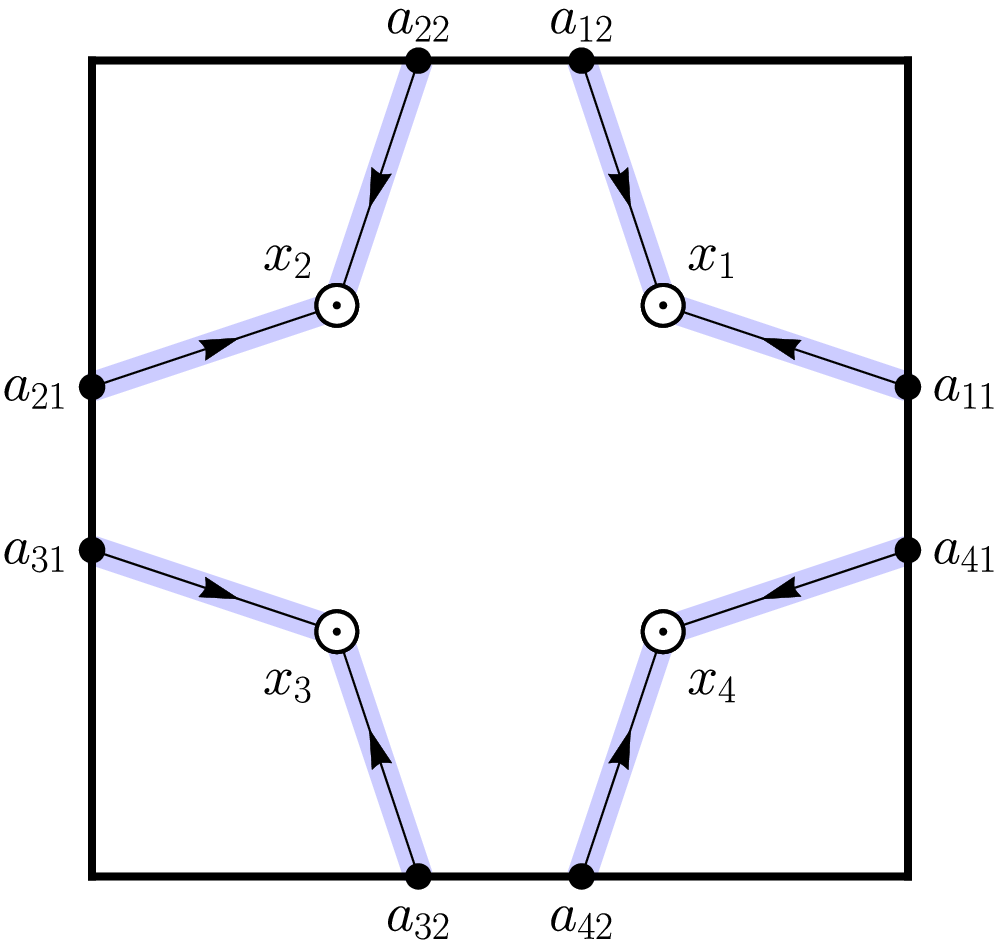}}\hspace{0.4cm}
		\subfloat[]{\includegraphics*[trim={1.2cm -1cm 1.cm -0cm},clip,width=0.28\textwidth]{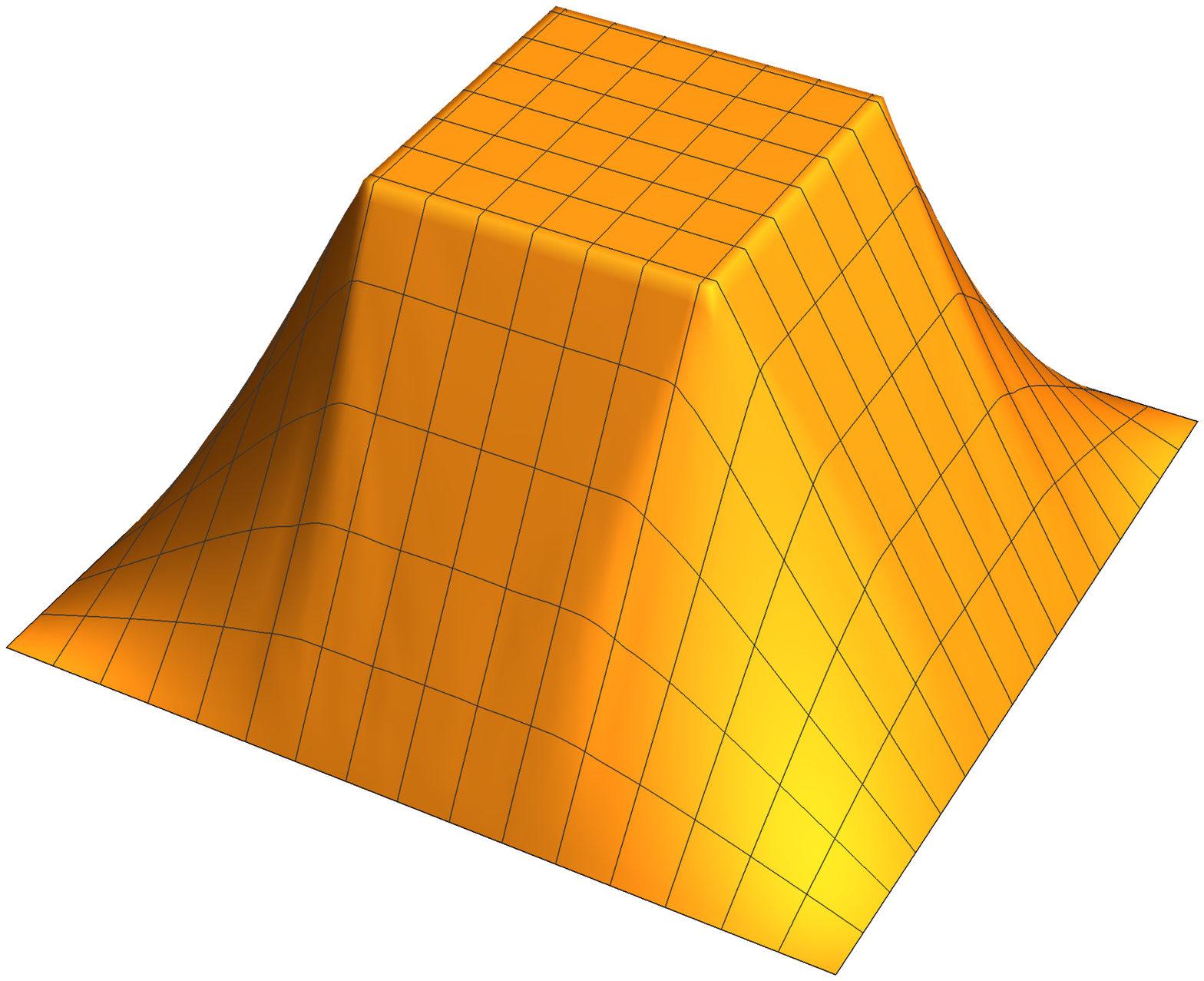}}
		\caption{Numerical prediction of solution of the four force problem: (a) optimal $\sigma_\Pi$ (and an alternative solution in the top right corner); (b) optimal $\lambda_\pi$; (c) optimal $u$.}
		\label{fig:4_point_loads}
\end{figure}
Let us now interpret these results. The truss, represented by the prestress $\sigma_{\Pi}$, consists of eight strings connecting points of load's application to the boundary and of a large family of thin strings contained in the central square $D=\mathrm{co}(\{x_i\})$; the non-zero transverse force $\lambda_\pi$ is present only on the eight strings. The function $u$ admits a plateau in $D$ where  $u \sim 0.529 $, whilst the optimal $w$  numerically matches with the identity map in this square and appears to be continuous on whole $\Ob$. For each of the eight strings we have:
\begin{equation} \label{expli12}
	\Pi\bigl(\bigl\{(x_i,a_{ij}) \bigr\} \bigr) = \frac{1}{2}, \qquad 	\pi\bigl(\bigl\{(x_i,a_{ij}) \bigr\} \bigr) = - \frac{u(x_i)}{2 \abs{x_i-a_{ij}}} \,  = - \frac{\sqrt{0.28}}{2 \abs{x_i-a_{ij}}} \qquad \qquad  \forall\,i,  j
\end{equation}
where $\sqrt{0.28}$ ($=\sqrt{1/2- \alpha -\alpha^2/2}$ for $\a=0.2$) is the expected exact plateau value $u$ on $D$.
 By slightly perturbing the algorithm with a penalty on the total number of strings,
we converge to the finite truss solution $(\tilde{\pi}, \tilde{\Pi})$ supported on $K:=\bigcup_{ij} \big\{ \{x_i\}\times \{a_{ij}\} \big\}$  
whose miniature is displayed in the top right corner of Fig. \ref{fig:4_point_loads}(a). Observe that $\spt\, (\tilde{\pi},\tilde{\Pi}) \subset (\bO \cup (\spt\, f))^2$ while $\lambda_{\tilde{\pi}} = \lambda_\pi$ on whole $\Ob$ and $\sigma_{\tilde{\Pi}} = \sigma_{{\Pi}}$ on $\Ob \setminus D$ (note that $D$ is closed).

An interesting issue would be to confirm that this 12-string structure suggested by the numerics (and encoded by \eqref{expli12})
 is indeed an optimal one. 
In view of  applying Theorem \ref{extendedOptiTwo}, we need to determine a suitable pair $(u,\mbf{v})$. 
One checks easily that if $\mbf{v}$ is an element of $\mbf{M}_\O$  vanishing in $D$ and if $u\in C_0(\Ob)$ admits $D$ as maximal plateau, then
$ \ell_{\mbf{v}}(x_i, a_{i,j}) = \min \{ \ell_{\mbf{v}}(x_i, y)\ :\ y\in \bO\}$. On the other hand, we have $u(x)-u(y) \le \ell_{\mbf{v}}(x,y) \ \ \forall (x,y) \in \spt (f)\times \bO,$ with an equality if $(x,y)\in K$. In order to satisfy all optimality conditions \eqref{OCtwo}, it remains  to find   $(u,\mbf{v})\in C_0(\R^d) \times \mbf{M}_\O$ matching  with the values prescribed above on $\bO \cup D$ and satisfying the admissibility condition $u(x)-u(y) \le \ell_{\mbf{v}}(x,y) $ for all 
$(x,y)\in\Ob^2$. Consequently, we are done if we can prove conjecture \eqref{finiteconj} or alternatively if the construction of such extension can be done by hand.
We know  for instance that $u$ (resp. $\mbf{v}$) need to be affine (resp. directionally affine) on the eight straight segments $[x_i,a_{ij}]$ which in turn are geodesics for the $c_\mbf{v}$ distance. 
To conclude this example, let us remark that if $({\pi}, {\Pi})$ given by \eqref{expli12} is optimal and if $(u,\mbf{v})$ can be constructed, then 
$(-\pi, \mbf{v})$ will satisfy the saddle-point relations
\eqref{extremality}  while taking $\mu=f$ and $\nu=0$. The minimum in \eqref{distMKv} is obtained for
$\mu_0=0$ and  $\nu_0=\sum_{i=1}^{4} \sum_{j=1}^{2} \frac{1}{2}\,\delta_{a_{ij}}$.

\end{example}

\begin{example}[\textbf{Five point forces}]
\label{ex:five_forces}
We modify the previous example by adding one point force in the centre of the square, i.e. $f =\delta_{x_0} + \sum_{i=1}^4 \delta_{x_i}$ where $x_0 = (0,0)$. Fig. \ref{fig:5_point_loads} presents numerical solutions $\sigma_{\Pi}$, $\lambda_{\pi}$ and $u$. We can see that the plateau is no longer present in the graph of $u$ and the  12 strings-structure $\sigma_\Pi$ seems stable (with respect to perturbations of the algorithm). In contrast to Example \ref{ex:four_forces} here the points $a_{ij}$ where strings connect to the boundary are slightly shifted towards the square's corners.

\begin{figure}[h]
		\centering
		\subfloat[]{\includegraphics*[trim={0cm 0cm -0cm -0cm},clip,height=0.28\textwidth]{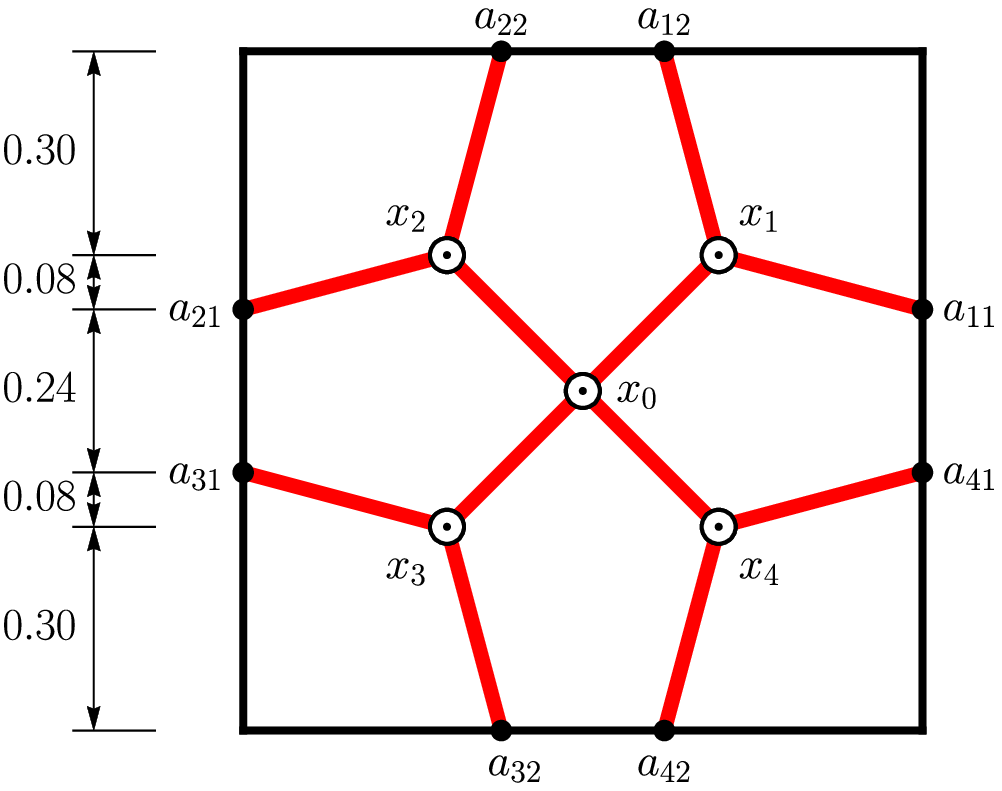}}\hspace{0.2cm}
		\subfloat[]{\includegraphics*[trim={0cm 0cm -0cm -0cm},clip,height=0.28\textwidth]{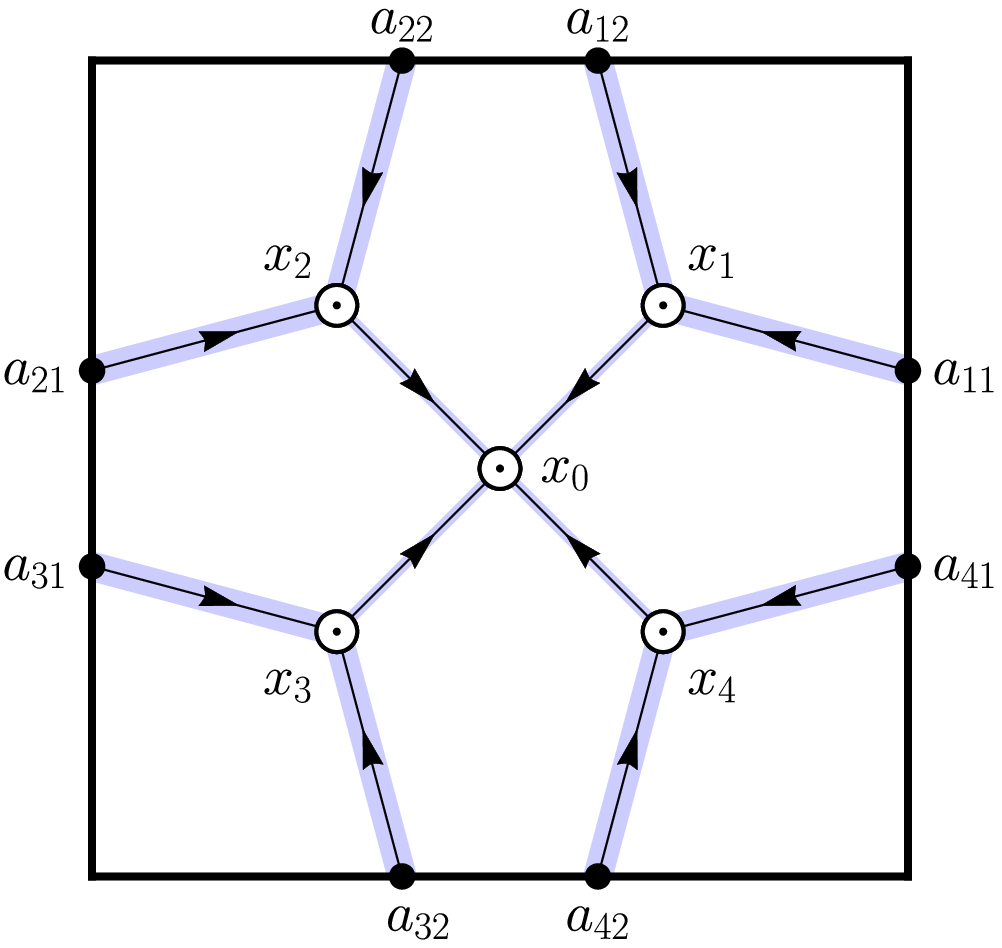}}\hspace{0.4cm}
		\subfloat[]{\includegraphics*[trim={1.2cm -1cm 1.cm -0cm},clip,width=0.28\textwidth]{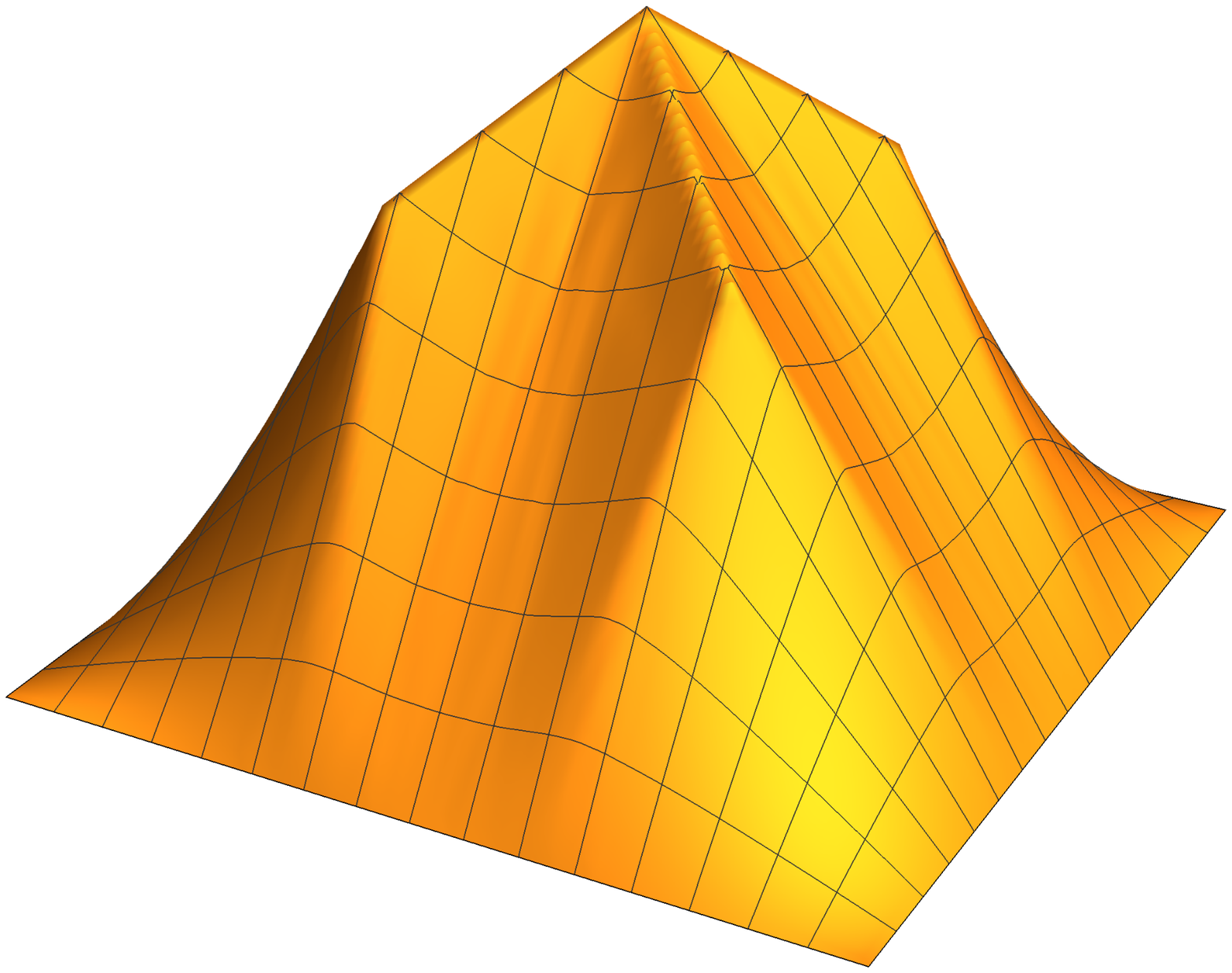}}
		\caption{Numerical prediction of solution of the five force problem: (a) optimal $\sigma_\Pi$; (b) optimal $\lambda_\pi$; (c) optimal $u$.}
		\label{fig:5_point_loads}
\end{figure}

The optimal $\pi$ given by the algorithm (which, up to the change of sign, can be seen also as the optimal transport plan $\gamma$ from the Monge-Kantorovich view-point) is given by
\begin{equation*}
\gamma = - \pi = \sum_{i=1}^4 \frac{1}{4}\, \delta_{(x_0,x_i)} + \sum_{i=1}^{4} \sum_{j=1}^{2} \frac{5}{8} \, \delta_{(x_i,a_{ij})}.
\end{equation*}
It encodes twelve $c_v$-geodesics connecting pairs of points $(x_0,x_i)$ and $(x_i,a_{ij})$ which are straight segments.
It is interesting to notice that the measure $\lambda_\pi$ can be also reconstructed from the alternative transport plan 
\begin{equation*}
	\tilde{\gamma} = \sum_{i=1}^{4} \sum_{j=1}^{2} \frac{1}{8}\, \delta_{(x_0,a_{ij})} + \sum_{i=1}^{4} \sum_{j=1}^{2} \frac{1}{2} \, \delta_{(x_i,a_{ij})}.
\end{equation*}
which involves  $c_v$-geodesics associated to pairs of points $(x_i,a_{ij})$, which  remain straight segments, and pairs $(x_0,a_{ij})$ whose $c_v$-geodesics become polygonal chains $[x_0,x_i] \cup [x_i,a_{ij}]$.
%
%
\end{example}

\begin{example}[\textbf{Pressure load}]
	\label{ex:pressure_load}
	
We consider now a distributed  pressure load, namely. $f = \mathcal{L}^2 \mres \Omega$.   Solutions $\sigma_\Pi$ and $u$ are showed in Fig. \ref{fig:pressure_load}(a) and (b), respectively.
	
	\begin{figure}[h]
		\centering
		\subfloat[]{\includegraphics*[trim={0cm -0.58cm -0cm -0cm},clip,width=0.28\textwidth]{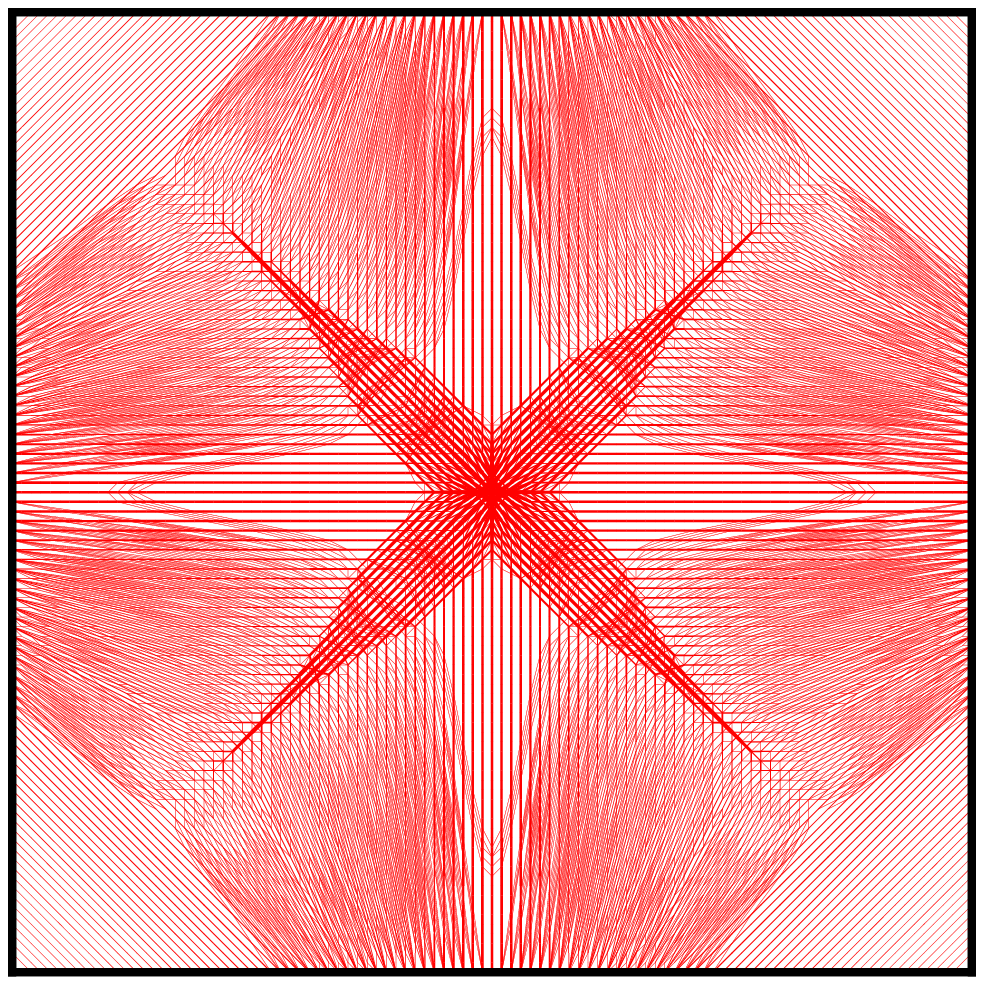}}\hspace{0.4cm}
		\subfloat[]{\includegraphics*[trim={1.2cm -1.5cm 1.cm -0cm},clip,width=0.30\textwidth]{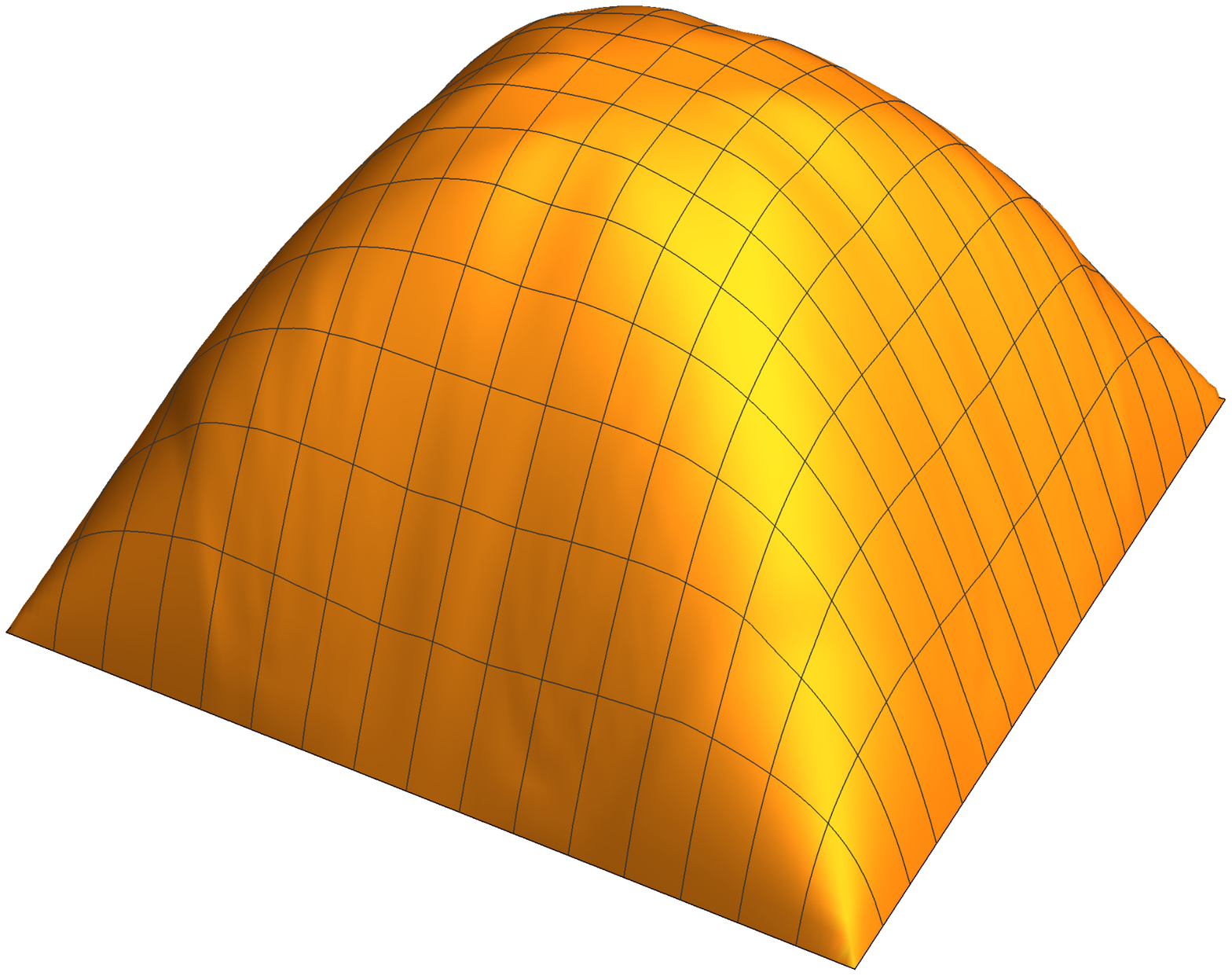}}\hspace{0.4cm}
		\subfloat[]{\includegraphics*[trim={0cm 0cm -0cm -0cm},clip,width=0.315\textwidth]{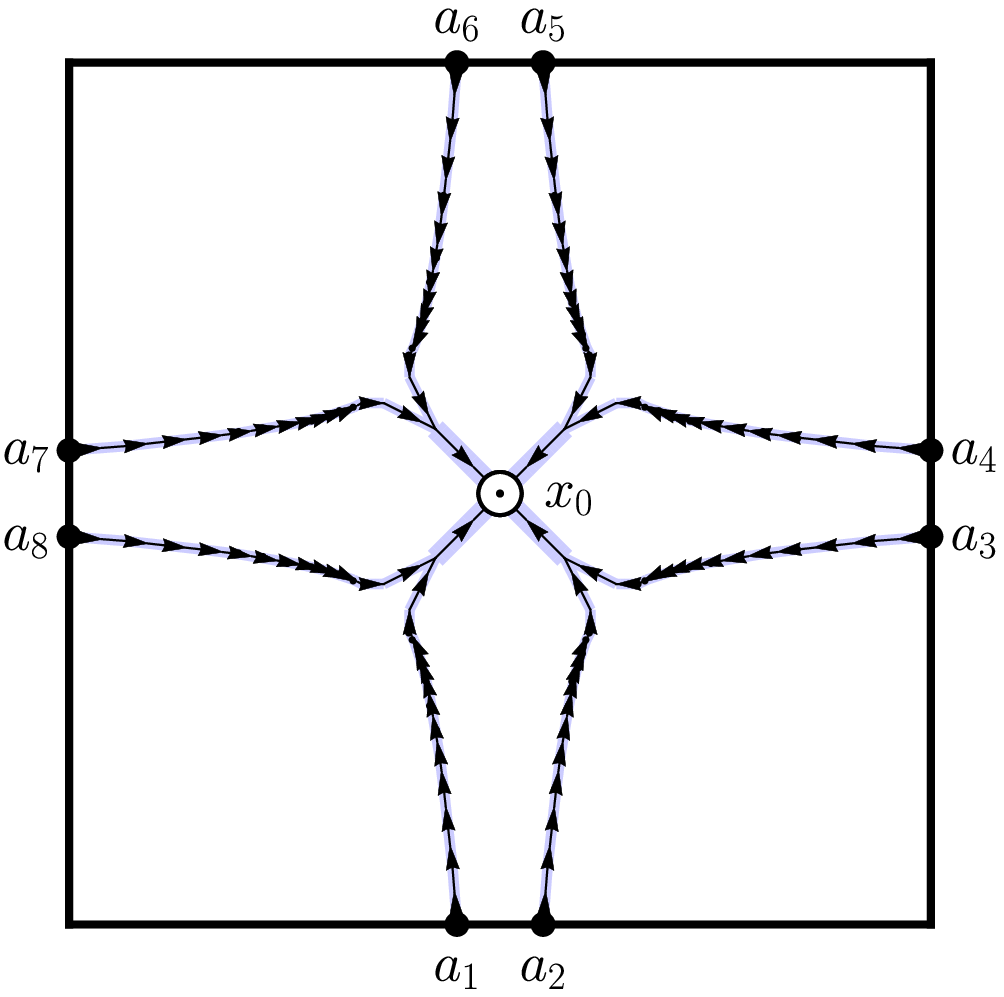}}
		\caption{Numerical prediction of optimal membrane for the uniform pressure load: (a) optimal $\sigma_\Pi$; (b) optimal $u$; (c) eight equivalent $c_{\mbf{v}}$-geodesics from the central point $x_0$ to $\bO$
		(computed for the numerical prediction of optimal $\mbf{v}$).}
		\label{fig:pressure_load}
	\end{figure}
	
The prediction $\sigma_{\Pi}$, that approximates an exact solution $\sigma \in \Mes(\Ob;\Sddp)$, is difficult to analyse. It seems that $\spt(\sig)=\Ob$ while $\sig \ll \mathcal{L}^2$. The exact pre-stress $\sigma$ is also suspected to be rank-one in a large portion of the domain except perhaps in the neighbourhood of diagonals, far from the corners, where $\sigma$ seems to be of full rank and possibly non-unique (the presumed rank-two region stands out visually in Fig. \ref{fig:pressure_load}(a)).
	Instead of presenting the solution $\lambda_{\pi}$, we investigate the form of geodesics with respect to the estimated  $c_{\mbf{v}}$-distance. Fig. \ref{fig:pressure_load}(c) shows a numerical solution of \textit{optimal transshipment problem} (OTP with constraint on the difference of marginals, cf. \cite{villani}) with respect to the $\ell_{\mbf{v}}$-cost. The eight paths numerically obtained thus approximate geodesics connecting pairs of points $(x_0,a_i)$. Based on the simulation it is fair to assume that these geodesics are curved and piecewise smooth. 
\end{example}

\begin{example}[\textbf{Diagonal load}]
\label{ex:diag_load}

While keeping  $\Sigma_0 = \bO$, we continue with a load  concentrated along the square's diagonals, namely  $f = \Ha^1\mres[a_1,a_3] + \Ha^1 \mres[a_2,a_4]$. Numerical computations provide  $\sigma_{\Pi}, \lambda_{\pi}, u$ displayed in Fig. \ref{fig:diagonal_load}
(for the sake of clarity, $\lambda_{\pi}$ is displayed with a lower resolution).

\begin{figure}[h]
		\centering
		\subfloat[]{\includegraphics*[trim={0cm 0cm -0cm -0cm},clip,height=0.315\textwidth]{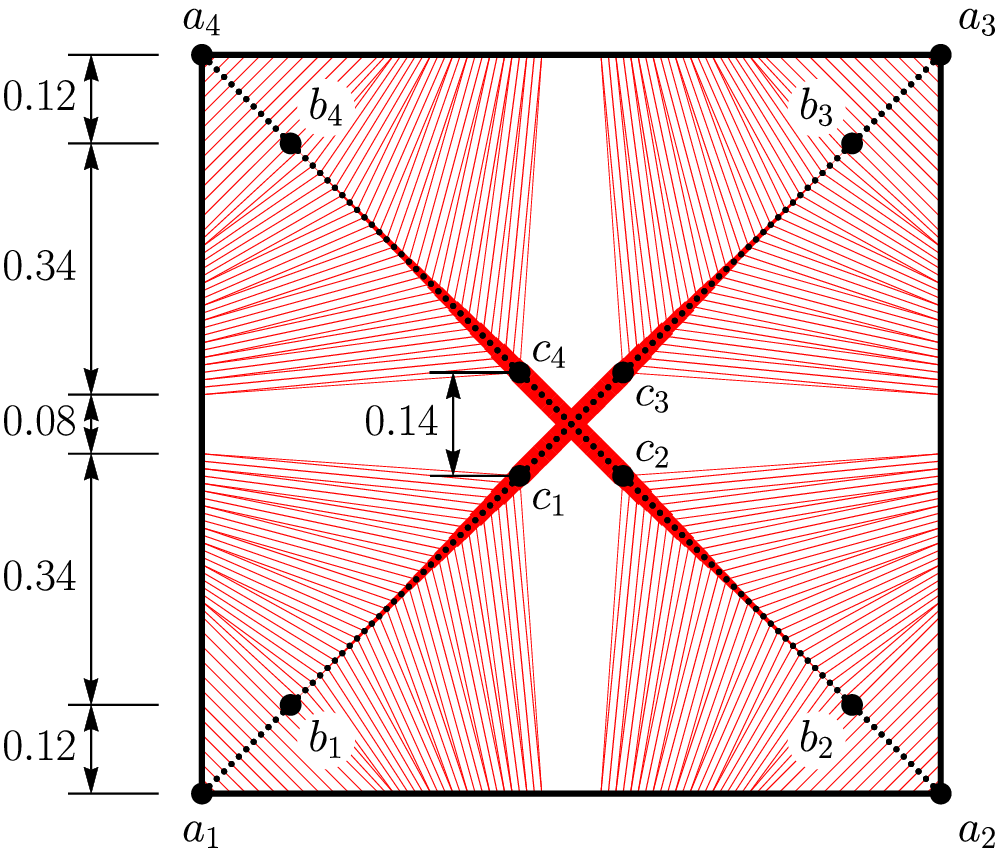}}\hspace{0.4cm}
		\subfloat[]{\includegraphics*[trim={0cm -0.55cm -0cm -0cm},clip,width=0.28\textwidth]{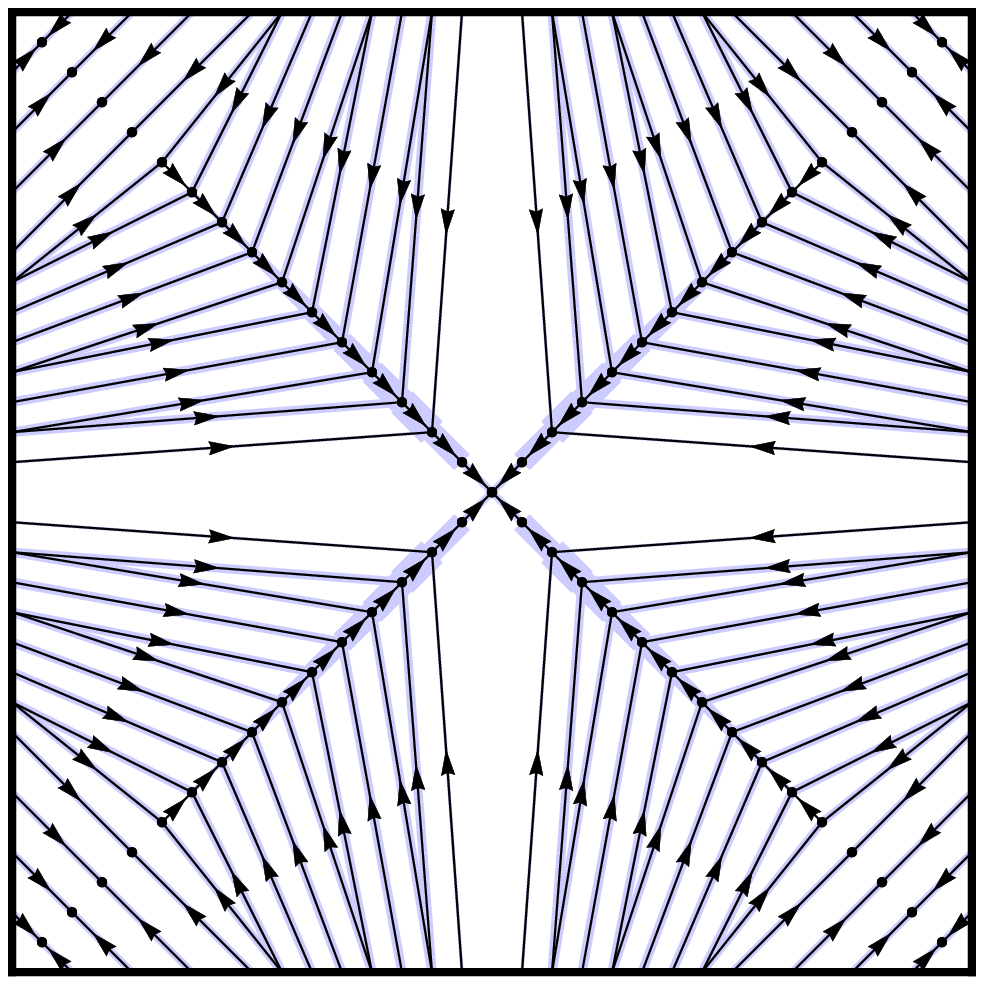}}\hspace{0.3cm}
		\subfloat[]{\includegraphics*[trim={1.2cm -2cm 1.cm -0cm},clip,width=0.28\textwidth]{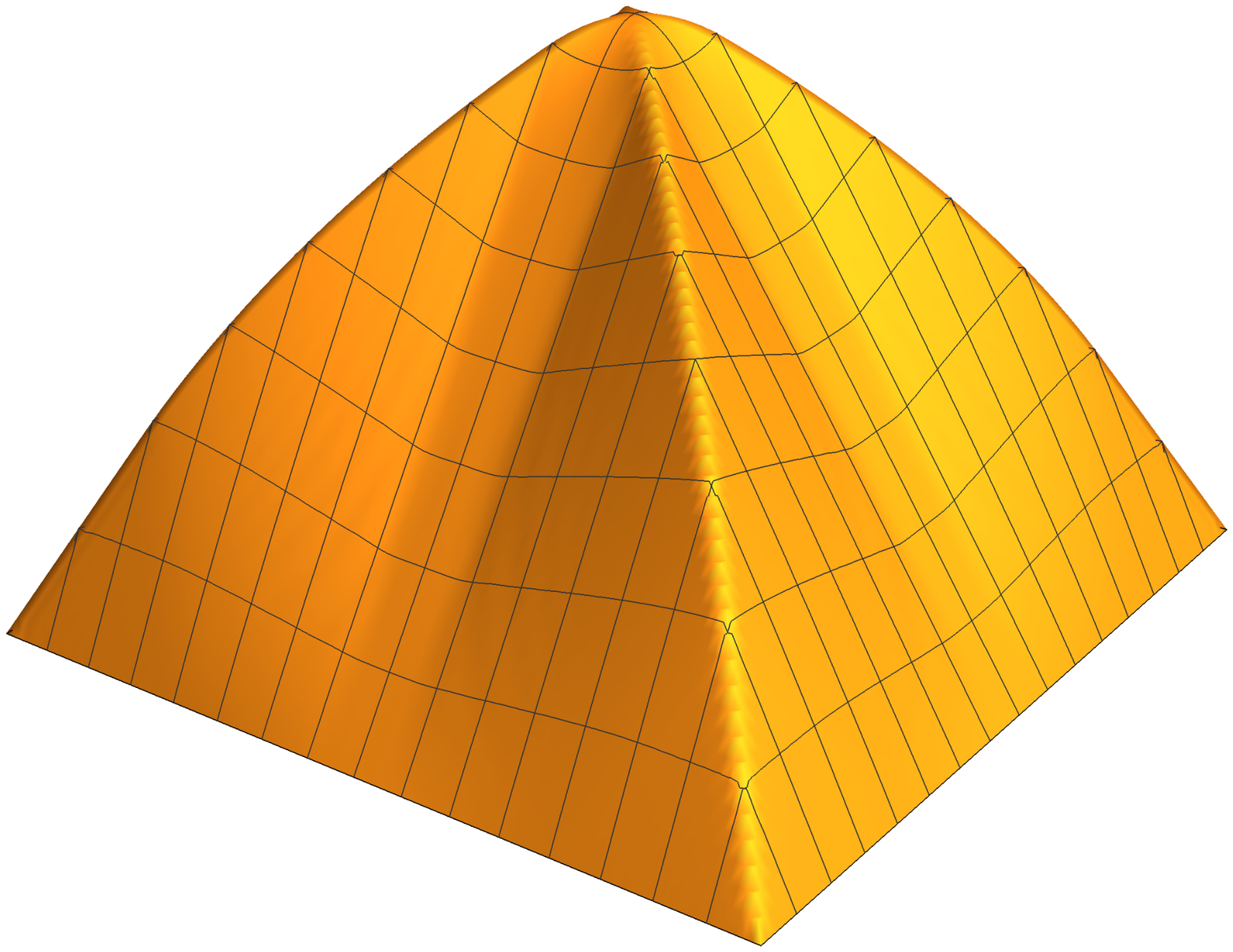}}
		\caption{Numerical prediction of optimal membrane for the diagonal load (the discretized load is denoted by black dots): (a) optimal $\sigma_\Pi$ (higher resolution); (b) optimal $\lambda_\pi$ (lower resolution); (c) optimal $u$.}
		\label{fig:diagonal_load}
\end{figure}

Similarly as in previous examples, the support of $(\pi,\Pi)$ is contained in $(\bO \cup \spt\, f)^2$, i.e. no intermediate points are essential for the optimal transmission of the load. The numerical display strongly indicate that exact optimal $\sigma$ splits  into an absolutely continuous part with respect to $\Ha^1 \mres [b_1,b_3] + \Ha^1 \mres [b_2,b_4]$ and into an absolutely continuous part with respect to Lebesgue measure restricted to four quadrilaterals. Based on Fig. \ref{fig:diagonal_load}(c) we may predict that the restrictions of $u\vert_{[b_1,b_3]}$ and $u\vert_{[b_2,b_4]}$ are strictly concave hence we expect that the problem $(\mathscr{P})$ with the initial load $f$ has no solution (see Remark \ref{notruss}). However, from the numerical solution $\lambda_\pi$, we can predict the structure of $c_v$-geodesics connecting points in $\spt f$ to boundary points. To this aim it is crucial to observe that $\lambda_{\pi}$ charges segments $[b_1,b_3]$, $[b_2, b_4]$. With $x_0$ denoting the centre of the square we can foresee that for an exact solution $v \in \mathrm{Lip}(\Ob;\Rd)$ there holds:
\begin{enumerate}[label={(\roman*)},leftmargin=1.5\parindent]
	\item for $\Ha^1$-a.e. $x \in [x_0,c_i]$ the geodesics may be characterized as polygonal chains
	\begin{equation}
		[x,z] \cup [z, \bar{y}_1(z)] \qquad \text{or} \qquad  [x,z] \cup [z, \bar{y}_2(z)]
	\end{equation}
	where $z$ is an arbitrary element in $[c_i,b_i]$ while $\bar{y}_1(z), \bar{y}_2(z) \in \bO$ are the boundary points uniquely determined for each such $z$ and positioned symmetrically with respect to diagonal ending at $a_i$;
	\item for $\Ha^1$-a.e. $x \in [c_i,b_i]$ the geodesic may be either a single segment or again a polygonal chain:
	\begin{equation}
	[x,\bar{y}_1(x)]\quad  \text{or} \quad [x,\bar{y}_2(x)] \qquad \text{or} \qquad [x,z] \cup [z, \bar{y}_1(z)] \quad \text{or} \quad  [x,z] \cup [z, \bar{y}_2(z)]
	\end{equation}
	where $z$ is an arbitrary element in $]x,b_i]$;
	\item for $\Ha^1$-a.e. $x \in [b_i,a_i]$ the geodesics are segments
	\begin{equation}
		[x,\hat{y}_1(x)]\quad  \text{or} \quad [x,\hat{y}_2(x)]
	\end{equation}
	where $\hat{y}_1(x), \hat{y}_2(x) \in \bO$ are projections of $x$ onto the boundary along direction orthogonal to diagonal ending at $a_i$.
\end{enumerate}
These predictions were confirmed numerically, i.e. by solving a suitable optimal transshipment problem as we did in  Example \ref{ex:pressure_load}. The (presumed) freedom of choosing the geodesics for points in $[x_0,b_i]$ is remarkable.  
\end{example}

\begin{example}[\textbf{Signed load}]
\label{eq:signed_forces}

Still taking $\Sigma_0 = \bO$, we investigate the case of a signed discrete load: $f = \delta_{x_0} - \delta_{x_1} + \delta_{x_2} - \delta_{x_3}+ \delta_{x_4}$, see Fig. \ref{fig:signed_load} where negative point forces are marked by $\otimes$ and numerical solutions $\sigma_{\Pi}$, $\lambda_{\pi}$, $u$ are showed.

\begin{figure}[h]
	\centering
	\subfloat[]{\includegraphics*[trim={0cm 0cm -0cm -0cm},clip,height=0.28\textwidth]{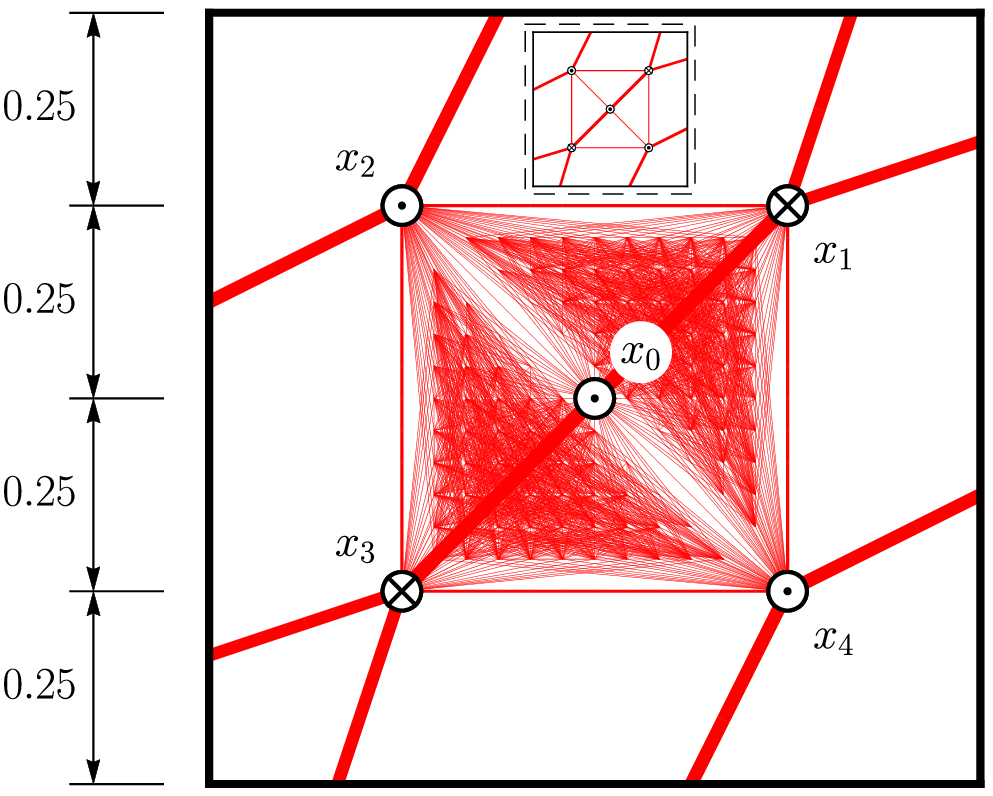}}\hspace{0.4cm}
	\subfloat[]{\includegraphics*[trim={0cm 0cm -0cm -0cm},clip,width=0.28\textwidth]{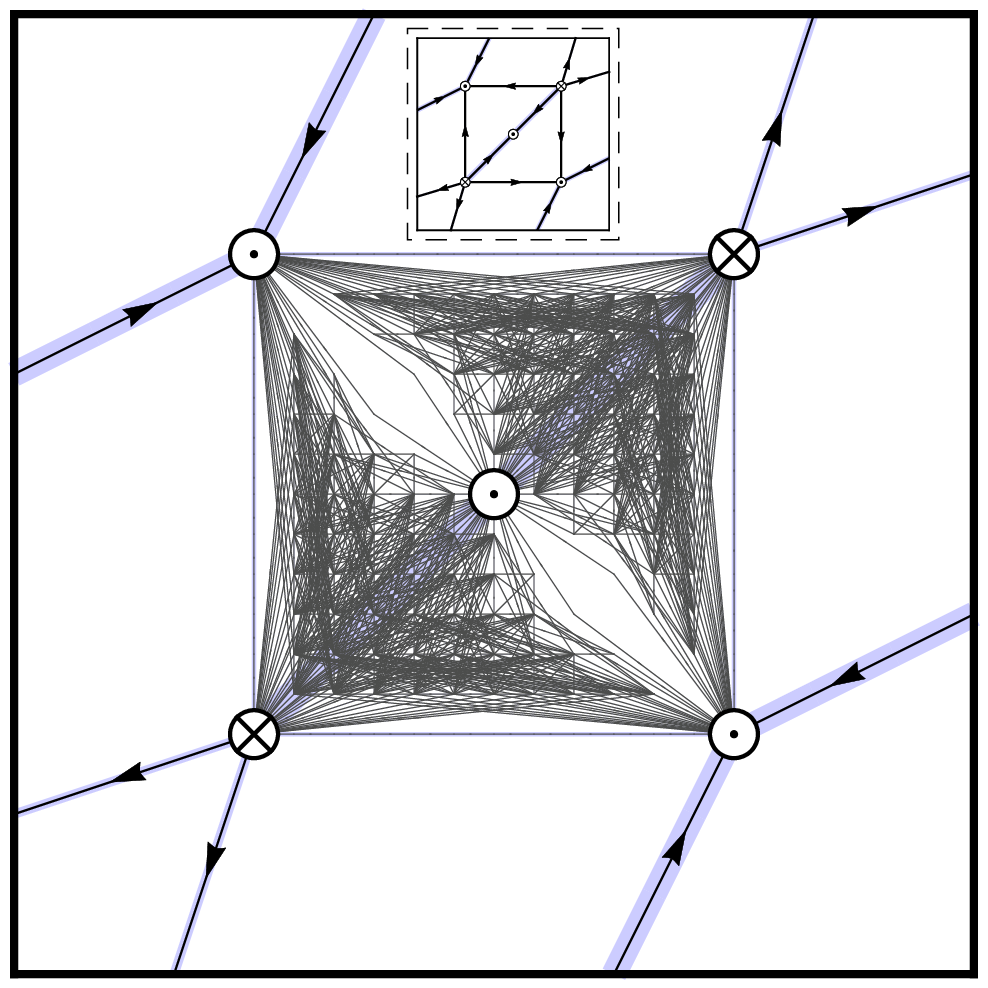}}\hspace{0.4cm}
	\subfloat[]{\includegraphics*[trim={1.cm -1cm 0.8cm -0cm},clip,width=0.28\textwidth]{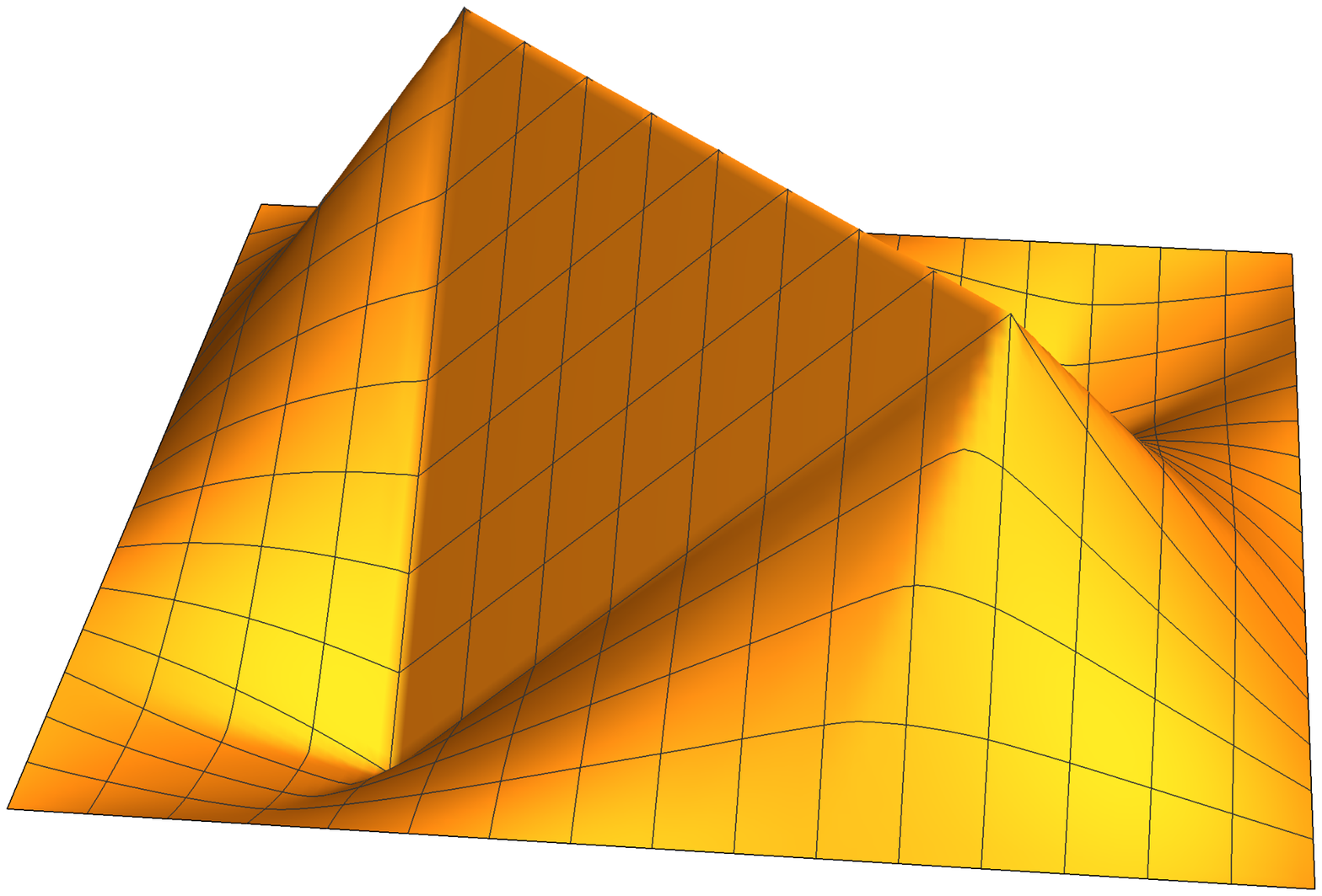}}
	\caption{Numerical prediction of optimal membrane for the signed load: (a) optimal $\sigma_\Pi$ (and an alternative solution); (b) optimal $\lambda_\pi$ (and alternative solution); (c) optimal $u$.}
	\label{fig:signed_load}
\end{figure}	
We can see that  $u$ is affine within two triangles $\mathrm{co}\bigl(\{x_2,x_3,x_4\}\bigr)$ and $\mathrm{co}\bigl(\{x_1,x_2,x_4\}\bigr)$ while $w$ turns out to be affine in the whole central square $\mathrm{co}\bigl(\{x_1,x_2,x_3,x_4\}\bigr)$ where it satisfies $\frac{1}{2}\,\nabla u \otimes \nabla u +e(w) = \Ident$. Accordingly, no pointwise constraint is prescribed for $\sigma_{\Pi}$ in this square and multiple solutions may appear as in the Example \ref{ex:four_forces}. On the top  of Fig. \ref{fig:dirichlet_zone}(a,b), we display miniatures of simplified solutions $\sigma_{\tilde{\Pi}}$ and  $\lambda_{\tilde{\pi}}$ which are obtained by slightly perturbing the algorithm.  Again we observe that $\spt\, (\tilde{\pi},\tilde{\Pi}) \subset (\bO \cup \spt f)^2$.
\end{example}

\begin{example}[\textbf{Eight-point Dirichlet zone}]\label{8points}
Again we consider a uniform pressure load $f = \mathcal{L}^2 \mres \O$ but we set an eight-point Dirichlet zone $\Sigma_0 = \{a_1,\ldots, a_8\}$, see Fig. \ref{fig:dirichlet_zone}(a). Numerical solutions $\sigma_{\Pi}$, $u$ are presented in Fig. \ref{fig:dirichlet_zone}(a,b) respectively while Fig. \ref{fig:dirichlet_zone}(c) shows the component of $w$ along diagonal $[a_4,a_2]$, i.e. $\pairing{w,\tau^{a_4,a_2}}$. The high resolution ($h = \frac1{200}$) allows to predict jump-type discontinuities of exact optimal $w \in BV(\Omega;\Rd)$ along segments $[a_5,a_6]$ and $[a_7,a_8]$, more precisely it is the normal component that is discontinuous (the component  $\pairing{w,\tau^{a_3,a_1}}$ is discontinuous along $[a_5,a_8]$ and $[a_6,a_7]$). Moreover, although in the numerical program $w$ is enforced to be zero on $\bO \cap \mathrm{X}_h$, from Fig. \ref{fig:dirichlet_zone}(c) we see that $w$ admits large values at the nodes close to $\bO$ -- this suggests that the multifunction $\mbf{v} = \mbf{i}^{-1}(\ident-w)$ is multivalued on $\bO$. Notwithstanding this, the graph shown in Fig. \ref{fig:dirichlet_zone}(b) seems to point to Lipschitz continuity for the exact solution $u$. 
	\begin{figure}[h]
		\centering
		\subfloat[]{\includegraphics*[trim={0cm 0cm -0cm -0cm},clip,height=0.28\textwidth]{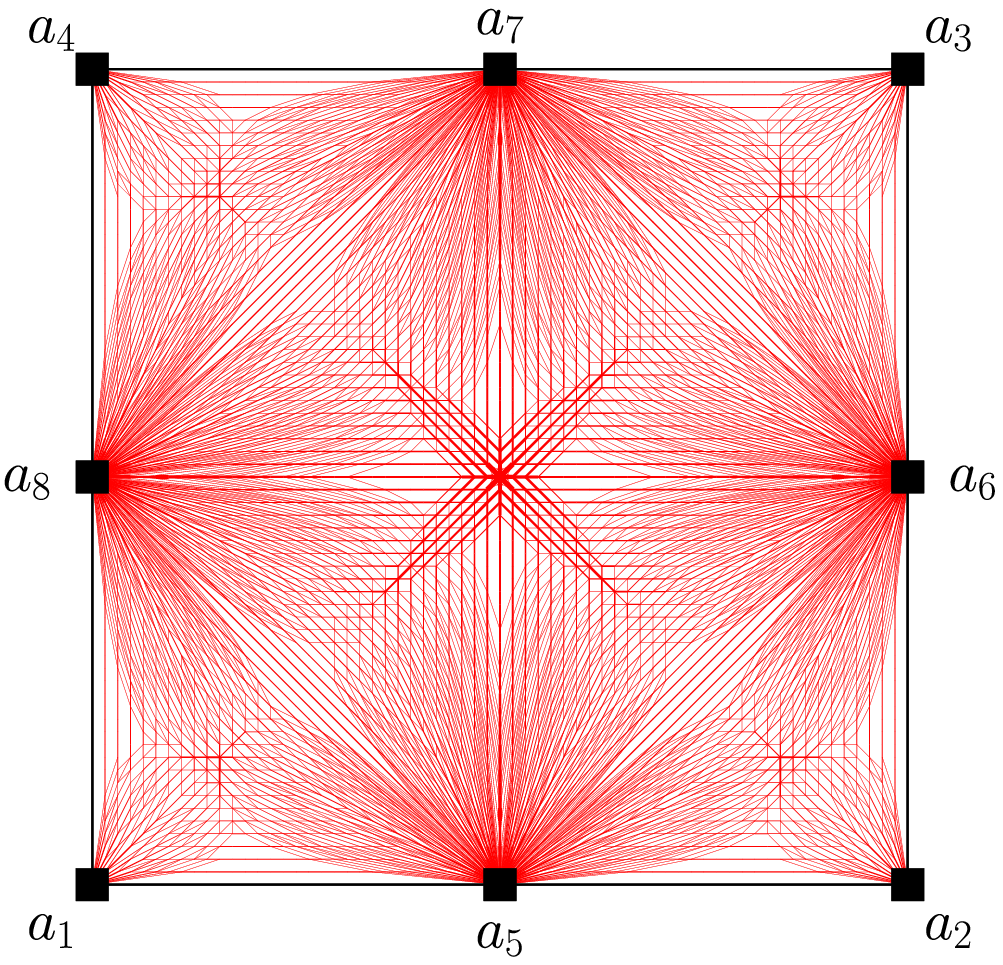}}\hspace{0.4cm}
		\subfloat[]{\includegraphics*[trim={1cm -1cm 1cm -0cm},clip,width=0.28\textwidth]{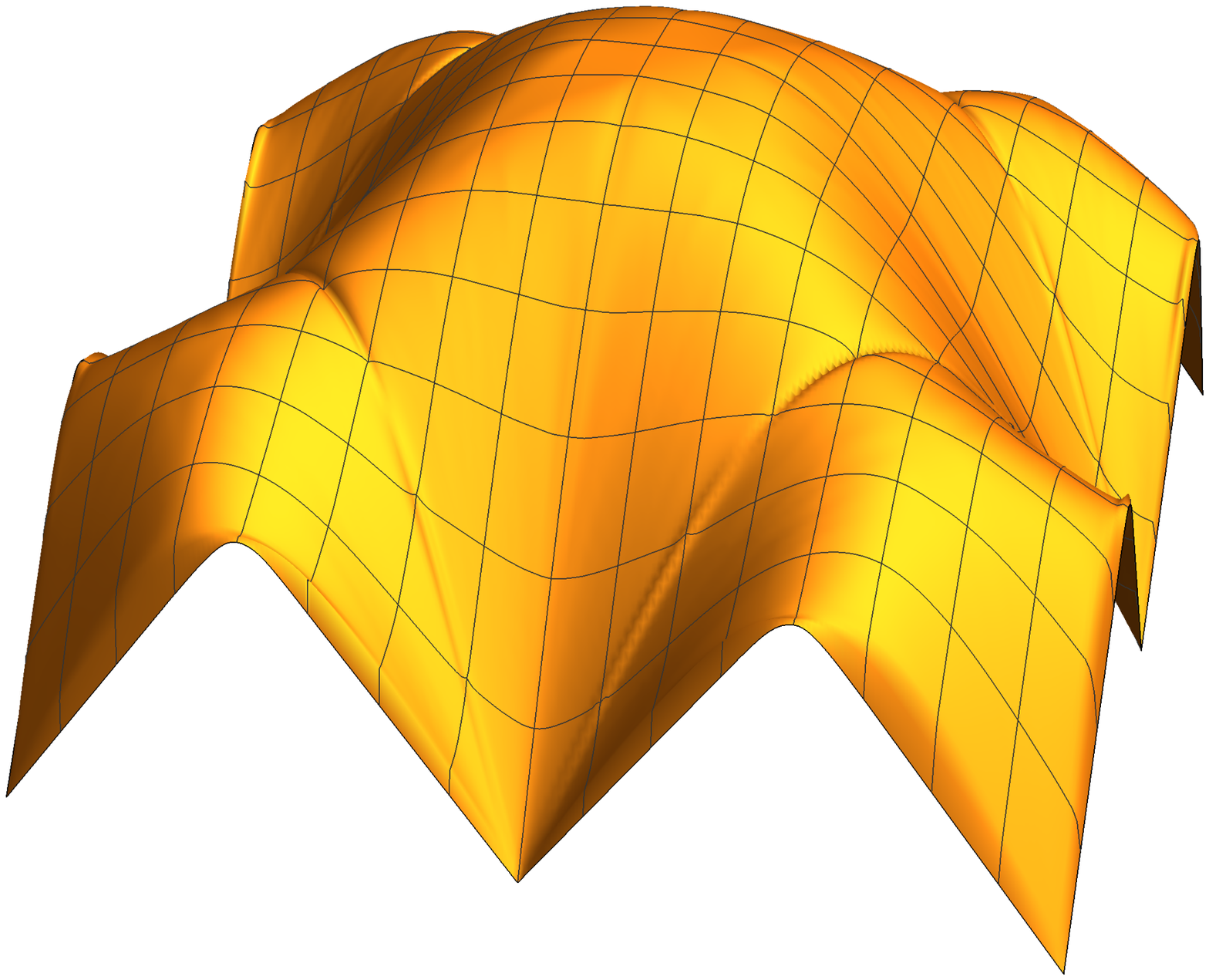}}\hspace{0.4cm}
		\subfloat[]{\includegraphics*[trim={0.7cm 2.2cm 0.5cm -0cm},clip,width=0.28\textwidth]{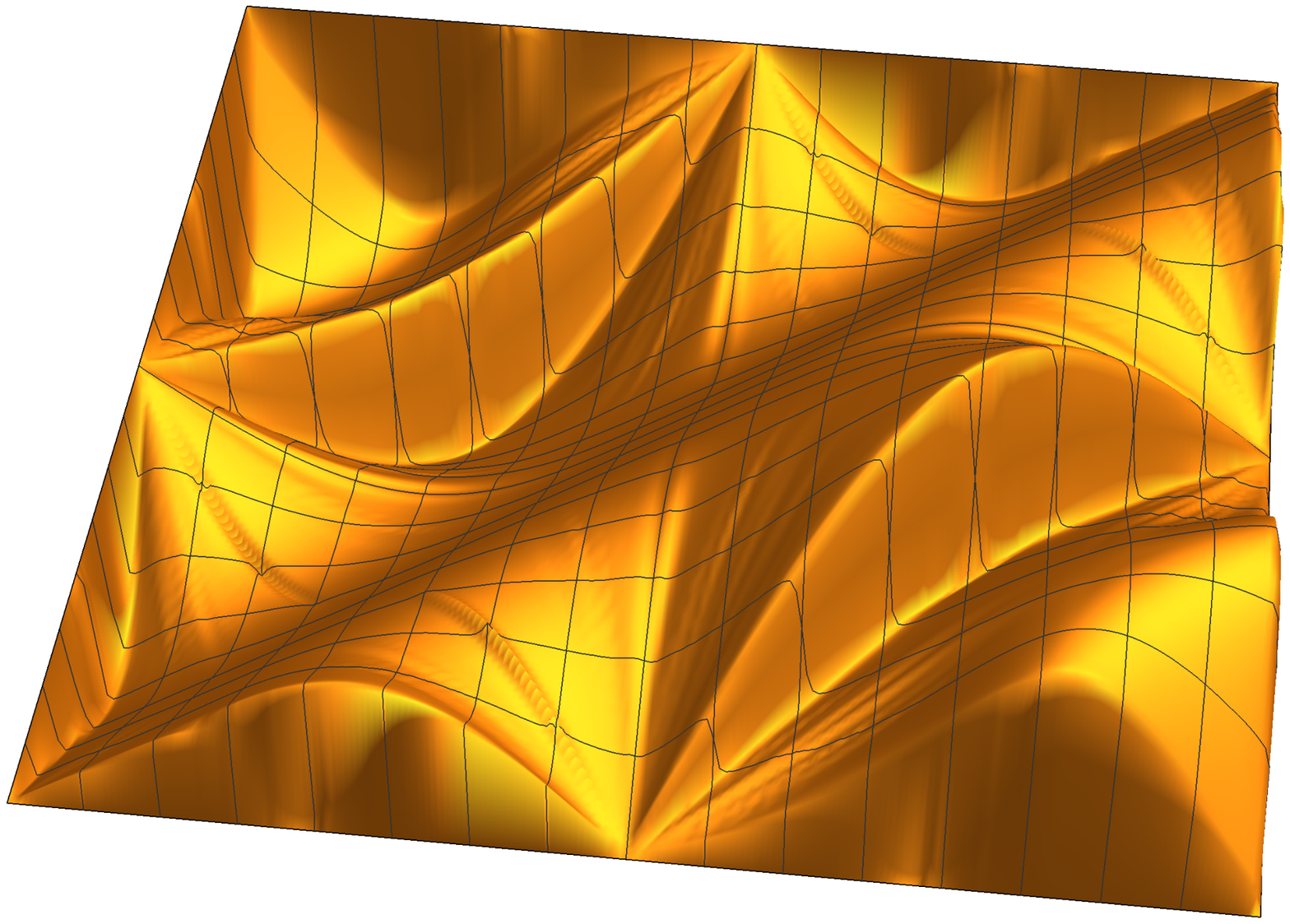}}
		\caption{ Uniform pressure load and  eight-points Dirichlet zone (denoted by solid squares): (a) optimal $\sigma_\Pi$ ; (b) optimal $u$; (c) component of $w$ parallel to diagonal $[a_4,a_2]$.}
		\label{fig:dirichlet_zone}
	\end{figure}
\end{example}

\begin{example}[\textbf{Miscellaneous}]\label{divers}
	We conclude the presentation with another three solutions for the square domain $\Omega$ and $\Sigma_0 = \bO$:  Fig. \ref{fig:miscellaneous}(a,b) show solutions $\sigma_{\Pi}$ for asymmetric positive loads: three point forces $f = \sum_{i=1}^{3} \delta_{x_i}$ in Fig. \ref{fig:miscellaneous}(a) and sum of a point force and load distributed along a line $f = \delta_{x_0} + \Ha^1 \mres [a,b]$ in Fig. \ref{fig:miscellaneous}(b).
	\begin{figure}[h]
		\centering
		\subfloat[]{\includegraphics*[trim={0cm 0cm -0cm -0cm},clip,height=0.28\textwidth]{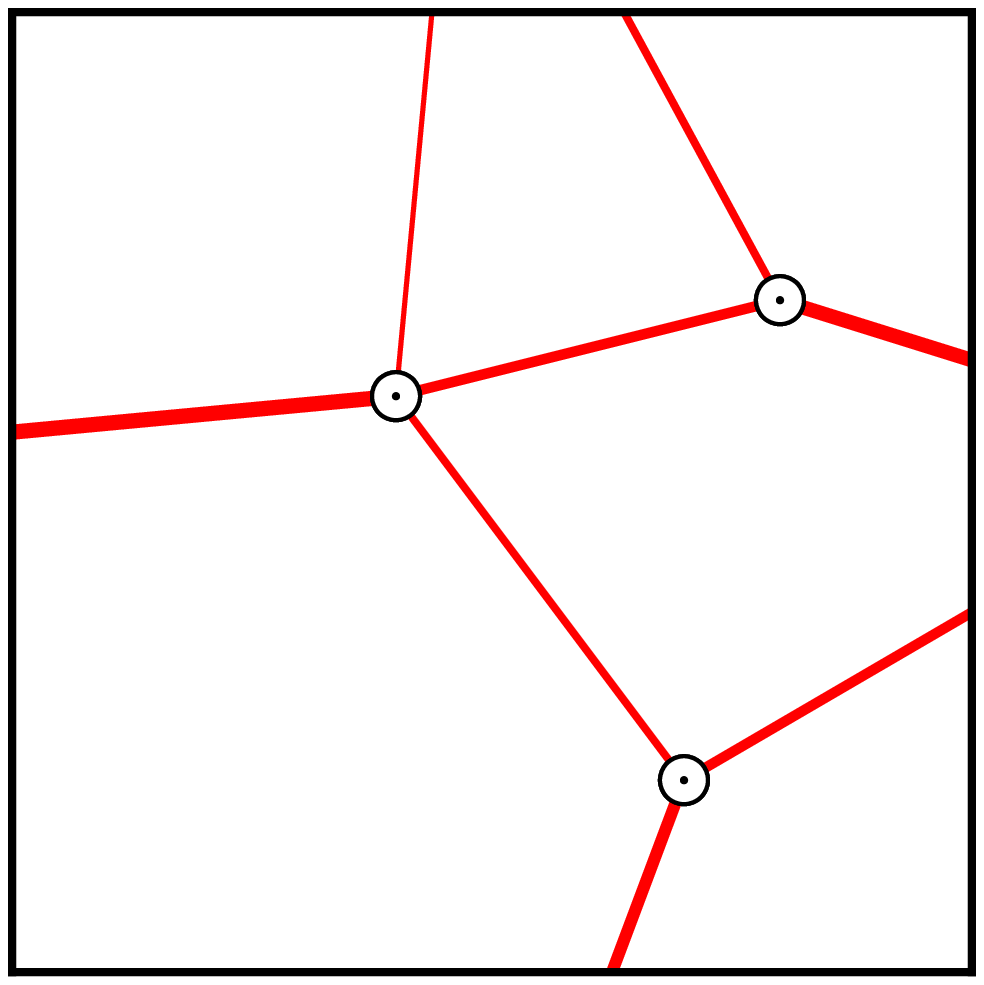}}\hspace{0.4cm}
		\subfloat[]{\includegraphics*[trim={0cm 0cm -0cm -0cm},clip,width=0.28\textwidth]{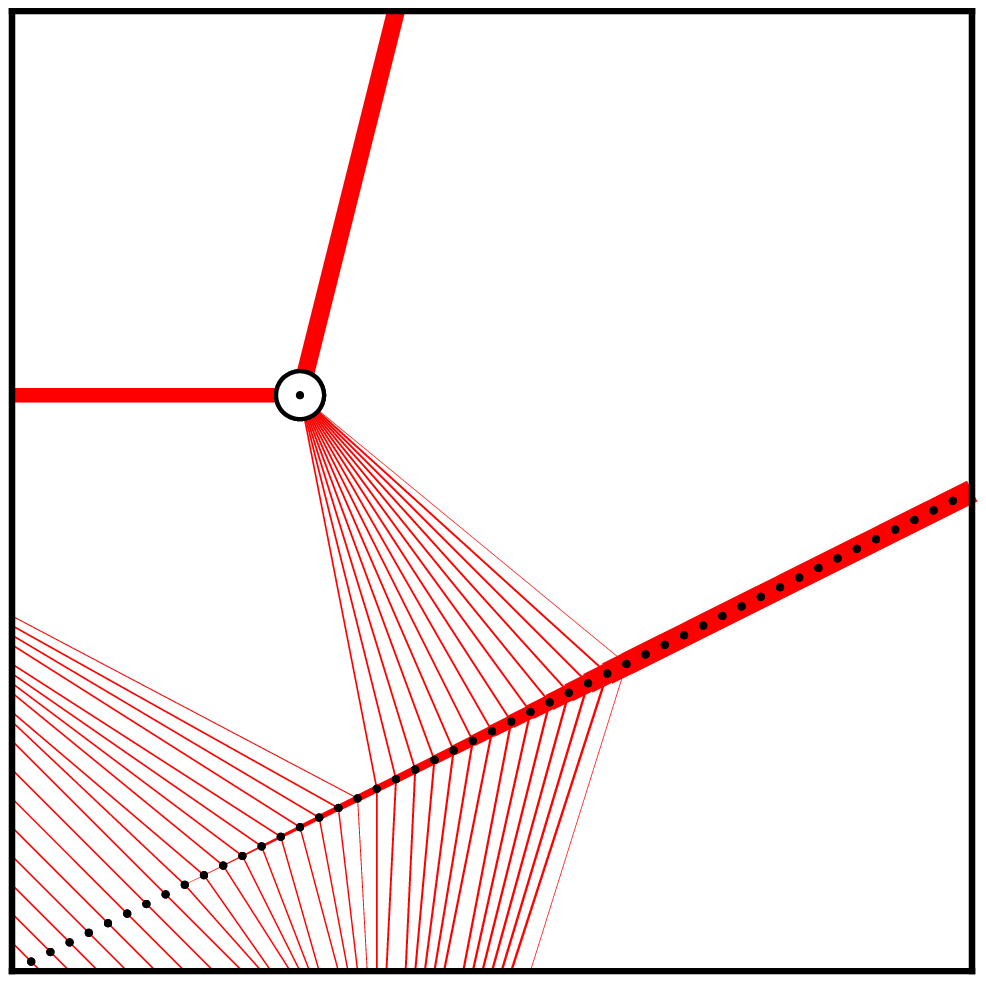}}\hspace{0.4cm}
		\subfloat[]{\includegraphics*[trim={0cm 0cm -0cm -0cm},clip,height=0.28\textwidth]{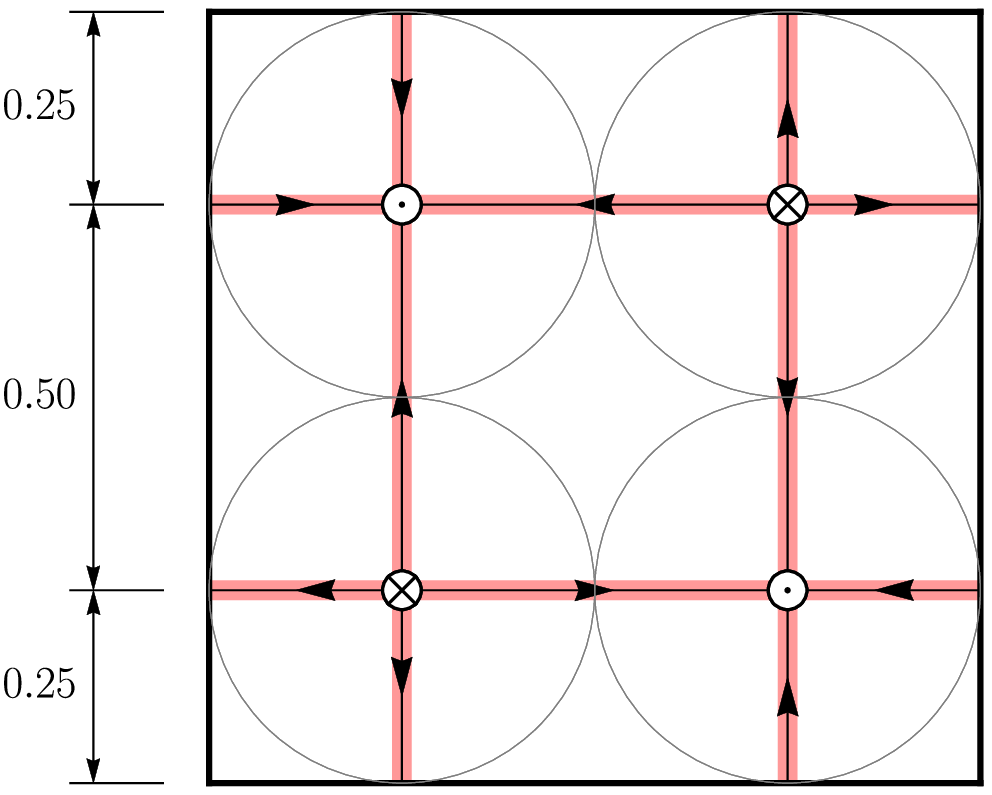}}
		\caption{ (a) optimal $\sigma_\Pi$ for three asymmetric point forces; (b) optimal $\sigma_\Pi$ for point force and force distributed along a line; (c) optimal $\sigma_\Pi$ and $\lambda_{\pi}$ for a signed load.}
		\label{fig:miscellaneous}
	\end{figure}
	Position of the signed load $f = \sum_{i=1}^{2} \delta_{x_i} - \sum_{i=3}^{4} \delta_{x_i}$ in Fig. \ref{fig:miscellaneous}(c) was strategically picked: the solution $w$ turns out to be zero implying that obtained $\sigma_{\Pi}$ solves the (FMD) problem while $\gamma = -\pi$ (Fig. \ref{fig:miscellaneous}(c) represents both optimal $\sigma_{\Pi}$ and $\lambda_{\pi}$, note the arrows) solves the optimal transport problem for the Euclidean distance. This implies that Corollary \ref{fmd<om} cannot be generalized to signed loads $f$. Solution from Fig. \ref{fig:miscellaneous}(c) may be proved to be exact by a simple adaptation of the proof of Proposition \ref{one-force_construction}: the graph of function $u$ can be constructed as an extension by zero of four cones of revolution with disks marked in Fig. \ref{fig:miscellaneous}(c) as bases and vertices $\bigl(x_i, \mathrm{sign}\bigl(f(x_i)\bigr) \sqrt{2}/4  \bigr)$. The key point is that in-flows and out-flows of $\lambda_{\pi}$ are matched at points where the pairs of disks touch and thus they play the role of additional points in the Dirichlet zone.
\end{example}

To conclude this section, we stress that in all the examples, a numerical solution $(\pi,\Pi)$ whose support is contained in $(\bO \cup \spt f)^2$ was found. In fact, the authors did not find a single counter-example to this rule, in particular in the case of a discrete load where an exact solution $(\pi,\Pi)$
should exist  according to Conjecture \ref{finiteconj}.
On the other hand, it appears that if  $\Omega$ is a convex domain and $\Sigma_0$ is the whole  boundary $\bO$, then 
 all  numerical predictions suggest that the exact optimal pair $(u,w)$ is Lipschitz continuous vanishing on the boundary. In contrast,
 if $\Sigma_0$ is a discrete subset of $\bO$ as it is in Example \ref{8points}, we have a strong indication that the exact solution $w$ could  exhibit discontinuities (that is $w\in BV\setminus W^{1,1}$).

\appendix

\section{Convex analysis}\label{secapp}


\subsection*{ Conjugate and biconjugate}

Let $X$ be a Banach space and $h:X\to \R \cup\{+\infty\}$ a convex function with non-empty domain $\mathrm{dom}(h)=\big\{x\in X : h(x)<+\infty\big\}.$
The conjugates $h^*: X^* \to \R \cup\{+\infty\}$ and $h^{**}:X\to \R \cup\{+\infty\}$ are defined respectively by
$$  h^*(x^*) = \sup_{x\in X}  \Big\{ \pairing{x,x^*} - h(x)\Big\},\qquad  h^{**}(x) = \sup_{x^*\in X^*}  \Big\{ \pairing{x,x^*} - h^*(x^*) \Big\}.$$

\begin{lemma}\label{pertu}\ Let $h:X\to \R \cup\{+\infty\}$ be a convex function such that $h(0)<+\infty$ and $h$ is lower semicontinuous
at $0$. Then it holds that $h(0) = h^{**}(0)= - \inf h^*$.
 Assume in addition that $h$ is continuous at $x=0$, then $h^*$ reaches its minimum on $X^*$. 

\end{lemma}
\begin{proof} See for instance \cite{bouchitte2006} or  \cite[Thm I-12]{castaing}.  \end{proof}

\subsection*{ Minimax Theorem}

\begin{theorem}\label{ky-Fan} \ (\textit{Ky-Fan})\quad For X,Y being two topological vector spaces, let  
 $\A\subset X$, $\B\subset Y$ be two non-empty convex subsets and let $\LL : \A\times\B \to \R$
 be a convex-concave Lagrangian (i.e. $u\in \A \mapsto \LL(u,v)$ is convex $\forall v\in \B$ and 
 $v\in \B \mapsto \LL(u,v)$ is concave  $\forall u\in \B$).
 Assume the followings:
 \begin{itemize}[leftmargin=1.5\parindent]
\item [(i)] \ $\A$ is a compact subset of $X$,
\item  [(ii)]\ for every $v\in\B$, the map $u\in \A \mapsto \LL(u,v)$ is lower semicontinuous.
\end{itemize}
Then the following equality holds (in $\R \cup \{+\infty\}$):
$$ \min_{u\in \A} \ \sup_{v\in\B} \ \LL(u,v) \quad =\quad \sup_{v\in\B} \  \min_{u\in \A} \ \LL(u,v).$$
Assume in addition that $\B$ is a convex compact subset of $Y$. Then $\LL$ admits 
 a saddle point $(\bar u,\bar v)\in \A\times\B$, that is:
 $$   \LL(\bar u, v) \ \le\   \LL(\bar u,\bar v)\ \le\  \LL(u,\bar v)\qquad \forall u\in \A \ ,\
 \forall v\in \B. $$

\end{theorem}

 \begin{proof} See for instance \cite[Theorem 4.36, p.73]{clarke-calvar} and \cite[Chap VI]{ekeland1999} for the saddle point statement. \end{proof}

\section{Tangential calculus with respect to a measure}
\label{appendix_mu_calculus}

For more details on this theory we refer to \cite{BBS,JFAmu,survey}
and to \cite{BCJ} for the specific case of Lipschitz functions. In what follows 
$ \mathrm{Lip}(\Ob)$ will be embedded with the weak* topology of $W^{1,\infty}(\O)$, which amounts to saying that $\f_n\, \weakstar\, \f$ in $\mathrm{Lip}(\Ob)$
if and only if $\f_n\to \f$ uniformly in $\Ob$ while  $\{\f_n\} $ is equi-Lipschitz.
Let $\la\in \Mes(\Ob;\R^d)$ and consider a  decomposition  $\lambda=\xi\, \mu$
with $\mu \in \Mes_+(\Ob)$ such that $\mu(\Ob)<+\infty$ and $\xi\in L^1_{\mu}(\Ob;\R^d)$.
We notice that if $\dive\, \la\in\Mes(\Ob)$ (the distributional divergence is intended in whole $\R^d$), then
 for any  sequence    
$(\f_n)$ in $C^{\infty}(\Ob)$ such that 
$ \f_n\,\weakstar\, 0$ in $\mathrm{Lip}(\Ob)$  one has
$\langle -\dive\,\lambda, \varphi_n\rangle\;=\;\langle \lambda, \nabla\varphi_n\rangle=\int_\Ob \sigma\cdot\nabla \varphi_n\ d\mu\;\to\; 0.$
This implies that  $\int_\Ob \pairing{\xi, \zeta}\, d\mu=0$ holds for every $\zeta$ in the following set:
\begin{align*}
\NN:=\bigg\{ \zeta\in L^\infty_{\mu}(\Ob;\R^d):\ \exists (u_n)_n, \ \ u_n\in C^{\infty}(\O),
\quad  u_n \rightarrow 0\ \text{ uniformly}, \ \ \nabla u_n\,\weakstar\, \zeta \  \ \text{in } (L^\infty_\mu)^d \bigg\}.
\end{align*} 
The orthogonal complement of $\NN$ in $L^1_{\mu}(\Ob;\R^d)$ defined by
$$\NN^\perp \ :=\ \left\{\eta \in  L^1_{\mu}(\Ob;\R^d)\ :\ \int_\Ob \pairing{\eta,\zeta} \, d\mu  = 0 \quad \hbox{for all $\zeta \in \NN $} \right\}$$
is a closed vector subspace of $L^1_{\mu}(\Ob;\R^d)$ which is stable by the multiplication by smooth scalar functions. 
Following \cite{JFAmu,BCJ} we define the
 tangent space $T_{\mu}$ to the measure $\mu$ through the following local characterization of $\NN^{\perp}$:
%
\begin{proposition} \label{Tmu} The following statements hold true:
\begin{itemize}[leftmargin=1.5\parindent]
	\item[(i)] There exists a $\mu$-measurable multifunction $T_{\mu}$ from $\Ob$ to the  linear subspaces of $\R^{d}$ such that:
	$$\eta\in \NN^\perp\quad \Longleftrightarrow\quad \eta(x)\in T_{\mu}(x) \qquad \text{for $\mu$-a.e.}\ x\in\R^d;$$ 
	\item[(ii)] The linear operator $ u\in C^1(\Ob) \mapsto  P_\mu(x) \nabla u (x) \in L^\infty_\mu(\Ob;\R^d)$, where $P_\mu(x)$ denotes the orthogonal projector onto $T_\mu(x)$, can be uniquely extended to a linear continuous operator 
	$$\nabla_\mu: u \in \mathrm{Lip}(\Ob) \  \mapsto \ \nabla_\mu u \in L^\infty_\mu(\Ob;\R^d), $$ 
	$ \mathrm{Lip}(\Ob)$ and  $ L^\infty_\mu(\Ob;\R^d)$ being equipped with their  weak* topology;
	\item[(iii)] The linear operator $ w\in C^1(\Ob;\R^d) \mapsto  P_\mu(x)\, e(w)\,  P_\mu(x) \in L^\infty_\mu(\Ob;\Sdd)$   extends in a unique way to a linear  continuous operator 
	$$e_\mu: w \in \mathrm{Lip}(\Ob;\R^d) \ \mapsto \ e_\mu(w) \in L^\infty_\mu(\Ob;\Sdd), $$ 
	$ \mathrm{Lip}(\Ob;\Rd)$ and $ L^\infty_\mu(\Ob;\Sdd)$ being equipped with the weak star topology.
\end{itemize}
\end{proposition}

\begin{remark}\label{k-manifold}  By virtue of the second assertion in Proposition \ref{Tmu} any Lipschitz function admits,
for every measure $\mu$, a $\mu$-a.e. defined tangential gradient $\nabla_\mu u$. In the case where $\mu$ is the $k$-dimensional
Hausdorff measure restricted to a smooth $k$-dimensional manifold in $\R^d$, this tangential gradient coincides with the one 
which is obtained  by using  Rademacher theorem on local charts representing the manifold.
If $\mu$ is a discrete measure, then $T_\mu(x)=\{0\}$ and the tangential gradient vanishes.
If $\mu$ is a multidimensional measure of the kind $\mu= \sum_i \mu_i$, where the $\mu_i$'s are mutually singular and $\mu_i=\Ha^{k_i} \mres S_i$, 
$S_i$ being a smooth $k_i$-dimensional manifold in $\R^d$ ($1\le k_i\le d$), then $T_\mu(x)=T_{\mu_i} (x)$ and $\nabla_\mu u
= \nabla_{\mu_i} u$ \ $\mu_i$-a.e. for any $u\in \mathrm{Lip}(\Ob)$.
\end{remark}

\bigskip
In view of Proposition \ref{Tmu} we may define in an intrinsic way the set of {\em tangential} vector measures
$$ \Mes_T(\Ob;\R^d) \ :=\ \Big\{ \la = \xi \,\mu \ : \  \mu\in \Mes_+(\Ob), \ \ \xi(x) \in T_\mu(x)  \ \  \mu\text{-a.e}\Big\}.$$
It can be shown that the property $\la\in \Mes_T(\Ob;\R^d)$ is independent of the chosen decomposition $\la=\xi\mu$ (see for instance \cite{BBS}).

\begin{remark}\ If $\xi\in L^1(\Omega;\R^d)$, then the measure
$\xi\, \LL^d \mres \O$ is an element of $\Mes_T(\Ob;\R^d)$ since 
$T_{\LL^d}(x)=\R^d\-$ a.e. in $\O$. On the other hand, if  $\la\in \Mes_T(\Ob;\R^d)$, the condition
$\frac{d\la}{d|\la|} \in T_{|\la|}(x)$ implies  $|\la|$-a.e that ${\rm dim} \big(T_{|\la|}(x)\big) \ge 1 \ \ |\la|$-a.e. 
As a consequence, elements of  $\Mes_T(\Ob;\R^d)$ are atomless.
\end{remark}

We are now in position to give the desired integration by parts formulae for tangential vector (resp. symmetric tensor) measures 
\begin{proposition} \label{byparts} The following statements hold true:
\begin{itemize}[leftmargin=1.5\parindent]
	\item[(i)] Let $\la\in \Mes(\Ob;\R^d)$ be such that  $-\dive\, \la \in
	\Mes(\Ob)$. Then  $\la\in \Mes_T(\Ob;\R^d)$ and for any decomposition $\la=\xi\, \mu$  we have 
	$$ \langle - \dive\,\la, u\rangle \ =\ \int_\Ob \pairing{\xi,\nabla_\mu u} \, d\mu\qquad \forall\, u\in \mathrm{Lip}(\Ob).$$
	\item[(ii)] Let $\sigma\in \Mes(\Ob;\Sdd)$ be such that  $-\DIV\, \sigma =0$ in $\O$ and let $\sigma=S\mu$  be a decomposition of $\sigma$.
	Then it holds that $P_\mu(x)\, S \, P_\mu(x)= S$ \ $\mu$-a.e. and   we have
	$$ \int_\Ob \pairing{S , e_\mu (w)} \, d\mu\ =\ 0   \qquad \forall w\,\in \mathrm{Lip}_0(\O;\R^d). $$
\end{itemize}
\end{proposition}
\begin{proof} For the assertion (i) we refer to \cite[Prop. 3.5]{BCJ}. To show (ii) we consider a localization function $\theta_\d \in \D(\O; [0,1])$
such that $\theta_\d = 1 $ on $\O_\d:=\{x\in\O : d(x, \bO) > \delta\}$. Then, the measure $\tilde \sigma:=\theta_\d \sigma$ satisfies
 $\DIV\, \tilde \sigma=
S \, \nabla \theta_\d \, \mu \in \Mes(\Ob;\R^d)$, hence all its rows or columns are tangential measures.  It follows that the rows and columns of $S$
belong to $T_\mu(x)$, hence $ P_\mu(x)\, S\, P_\mu(x)= S$ \ $\mu$-a.e. in $\O_\d$. Eventually let $w\in  \mathrm{Lip}_0(\O;\R^d)$ and let us  prove that
 $\int_\Ob \pairing{S , e_\mu (w)} \, d\mu = 0 $.  As $w$ can be approximated by a sequence $(w_n)\in \mathrm{Lip}(\O;\R^d)$ such that 
 $w_n$ is compactly supported in $\O$ while $w_n\,\weakstar\, w$ and  $\int_\Ob \pairing{S , e_\mu (w_n)} \, d\mu \to
\int_\Ob \pairing{S , e_\mu (w)} \, d\mu  $  (see assertion (iii) in Proposition \ref{Tmu}), it is not restrictive to assume that
$w$ is itself compactly supported in $\O$, hence in $\O_\d$ for $\d$ small enough. 
 Then, since $\tilde \sigma =\sigma= S\, \mu$  and $\DIV\, \tilde \sigma=0$ in $\O_\d$, by applying the integration by parts formula from assertion (i), we infer that $ 0 =  \langle - \DIV\, \tilde \sigma  , w\rangle \  =\ \int_\Ob \pairing{S,e_\mu (w)} \, d\mu. $
\end{proof}

\null

\section{Mollifications of convex functions of measures }\ Let $h: \R^m\to [0,+\infty]$ be a convex, l.s.c. and positively one homogeneous integrand
and let $\chi\in \Mes(\R^d;\R^m)$
 to which we associate the scalar non-negative Borel measure $h(\chi)$ on $\R^d$ (see  \cite{Goffman}).
 Let us apply to $\chi$  a smooth convolution kernel $\theta_\e(x)= \e^{-d} \theta\big(\frac{x}{d}\big)$ where  $\e>0$ and $\theta$ is a radial symmetric element of $\D^+(\R^d)$ such that $\int \theta =1$.
 \begin{lemma} \label{mollih}
 We have $h(\chi * \theta_\e) \le h(\chi)$ for every $\e>0$ and $\lim_{\e\to 0} \int_A h(\chi * \theta_\e) = \int_A h(\chi)$ (possibly equal to $+\infty$)
 for every open subset $A\subset \R^d$.
 \end{lemma}
 \begin{proof} The integrand $h$ is the support function of a closed convex subset $K\subset \R^m$ such that $0\in K$. 
 Let $A$ be an open subset of $\R^d$ and let $\zeta\in C_0(A;\R^m)$ be such that $\spt(\zeta)\subset A$ and $\zeta(x)\in K$ for all $x\in A$.
 Then, for small $\e>0$, the convolution $\zeta*\theta_\e$ satisfies the same properties as listed above for $\zeta$. Hence, since $\theta_\e$ is symmetric and inequality $ \big\langle \zeta*\theta_\e, \frac{d \chi}{d |\chi|} \big\rangle \le  \one_A\,  h(\frac{d \chi}{d |\chi|})$ holds, we get
 $$ \pairing{ \chi*\theta_\e, \zeta}\ =\  \pairing{ \chi, \zeta*\theta_\e} \le \int_A h(\chi).$$
 Passing to the supremum with respect to all $\zeta$, we infer that $\int_A h(\chi*\theta_\e)\le \int_A h(\chi)$. This is true for every open subset $A$
 and the first claim of Lemma \ref{mollih} follows. On the other hand, as it holds that $\chi*\theta_\e \,\weakstar\, \chi$ in $\Mes(A; \R^m)$, we have the classical lower semicontinuity property
 $ \liminf_{\e\to 0} \ \int_A  h(\chi*\theta_\e) \ \ge \   \int_A  h(\chi),$
 from which we obtain the desired convergence.
  
 \end{proof}
%
%
%
%
%

\bibliographystyle{plain}      

\bigskip\small\noindent
Karol Bo{\l}botowski:
Department of Structural Mechanics and Computer Aided Engineering \\
 Faculty of Civil Engineering, Warsaw University of Technology\\
   16 Armii Ludowej Street, 00-637 Warsaw - POLAND\\
{\tt k.bolbotowski@il.pw.edu.pl}

\bigskip
{\small\noindent
Guy Bouchitt\'e:
Laboratoire IMATH, Universit\'e de Toulon\\
BP 20132, 83957 La Garde Cedex - FRANCE\\
{\tt bouchitte@univ-tln.fr}

\end{document}